\documentclass{memo-l}

\usepackage{mathabx} 

\usepackage{graphicx}   
\usepackage{svg}        
\usepackage{tikz-cd}    
\tikzcdset{arrow style=math font}
\usepackage[all]{xy}    

\usepackage{enumitem}   
\usepackage{comment}    
\usepackage{pifont}     
\usepackage{footnote}
\usepackage{multirow}   
\usepackage{etoolbox}   
\usepackage[T1]{fontenc}
\usepackage{hyperref}
\hypersetup{
	colorlinks=true,
	breaklinks,
	linkcolor=blue,
	filecolor=magenta,
	urlcolor=blue,
	citecolor=magenta,
	linktoc=all,
}
\usepackage{aliascnt}
\usepackage[nameinlink]{cleveref}

\numberwithin{equation}{section} 

\AtBeginEnvironment{figure}{\refstepcounter{equation}}
\AtBeginEnvironment{table}{\refstepcounter{equation}}

\Crefname{figure}{Figure}{Figures}
\crefname{figure}{fig.}{Fig.}
\Crefname{table}{Table}{Tables}

\newtheorem{thmA}{Theorem}
\Crefname{thmA}{Theorem}{Theorems}

\theoremstyle{plain}
\newaliascnt{thm}{equation}
\newtheorem{thm}[thm]{Theorem}
\aliascntresetthe{thm}
\Crefname{thm}{Theorem}{Theorems}
\newaliascnt{lemma}{equation}
\newtheorem{lemma}[lemma]{Lemma}
\aliascntresetthe{lemma}
\Crefname{lemma}{Lemma}{Lemmas}
\newaliascnt{prop}{equation}
\newtheorem{prop}[prop]{Proposition}
\aliascntresetthe{prop}
\Crefname{prop}{Proposition}{Propositions}
\newaliascnt{cor}{equation}
\newtheorem{cor}[cor]{Corollary}
\aliascntresetthe{cor}
\crefname{cor}{corollary}{corollaries}
\Crefname{cor}{Corollary}{Corollaries}

\theoremstyle{definition}
\newaliascnt{definition}{equation}
\newtheorem{definition}[definition]{Definition}
\aliascntresetthe{definition}
\Crefname{definition}{Definition}{Definitions}
\newaliascnt{rem}{equation}
\newtheorem{rem}[rem]{Remark}
\aliascntresetthe{rem}
\Crefname{rem}{Remark}{Remarks}
\newaliascnt{example}{equation}
\newtheorem{example}[example]{Example}
\aliascntresetthe{example}
\Crefname{example}{Example}{Examples}
\newaliascnt{question}{equation}

\aliascntresetthe{question}
\Crefname{question}{Question}{Questions}
\newaliascnt{notation}{equation}
\newtheorem{notation}[notation]{Notation}
\aliascntresetthe{notation}
\crefname{notation}{notation}{notations}
\Crefname{notation}{Notation}{Notations}
\newaliascnt{block}{equation}
\newtheorem{block}[block]{}
\aliascntresetthe{block}

\crefformat{enumi}{#2\textup{#1}#3} 

\newcounter{dummy}
\renewcommand{\thedummy}{\roman{dummy}}
\crefformat{dummy}{#2(#1)#3}
\crefrangeformat{dummy}{(#3#1#4-#5#2#6)}
\newenvironment{blist}{
	\begin{list}{(\thedummy)}{
			\setlength\labelsep{4pt}
			\setlength\itemindent{4pt}
			\setlength\leftmargin{0pt}
			\setlength\labelwidth{0pt}
			\usecounter{dummy}
		}
	}{\end{list}}

\newcommand{\C}{\mathbb{C}}
\newcommand{\CP}{\mathbb{CP}}
\newcommand{\N}{\mathbb{N}}
\newcommand{\Z}{\mathbb{Z}}
\newcommand{\R}{\mathbb{R}}
\newcommand{\Q}{\mathbb{Q}}
\newcommand{\D}{\mathbb{D}}

\newcommand{\W}{\mathcal{W}}

\newcommand{\V}{\mathcal{V}}
\newcommand{\A}{\mathcal{A}}
\newcommand{\E}{\mathcal{E}}

\newcommand{\Cinf}{\mathcal{C}^\infty}

\renewcommand{\O}{\mathcal{O}}

\DeclareMathOperator{\rk}{rk}

\DeclareMathOperator{\Jac}{Jac}

\renewcommand{\Re}{\mathop{\mathrm{Re}}}
\renewcommand{\Im}{\mathop{\mathrm{Im}}}

\DeclareMathOperator{\supp}{supp}

\DeclareMathOperator{\wt}{wt}

\renewcommand{\mod}{\mathop{\mathrm{mod}}\nolimits}
\DeclareMathOperator{\id}{id}

\DeclareMathOperator{\GL}{GL}

\DeclareMathOperator{\init}{in}

\newcommand{\Ndmin}{\mathcal{N}_{\mathrm{min}}}

\newcommand{\Mat}{\mathop{\rm Mat}\nolimits}

\newcommand{\Yro}{Y^{\mathrm{ro}}}
\newcommand{\Ymin}{Y_{\mathrm{min}}}

\newcommand{\degmin}{\deg_{\mathrm{min}}}
\newcommand{\Cmin}{C_{\mathrm{min}}}

\newcommand{\Dmin}{D_{\mathrm{min}}}
\newcommand{\Pmin}{P_{\mathrm{min}}}

\newcommand{\Ypol}{Y_{\mathrm{pol}}}
\newcommand{\Ypolte}{Y_{\mathrm{pol},\theta}}
\newcommand{\Yropol}{Y^{\mathrm{ro}}_{\mathrm{pol}}}
\newcommand{\Yropolte}{Y^{\mathrm{ro}}_{\mathrm{pol},\theta}}

\newcommand{\Droup}{D^{\mathrm{ro}}_{\Upsilon}}
\newcommand{\Cpol}{C_{\mathrm{pol}}}

\newcommand{\Dpol}{D_{\mathrm{pol}}}
\newcommand{\Ppol}{P_{\mathrm{pol}}}
\newcommand{\Gpol}{G_{\mathrm{pol}}}
\newcommand{\Gapol}{\Gamma_{\mathrm{pol}}}
\newcommand{\Gamin}{\Gamma_{\mathrm{min}}}
\newcommand{\Vmin}{\mathcal{V}_{\mathrm{min}}}
\newcommand{\Vpol}{\mathcal{V}_{\mathrm{pol}}}

\newcommand{\pimin}{\pi_{\mathrm{min}}}
\newcommand{\pipol}{\pi_{\mathrm{pol}}}
\newcommand{\piropol}{\pi^{\mathrm{ro}}_{\mathrm{pol}}}
\newcommand{\xipol}{\xi_{\mathrm{pol}}}
\newcommand{\xiropol}{\xi^{\mathrm{ro}}_{\mathrm{pol}}}
\newcommand{\garopol}{\gamma^{\mathrm{ro}}_{\mathrm{pol}}}
\newcommand{\gapol}{\gamma_{\mathrm{pol}}}
\newcommand{\Aropol}{A^{\mathrm{ro}}_{\mathrm{pol}}}
\newcommand{\Apol}{A_{\mathrm{pol}}}
\newcommand{\hxiro}{\hat{\xi}^{\mathrm{ro}}}

\newcommand{\Siro}{\Sigma^{\mathrm{ro}}}
\newcommand{\tSiro}{\tilde\Sigma^{\mathrm{ro}}}
\newcommand{\tSro}{\tilde{S}^{\mathrm{ro}}}

\newcommand{\Sro}{S^{\mathrm{ro}}}
\newcommand{\Sroup}{S^{\mathrm{ro}}_{\Upsilon}}
\newcommand{\Sroupte}{S^{\mathrm{ro}}_{\Upsilon,\theta}}
\newcommand{\Uro}{U^{\mathrm{ro}}}

\newcommand{\Cro}{\C^{\mathrm{ro}}}
\newcommand{\tUro}{\tilde U^{\mathrm{ro}}}
\newcommand{\hUro}{\hat U^{\mathrm{ro}}}
\newcommand{\gro}{g^{\mathrm{ro}}}
\newcommand{\xiro}{\xi^{\mathrm{ro}}}
\newcommand{\karo}{\kappa^{\mathrm{ro}}}
\newcommand{\phiro}{\phi^{\mathrm{ro}}}
\newcommand{\txiro}{\tilde{\xi}^{\mathrm{ro}}}
\newcommand{\xiror}{\xi^{\mathrm{ro},r}}
\newcommand{\xiros}{\xi^{\mathrm{ro},s}}
\newcommand{\xirot}{\xi^{\mathrm{ro},t}}

\newcommand{\xiroal}{\xi^{\mathrm{ro},\alpha}}
\newcommand{\xirobe}{\xi^{\mathrm{ro},\beta}}

\newcommand{\xirov}{\xi^{\mathrm{ro},v}}
\newcommand{\xiroup}{\xi^{\mathrm{ro}}_{\Upsilon}}
\newcommand{\xiroupte}{\xi^{\mathrm{ro}}_{\Upsilon,\theta}}
\newcommand{\piro}{\pi^{\mathrm{ro}}}
\newcommand{\Iro}{I^{\mathrm{ro}}}
\newcommand{\fro}{f^{\mathrm{ro}}}
\newcommand{\Fro}{F^{\mathrm{ro}}}
\newcommand{\Froupte}{F^{\mathrm{ro}}_{\Upsilon,\theta}}
\newcommand{\phiroupte}{\phi^{\mathrm{ro}}_{\Upsilon,\theta}}
\newcommand{\phiroup}{\phi^{\mathrm{ro}}_\Upsilon}
\newcommand{\Rro}{R^{\mathrm{ro}}}

\newcommand{\tBro}{\tilde{B}^{\mathrm{ro}}}
\newcommand{\Dro}{D^{\mathrm{ro}}}

\newcommand{\tDroc}{\tilde{D}^{\mathrm{ro}, \circ}}
\newcommand{\Droc}{D^{\mathrm{ro}, \circ}}
\newcommand{\Drosc}{D^{\mathrm{ro}, {\scriptscriptstyle \cup}}}
\newcommand{\Tu}{{\mathrm{Tub}}}

\newcommand{\hot}{{\mathrm{h.o.t.}}}
\newcommand{\Hess}{{\mathrm{Hess}}}
\newcommand{\Ind}{{\mathrm{Ind}}}
\newcommand{\ang}{{\mathrm{ang}}}
\newcommand{\gstd}{g_{\mathrm{std}}}

\newcommand{\Zmax}{Z_{\mathrm{max}}}
\newcommand{\bZmax}{\bar Z_{\mathrm{max}}}

\newcommand{\ord}{\mathrm{ord}}

\renewcommand{\epsilon}{\varepsilon}
\newcommand{\set}[2]{\left\{ #1 \,\middle\vert\, #2 \right\}}
\newcommand{\icol}[1]{\left(\begin{smallmatrix}#1\end{smallmatrix}\right)}
\newcommand{\irow}[1]{\begin{smallmatrix}(#1)\end{smallmatrix}}

\makeindex

\begin{document}

\frontmatter

\title{The Total Spine of the Milnor Fibration of a Plane Curve Singularity}


\author{Pablo Portilla Cuadrado}
\address{
	Universidad Politécnica Madrid,
 Escuela Técnica Superior de Ingenieros Industriales, 
 José Gutiérrez Abascal 2, 28006 Madrid}
\email{p.portilla89@gmail.com}
\thanks{The first author is supported by RYC2022-035158-I, funded by 
	MCIN/AEI/10.13039/501100011033 and by the FSE+ and also was supported  by 
	the Labex CEMPI (ANR-11-LABX-0007-01)}

\author{Baldur Sigur{\dh}sson}
\address{
	Universidad Politécnica Madrid,
	Matem\'atica e Inform\'atica Aplicadas a las Ingenier\'ias Civil y Naval,
	C. del Profesor Aranguren 3, 28040 Madrid.}
\email{baldursigurds@gmail.com}
\thanks{The second
	author B.S. was partly supported by The Simons Foundation Targeted  Grant
	(No.
	558672) for the Institute of Mathematics, Vietnam Academy of Science and
	Technology, and  Contratos María Zambrano para la atracción de talento
	internacional at Universidad Complutense de Madrid and
    Universidad Polit\'ecnica de Madrid.}

\date{1st September 2023}

\subjclass[2020]{Primary 32S05, 14H50, 37B35}

\keywords{singularity theory, plane curve singularities, Milnor fibration, 
spine, gradient, low-dimensional topology.}

	\begin{abstract}
		For any plane curve singularity defined by an analytic function germ 
		$f$, we construct a spine on each Milnor fiber simultaneously, that 
		realizes the vanishing topology. In order to do so, we study the 
		separatrices at the origin of the vector field $-\nabla \log |f|$. 
		Under some genericity conditions on
		the metric, we produce a natural partition of the set of separatrices, 
		$S$,
		into a finite collection smooth strata.
		As a byproduct of this theory,
		we construct a smooth fibration which is
		equivalent to the Milnor fibration, and lives
		on a quotient of the Milnor fibration at radius $0$.
		The strict transform of $S$ in this space induces the aforementioned 
		spine for each fiber of this fibration.
		These fibers are naturally endowed with a vector field in
		such a way that the spine consists of trajectories which do not escape
		through the boundary.
	\end{abstract}
	
	\maketitle
	
	\tableofcontents
	
	\mainmatter
	
%
%
%
\chapter{Introduction}

\section{Motivation}
The motivation for this work comes from a question that Thom posed to L\^e  
\cite{le_poly}. More concretely, let $f:\C^n \to \C$ be a holomorphic map. 
Then there exists $\epsilon, \eta >0$ with $\epsilon$ small enough and 
$\eta$ small with respect to $\epsilon$ such that the restriction of $f$ 
yields a locally trivial fibration: 
\begin{equation}
	f|_{B_\epsilon \cap f^{-1}(D_\eta)^*}: B_\epsilon \cap f^{-1}(D_\eta^*) 
	\to D_\eta^*
\end{equation}
This is the Milnor fibration over the punctured disk. The following two 
results were proven by Milnor \cite{Milnor_hyp}
\begin{enumerate}
	\item The pair of spaces $(f^{-1}(0) \cap B_\epsilon, B_\epsilon)$ is 
	homeomorphic to the pair of cones $(C(f^{-1}(0) \cap  S_\epsilon ), 
	C(S_\epsilon))$.
	
	\item When $f$ has an isolated critical point at the origin, the Milnor 
	fiber $F$ has the homotopy type of a bouquet of spheres of real 
	dimension $n-1$.
\end{enumerate}

In this context, according to \cite{le_poly}, Thom asked L\^e if one could 
find
a polyhedron $P$ contained in this fiber $F$ (the Milnor fiber), of real
dimension $n-1$, such that $F$ would be a regular neighborhood
of $P$. From this
result one would obtain the construction of a continuous application which
would send $F$ to the special fiber $F_0$, sending $P$ to $\{0\}$ and $F
\setminus P$ homeomorphically to  $F_0 \setminus \{0\}$. This application 
would
geometrically realize the collapsing of the homology $H(F)$ of $F$ on the
trivial homology of $F_0$ and this would give a geometric realization of the
vanishing cycles of the function at the isolated critical point.

A solution to this problem in the more general context of germs of analytic
maps with isolated singularity defined on analytic complex spaces $f:X \to 
\C$
was sketched by L\^e in \cite{le_poly} and later \cite[Theorem
1]{le_menegon_poly}, detailed by L\^e and Menegon. More concretely, they 
proved
the following: let $F_t$ be the Milnor fiber of $f$ and let $F_0$ be the
central fiber. Then there exists a polyhedron $P_t \subset F_t$ of real
dimension $\dim_\C F_t$ and a continuous simplicial map $\partial F_t \to 
P_t$
such that $F_t$ is homeomorphic to the mapping cylinder of the map above.
Furthermore, there exists a continuous map of pairs $\Psi_t: (F_t, P_t) \to
(F_0,0)$ such that its restriction $F_t \setminus P_t \to F_0 \setminus 
\{0\}$
is a homeomorphism.
Their construction can be realized simultaneously only
over a proper sector of the
base space of the fibration, which is a punctured disk.

The starting point of this work is based on an original idea by A'Campo
\cite{AC_lag} that can be summarized as follows:
\begin{enumerate}
	\item
	Consider the inwards pointing
	radial vector field $\frac{-t}{|t|} \frac{\partial}{\partial t}$ on the 
	disk 
	$D_\eta$.
	
	\item Lift the vector field to a vector field $\xi^{\mathrm{lift}}$ on
	$B_\epsilon \setminus F_0$ using the connection given by the family of 
	tangent
	planes which are the symplectic orthogonals to the vertical tangent 
	bundle
	associated with the Milnor fibration.
	
	\item Integrate the vector field and analyze the uniparametric family 
	of diffeomorphisms that takes Milnor fibers along a ray all the way to 
	the central fiber.
	
	\item Analyze the maps $\rho_t:F_t \to F_0$ resulting from taking the
	limit of these diffeomorphisms and study the set
	$S_t=\rho_t^{-1}(0) \subset F_t$.
	
\end{enumerate}
The set $S_t$ is the expected
{\em spine}
for the Milnor fiber $F_t$ but it 
is 
not
clear, a priori, what kind of structure it has. At this moment a simple
observation is due: Analyzing the set $S_t$ amounts to analyzing the
separatrices at $0$ of $\xi^{\mathrm{lift}}$, that is, the set of integral
lines $S$ of $\xi^{\mathrm{lift}}$ that converge to the origin. Intersecting
this set with each Milnor fiber $S \cap F_t$ equally defines the set $S_t$.
From now on, we consider the case of plane curves, that is, $n=2$.
We do not assume, however, that the plane curve is reduced.

Our program departs from A'Campo's idea but it quickly bifurcates. The lift 
of
the vector field $\frac{-t}{|t|} \frac{\partial}{\partial t}$ is easily 
seen to 
be, up to multiplication by a
positive real function, the vector field
\[
\xi = - \nabla \log |f|.
\]
In particular, $\xi$ and $\xi^{\mathrm{lift}}$ have the same integral 
lines. So this  problem turns out to be closely related to the study of 
integral lines of gradients of absolute values of complex analytic maps.

The set $S$ is very complicated near the origin so a natural idea to study 
it,
is to follow this strategy: with the notation $C = F_0$, and
$\tilde C$ the strict transform of $C$, blow up the origin 
\[
\pi_0:(Y_0, \tilde{C} \cup D_0) \to (\C^2,C)
\] 
and see if the strict transform of the separatrices  
$\overline{\pi_0^{-1}(S)}$ intersects in a reasonable way the new 
exceptional divisor $D_0$. It happens that, under some genericity 
conditions on the metric, some of the trajectories in $S$ now meet $D_0$ 
transversely at a finite set of points $\Sigma_0$. Moreover, one is able to 
suitably rescale $\pi_0^* \xi$ in such a way that the rescaled vector field 
extends over $D_0^\circ = D_0 \setminus \tilde{C}$ to a vector field  
$\hat{\xi}_0$ and its restriction $\xi_0 = \hat{\xi}_0|_{D_0^\circ}$ 
satisfies that
\begin{enumerate}
	\item $\xi_0$ is the gradient with respect to the Fubini-Study metric 
	on $D_0$ of a Morse function $D_0^\circ \to \R$.
	\item The critical set of $\xi_0$ is precisely $\Sigma_0$.
	\item $\xi_0$ has fountain singularities and saddle points on interior 
	points of $D_0^\circ$ but it does not have sink singularities. 
	\item The intersection points of $D_0$ with the strict transform 
	$\tilde{C}$ behave like sink singularities, that is, $\xi_0$ points 
	towards the strict transform near these intersection points.
\end{enumerate} 

While we have {\em simplified} our set $S$, the problem is far from being
solved: there are still many trajectories in $\pi_0^{-1}(S)$ that converge 
to
intersection points in $\tilde{C} \cap D_0$. Some wishful thinking leads 
one to
believe that further repeating of this technique will eventually {\em 
	resolve}
the set $S$ and one shall be able to study these separatrices by studying 
some
sort of nice Morse vector fields on the divisors $D_i$ of, let's say, an
embedded resolution $\pi:(Y, D = \bigcup_i D_i) \to (\C^2,C)$. This 
couldn't be
further from the truth since a lot of things that behave nicely on the first
blow up, go wrong. In particular, it is not {\em always} possible to 
re-scale
and extend the (pullback of) the vector field $\xi$ to the divisors $D_i$. 
This
phenomenon occurs because the extension depends on the angle $\theta$ from
which one
approaches the origin. Once spotted the problem and its origin, it becomes
natural to perform a {\em real oriented blow up} of the resolution space $Y$
along the exceptional set $D$. The real oriented blow up substitutes a 
manifold
by the moduli of all the normal directions to this manifold. Note that the
function $\arg(f)$ that gives the argument of $f$, is only well defined 
outside
$C$. The real oriented blow up 
\[
\sigma:(\Yro, \Dro) \to (Y,D)
\]
{\em resolves this indeterminacy}: the function $\arg(f)$ lifts to a 
function
$\arg^\mathrm{ro}(f)$ that naturally extends over all $\Yro$.
With $\piro = \pi\circ\sigma$, this same
phenomenon allows us to rescale and extend the vector field $(\piro)^*\xi$ 
to
vector fields $\xiro_i$ that are defined on, and are tangent to, the
exceptional divisors $\Dro_i = \sigma^{-1}(D_i)$.

The space $\Yro$ is a manifold with corners. As a topological manifold it 
has a
boundary $\partial \Yro = \Dro$ that fibers over the circle $\R / 2 \pi \Z$ 
via
the map $\arg^\mathrm{ro}(f)$. This topological fibration is equivalent to 
the
Milnor fibration and it is usually known as {\em the Milnor fibration at 
	radius
	$0$} and its fiber $\Fro_{\theta}$
{\em the Milnor fiber at radius $0$}. This fibration was
described by A'Campo \cite{ACampo} to describe a zeta function for the
monodromy. See also \cite{Kat_Nak} for a very similar construction.

When the metric used to define $\xi$ satisfies other genericity conditions 
we show that the vector fields $\xiro_i$ behave in a nice way:
\begin{enumerate}
	
	\item
	$\xiro_{i,\theta} = \xiro_i |_{\Fro_{\theta} \cap \Dro_i}$
	has a saddle-point singularity at a point
	$p\in \Droc_{i,\theta}$
	if and only if $\sigma(p) \in D_i$ is the intersection point
	of the strict transform of a particular relative polar curve and $D_i$.
	
	\item
	$\xiro_{i,\theta}$ has no other singularities on $p\in 
	\Droc_{i,\theta}$.
\end{enumerate}

At this point we have simplified enough the set of separatrices $S$ and we 
conclude that it can be partitioned into manifolds of trajectories that 
converge either to some saddle point of some $\xiro_{i,\theta}$ or to a 
repeller on the first blow up. More concretely:

\begin{thmA}[\Cref{thm:spine}] \label{thm:total_spine}
	Assuming that $\C^2$ is endowed with a generic metric,
	the punctured total spine $S\setminus\{0\}$
	is the disjoint union of strata, each of which is a punctured
	disk, an open solid torus or an open solid Klein bottle.
	Each stratum corresponds to a point $p$ on the exceptional divisor of
	an embedded resolution of $(C,0)$, in that it is the union of 
	trajectories of $\xi$ whose lift to the resolution converges to
	$p$.
\end{thmA}

In \Cref{s:invariant_spine} we define the
\emph{invariant Milnor fibration at radius zero} which as a set,
is a quotient of the Milnor fibration at radius zero.

\begin{thmA}[\Cref{thm:invariant_spine}] \label{thm:invariant}
	The invariant Milnor fibration naturally admits a smooth structure and 
	a smooth
	vector field tangent to each Milnor fiber.
	The union of trajectories of this vector field, which do not escape 
	through
	the boundary, is a spine for each Milnor fiber.
	As a set, this spine coincides with the intersection of strict transform
	of the total spine with the Milnor fibration at radius zero.
	Moreover, this spine has the structure of a piecewise smooth
	$1$-dimensional CW complex and all the $1$-cells meet transversely at 
	the
	$0$-cells 
\end{thmA}

Note that we are able to construct this spine for all $\theta \in \R /2 \pi 
\Z$ at the same time and so, necessarily, the topological type of the spine 
changes with $\theta$ because, otherwise, we would get a finite order 
algebraic monodromy which is not always the case for plane curve 
singularities.

\section{Future work}

A large motivating factor for this project is to understand the integral 
monodromy and the variation map. A natural continuation of this work 
provides 
us with methods to explicitly describe both operators. This is the subject 
of a 
manuscript already in preparation.

A similar program in higher dimensions quickly becomes more complicated. 
For 
instance, the resolution complex in higher dimensions is not as well 
understood 
as the dual resolution graph of a plane curve singularity, whose structure 
is 
repeatedly used in this text. Another reason is that relative polar curves, 
which are heavily used in this work, have a simpler (but still very rich) 
nature than relative polar varieties which, for example, do not intersect 
the 
exceptional divisors in just isolated points.

\section{Organization of the paper}
In \Cref{s:general_setting} we recall the ideas of A'Campo \cite{AC_lag} 
and 
motivate our problem. Using results of {\L}ojasiewicz \cite{loj_coll}, we 
prove 
that the ideas of A'Campo yield a well defined collapsing map 
\[
\rho_t: F_t \to C
\]
from each Milnor fiber over $t \in D_\eta^*$ to the central fiber $C$. We 
show 
that this map is a local diffeomorphism when restricted outside the 
preimage of 
the origin. One of the main goals of the rest of the paper is to understand 
the 
set 
\[
S_t = \rho_t^{-1}(0).
\]

In \Cref{s:embedded_resolutions} we mainly recall the theory of embedded 
resolutions $\pi:(Y,D) \to (\C^2,C)$, this serves as well to introduce some 
important notation and invariants. We introduce a structure of directed 
graph to the dual graph of the resolution $\Gamma$ and describe this 
structure in terms of defined invariants associated with the resolution. 

\Cref{s:real_oriented} is devoted to understand the real oriented blow up 
of the resolution along the exceptional set $D$. This procedure produces a 
manifold with corners $\Yro$ that is naturally foliated by manifolds 
$\Yro_\theta$ where $\theta = \arg(f)$. A topological locally trivial 
fibration naturally appears on the boundary of the real oriented blow up: 
the Milnor fibration at radius $0$. 

In \Cref{s:resolving_polar_locus}, we recall some of the work by Tessier
\cite{Tei} and we describe a particular kind of embedded resolutions
$\pipol:\Ypol \to \C^2$: those
that resolve also all the relative polar curves that live in a special dense
and equisingular family. These are the resolutions used throughout the rest 
of
the paper. 

In \Cref{s:invariant_non_invariant} we introduce an important invariant
associated to each divisor of an embedded resolution, {\em the polar weight}
$\varpi_i \in \Z_{\geq 0}$. The vertices where this invariant vanishes 
describe
a special connected subgraph $\Upsilon$ of the resolution graph, {\em the
	invariant subgraph}. It is, furthermore, characterized as the subgraph
consisting of all the geodesics in $\Gamma$ that join the vertex $0$ with
vertices adjacent to some arrow head (which correspond to components of the 
strict transform of the curve). The complement of $\Upsilon$ is the (in
general disconnected) non-invariant graph. This section ends with a detailed
study of the generic relative polar curves of a plane curve singularity. In
particular, using previous work of Michel \cite{Mich_Jac} we describe base
points of the Jacobian ideal and some obstruction as to where the strict
transforms of generic relative polar curves intersect the
exceptional divisors of the
embedded resolution $\pipol: \Ypol \to \C^2$.

As we have said in the introduction, our program has a special treatment for
the divisor $D_0$ and its corresponding counterpart $\Dro_0$ in the real
oriented blow up. Our program is carried away completely in this case in
\Cref{s:1st_blowup}. More concretely, we prove that, after appropriately
rescaling the vector field $(\pipol)^*\xi$ we can extend it over to a vector
field $\xi_0$ defined on $D_0^\circ$. Moreover, for a dense set of linear
metrics which is made precise in that section, the restriction to 
$D_0^\circ$ is
the gradient of a Morse function $D_0^\circ \to \R$. We finish the section 
by
describing the spine a radius $0$ over the $1$st blow up. Note that this 
case
is important since it completely finishes the program for
homogeneous singularities, since they
are resolved by one blow up. One can also think
about this as the part that deals with the tangent cone of the singularity.

In \Cref{s:corners,s:others} we deal with the extension of the vector field
$(\piropol)^*\xi$ over the the boundary of the real oriented blow up 
$\Dro$. In
order to do so, we introduce the numerical invariant called {\em the radial
	weight} that measures the order of the poles that $(\pipol)^*\xi$ has 
over the
divisors $\Dro_i$ and, thus, indicates how to rescale $(\piropol)^*\xi$.

\Cref{s:sings_on_boundary} constitutes one of the main technical 
contributions
of the paper. We prove that the vector fields extended in the previous 
sections
have a manifold of trajectories converging to each saddle point in the 
interior
of some divisor. The main ingredient for this part is the theory of
center-stable manifolds as developed by A. Kelley in
\cite{kelley_stab_central,kelley_stable}.

In \Cref{s:total_spine} we prove the first main theorem
(\Cref{thm:total_spine}) of this paper, showing that the total spine $S$ 
naturally
decomposes into smooth manifolds corresponding to the singularities of
the extended vector fields. In particular,  we prove that if a trajectory
$\gamma$ converges to the origin in $\C^2$, its strict transform by the real
oriented blow up converges to one of these singularities.

In \Cref{s:invariant_spine}, we prove the second main result
of the paper, \Cref{thm:invariant},
which also depends on the first one. First we describe the
invariant Milnor fibration which is the natural fibration induced by the
function $\arg^\mathrm{ro}$ on the boundary of the real oriented blow up of
$(Y,D)$ along a certain subset of divisors. We put a smooth structure on 
this
space and show that our vector fields naturally extend and glue to give a
smooth vector field. Then, the regularity of the extended vector field and 
the
previous main theorem allows us to prove \Cref{thm:invariant}.

Finally, in \Cref{s:examples} we work out a detailed example where all the 
theory developed in the paper comes into play.

%
%
%
	\chapter{General Setting}
\label{s:general_setting}
Let $f\in \O_{\C^2,0}$ be a germ of a holomorphic function on $\C^2$ which 
is not necessarily reduced. Let $C \subset \C^2$ be a representative of the 
germ  $(Z(f),0)$ defined by $f$. Thus, $C$ is an analytic plane curve with 
possibly non-reduced components.

Let $B_\epsilon$ be a closed Milnor ball for $C$, that is, a ball such that 
all spheres $S_{\epsilon'}$ with $0<\epsilon'\leq \epsilon$ intersect $C$ 
transversely. Let $\eta$ be small enough with respect to $\epsilon$.
We define the set
\[
\Tu = \Tu(\epsilon,\eta)=B_\epsilon \cap f^{-1}(D_\eta)
\]
and call it a
\index{Milnor tube}
{\em Milnor tube}.
In particular, if $\eta$ is chosen small 
enough, Ehresmann's fibration theorem implies that 
\[
f|_{\Tu^*}: \Tu^* \to D_\eta^*
\] 
is a locally trivial fibration which is called
\index{Milnor fibration}
{\em the Milnor fibration}.
Here $\Tu^* = \Tu \setminus C$ and $D_\eta^* = D_\eta \setminus \{0\}$. For 
each $t\in D_\eta$, we denote by $F_t$ the sets
\[
F_t = f^{-1}(t) \cap B_\epsilon \subset \Tu.
\]
When $t \neq 0$, the set $F_t$ is {\em the Milnor fiber} at $t$. In this 
case, $F_t$ is a, possibly disconnected, compact surface with non-empty 
boundary. 
The set $F_0 = C \cap B_\epsilon$ is called the {\em central fiber} or the 
{\em special fiber}.  For notational convenience we also denote $F_0$ by 
$C_\epsilon = C \cap B_\epsilon$.

\section{The complex gradient as a lift} 

For the most of this section we essentially follow \cite{AC_lag}.

Let $\omega$ be the standard real symplectic structure on $\C^2$. We define
a connection associated to the Milnor fibration on the tube as
follows. For each $p \in \Tu \setminus \{0\}$ we define the vector space
\[N^\omega_p:=\{v \in T_p \Tu: \omega(v,u)=0 \text{ for all } u \in T_p 
F_{f^{-1}(f(p))}\}.\]
That is, $N_p^\omega$ is the symplectic orthogonal to the tangent space of 
the Milnor fiber that contains $p$. We consider the plane field 
\[
N^\omega := \bigcup_{p\in \C^2\setminus \{0\}} N_p^\omega
\]
which defines a connection on $\Tu \setminus \{0\} \to D_\eta^*$. We call 
$N^\omega$ the
\index{symplectic connection}
{\em symplectic connection}
associated with $f$.

\begin{rem}\label{rem:vect_properties}
	In this remark we expose some properties of the symplectic connection.
	\begin{enumerate}
		
		\item By the non-degeneracy of $\omega$ we find that $\dim_\R 
		N^\omega_p =2.$
		
		\item \label{v_prop:ii} Since each fiber is a complex curve, the 
		complex structure $J:T\C^2 \to T\C^2$ leaves invariant the tangent 
		bundle to the fibers. So $$\omega(u,v)=0 \Leftrightarrow 
		\gstd(u,v)=\omega(u,Jv)=0$$  where $\gstd$ is the standard 
		Riemannian 
		metric on 
		$\C^2$. Thus, the symplectic orthogonal to the 
		tangent spaces of fibers coincides with the Riemannian orthogonal 
		$N^{\gstd}_p$ for each $p \in \Tu \setminus \{0\}$.
		
		\item \label{v_prop:iii}Since $f$ is a submersion on $\Tu^*$ we 
		find that the map $df_p|_{N_p^\omega}:N_p^\omega \to T_{f(p)}\C$ is 
		a $\C$-linear isomorphism (in particular conformal) for all $p \in 
		\Tu^*$, so $N^\omega$ is a connection for the Milnor fibration 
		$\Tu^*\to D_\eta^*$.
	\end{enumerate}
	
\end{rem}

Let now $-t/ |t|$ be the unit real radial vector field pointing inwards on 
$D_\eta^*$. 

\begin{definition}
	Let $\xi^{\mathrm{lift}}$ be the symplectic lift of $-t/|t|$ to $\Tu^*$ 
	using the symplectic 
	connection $N^\omega$.
\end{definition}

In the next lemma we recall an explicit formula the gradient of the 
logarithm of the absolute value of a holomorphic function and a 
characterization of it.

\begin{lemma}\label{lem:cplex_gradient}
	Let $U\subset \C^n$ be an open subset and let $g:U\to \C$ be a 
	holomorphic 
	function. Then, 
	
	\begin{enumerate}
		\item \label{it:complex_gradient} 
		\[
		\nabla \log |g|
		= 
		\left( 
		\frac{\overline{\partial_{1} g }}{\bar{g}}, 
		\cdots,\frac{\overline{\partial_{n} g }}{\bar{g}}
		\right)^\top 
		\]
		
		\item \label{itcp:ii} 
		\[
		(Dg)(\nabla \log |g|)= \|\nabla \log |g|\|^2 g 
		\]
	\end{enumerate}
	where $\nabla$ indicates the real gradient with respect to the standard 
	Riemannian metric on $\C^n$ and $Dg$ is the differential of the 
	function $g$.
\end{lemma}

\begin{proof}
	The first item \cref{it:complex_gradient} is a direct computation. For 
	\cref{itcp:ii} again, the following direct computation
	\[
	\begin{split}
		(Dg)(\nabla \log |g|) &= \left( \partial_{1} g , 
		\cdots,\partial_{n} g
		\right) \cdot \left( \frac{\overline{\partial_{1} g }}{\bar{g}}, 
		\cdots,\frac{\overline{\partial_{n} g }}{\bar{g}}
		\right)^\top  \\
		& =  \frac{|\partial_{1} g|^2}{\bar{g}} + \cdots + 
		\frac{|\partial_{n} 
			g|^2}{\bar{g}} \\
		& = \left( \frac{|\partial_{1} g|^2}{|g|^2} + \cdots + 
		\frac{|\partial_{n} 
			g|^2}{|g|^2} \right) g \\
		& = \|\nabla \log |g|\|^2 g 
	\end{split}
	\]
	yields the result.
\end{proof}

The following lemma relates the lift $\xi^{\mathrm{lift}}$ of the unit 
radial vector field and the gradient described before.

\begin{lemma}\label{lem:vector_in_cords} The equality 
	\[
	\xi^{\mathrm{lift}} 
	= 
	\frac{-1}{\|\nabla \log |f|\|^2 |f|} \cdot \nabla \log |f|
	\] 
	holds.
\end{lemma}

\begin{proof}
	Observe that the function $-\log |f|: \Tu^*  \to \R_{> 0}$ takes 
	constant
	values on $f^{-1}(\mathbb{S}^1_{\eta'}) \cap B_\epsilon \subset \Tu^*$, 
	that
	is, on the preimage by $f$ of  circles of radius $\eta'$ with
	$0<\eta'<\eta.$ By the definition of gradient, $- \nabla \log |f|$ is
	orthogonal with respect to the Riemannian metric $\gstd$ to the tangent 
	spaces
	$T_pf^{-1}(\mathbb{S}^1_{\eta'})$, so in particular it is 
	$g$-orthogonal to the
	tangent spaces of the Milnor fibers and thus, by
	\Cref{rem:vect_properties}\cref{v_prop:ii},
	it is contained in $N^\omega_p$. It follows from construction
	that $\xi^{\mathrm{lift}}(p)$ is orthogonal to $T_pF_{f(p)}$. Since
	$df_p|_{N_p^\omega}:N_p^\omega \to T_{f(p)}\C$ is conformal
	(\Cref{rem:vect_properties} \cref{v_prop:iii}), the orthogonal 
	complement to
	$\xi^{\mathrm{lift}}(p)$ in $N^\omega_p$ is 
	\[
	T_pf^{-1}(\mathbb{S}^1_{\eta'}) \cap N^\omega_p.
	\]
	This shows that there exists a positive (since both vector fields point 
	in the same direction) real function $k$ such that 
	$\xi^{\mathrm{lift}}(p) = k \cdot (- \nabla \log |f|)$ and we can 
	compute it:
	\[
	k^{-1}= 
	\frac{\| Df(\nabla \log |f|)\|}{\|Df(\xi^{\mathrm{lift}})\|} = 
	\| Df(\nabla \log |f|)\| 
	\]
	because $\| Df(\xi^{\mathrm{lift}}) \| = 1.$  So we conclude by 
	applying \Cref{lem:cplex_gradient}, \ref{itcp:ii}.
\end{proof}

Once and for all, we put a name to this important vector field which is the 
central object studied in this work.

\begin{definition}\label{def:xi}
	We define the vector field 
\index{$\xi$}
	\[
	\xi = -\nabla \log |f|.
	\]
\end{definition}

\begin{rem} \label{rem:R_and_xi}
	\begin{blist}
		\item \label{it:i_R_and_xi_traj}
		The vector fields $\xi$ and $\xi^{\mathrm{lift}}$ have the same 
		trajectories on $\Tu^*$ since they differ by a positive scalar 
		function. 
		For computational convenience, throughout this article we 
		work with $\xi$ instead of $\xi^{\mathrm{lift}}$.
		
		\item \label{it:ii_R_and_xi_traj}Furthermore, an analogous 
		reasoning yields that these vector fields also have the same 
		trajectories as $- \nabla |f|^2$.
		\item \label{it:iii_and_xi_tang} 
		Since $\arg (f)$ is constant along trajectories of 
		$\xi^{\mathrm{lift}}$ (and $\xi$), these vector fields are tangent 
		to $\arg (f)^{-1}(\theta) \subset \Tu^*$ for all $\theta \in \R / 2 
		\pi \Z$. 
	\end{blist}
\end{rem}
\section{The collapsing map}

\begin{block}\label{blc:loja}
	The collapsing map,
$\rho:\Tu \to C$, is defined as follows.
	Choose $\epsilon$, and therefore also
	$\eta$, small enough. This means that $\Tu = \Tu(\epsilon,\eta)$ is
	contained in an arbitrarily small neighborhood of $0$ in $\C^2$.
	In fact, given a small neighborhood $U$ of $0 \in \C^2$,
	by a theorem of \L ojasiewicz \cite{loj_coll},
	we can choose $\epsilon$ small enough so that
	any trajectory of the vector field $-\nabla|f|^2$ starting
	at $p \in \Tu$, does not escape $U$. Furthermore, each such trajectory 
	$\gamma_p$ satisfies:
	\begin{enumerate}[label= (\alph*)]
		\item \label{it:def_R} it is defined on $\R_{\geq 0}$,
		\item \label{it:fin_arc} it has finite arc length, and 
		\item \label{it:converg} it converges to some point in $U$
		where $-\nabla|f|^2$ vanishes, i.e. a point of $C \cap U$.
	\end{enumerate}
	In this construction, we can clearly replace the vector field
	$-\nabla|f|^2$ by $\xi = -\nabla\log|f|$, or the lifting 
	$\xi^{\mathrm{lift}}$, since
	these define the same trajectories in $U \setminus C$ (recall 
	\Cref{rem:R_and_xi} \cref{it:i_R_and_xi_traj} and 
	\cref{it:ii_R_and_xi_traj} above).
	Note that if $p \in C$, then the trajectory of $-\nabla|f|^2$ is
	the constant path at $p$, and so the above results hold trivially.
\end{block}
\begin{rem}
	For completeness, we recall how the result cited above follows from
	the existence of a \L ojasiewicz exponent.
	Set $g = -|f|^2$.
	There exist $c > 0$ and $0<\theta<1$ so that in a neighborhood of the 
	origin,
	\[
	\|\nabla g\| \geq c |g|^\theta.
	\]
	Let $x(t)$ be a trajectory of $\nabla g$ parametrized by arc length.
	This way, we have
	\[
	\dot x
	=
	\frac{\nabla g}{\|\nabla g\|}
	=
	\frac{-\nabla\log|f|}{\|\nabla\log|f|\|}
	=
	\frac{\xi}{\|\xi\|}.
	\]
	Here we assume that $x$ is not a constant trajectory.
	This gives
	\[
	\frac{dg (x(t))}{dt}
	=
	\langle \nabla g(x(t), \dot x \rangle
	=
	\|\nabla g(x(t)\|
	\geq
	c|g(x(t)|^\theta,
	\]
	and so (since $|g| = -g$)
	\[
	\frac{d|g(x(t))|^{1-\theta}}{dt}
	=
	(1-\theta)|g(x(t))|^{-\theta} \frac{d|g(x(t))|}{dt}
	\leq
	-(1-\theta)c.
	\]
	Therefore, a trajectory starting at $p$ cannot be longer than
	\[
	(1-\theta)c|g(p)|^{1-\theta} = (1-\theta)c|f(p)|^{2-2\theta}.
	\]
	This shows that $x(t)$ has a limit, which must be a point where the
	vector field $\nabla g$ vanishes, i.e. a point on $C$.
\end{rem}

\begin{definition}
	With $\Tu$ as before, define the
\index{collapsing map}
{\em collapsing map} $\rho:\Tu\to C$ 
	by setting
	\[
	\rho(p) = \lim_{t\to\infty} \gamma_p(t),\qquad p\in\Tu,
	\]
	where $\gamma_p$ is the trajectory of $-\nabla|f|^2$ starting at $p$ 
	(which by \cref{it:def_R} is defined for all $t\in \R_{\geq 0}$).
	For $t \in D_\eta$, we set $\rho_t = \rho|_{F_t}$.
\end{definition}

\begin{definition}\label{def:spine}
	We say that a set $K$ in a manifold with boundary $M$ is a
\index{spine}
{\em spine} 
	for 
	$M$ if $M \setminus K$ is a collar neighborhood of $\partial M$.
\end{definition}

\begin{figure}[h!]
	\centering
	\includegraphics*{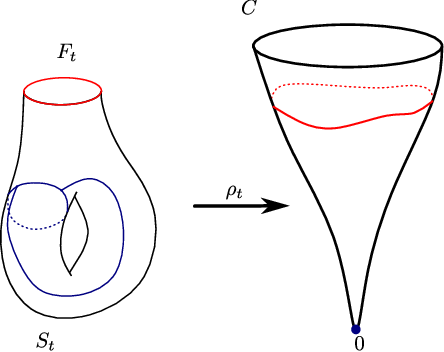}
	\caption{On the left, the Milnor fiber of the singularity $y^2+x^3$. In 
		blue we see the part of the Milnor fiber that maps to the origin by 
		the 
		collapsing map. In the central fiber we see the image of the 
		boundary 
		component of the Milnor fiber (red).}
\end{figure}

\begin{lemma} \label{lem:col_prop}
	
	Assume that $f = f_1^{r_1} \cdots f_s^{r_s}$, where $f_1,\ldots,f_s$ are
	irreducible, and set $C_i^{\phantom{1}} = f_i^{-1}(0)$.
	The collapsing map $\rho$ has the following properties:
	\begin{blist}
		\item \label{it:col_prop_prop}
		$\rho$ is continuous and proper.
		\item \label{it:col_prop_diff}
		The restriction of $\rho_t$ to $F_t \setminus \rho_t^{-1}(0) \to C 
		\setminus 
		\{0\}$
		is a local diffeomorphism.
		\item \label{it:col_prop_r}
		Let $q\in C_i \setminus \{0\}$, with
		$\|q\|<\epsilon$. Then, for $|t|$ small enough, the fiber 
		$\rho_t^{-1}(q)$  
		consists of  $r_i$ points.
	\end{blist}
\end{lemma}
\begin{proof}
	For \cref{it:col_prop_prop}, it suffices to show that $\rho$ is 
	continuous
	since $\Tu$ is compact.
	Let $p\in \Tu$, set $q = \rho(p)$ and let $V$ be a neighborhood of $q$ 
	in
	$\C^2$. By \cref{blc:loja},
	there exists a smaller neighborhood of $q$, $V' \subset V$ so that 
	trajectories
	starting in $V'$ do not escape $V$. The trajectory of $\xi$ starting
	at $p$ ends up in $q \in V'$, which means that the flow of $\xi$ at 
	some time
	sends a neighborhood $U'$ of $p$ to $V'$. As a result, trajectories
	starting in $U'$ ultimately do not escape $V$, and so
	$\rho(U') \subset \overline V$, which proves continuity.
	
	Next we prove \cref{it:col_prop_diff} and \cref{it:col_prop_r}.
	First, consider the case when
	$f = f_1^{r_1}$ has one branch with some multiplicity $r_1 \geq 1$, and
	$f_1$ is smooth. In this case, the vector fields $-\nabla|f|^2$ 
	and $-\nabla|f_1|^2$ have the same trajectories, and the function
	$-|f_1|^2$ is Morse-Bott. In this case, the Milnor fiber $F_t$ is the 
	disjoint
	union of $r_i$ disks, and $\rho_t$ embeds each of them in the smooth
	curve $\{f_1 = 0\}$. Thus, \cref{it:col_prop_diff} and 
	\cref{it:col_prop_r} hold
	in this case.
	
	In general, assume that $p\in F_t \setminus \, \rho_t^{-1}(0)$ with 
	$t\in 
	D_\eta^*$, and set
	$q = \rho(p)$. Then there exist small neighborhoods
	$V\supset V' \ni q$ so that $C\cap V$
	is a smooth curve with some multiplicity, and trajectories starting in 
	$V'$
	do not escape $V$. The flow of $\xi^\mathrm{lift}$ at some time sends a 
	neighborhood of
	$p$ in $F_t$ diffeomorphically to a neighborhood in $V' \cap F_{t'}$ for
	some $t'$. By the above case, for $|t'|$ small enough, 
	$\rho_{t'}|_{V'}: V' \to C$
	is a local diffeomorphism.
	This proves \cref{it:col_prop_diff}.
	
	For \cref{it:col_prop_r}, we can choose
	small neighborhoods $V'$ and $V$ in $\C^2$ with $q \in V' \subset V$ so 
	that if $|t|$ is small, then
	$\rho_t^{-1}(q) \cap V'$ consists of $r_i$ points. This follows from 
	the case above,
	where $f = f_1^{r_i}$ and $f_1$ is smooth.
	If $V''\subset V'$ is a closed neighborhood
	of $q$, then the complement $(C \cap B_\epsilon) \setminus V'$ can be 
	covered with finitely many
	neighborhoods with the property that if $p$ is in any of them, then
	the trajectory starting at $p$ does not escape $\Tu \setminus V''$.
	As a result, $\rho_t^{-1}(q) \subset V''$ for $|t|$ small enough.
\end{proof}

\begin{cor}\label{cor:spine}
	For $t \in D^*_t$, the complement of the spine in the Milnor fiber $F_t$
	is a collar neighborhood of the boundary,
	\[
	F_t \setminus \rho_t^{-1}(0)
	\cong
	\partial F_t \times (0,1].
	\]
\end{cor}
\begin{proof}
	The punctured central fiber is the union of trajectories of the vector
	field
	\[
	\nabla((|x|^2 + |y|^2)|_{C \setminus \{0\}})
	\]
	That is, the gradient (with respect to $g|_{C \setminus \{0\}}$) of the 
	restriction of the square distance function to the central fiber. Each 
	of these trajectories is transverse to all
	spheres $S_{\epsilon'}^3$, for $0 \leq \epsilon' \leq \epsilon$.
	This gives a product structure
	\[
	C\cap B_{\epsilon'}\setminus \{0\}
	\cong
	(C \cap S_{\epsilon'}^3) \times (0,1]
	\]
	for $\epsilon'$ small.
	This pulls back to a product structure on $F_t \setminus \, 
	\rho_t^{-1}(0)$.
\end{proof}

As we said in the introduction, this work is devoted to study the set of 
trajectories of $\xi$ that converge to the origin. The closure of this set, 
by 
definition, contains the origin. After \Cref{def:spine} and 
\Cref{cor:spine}, 
the following definition is justified.

\begin{definition} \label{def:total_spine}
	The
\index{spine!total}
\emph{total spine}
$S \subset \Tu$ and the \emph{spine} $S_t 
	\subset 
	F_t$
	of each Milnor fiber are defined as
	\[
	S = \rho^{-1}(0),\qquad
	S^{\phantom{1}}_t = \rho_t^{-1}(0),\quad t\in D_\eta.
	\]
\end{definition}

\section{Isometric coordinates}
In this subsection we define the notion of
\index{isometric coordinates}
{\em isometric coordinates}
that 
we use throughout the text. This captures the idea that changing the metric 
by a linear transformation and taking a linear change of coordinates are 
equivalent operations and thus, one can always work with the standard 
metric if one is willing to change the equation of $f$.

\begin{definition}\label{def:isometric_coords}
	For $A \in \GL(\C,2)$ define a Hermitian metric
	on $\C^2$ by setting \[h_A(v,w) = w^* A^*Av.\]
	In particular $h_{\id}$ is the standard metric. Equivalently, $h_A$ is 
	the metric obtained by pulling back
	the standard metric on $\C^2$ via the linear change of coordinates
	$A:\C^2 \to\C^2$. Associated with $h_A$ we have its real part $g_A$ 
	which is a Riemannian metric and its imaginary part $\omega_A$ which is 
	a symplectic form.
	
	Two linear functions $x,y:\C^2 \to \C$ form \emph{isometric coordinates}
	with respect to $A$ if the induced linear map
	$(x,y):(\C^2, g_A) \to (\C^2,g_{\id})$ is an isometry.
\end{definition}

\section{Vector fields}
\label{ss:vector_fields}

In this subsection we recall some notions about vector fields and fix some 
notation.

Let $\zeta$ be a vector field defined on an open neighborhood $U$ of an 
$n$-dimensional manifold $M$ around a point $p\in M$. In local coordinates, 
$\zeta$ is given by

\[
\zeta=(\zeta^1, \ldots, \zeta^n): U \to \R \times \R^n.
\]

The following definition is inspired by the usual notion of Hessian when 
the vector field admits a potential.
\begin{definition}\label{def:hess_vf}
	Assume that $\zeta(p)=0$. We define the
\index{Hessian matrix}
{\em Hessian matrix}
of $\zeta$ 
	at $p \in M$ as the matrix of partial derivatives
	\[\Hess_p \zeta  = 	\left(
	\begin{matrix}
		\partial_1 \zeta^1 (p) & \partial_2 \zeta^1 (p) &
		\cdots & \partial_n \zeta^1 (p)\\
		\partial_1 \zeta^2 (p) & \partial_2 \zeta^2 (p) &
		\cdots & \partial_n \zeta^2 (p)\\
		\vdots & \vdots &
		\ddots & \vdots\\
		\partial_1 \zeta^n (p) & \partial_2 \zeta^n (p) &
		\cdots & \partial_n \zeta^n (p)\\
	\end{matrix}
	\right)
	.
	\]
	The {\em Hessian} of $\zeta$ at $p$ is the linear operator $T_pM \to 
	T_pM$ represented by the above matrix. 
\end{definition}

\begin{definition}\label{def:singularity_vf}
	A {\em singularity of a vector field} is a point $p$ where the vector 
	field is not defined (for example a puncture of $M$) or where the 
	vector field takes the value $0$. 
\end{definition}
We follow the nomenclature of the book \cite{Abra}, in particular Corollary 
22.5 and text after.
\begin{definition}\label{def:elementary_vf}
	A linear operator is
{\em elementary}
if its eigenvalues all have 
	non-zero real part.
	
	A singularity $p$ of a vector field $\zeta$ where $\zeta(p)=0$ is an 
\index{elementary singularity}
{\em elementary singularity}
if the Hessian of the vector field is 
	elementary as an operator.
	
	We say that a vector field $\xi$ is {\em elementary} if it only has 
	elementary singularities.
\end{definition}

Elementary singularities allow us to apply the theorem of Grobman-Hartman 
and {\em linearize} the vector field near the singularity. 

\begin{example} \label{ex:saddles}
	A consequence of Grobman-Hartman  is that there are only three types of 
	elementary singularities (see \Cref{fig:elementary_sing}) that can 
	appear in a vector field defined on a topological surface:
	\begin{enumerate}
		\item a
\index{fountain}
{\em fountain}
or
\index{repeller}
{\em repeller},
which happens when the 
		Hessian has two eigenvalues with positive real part,
		\item a
\index{sink}
{\em sink}
or
\index{attractor}
{\em attractor},
when the two eigenvalues of 
		the Hessian have negative real part, and
		\item a
\index{saddle}
{\em saddle point},
when one eigenvalue has positive real 
		part and the other has negative real part.
	\end{enumerate}
	
	\begin{figure}[h!]
		\centering
		\includegraphics*[scale=1]{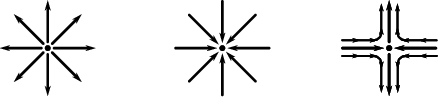}
		\caption{From left to right: a fountain, a sink and a saddle point.}
		\label{fig:elementary_sing}
	\end{figure}
	
\end{example}
Next, we recall the notions of winding number function of a curve relative 
to a vector field and Poincar\'e-Hopf index of a vector field. For more on 
winding number functions see \cite{Humph} or \cite{Chill}.

\begin{definition}\label{not:index_vf}
	Let $\zeta$ be a vector field on an oriented surface $M$ with punctures 
	and boundary components. Let $\gamma:S^1 \to  M$ be a simple closed 
	oriented curve in $M$. Assume that $\zeta$ does not vanish at any point 
	of $\gamma$. We define the
\index{winding number}
{\em winding number of $\zeta$ around 
		$\gamma$} by 
	
	\[
	w_\zeta (\gamma) = \frac{1}{2 \pi}\int_{S^1} d \, \ang 
	\left((\gamma'(t),\zeta(\gamma(t)) \right)
	\]
	where $\ang(\cdot, \cdot)$ is the angle function with respect to some 
	metric on $M$. It can be shown that it does not depend on the chosen 
	metric and that it is also invariant under isotopies of $\gamma$ that 
	do not cross any singularities of $\zeta$.
	
	Assume that the curve $\gamma$ is the boundary of a small disk centered 
	at an isolated singularity $p$ (puncture or zero) of $\zeta$, and that 
	$\gamma$ is oriented as the boundary of such disk. Then $w_\zeta 
	(\gamma)$ relates to the classical
\index{Poincar\'e-Hopf index}
Poincar\'e-Hopf index
of $\zeta$ at 
	$p$ denoted by $\Ind_\zeta (p)$ as follows
	\begin{equation} \label{eq:wind_ph}
		w_\zeta (\gamma) = 1 - \Ind_\zeta (p)  .
	\end{equation}
	
	Let $L$ be a boundary component of $M$ and let $N(L)$ be a small collar 
	neighborhood of $L$. Assume that $\zeta$ does not vanish at any point 
	of $N(L) \setminus L$. We define the {\em winding number of $\zeta$ 
		around $L$}, denoted by $w_\zeta (L)$, as 	$w_\zeta (L')$ where $L'$ 
	is a closed curve in $N(L)\setminus L$ parallel to $L$ and oriented 
	with the orientation induced by $M$ on $L$.
	
\end{definition}

\begin{example}\label{ex:ind_k}
	Consider the vector field $\zeta=\bar{z}^k$ on $\C$ with $k \geq 0$. 
	This vector field has, outside the origin and up to scaling, the same 
	integral lines as the vector field $z^{-k}$. Therefore we can compute 
	its index at the origin by the Cauchy integral formula and get
	\begin{equation}\label{eq:winding_index}
		\Ind_\zeta(0) = -k.
	\end{equation}
	Note that in this case the vector field $\bar{z}^k$ defines a 
	$k+1$-pronged singularity at the origin (see the left part of 
	\cref{fig:k_pronged}).
	\begin{figure}[h!]
		\centering
		\includegraphics*[scale=1]{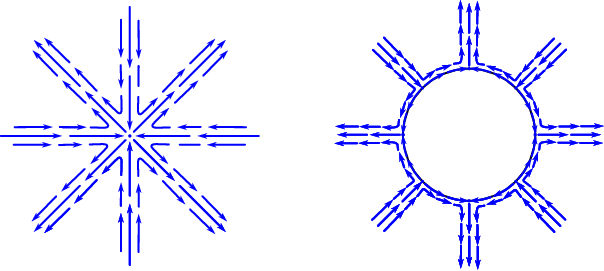}
		\caption{A $4$-pronged singularity on the left corresponding to the 
			vector field $\bar{z}^3$. And a $4$-pronged boundary on the right.}
		\label{fig:k_pronged}
	\end{figure}
	If the surface $M$ has a boundary component $L$ such that after being 
	contracted to a point, looks like a $k+1$-pronged singularity (see the 
	right part of \cref{fig:k_pronged}), then 
	\[
	w_\zeta (L)=1 + k
	\]
	
	Similarly, for the vector field $\zeta=z^k$  on $\C$ with $k \geq 0$, 
	we get a $k+1$-petal singularity (see \cref{fig:petal}). By a similar 
	computation the index in this case is
	\begin{equation}\label{eq:winding_index_petal}
		\Ind_\zeta(0) = k.
	\end{equation}
	And, if $M$ has a boundary component $L$ such that after being 
	contracted to a point, looks like a $k+1$-petal singularity, then 
	\[
	w_\zeta (L)=1-k
	\]
	\begin{figure}[h!]
		\centering
		\includegraphics*{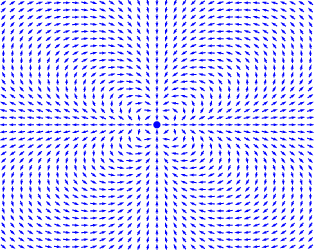}
		\caption{The vector field near the singularity of the vector field 
			given by $z^3$: a $4$-petal singularity.}
		\label{fig:petal}
	\end{figure}
	
\end{example}

%
%
%
\chapter{Embedded Resolutions} 
\label{s:embedded_resolutions}

In this section we explain several notions related to the resolutions of 
plane curve singularities, we also introduce notation and fix our 
conventions. We start by talking about different embedded resolutions and 
their  dual graphs. We introduce the {\em group of cycles} associated with 
a modification and the notion of {\em initial part} and {\em vanishing 
	order} of a function with respect to a variable or hypersurface. The latter 
two concepts play a central role in many of the subsequent proofs of this 
work.

After this, we introduce the {\em maximal} and {\em canonical} cycles which 
lead to the definition of two numerical invariants, $c_0$ and $c_1$, which 
allow us later to describe precisely the spine. We also include easy 
algorithms to compute them from the resolution process. For more details on 
the invariants and classical objects defined in these subsections, see 
\cite{Nemethi_FL} or \cite{Wall_ctc}.

In the last two subsections we introduce the structure of directed graph on 
the dual resolution graph $\Gamma$ and we introduce the notion of tangent 
associated with a divisor.

\section{Embedded resolutions}
\label{ss:embedded_resolutions}

\begin{notation} \label{not:pi}
	Let $(C,0) \subset (\C^2,0)$ be an isolated plane curve singularity 
	defined by a germ $f:(\C^2,0) \to (\C,0)$. Let $\pi:Y \to \C^2$ be an 
\index{embedded resolution}
\index{$\pi$}
	embedded resolution, that is, $\pi$ is a modification of $\C^2$ with 
	$\pi^{-1}(0)$ a strict normal crossings divisor.
	Factorize $\pi$ as
	\begin{equation} \label{eq:factorize_pi}
		\pi = \pi_0 \circ \pi_1 \circ \cdots \circ \pi_s
	\end{equation}
	where $\pi_i:(Y_i,D_i) \to (Y_{i-1},p_{i})$ are blow-ups.
	Set also $\bar\pi_i = \pi_i\circ\cdots\circ\pi_0:Y_i \to \C^2$.
	In particular, $\pi_0$ is the usual blow-up of $\C^2$ at the origin.
	We identify the divisors
	$D_i \in Y_i$ with their strict transforms in $Y_j$ for $j > i$.
	Set 
	\[
	D = \bigcup_{i=0}^s D_i \subset Y = Y_s.
	\]
	Let $\tilde{C}$ be the strict transform of $C$ in $Y$.
	
\end{notation}

Examples of embedded resolutions that appear in this work are $$\pimin: 
(\Ymin, \Dmin) \to (\C^2,0)$$ which is the
\index{$\pimin$}
\index{embedded resolution!minimal}
minimal embedded resolution. 
\index{$\pipol$}
Similarly, we have $$\pipol: (\Ypol, \Dpol) \to (\C^2,0)$$ which is the 
\index{embedded resolution!resolving polar locus}
minimal embedded resolution that also resolves the generic polar locus (cf. 
\Cref{ss:mp_p}).

\begin{notation}\label{not:dual_graph}
	Let
\index{$\Gamma$}
$\Gamma$
\index{dual graph}
be the dual graph associated with the resolution $Y\to 
	\C^2$ and
	the function $f$.
	This graph has set of vertices $\W = \V \cup \A$
	where $\V = \{0,1,\ldots,s\}$,
	and $\A$ is in bijection with the branches
	of $(C,0)$, where $(C,0) = \bigcup_{a\in \A} (C_a,0)$.
	Then $i,j\in \W$ are joined by an edge $ij$ in $\Gamma$ if and
	only if $D_i \cap D_j \neq \emptyset$. Here,
	if $a \in \A$, we are denoting by $D_a$ the strict transform of $C_a$
	in $Y$. In this case, we denote by $p_{ij}$ the intersection point $D_i 
	\cap D_j$.

	Set $D_\V = D$, as well as $D_\A = \bigcup_{a\in\A} D_a$ and
	$D_\W = \bigcup_{w\in \W} D_w$. For $i \in \W$,  set 
	\[
	D_i^\circ = D_i \setminus \bigcup_{j\in \W\setminus \{i\}} D_j.
	\]
	
	Similarly we can define the graphs $\Gamma_i$ associated with the 
	modifications $\bar{\pi}_i:Y_i \to \C^2$ for $i=0,1, \ldots, s$. These 
	have a set of vertices $\W_i =  \V_i \cup \A_i$. Furthermore, for each 
	$j,k$ with $j<k$ there is a canonical injective map $\V_j \to \V_k$ and 
	a bijective map $\A_j \to \A_k$.
	
	The graphs $\Gamma$ and $\Gamma_i$ (for all $i$) are trees (see 
	\cite[Section 3.6]{Wall_ctc}). In particular, given two vertices $j,k 
	\in \W$ or $\W_i$, there is a unique {\em geodesic} (path with the 
	smallest number of edges possible) joining $j$ and $k$.  
\end{notation}

Next we recall the notion of {\em group of cycles} associated with a 
modification of $\C^2$.

\begin{definition} \label{def:group_cycles}
	The
\index{group of cycles}
\emph{group of cycles}
$L$ associated with the modification
	$Y\to\C^2$ is the free Abelian group generated by the exceptional
	divisors $D_i$ for $i\in \V$. 
\end{definition}

\begin{rem}
	Set $V = \pi^{-1}(B_\epsilon)$,
	where $B_\epsilon$ is a Milnor ball. By, \cite[Lemma 8.1.1]{Wall_ctc}, 
	the group of cycles is isomorphic to
	$H_2(V;\Z) \cong H^2(V,\partial V;\Z)$, that is, this homology group is
	freely generated by the homology classes of exceptional components
	$D_i$, $i\in \V$. Furthermore, since $\partial V \cong S^3$,
	the long exact sequence of the pair $(V,\partial V)$ induces
	an isomorphism 
	\[
	H_2(V;\Z) \to H_2(V,\partial V; \Z) \cong H^2(V;\Z).
	\]
	As a result, with the identification $L = H_2(V,\Z)$,
	Poincar\'e duality induces a unimodular pairing
	\[
	(\cdot,\cdot):L\times L\to\Z.
	\]
\end{rem}
\begin{definition} \label{def:anti_dual_basis}
	We define the
\index{antidual basis}
\emph{antidual basis}
$D_j^*$, $j\in \V$ as the unique
	elements $D_j^*\in L$ satisfying $(D_i,D_j^*) = -\delta_{ij}$.
\end{definition}

\section{Initial part and vanishing order} 

Here we define several notions related to the order of vanishing of a real 
analytic function taking values in $\C$.

\begin{definition}
	Let $g:(\C^2,0)\to(\C,0)$ be a Laurent series in coordinates $u,v$ 
	given by
	\[
	g(u,v)= \sum_{i,j,k,l\in \Z} a_{ijkl} u^i v^j \bar{u}^k \bar{v}^l.
	\] 
	Let $\ord g = \min \set{i+j+k+l}{a_{ijkl} \neq 0}$. Then, we define the 
\index{initial part}
{\em initial part}
of $g$ as 
	\[
	\init g (u,v) = \sum_{i+j+k+l=\ord g} a_{ijkl} u^i v^j \bar{u}^k 
	\bar{v}^l.
	\]	
\end{definition}

\begin{definition}\label{def:initial_part}
	Let $g:(\C^2,0)\to(\C,0)$ be a Laurent series in coordinates $u,v$ 
	given by \[g(u,v)= \sum_{i,j,k,l\in \Z} a_{ijkl} u^i v^j \bar{u}^k 
	\bar{v}^l.\] 
	We define the
{\em vanishing order}
of $g$ with respect to $u$ as
	\[
	\ord_u g = \min \set{i+k}{a_{ijkl} \neq 0}.
	\]
Then, we define the {\em initial part} of $g$ with respect to $u$ as 
	the Laurent series given by  
	\[
	(\init_u g) (u,v) = \sum_{i+k=\ord_u g} a_{ijkl} u^i v^j \bar{u}^k 
	\bar{v}^l
	\]
\end{definition}

\begin{definition}
	Let $h:(\C^2,0)\to(\C,0)$ be the germ of a real analytic function, and 
	let $\pi:Y \to \C^2$ be as in \Cref{not:pi} and let $p\in D_i\subset Y$ 
	with $i\in \V$. And let $U\subset Y$ a coordinate neighborhood around 
	$p$ with coordinates $u,v$ so that $D_i \cap U$ is defined by $u=0$. We 
	define
\index{vanishing order!along $D_i$}
\index{$\ord_{D_i}$}
	the {\em vanishing order of $h$ along $D_i$} by
	\[
	\ord_{D_i} (h) = \ord_u (\pi^* h).
	\]
\end{definition}

\begin{rem}\label{rem:asso_div}
	Given a function $h:(\C^2,0)\to(\C,0)$ and a modification $\pi:Y \to 
	\C^2$ with dual graph $\Gamma$, there is a cycle associated to $\pi^*h$ 
	defined by 
	\[
	(\pi^*h) =\sum_{i\in \V} \ord_{D_i} (h) D_i.
	\]
	In the antidual basis $D^*_j$, the cycle takes the form 
	\[
	(\pi^*h) =\sum_{i\in \V} a_i D^*_i,
	\]
	where $a_i$ is the multiplicity intersection of $D_i$ and the strict 
	transform of $\{h=0\}$, that is, the number of intersection points of 
	$D_i$ with a small perturbation of the strict transform of $\{h=0\}$. 
	(See \cite{Nemethi_FL}).
\end{rem}

\begin{definition}
	For $i \in \W$, we set
	\[
	m_i = \ord_{D_i} (f).
	\]
	Note that, if $a \in \A$, then $m_a$ is the vanishing order of $f$ 
	along one of the branches of $C$ (which might be bigger than $1$ if $f$ 
	has non isolated critical points).
\end{definition}
\section{Invariants of the resolution: $c_0, c_1$ and the canonical 
	divisor}
\label{ss:modification_c}

In this subsection, we introduce numerical invariants of the modification
$Y \to \C^2$. In particular, these do not depend on $f$ or the metric $g$.
All of this subsection is either part of the classical theory of plane 
curve singularities, or it may be deduced from \cite{BFP_inner_rate}, (for 
example 
\Cref{lem:c_01_recursive}  is implied by
\cite[Lemma 3.6]{BFP_inner_rate}). However, since our setting is different 
and we need to fix notation, we briefly rework this theory.

\begin{definition} \label{def:max_cycle}
	Denote by $\Zmax$ the
\index{maximal cycle}
\index{$c_{0,i}$}
\emph{maximal cycle}
in $Y$, and denote by
	$c_{0,i}$ its coefficients. Thus,
	\[
	\Zmax = \sum_{i\in \V} c_{0,i} D_i.
	\]
	and we have $c_{0,i} = \ord_{D_i}(\ell)$, where $\ell:\C^2\to\C$ is a 
	generic linear function. Here {\em generic} means that the strict 
	transform of $\{\ell=0\}$ intersects $D_0^\circ$, or equivalently, that 
	$\{\ell=0\}$ is not a tangent of $C$ at the origin.
	
	Since $c_{0,i} = \ord_{D_i}(\ell)$ we can similarly define these 
	numbers for arrowheads, that is for $i \in \A$. In this case $c_{0,i}= 
	\ord_{D_i}(\ell) = 0$.
\end{definition}

\begin{rem}\label{rem:vanishing_linear}
	We recall a basic fact about vanishing orders of linear functions (see 
	for example \cite{Nemethi_FL}).
	If $\ell$ is a generic linear function and $\ell'$ is any other linear 
	function then, 
	\[
	\ord_{D_i}(\ell') \geq \ord_{D_i}(\ell)
	\]
	for all $i \in \V$. We also have 
	\[
	c_{0,i} = \min \{\ord_{D_i}(x), \ord_{D_i}(y)\}
	\]
	for $x,y$ any coordinate system in $\C^2$.
\end{rem}
\begin{definition} \label{def:canonical_cycle}
	Denote by
\index{$K$}
$K$
the
\index{canonical cycle}
{\em canonical cycle}
on $Y$ and by
\index{$\nu_i$}
$\nu_i -1$ its 
	coefficients.
	Thus,
	\[
	K = \sum_{i\in \V} (\nu_i-1) D_i.
	\]
\end{definition}

\emph{A canonical divisor} on a smooth variety $X$ of dimension $d$
is defined as the divisor defined by any section of the canonical
sheaf $\Omega_X^d$. The canonical cycle defined above is the unique
canonical divisor supported on the exceptional set of $\pi$.
In fact, for $d=2$
take the standard holomorphic two-form $dx \wedge dy \in \Omega_{\C^2}^2$
on $\C^2$. Then $K$ is the divisor defined by the section
$\pi^*(dx \wedge dy)$ of $\Omega_Y^2$.

$K$ can be seen as a Cartier divisor as follows.
If $U \subset Y$ is a coordinate chart, with coordinates $u,v$, we have
a Jacobian matrix
\[
\Jac \pi(u,v)
=
\left(
\begin{matrix}
	\partial_u \pi^*x & \partial_v \pi^*x \\
	\partial_u \pi^*y & \partial_v \pi^*y
\end{matrix}
\right).
\]
The pairs $(U, \det \Jac \pi)$ constitute a Cartier divisor,
whose associated Weil divisor is $K$.

Yet another characterization goes as follows.
$K$ is the unique exceptional divisor satisfying
the \emph{adjunction formulas}:
\[
(K,D_i) = -D_i^2 - 2, \qquad i\in\V.
\]

\begin{block} \label{block:c_1}
	Let $p \in D^\circ_i$ for some $i\in \V$, and choose a (small) 
	coordinate
	neighborhood $U$ of $p$.
	Assume that $x,y$ are coordinates in $\C^2$, with $x$ a generic linear 
	function.
	Then the function $\pi^*x$ vanishes with order $c_0 = c_{0,i}$ along 
	$D_i$,
	and does not vanish on $U \setminus D_i$.
	As a result, coordinates $u,v$ can be chosen for $U$ so that 
	$\pi^*x = u^{c_0}$. Expand the second coordinate
	around $p\in U$ as
	\[
	\pi^*y = \sum_{k \geq c_0} u^k g_k(v),
	\]
	where $g_k$ are holomorphic functions defined in a neighborhood of
	$p$ in $D_i^\circ$ (possibly smaller than $U\cap D_i^\circ$). We are 
	interested in the smallest number $k$
	so that $g_k$ really depends on $v$, i.e. is not constant.
	Denote this number by $c_1$. The Jacobian of $\pi$ in these coordinates 
	has a triangular form
	\begin{equation} \label{eq:partials}
		\left(
		\begin{matrix}
			\partial_u \pi^*x & \partial_v \pi^*x \\
			\partial_u \pi^*y & \partial_v \pi^*y
		\end{matrix}
		\right)
		=
		\left(
		\begin{matrix}
			c_0 u^{c_0-1} & 0 \\
			\sum_k k u^{k-1} g_k(v) & \sum_k u^k g_k'(v)
		\end{matrix}
		\right)
	\end{equation}
	and so $c_1$ is the order of vanishing of $\partial_v \pi^* y$ along 
	$D_i^\circ$. That is 
	\[
	\ord_{u}(\partial_v \pi^* y)= c_1.
	\]
	Furthermore, we see that 
	\[
	\ord_u(\det\Jac\pi)=c_0 + c_1 - 1.
	\]
	But this order of vanishing is $\nu_i-1$ (recall 
	\Cref{def:canonical_cycle}), which is independent of the choices
	we made so far. This motivates the following definition of the numbers 
	$c_{1,i}$:
\end{block}

\begin{definition}\label{def:c_1}
	For each vertex $i\in \V$, set
\index{$c_{1,i}$}
	\[
	c_{1,i} = \nu_i - c_{0,i}.
	\]
	
	As we did in \Cref{def:max_cycle}, just set $c_{1,i} =0$ for $i \in \A$.
\end{definition}

\begin{lemma} \label{lem:c_01_recursive}
	Let $\pi_i:(Y_i,D_i) \to (Y_{i-1},p_i)$ be the $i$-th blow up as in 
	\Cref{not:pi}. Then, 
	\begin{itemize}
		\item If $i=0$, then 
		\[c_{0,0}=c_{1,0}=1.\]
		\item
		If there is $j <i$, such that $p_i$ lies on a smooth point of the 
		exceptional divisor $D_j$,
		that is $p_i \in D_j^\circ$, then
		\[
		c_{0,i} = c_{0,j},\qquad
		c_{1,i} = c_{1,j} + 1.
		\]
		\item
		If there are $j,k<i$, such that $p_i$ lies at the intersection, 
		$p_i \in D_j \cap D_k$, then
		\[
		c_{0,i} = c_{0,j} + c_{0,k},\qquad
		c_{1,i} = c_{1,j} + c_{1,k}.
		\]
	\end{itemize}
\end{lemma}

\begin{proof}
	The first equality is a direct computation using the usual coordinates 
	$\pi_0^*x=u, \pi_0^*y=uv$.
	The two formulas for $c_{0,i}$ follow immediately, since this is the 
	vanishing
	of a generic linear function along the divisor $D_i$ (see the more 
	general, \cite[Lemma 8.1.2]{Wall_ctc}).
	The formula for $c_{1,i}$ follows, using the \Cref{def:c_1},
	and the well known fact (see the proof of \cite[Proposition 
	8.1.8]{Wall_ctc} that
	$\nu_i = \nu_j + 1$ if $p_i$ is a smooth point, and
	$\nu_i = \nu_j + \nu_k$ if $p_i$ is an intersection point of two 
	divisors.
\end{proof}

As a direct consequence of the recursive formulas, we have the following 
corollary.
\begin{cor}
	For any $i\in \V$, we have $c_{1,i} \geq c_{0,i}$, with equality if
	and only if $i = 0$.
\end{cor}

The next important corollary yields a hierarchy on the set of vertices 
$\V$. We make explicit use of this hierarchy in the next subsection and 
throughout the rest of the paper.

\begin{cor} \label{cor:c01_det}
	If $i,j$ are neighbors in the resolution graph $\Gamma$, then
	we have a nonzero determinant
	\begin{equation} \label{eq:c01_det}
		\left|
		\begin{matrix}
			c_{0,i} & c_{1,i} \\
			c_{0,j} & c_{1,j}
		\end{matrix}
		\right|
		\neq 0.
	\end{equation}
	Furthermore, the above determinant is negative if and only if the 
	geodesic (in $\Gamma$) from $0$ to $i$ passes through $j$.
\end{cor}
\begin{proof}
	We use induction on $i=0,1, \ldots, s$. Assume the statement holds for 
	$\Gamma_{i-1}$. We recall from \Cref{not:dual_graph} that the set of 
	vertices $\V_{i-1}$ of $\Gamma_{i-1}$ naturally embeds in the set of 
	vertices $\V_i$ of $\Gamma_i$. Actually, there is a unique extra vertex 
	$i \in \V_i$ that is not in $\V_{i-1}$. We treat two different cases 
	depending on whether $p_i$ is a smooth point or an intersection point 
	between two divisors (see also \cref{fig:new_blow_up}).
	
	Case 1. If $p_i$ lies on a smooth point of $D_{j}$, then there is a 
	unique edge $ij$ in $\Gamma_i$ that is not in $\Gamma_{i-1}$. In this 
	case $i$ is further from $0$ than $j$, and \Cref{lem:c_01_recursive} 
	gives  
	\[
	\left|
	\begin{matrix}
		c_{0,i} & c_{1,i} \\
		c_{0,j} & c_{1,j}
	\end{matrix}
	\right| = c_{0,i} c_{1,j} -  c_{1,i}c_{0,j} = c_{0,i} c_{1,j} -  
	(c_{1,j}+1)c_{0,i} = -c_{0,i} < 0.
	\]
	
	Case 2. If $p_i$ is the intersection point $D_j \cap D_k$, then there 
	are two edges $ik$ and $ij$ in $\Gamma_i$ that are not in 
	$\Gamma_{i-1}$. In this case, \Cref{lem:c_01_recursive} gives 
	
	\[
	\left|
	\begin{matrix}
		c_{0,i} & c_{1,i} \\
		c_{0,j} & c_{1,j}
	\end{matrix}
	\right| = 
	\left|
	\begin{matrix}
		c_{0,k} & c_{1,k} \\
		c_{0,i} & c_{1,i}
	\end{matrix}
	\right|= 
	\left|
	\begin{matrix}
		c_{0,k} & c_{1,k} \\
		c_{0,j} & c_{1,j}
	\end{matrix}
	\right|
	\]
	because 
	\[
	c_{0,i} = c_{0,j} + c_{0,k},\qquad
	c_{1,i} = c_{1,j} + c_{1,k}
	\]
	and row operations preserve determinants.
\end{proof}

\begin{cor} \label{cor:c01_1}
	If $i\in \V$ is a neighbor of the vertex $0$, which corresponds to the
	first blow-up, then $c_{i,1} = c_{0,i}+1$. 
\end{cor}

\section{Directions on the dual graph}

We endow $\Gamma$ with the structure of an directed graph and deduce some 
combinatorial properties of this structure in terms of $c_0$ and $c_1$.

\begin{definition} \label{def:directed}
	The graph $\Gamma$ is seen as a directed graph as follows.
	Let $i,j$ be neighbors in $\Gamma$. The edge $ij$ is directed from
	$j$ to $i$ if and only if the determinant \cref{eq:c01_det} is 
	negative. In order to express that {\em the edge $ij$ is oriented from 
		$j$ to $i$} we write $j \to i$.
	
	Similarly, the graphs $\Gamma_i$ corresponding to the modifications 
	$\bar\pi_i : Y_i \to \C^2$ are directed for each $i=0,1,\ldots,s$ 
	following the same criterion.
\end{definition}
\begin{figure}[!ht]
	\centering
	\includegraphics[scale=0.7]{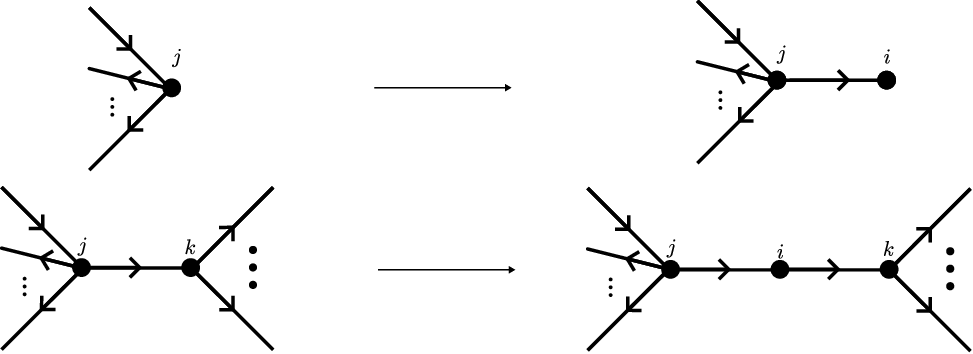}
	\caption{Effect of a blow-up on directed graphs. On top the effect of 
		blowing up a smooth point at $D_j$. On bottom, the effect of blowing up 
		the intersection point $D_j \cap D_k$.}
	\label{fig:new_blow_up}
\end{figure}
\begin{rem}
	The orientations of the edges defined in \Cref{def:directed} are the 
	same as the ones induced by the criterion: 
	
	\centerline{ ``orient an edge $ij$ from $j$ to $i$ if $i$ is further 
		from $0$ than $j$.''}
\end{rem}

Observe that each vertex $i \in \V$ with $i\neq 0$ has exactly one neighbor 
$j$ such that $j \to i$. This allows the following definition.

\begin{definition} \label{def:Ga_branch}
	Let $i$ be a vertex in the directed graph $\Gamma$, as above. If 
	$i\neq0$, let $j$ be the unique neighbor of $i$ such that $j \to i$.
	The
\index{branch}
\emph{branch}
(\cref{fig:branch}) of $\Gamma$ at $i$ is 
	\begin{enumerate}
		\item $\Gamma$, if $i = 0$ and,
		\item otherwise, the connected component of $\Gamma$ with the edge 
		$ji$ removed, containing $i$.
	\end{enumerate}
	
\end{definition}

\begin{rem} \label{rem:branch_ah}
	
	The branch of $\Gamma$ at $i$ is a tree rooted at $i$ (note that $i$ is 
	the vertex which is the closest to $0$ among all vertices of the 
	branch). Moreover, if $\Xi$ is a branch of $\Gamma$ at $i$ then, it 
	satisfies that if $j \in \W_\Xi$ and $j \to k$, then $k \in \W_\Xi$.
	
	Arrowheads in $\Gamma$ are included in the above definition. Thus,
	the branch at $i$ contains a subset of the arrowheads of $\Gamma$.
\end{rem}

\begin{figure}[h!]
	\centering
	\includegraphics*[scale=1]{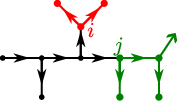}
	\caption{In red the branch at the vertex $i$ and in green the branch at 
		the vertex $j$.}
	\label{fig:branch}
\end{figure}

\begin{lemma} \label{lem:c0ord}
	Let $i,j\in\V$, and assume that $\Gamma$ has an edge $j \to i$. If $g 
	\in \O_{\C^2,0}$, then
	\[
	\det
	\left(
	\begin{matrix}
		c_{0,i} & \ord_{D_i}(g) \\
		c_{0,j} & \ord_{D_j}(g)
	\end{matrix}
	\right)
	\leq 0.
	\]
	If $\bar \V \subset \V$ is the set of vertices on the branch of $\Gamma$
	at $i$, then equality holds if and only if the strict transform of
	the curve defined by $g = 0$ does not intersect any $D_k$ for
	$k \in \bar \V$.
\end{lemma}

\begin{proof}
	In this proof, cycles are considered in the group
	$L \otimes \Q \simeq H_2(V; \Q)$.
	Set $n_k = \ord_{D_k}(g)$ for $k\in\V$, and $Z = \sum_k n_k D_k \in L$. 
	Thus, $Z$ is the cycle associated with $\pi^*g$.
	This cycle can be written in terms of the antidual basis as
	$Z = \sum_{k\in \V} a_k D_k^*$, with $a_k \geq 0$ (recall 
	\Cref{rem:asso_div}).
	
	Let $\bar\Gamma$ be the
	branch of $\Gamma$ at $i$, with vertex set $\bar \V \subset \V$.
	Since the intersection form is negative definite, there exist unique
	rational cycles $\bar D_k^*$ for $k\in\bar \V$
	supported on $\bar\Gamma$ satisfying $(\bar D_k^*, D_l) = -\delta_{k,l}$
	for $l \in \bar\V$, and furthermore, $\bar D_k^*$ has positive 
	coefficients in the standard basis $\{D_i\}_{i\in \bar{\V}}$.
	
	Define $\bar Z = \sum_{k\in \bar \V} a_k \bar D^*_k$. Then
	$\bar Z = \sum_{k\in\bar\V} \bar n_k D_k$, with nonnegative coefficients
	$\bar n_k \in \Q$. For $k \in \bar \V$, set $n_k' = n_k - \bar n_k$ and
	$Z' = \sum_{k\in \bar \V} n_k' D_k$.
	Thus, $Z'$ is supported on $\bar\Gamma$, and we have
	\[
	(Z',D_k) = -n_j \delta_{ik},\qquad
	k\in \bar \V.
	\]
	Indeed, if $k \in \bar \V$, with $k \neq i$, then $D_k$ only intersects
	exceptional divisors $D_l$ with $l \in \bar \V$, and so
	\[
	(Z',D_k)
	= (Z,D_k) - (\bar Z, D_k)
	= (-a_k) - (-a_k)
	= 0.
	\]
	If, however, $k = i$, then $D_k = D_i$ intersects $D_j$ once, but
	no other $D_l$ with $l \in \V \setminus \bar \V$, and so in this case
	\[
	(Z',D_k)
	= (Z - n_jD_j,D_k) - (\bar Z, D_k)
	= (-a_k) - n_j - (-a_k)
	= -n_j.
	\]
	As a result, $Z' = n_j \bar D_i^*$.
	
	In a similar way, if we set $\bZmax = \sum_{k\in \bar \V} c_{0,k} D_k$,
	we have $(\bZmax, D_k) = -c_{0,j} \delta_{ik}$, i.e.
	$\bZmax = c_{0,j}\bar D^*_i = c_{0,j} / n_j Z'$.
	As a result,
	\[
	n_i = n_i' + \bar n_i \geq n_i' = \frac{n_j c_{0,i}}{c_{0,j}},
	\]
	and so the determinant from the statement of this lemma equals $c_{0,i} 
	n_j -c_{0,j} n_i \leq 0$.
	Furthermore, we have equality if and only if all the
	$a_k$ vanish for $k \in \bar \V$. But these coefficients
	are nothing but the intersection
	between the strict transform of the vanishing set of $g$ and $D_k$.
\end{proof}

\begin{example} \label{ex:toric}
	We consider the case of a
\index{toric modification}
toric modification
of $\C^2$,
	see \cite{Fulton_toric} for definitions used here.
	Let $\triangle$ be a regular subdivision of the positive quadrant.
	This gives a modification $\pi_\triangle:Y\to\C^2$, where $Y$ is smooth.
	An irreducible component of the exceptional divisor corresponds
	to a ray $\rho \in \triangle$ generated by a unique primitive vector
	$(a,b) \in \Z^2$ with $a,b > 0$ and $\gcd(a,b) = 1$.
	Call the component $D_\rho$.
	We can choose an adjacent ray $\theta$ generated similarly by $(c,d)$,
	so that $\eta = \rho + \theta$ is a two-dimensional cone in $\triangle$ 
	and
	$ad-bc = 1$. The affine toric variety $U_\eta \cong \C^2$ provides a 
	coordinate
	neighborhood containing all but one point of $D_\rho$.
	Denoting by $u,v$ the coordinates in $D_\eta$, we have
	\[
	\pi_\triangle(u,v) = (u^av^c,u^bv^d),\qquad
	U_\eta \cap D_\rho = \{u = 0\}.
	\]
	We therefore have $\ord_{D_\rho} x = a$, $\ord_{D_\rho} y = b$ and so
	$c_{0,\rho} = \min\{a,b\}$.
	Assuming that $a \leq b$, take a local coordinate change
	$\tilde u = u v^{a/c}, \tilde v = v$. Then we have $x = \tilde 
	u^{c_{0,\rho}}$,
	and $y = \tilde u^{b} \tilde v^{d - bc/a}$.
	It follows that $c_{1,\rho} = b = \max\{a,b\}$.
\end{example}

\section{The tangent associated with a divisor}

We associate a tangent to each connected component of $\Gamma \setminus 
\{0\}$ and prove an easy, but useful, inequality of vanishing orders of 
linear functions.

\begin{definition}\label{def:tangent}
	Setting $\pi' = \pi_1 \circ \cdots \circ \pi_s$, we have
	a factorization $\pi = \pi_0 \circ \pi'$. For $p \in D \subset Y$, we 
	define
	the \emph{tangent associated with} $p$ as the line in $\C^2$ 
	corresponding
	to the point $\pi'(p) \in D_0 \subset Y_0$.
	If $i \neq 0$, then the tangent associated with $p \in D_i$ is 
	independent of $p$. So whenever $i \neq 0$, we also speak of the
\index{tangent associated with $D_i$}
{\em tangent associated} with $D_i$.
\end{definition}

The following lemma expresses the idea that tangents are precisely the {\em 
	non-generic} complex lines relative to a given plane curve.

\begin{lemma} \label{lem:tang_van}
	Let $i \in \V\setminus \{0\}$, and let $x,y$ be linear coordinates in
	$\C^2$ such that $\{y=0\}$ is the tangent associated with $D_i$. Then
	\[
	\ord_{D_i}(y) > \ord_{D_i}(x) = c_{0,i}.
	\]
\end{lemma}
\begin{proof}
	Since $\{x=0\}$ is not the tangent associated with $D_i$, the vanishing 
	order of $\pi^*x$ along $D_i$ is that of a generic linear function 
	which is, by definition (recall \Cref{def:max_cycle}) $c_{0,i}$. In 
	particular, since $y$ is another linear function, by 
	\Cref{rem:vanishing_linear} we have the inequality
	\[
	\ord_{D_i}(y) \geq \ord_{D_i}(x).
	\]
	The fact that $\{y=0\}$ is the tangent associated to $D_i$ implies that 
	its strict transform by the first blow up $\pi_0$ does not pass through 
	$D_0^\circ$ (it has to pass through one of the intersection,, let's say 
	$p$ points of the strict transform of $C$ by $\pi_0$). Since we are 
	assuming that $i \neq 0$, we know that the resolution $\pi= \pi_0 \circ 
	\pi_1 \circ \cdots \circ \pi_s$ includes the blow up at $p \in D_0$. 
	Let $\pi_k: Y_k \to Y_{k-1}$ be the blow up at $p$. Then $\ord_{D_k} 
	(x)  = 1$ and  $\ord_{D_k} (y)  = 2$. Necessarily $i \geq k$ and after 
	each new blow up, the order of vanishing of $y$ increases at least as 
	that of $x$. This settles the strict inequality.
\end{proof}

%
%
%
\chapter{The Real Oriented Blow-up} 
\label{s:real_oriented}

In this section, we recall the definition of the
real oriented blow-up of an embedded resolution along the exceptional set. 
This construction has been used before by other authors, see for example
\cite{ACampo}  or, in a more general setting,  \cite{Kat_Nak}.

Let $(Y,D) \to (\C^2,0)$ be a modification of the origin as above 
(\Cref{not:pi}). 
Denote by $\sigma_i:\Yro_i \to Y$
the
\index{real oriented blow-up}
\emph{real oriented blow-up}
of $Y$ along the submanifold $D_i \subset Y$
for $i \in \W$ constructed as follows.
If $U \subset Y$ is a chart with coordinates $u,v$ such that $u = 0$ is
an equation for $D_i \cap U$,
then we take a coordinate chart
$\Uro  \subset \Yro_i$ with coordinates $r,\alpha,v \in \R_{\geq 0} \times 
\R / 2 \pi\Z \times \C$, where $r$ and $\alpha$ are polar coordinates for 
$u$, that is:
\begin{equation} \label{eq:polar_coords}
	r = |u|:\Uro \to \R,
	\qquad
	\alpha = \arg(u):\Uro \to \R / 2\pi \Z.
\end{equation}
We can cover $D_i$ by such charts in order to define an atlas for $\Yro_i$. 
In each of these charts, the map $\sigma_i$ is given by 
$\sigma_i(r,\theta,v)= (re^{i\theta}, v)$.
Denote the fiber product of these maps by
\begin{equation}\label{eq:fiber_product}
	\sigma = \bigtimes_{i\in\W} \sigma_i:(\Yro,\Dro_\W) \to (Y,D_\W).
\end{equation}
The spaces $\Yro_i$ are manifolds with boundary, and the space $\Yro$ is
a manifold with corners and, as a topological manifold, its boundary 
$\partial \Yro$ coincides with $\Dro_\W$. Denote by
\index{$\piro$}
$\piro$
the composition
\[
\piro =  \pi \circ \sigma: \Yro \to \Tu. 
\]
Similarly, we can consider the real oriented blow up of $\C$ at the origin
\[
\sigma': \C^{\mathrm{ro}} \to \C.
\]
In this case, $\Cro \simeq \R_{\geq 0} \times \R / 2 \pi \Z$ and $\sigma' 
(r, \theta) = r e^{2 \pi i \theta}.$
The real-oriented blow up construction is functorial, that is, there exists 
a well defined map
\index{$\fro$}
$\fro: \Yro \to \Cro$
that makes the diagram in 
\cref{eq:ro_cd} commutative. We also say that the map $\sigma$ resolves the 
indeterminacy locus of the function $\arg(f)$ which results in a well 
defined function
\index{$\arg(\fro)$}
\[
\arg(\fro)=\arg \circ \fro:\Yro \to \R / 2 \pi \Z.
\]

We fit all this information in the following commutative diagram:
\begin{equation} \label{eq:ro_cd}
	\begin{tikzcd}[row sep=huge,column sep=huge]
		\partial \Yro 
		\arrow[hookrightarrow]{r} 
		\arrow[swap, bend right=10]{ddr}{\arg(\fro)|_{\partial \Yro} } 
		&\Yro 
		\arrow[ bend right= 10]{drr}{(\piro)^*f} 
		\arrow{r}{\sigma} 
		\arrow{d}{\fro} 
		\arrow[bend left=25]{rr}{\piro}
		\arrow[swap, bend right=25]{dd}{\arg(\fro)}
		& Y 
		\arrow{dr}{\pi^*f}  
		\arrow{r}{\pi}
		& \Tu  
		\arrow{d}{f} 
		\arrow[hookleftarrow]{r} 
		& \Tu^* 
		\arrow[bend left=25]{ddlll}{\arg(f)}\\
		&\C^{\mathrm{ro}} 
		\arrow{d}{\arg} 
		\arrow[bend right=20]{rr}{\sigma'}
		& & \C &\\
		&\R / 2 \pi \Z & & &
	\end{tikzcd}
\end{equation}

\begin{notation}\label{not:strat}
	The corners of $\Yro$ induce a stratification indexed by the graph
	$\Gamma$ as follows.
	Set $\Dro_\emptyset = \Yro \setminus \Dro_\W$.
	For $i,j \in \W$, we define the following subspaces of $\partial \Yro$:
\index{$\Dro_i$}
	\[
	\Droc_i = \sigma^{-1}(D_i^\circ), \qquad
	\Dro_i  = \sigma^{-1}(D_i),\qquad
	\Dro_{i,j} = \sigma^{-1}(D_i \cap D_j).
	\]
	Note that $\sigma^{-1}(D_i)= \overline{\Droc_i}$.
\end{notation}

\section{An action on the normal bundle}
For an $i\in \W$, the elements of $\sigma_i^{-1}(D_i) = \partial \Yro_i$
correspond to oriented real lines in the normal bundle
of $D_i \subset Y$, which has a complex structure.
This induces a $\C^*$-action on $\partial \Yro_i$, with stabilizer
$\R_{>0} \subset \C^*$,
which reduces to a principal action of
$\R/2\pi\Z \cong \C^*/\R_{> 0}$. 
In particular,  $\Droc_i$ is naturally a principal $S^1$-bundle over 
$D_i^\circ$,
and $\Dro_{i,j} \cong S^1 \times S^1$.
Denote this action by
\index{$\cdot_i$}
$\cdot_i$.
In the coordinates $r,\alpha,v$ on $\Uro$ introduced before in 
\Cref{s:real_oriented},
this action is given by
\[
\eta \cdot_i (0,\alpha,v) = (0,\alpha+\eta,v),\qquad
\eta \in \R / 2\pi \Z.
\]
for a point $(0,\alpha,v) \in \partial \Uro$.

Let $p \in D_i \cap D_j$ be an intersection point, and $U \ni p$ a
coordinate chart with $U \cap D_i = \{u=0\}$ and $U \cap D_j = \{v = 0\}$.
We have polar coordinates $(r,s,\alpha,\beta)$ on $\sigma^{-1}(U) = \Uro$
so that $\sigma(r,s,\alpha,\beta) = (r e^{i\alpha}, s e^{i\beta})$.
The corner
\[
\sigma^{-1}(p) = \set{(r,s,\alpha,\beta)\in \Uro}{r=s=0}
\]
has two $\R/2\pi\Z$-actions. For $\eta, \nu \in \R/2\pi\Z$
\[
\eta \cdot_i \nu \cdot_j (0,0,\alpha,\beta)
= (0,0,\alpha+\eta,\beta+\nu).
\]
Thus, the torus $\sigma^{-1}(p)$ is an $\R/2\pi\Z \times \R/2\pi\Z$-torsor.

\begin{definition}\label{def:equivariant}
	Let $U \subset Y$ and $\Uro = \sigma^{-1}(U)$.
	A function $f:\Uro \to \C$ is $\cdot_i$-\emph{equivariant of weight}
	$k \in \Z$
	if for all $p \in \Dro_i \cap \Uro$ and $\eta \in \R/2\pi \Z$, we have
	$f(\eta\cdot_i p) = e^{ik \eta} f(p)$. A function which is equivariant 
	of weight $0$, is said to be {\em invariant}.
\end{definition}

\begin{lemma} \label{lem:arg_equivariant}
	If $U \subset Y$ is open, and
	$h:U\to \C$ is a meromorphic function which is holomorphic on $U 
	\setminus D_i$ and has vanishing order $k\in\Z$
	along $D_i \cap U$. Then the pull-back $r^{-k} \sigma^*h$ of $|u|^{-k} 
	h$ by $\sigma$, extends
	over the boundary
	to a real analytic equivariant function of weight $k$.
	
	Similarly, the pull-back of $|u|^{-k} \bar h$ extends
	over the boundary
	to a real analytic equivariant function of weight $-k$.
\end{lemma}

\begin{proof}
	Since the statement of the lemma is local, we can assume that
	$U$ is a small coordinate chart around some $p \in D_i$, with
	coordinates $u,v$ so that the set $U \cap D_i$ is defined by $u = 0$.
	Expand $h = \sum_{j=k}^\infty a_j(v) u^j$, where $a_i$ are holomorphic
	functions in the variable $v$. Then, using polar coordinates as in
	\cref{eq:polar_coords},
	\[
	r^{-k} \sigma^*h(r,\alpha,v)
	= \sum_{j=k}^\infty r^{-k+j} a_k(v) e^{ij\alpha}
	\]
	which shows that $	r^{-k} \sigma^*h $ is real analytic function which 
	is well defined on all $\Uro$. Restricting to the boundary, that is 
	$r=0$, we get:
	\[
	\left. (r^{-k} \sigma^*h) \right|_{r=0} 
	= a_k(v) e^{k i \alpha}
	\]
	and the first statement follows.
	The second statement is completely analogous.
\end{proof}

\section{Milnor fibration at radius zero}

In this subsection we introduce the {\em Milnor fibration at radius zero}, 
which is the fibration given by the map
\[
\arg(\fro)|_{\partial \Yro} : \partial \Yro \to \R / 2 \pi \Z.
\]

\begin{definition}\label{def:angled_ray}
	For $\theta \in \R/2\pi \Z$, we define the {\em Milnor ray} at angle 
	$\theta$, as  
	\[
	\Tu^*_\theta = \arg(f)^{-1}(\theta) \subset \Tu^*,
	\]
	as well as the subsets of $\Yro$:
	\begin{equation}
		\begin{aligned}
			\Yro_\theta &= \arg(\fro)^{-1}(\theta), \qquad 
			&\Dro_{i,\theta} &= \Yro_\theta \cap \Dro_i, \\
			\Droc_{i,\theta} &= \Yro_\theta \cap \Droc_i, 
			&\Dro_{i,j,\theta} &=  \Yro_\theta \cap \Dro_{i,j}.
		\end{aligned}
	\end{equation}
\end{definition}

\begin{lemma}\label{lem:transverse_strata}
	For any $\theta \in \R/2\pi \Z$, the set $\Yro_\theta$ is a submanifold
	of $Y$ which intersects the boundary strata transversely.
\end{lemma}
\begin{proof}
	It suffices to show that for every stratum $S \subset \Yro$, the
	restriction \[
	\arg(\fro)|_S:S\to \R / 2 \pi \Z
	\]
	is a submersion.
	For $S = \Dro_\emptyset$, this follows from
	the chain rule, and
	the fact that
	if $p\in \Yro \setminus \Dro_\W \cong \Tu \setminus C$, then
	\[
	D((\piro)^*f)_p:TY_p \to T\C_{f(p)}
	\]
	is a non-zero complex linear map, in
	particular, it is surjective.
	
	Now take $p \in S = \Droc_i$ and let $\Uro$ be a chart around $p$ with 
	coordinates $r,\alpha, v$ and such that $\Droc_i \cap \Uro$ is defined 
	by $r=0$. The function $f$ in the chart $\Uro$ takes the form 
	\[
	f = r^{m_i}e^{i m_i \alpha} \tilde{f}
	\]
	with $\tilde{f}$ a unit on $\Uro$. Now, the function $\fro$ (recall 
	\cref{eq:ro_cd}), in these coordinates, looks like
	\[
	\fro (r,\alpha,v) = (r^{m_i}|\tilde{f}|, m_i \alpha + \arg(\tilde{f}) ) 
	\in \R_{\geq 0} \times \R / 2 \pi \Z.
	\]
	Now, we can compute 
	\[
	\arg(\fro)(r,\alpha, v) = m_i \alpha + \arg(\tilde{f})
	\]
	which, restricted to $S$, takes the form
	\[
	\arg(\fro)|_{S}(\alpha, v) = m_i \alpha + \arg(\tilde{f}).
	\]
	
	Since $\partial_\alpha (\arg(\fro)|_{S}) = m_i \neq 0$, because  
	$\arg(\tilde{f})$ does not depend on $\alpha$, we get again that 
	$\arg(\fro)|_S$ is a submersion. 
	
	A similar argument applies to $\Dro_{i,j}$.
\end{proof}

As a consequence of \Cref{lem:transverse_strata} the map 
\[
\arg(\fro)|_{\partial \Yro}: \partial \Yro \to \R/2 \pi \Z
\]
is a locally trivial topological fibration which is smooth when restricted 
to each of the strata (\Cref{not:strat}) of $\partial \Yro$. Moreover, it 
is equivalent as a $C^0$-fibration to the Milnor fibration.

\begin{definition}\label{def:milnor_fib_radius_0}
	The locally trivial fibration induced by $\arg(\fro)|_{\partial \Yro}$ 
\index{Milnor fibration!at radius zero}
	is called {\em the Milnor fibration at radius $0$}.
\index{$\Fro_\theta$}
	We call its fiber $\Fro_\theta = \arg(\fro)|_{\partial 
		\Yro}^{-1}(\theta)$ {\em the Milnor fiber at radius $0$ over the angle 
		$\theta$}.
\end{definition} 

\begin{rem}
	Note that $\Fro_\theta$ is a topological, piecewise smooth manifold and 
	that, for each $\theta \in \R / 2 \pi \Z$ it has a natural partition 
	induced by the stratification of the exceptional set $D$:
	\[
	\Fro_\theta = \bigsqcup_{i \in \W} \Droc_{i,\theta} 
	\bigsqcup_{\substack{i,j \in \W \\ D_i \cap D_j \neq \emptyset}} 
	\Dro_{i,j,\theta}.
	\]
\end{rem}

\section{Scaling and extending the metric}

In this subsection we explain how the pullback metric on $\Yro$ induces a 
metric on
$\Droc_{i,\theta}$
for each $i \in \V \setminus \{0\}$ and each 
$\theta \in \R/2\pi \Z$.

Fix a vertex $i \in \V\setminus \{0\}$ and a value $\theta\in\R/2\pi\Z$.
The set
$\Droc_{i,\theta}$ covers the Riemann surface $D^\circ_i$, and so
is a Riemann surface. Choose isometric coordinates 
(\Cref{def:isometric_coords}), $x,y$ so that $\{y=0\}$ is the tangent 
associated with $D_i$. Let $U$ be a small coordinate chart of some point
as in \cref{block:c_1}, with coordinates $u,v$ chosen so that the pullback 
$\pi^*x = u^{c_{0,i}}$, and in particular ${u=0}$ defines $D_i\cap U$. 
Recall that by \Cref{lem:tang_van},
$\pi^*y$ vanishes with strictly higher order along $D_i$ than $\pi^*x$.
By abuse of notation, we  write $x$ rather than $\pi^*x$, and denote its
partials with respect to $u$ and $v$ by $x_u$ and $x_v$, and so forth.
With these assumptions, we have $x_v = 0$.
Denote by $\Jac\pi$ the Jacobian matrix \cref{eq:partials}. Then, outside
of the exceptional divisor, the pullback of the standard Hermitian
metric in $\C^2$ is given by the Hermitian matrix
\[
\Jac\pi^* \cdot \Jac\pi
=
\left(
\begin{matrix}
	\bar x_u &
	\bar y_u \\
	0 &
	\bar y_v
\end{matrix}
\right)
\left(
\begin{matrix}
	x_u &
	0 \\
	y_u &
	y_v
\end{matrix}
\right)
=
\left(
\begin{matrix}
	|x_u|^2 + |y_u|^2 &
	\bar y_u y_v \\
	\bar y_v y_u &
	|y_v|^2
\end{matrix}
\right)
\in \Mat(2\times 2, \C)
\]
which does not extend to give a Hermitian metric on all of $U$. Here $\Jac 
\pi^*$ denotes the hermitian adjoint of the complex matrix $\Jac \pi$.
However, the lower diagonal $|y_v|^2$ entry vanishes with order precisely
$2c_{1,i}$. Thus, if we scale the metric by the function
$|u|^{-2c_{1,i}} = |x|^{-2c_{1,i} / c_{0,i}}$,
this lower diagonal entry gives a real $2\times 2$ matrix which is 
diagonal, and
whose diagonal entries are the restriction to $U\cap D_i$
of the extension of $|y_v|^2/|u|^{2c_{1,i}}$ over $U\cap D_i$.
Denote this function by $h_U$. Thus, if $(0,v) \in D_i \cap U$, then
\[
h_U(v) = \lim_{u\to 0} \frac{|y_v|^2}{|u|^{2c_{1,i}}}.
\]
Since this construction was made by scaling the pullback of the
standard metric by the function $|x|^{-2c_{1,i} / c_{0,i}}=|u|^{-2c_{1,i}} 
$, this gives
a metric of the whole set $D^\circ_i$, which in the coordinate
neighborhood $U\cap D^\circ_i$ is given by the real $2\times 2$ matrix
\[
\left(
\begin{matrix}
	h_U &
	0 \\
	0 &
	h_U
\end{matrix}
\right)
\in \Mat(2\times 2,\R)
\]
In fact, one may verify that if $U'$ is another coordinate neighborhood,
then the above matrices define the same metric on the intersection.
Since $\Droc_{i,\theta}$ covers $D^\circ_i$, the Riemann surface
$\Droc_{i,\theta}$ inherits a Riemannian metric.

\begin{definition} \label{def:exc_metric}
	Denote by $g_i$ and $\gro_{i,\theta}$ the Riemannian metrics on
\index{metric!on $\Droc_{i,\theta}$}
\index{metric!on $D_i^\circ$}
	$D_i^\circ$ and $\Droc_{i,\theta}$ constructed above.
\end{definition}

\begin{rem} \label{rem:exc_metric}
	A holomorphic coordinate $v$ in a chart $U \cap D_i^\circ$, as above, 
	induces
	a trivialization of the tangent and cotangent bundles of $U \cap 
	D_i^\circ$,
	which identifies a vector field or a one form with a complex function.
	Write $v = s + it$, so that $s,t$ are real coordinates.
	A vector field $\zeta = a\partial_s + b\partial_t$
	on $U \cap D_i^\circ$ corresponds to the complex function $a + ib$, and
	a differential form $\eta = \alpha ds + \beta dt$ corresponds to
	the function $\alpha + i \beta$. 
	
	The form $\eta$ is dual to the vector field
	$\zeta$ with respect to the metric $g_i$, that is, $\eta = 
	g_i(\zeta,\cdot)$
	if and only if
	\[
	\alpha+i\beta = h_U\cdot(a+ib).
	\]
	If $\phi$ is a real function on $D^\circ_i \cap U$, then the gradient 
	of $\phi$
	with respect to the metric $g_i$ is given by
	\[
	\nabla\phi
	=
	h_U^{-1} \cdot(\phi_s\partial_s + \phi_t\partial_t).
	\]
\end{rem}

%
%
%
	\chapter{Resolving the Polar Locus} \label{s:resolving_polar_locus}

In this section we specify a genericity condition for the metric that 
involves the relative polar curves of the plane curve singularity. Before 
doing so, we review part of the  needed theory.

\section{Generic polar curves} \label{ss:mp_p}
In this subsection, we fix notation regarding relative polar curves and 
recall a result from Tessier that says that there is an equisingular family 
of relative polar curves which is dense among all polar curves.

Let $w = (w_1,w_2) \in \C^2 \setminus \{0\}$ be a vector.
By the canonical isomorphism $T_p\C^2 \cong \C^2$ for any $p\in \C^2$,
we identify $w$ with a vector field $\partial_w$
on $\C^2$. We denote the corresponding partial $\partial_w f$ of $f$ by 
$f_w$.

\begin{definition}\label{def:polar_curve}
	The function $f_w: \C^2 \to \C$ defines a plane curve. We denote this 
	plane curve by 
	\[
	P_w = \{f_w = 0\}.
	\] 
	That is, $P_w$ is the vanishing set of $f_w$. We call $P_w$ the
\index{polar curve}
{\em relative polar curve}
of $f$ associated with $w$. 
	
	Observe that $P_w = P_{\lambda w}$ for any $\lambda \in \C^*$, and so
	$P_w$ only depends on the equivalence class $[w] \in \CP^1$.
	Denote the strict transform of $P_w$ in a modification $Y$ of $\C^2$ by 
	$\tilde P_w$.
\end{definition}

\begin{block}\label{blc:teissier}
	By \cite[pg. 269-270, Th\'eor\`eme 2.]{Tei}, the family $P_w$ for
	$[w] \in \CP^1$ is equisingular on
	a Zariski open dense subset $\Omega$ of $\CP^1$.
	More concretely, there exists a modification
	\[
	\pi: Y \to \C^2
	\] 
	as in \cref{eq:factorize_pi} and a set $\Omega \subset \CP^1$
	satisfying the properties
	\begin{enumerate}
		\item \label{it:normal_cross}
		if $[w] \in \Omega$,
		then $D \cup \tilde X \cup \tilde P_w$ is a normal crossing 
		divisor. That is, $\pi$ is an embedded resolution for each $P_w$ 
		with $w \in \Omega$,
		\item \label{it:equal_vanishing}
		if $[w],[w'] \in \Omega$ and $i \in \V$, then the strict transforms 
		of both 
		associated polar curves $\tilde{P}_w$ and $\tilde{P}_{w'}$ 
		intersect each 
		divisor in the same number of points, that is
		$|D_i \cap \tilde P_w| = |D_i \cap \tilde P_{w'}|$. Equivalently, \[
		\ord_{D_i} \pi^* f_w = \ord_{D_i} \pi^* f_{w'}.
		\]
		Or, equivalently, the cycles associated with $\pi^*f_w$ and 
		$\pi^*f_{w'}$ 
		coincide (recall \Cref{rem:asso_div}). We put a name to this cycle 
		in 
		\Cref{def:pi}.
		\item \label{it:not_tangent} $\Omega$ does not contain  any of the 
		tangents to $(C,0)$.
	\end{enumerate}
\end{block}

\begin{definition} \label{def:pipol}
	We denote by
\index{$\pipol$}
$\pipol: (\Ypol, \Dpol) \to (\C^2,0)$
	the minimal embedded resolution of $(C,0)$ satisfying the above 
	properties,
	and by $\Omega\in\CP^1$ the corresponding open set.
	Denote by $\Gapol$ its resolution graph, with $\Vpol$ the set of
	vertices corresponding to exceptional components.
	Denote by $\Cpol$ and $\Ppol$ the strict transform of the curve $C$ and
	a polar curve $P_w$ for a $w\in \Omega$, respectively.
\end{definition}

\begin{rem}\label{rem:resolv_jac}
	\begin{enumerate}
		\item \label{it:repeat_blow} Observe that $\pipol$ can be obtained 
		from $\pimin$ by iteratively blowing map at intersection points of 
		$\Pmin$ with the exceptional set.
		
		\item \label{it:jacobian_res} Any embedded resolution $\pi:Y \to 
		\C^2$
		that resolves the Jacobian ideal $J= (f_x, f_y)$, i.e. such that 
		$\pi^*J$ is
		principal, satisfies properties \cref{blc:teissier}
		\ref{it:normal_cross}, \ref{it:equal_vanishing} and 
		\ref{it:not_tangent} above.
	\end{enumerate}
\end{rem}
Property \Cref{blc:teissier} \cref{it:equal_vanishing} above leads to the 
following definition:

\begin{definition}\label{def:pi}
	For $[w] \in \Omega$ and $i\in\Vpol$,
	define a cycle $Z_P \in L$ (in the group of cycles
	associated with the modification $\pipol$)
\index{$p_i$}
	and integers $p_i$ for $i\in \V$ by
	\[
	Z_{P}
	= \sum_{i\in \V} |\Ppol \cap D_i| D^*_i
	= \sum_{i\in \V} p_i D_i.
	\]
	For $a \in \A$ we define $p_a = \ord_{D_a}f_w = m_a - 1$.
\end{definition}

\begin{rem}\label{rem:proppi}
	Recall that by \Cref{rem:asso_div}, $\ord_{D_i}(f_w) = p_i$. An 
	equivalent
	definition of the numbers $p_i$ (for any resolution),
	which does not depend on any choice, is
	\begin{equation} \label{eq:def_p_eq}
		p_i = \min\{ \ord_{D_i} (f_x),\; \ord_{D_i}(f_y) \}.
	\end{equation}
	This is because a relative polar curve is defined by taking a linear 
	combination of the partials $f_x$ and $f_y$.
\end{rem}

The next lemma shows that the number of intersection points of a generic 
polar 
curve with the first exceptional divisor is controlled by the number of 
tangents of the plane curve.

\begin{lemma} \label{lem:polar_comp_0}
	For $w \in \Omega$, the number of intersection points of the divisor 
	$D_0$
	corresponding to the first blow-up and
	$\Ppol$ is  $|D_0 \cap \Ppol| = t - 1$,
	where $t$ is the number of different tangents of $(C,0)$.
\end{lemma}

\begin{proof}
	Let $e$ be the multiplicity of $(C,0)$ so that $f = \init(f) + \cdots$
	where $\init(f)$ is a homogeneous polynomial of degree $e$, and $\cdots$
	stands for higher order terms.
	Distinct linear factors of $\init(f)$ correspond to tangents of $(C,0)$.
	Therefore, we have distinct points $[a_i,b_i] \in \CP^1$, and
	positive integers $e_i$ so that
	$e_1+\ldots+e_t = e$ and
	\[
	\init(f)(x,y) = \prod_{i=1}^t (a_ix - b_iy)^{e_i}
	\]
	We can assume $w = (0,1)$ up to a change of coordinates.
	In this case $f_w=f_y$. Then the intersection points
	$D_0 \cap \Ppol$ correspond to the roots of the homogeneous polynomial
	$\init(f)_y$, which are not roots of $\init(f)$.
	Note that $\init(f_y)=\init(f)_y$, because,
	since $y$ is not a tangent, $\init(f)$ has the monomial $y^e$ with 
	nonzero
	coefficient.
	As a result, $\init(f_y)$ is a homogeneous polynomial of degree $e-1$,
	and has a root of multiplicity $(e_i-1)$ at $[a_i,b_i]$.
	
	To finish the proof, it is enough to show that
	the Jacobian ideal has no base point on $D_0$ outside
	the strict transform of $C$ in $Y_0$.
	In other words, it suffices to show that $\init f_x$ and $\init f_y$
	have no common factors other than the roots of $\init(f)$.
	But a common root of $\init (f_x), \init (f_y)$ is a root of $\init (f)$
	by Euler's identity $x\init (f_x) + y\init (f_y) = e\init (f)$.
\end{proof}

\begin{rem} \label{rem:tangents}
	In the above proof, we have seen that the Jacobian ideal
	$(f_x,f_y) \subset \O_{C,0}$ does not have base points on the
	exceptional divisor of the blow-up $\pi_0:Y_0 \to \C^2$, outside the
	strict transform of $C$. As a result,
	the vertex $0$ of $\Gapol$ has precisely $t$ neighbors in $\W$.
\end{rem}

\begin{definition} \label{def:gen_pol}
	With $\Omega$ as above, let $G_{\mathrm{pol}} \subset \GL(\C,2)$ be the 
	set of
	matrices $A\in\GL(\C,2)$ satisfying the following property:
	If $v,w \in \C^2\setminus\{0\}$ are such that $v$ is a tangent of 
	$(C,0)$ and
	$h_A(v,w) = 0$, then $[w] \in \Omega$. We call such a metric $h_A$, a 
\index{metric!$P$-generic}
	{\em $P$-generic metric}.
\end{definition}

%
%
%
	\chapter{The Invariant and Non-invariant Subgraphs}
\label{s:invariant_non_invariant}

In \Cref{ss:polar}, we define the numerical invariant $\varpi_i$
associated to each divisor of an embedded resolution of $C$. This invariant
distinguishes between two types of divisors: those where $\varpi_i$ vanishes
({\em invariant}) and those where it does not ({\em non invariant}).
In \Cref{s:others}, we see that invariance is precisely the condition
which allows for a scaling of the pullback of $\xi$ to the minimal 
resolution
to be extended over a component of the exceptional divisor.

In \Cref{ss:Upsilon}, we give a combinatorial description of the
induced subgraph of $\Gamin$ on the invariant vertices, only in
terms of $\Gamin$.

\section{The polar weight} \label{ss:polar}

In this subsection we define an integral invariant associated to an embedded
resolution and we prove a formula for it. We work with the embedded 
resolution
$\pipol:\Ypol\to \C$ (\Cref{def:pipol}) of the curve $(C,0)$. 
Recall the  properties of $p_i$ in \Cref{rem:proppi} and that $m_i = 
\ord_{D_i} 
(f)$. Also, as we already did in \Cref{s:real_oriented}, by abuse of 
notation, 
we sometimes write $x$ rather than $\pi^*x$, and denote its
partials with respect to $u$ and $v$ by $x_u$ and $x_v$, and so forth. 
Similarly $f_v$ denotes the partial derivative $\partial_v$ of $\pi^*f$.

\begin{definition}\label{def:polar}
	The
\index{polar weight}
\emph{polar weight}
of the vertex $i \in \V$ is the integer
\index{$\varpi_i$}
	\[
	\varpi_i = c_{1,i} - m_i + p_i.
	\]
	The vertex $i\in \V$, or the divisor $D_i$,
	is \emph{invariant} if $\varpi_i = 0$.
	An edge $ij$ joining $i,j \in \V$, or the intersection point
	in $D_i \cap D_j$, is \emph{invariant} if both $i$ and $j$
	are invariant. An edge $ij$ is said to be {\em non-invariant} if  
	$\varpi_i \neq 0$ or $\varpi_j \neq 0$.
\end{definition}

\begin{example}\label{ex:base_case}
	If $i = 0$, then $c_{1,i} = 1$, and $p_i = m_i - 1$, and so 
	$\varpi_0=0$.
	In other words, the divisor $D_0$ appearing in the first blow-up is 
	invariant.
\end{example}


\begin{lemma} \label{lem:varpi}
	Let $i \in \V$ and $p \in D_i^\circ$, and let $U \subset Y$ be a 
	coordinate
	neighborhood around $p$ with coordinates $u,v$ satisfying
	$\pi^* x = u^{c_{0,i}}$. Assume also that the standard metric in the 
	coordinates $x,y$ is $P$-generic. Then
	\begin{equation} \label{eq:varpi}
		\varpi_i = \ord_u f_v - \ord_u f.
	\end{equation}
	In particular, $\varpi_i \geq 0$ with a strict inequality if and only
	the initial part $\init_u f$ does not depend on $v$, and so
	$\init_u f$ is a nonzero constant multiple of $u^{m_i}$.
\end{lemma}

\begin{proof}
	As in \cref{block:c_1}, we have $x_v = 0$ and $\ord_u y_v = c_{1,i}$.
	Furthermore, by the chain rule,
	\[
	f_v = x_v f_x + y_v f_y = y_v f_y.
	\]
	Therefore,
	\[
	c_{1,i} - m_i + p_i = \ord_u f_v - \ord_u f,
	\]
	proving \cref{eq:varpi}. The other statements follow from the last 
	equation and the definition of $\varpi_i$.
\end{proof}

\section{The invariant subgraph $\Upsilon$ of $\Gamma_{\mathrm{min}}$} 
\label{ss:Upsilon}

In this subsection, we use arbitrary coordinates $x,y$ for $\C^2$,
with the only condition that $\{y=0\}$ is a tangent of $C$. In particular,
we make no use of a generic metric.
We work with the minimal resolution $\Ymin \to \C^2$.
We recall that the characterization of the numbers $p_i$ given in 
\cref{eq:def_p_eq}, for $i \in \Vmin$
(or for $i \in \V$ for $\V$ the vertex set of any embedded resolution), is 
independent of any genericity conditions.
The main result of this subsection is \Cref{lem:inv}, which
characterizes invariant vertices in $\Gamin$.

\begin{notation}\label{not:geodesic_degree}
	Since $\Gamin$ is a tree, any pair of vertices in $\Gamin$ is joined by
	a \emph{geodesic}, i.e. a shortest path, which is unique.
	The
\index{degree}
\emph{degree}
$\degmin(i)$ of a vertex $i \in \Vmin$ in $\Gamin$ is 
	the number of adjacent
	vertices, including arrowheads.
\end{notation}

\begin{definition}\label{def:upsilon}
	Let
\index{$\Upsilon$}
$\Upsilon \subset \Gamin$ be the smallest connected subgraph of
	$\Gamin$ 
	containing the vertex corresponding to the first blow-up $0 \in \Vmin$, 
	as well as any vertex in
	$\Vmin$ adjacent to an arrow-head $a \in \A$.
	Let $\V_\Upsilon \subset \Vmin$ be the vertex set of $\Upsilon$.
\end{definition}

The next lemma characterizes the vertex set of $\Upsilon$. We postpone its 
proof to the end of this subsection (until we have developed the necessary 
tools for it). The rest of the subsection is a collection of technical 
results, 
most of them of a combinatorial nature, that are of interest on their own.

\begin{lemma} \label{lem:inv}
The vertex $i \in \Vmin$ is
\index{invariant}
invariant, i.e. $\varpi_i = 0$, if and only if $i \in \V_\Upsilon$.
\end{lemma}

\begin{block} \label{block:Z_induction}
	Let $\V_y \subset \Vmin$ be the set of vertices of $\Gamin$ whose
	associated tangent is the line $\{y=0\}$. The induced
	subgraph $\Gamma_y \subset \Gamma$ is one of the connected
	components of the graph
	$\Gamin$, with the vertex $0$ removed.
	We have a coordinate chart in $Y_0$ with coordinates $s,t$ so that
	$\pi_0$ is given by $(s,t) \mapsto (s,st)$.
	Denote by $p_0 \in Y_0$ the intersection point of $D_0$ and the
	strict transform of $\{y=0\}$, given in $s,t$-coordinates as $(0,0)$, 
	and
	let $g \in \O_{Y_0,p}$ be the pullback of $f$, given by
	\[
	g(s,t) = f(s,st).
	\]
	The morphism $\Ymin \to Y_0$ (restricted to a preimage of a small 
	neighborhood
	of $p_0$) is the minimal resolution of the plane curve
	germ defined by $g$. For $i \in \V_y$, we denote by
	$c_{0,i,y}$, $c_{1,i,y}$, $p_{i,y}$, $m_{i,y}$ and $\varpi_{i,y}$
	the invariants so far introduced corresponding to this resolution.
	Note that we have $m_{i,y} = m_i$, since $g$ is the pullback of $f$.
	Since $s,t$ are coordinates around $p_0\in Y_0$, we have
	\[
	c_{0,i,y} = \min\{\ord_{D_i}(s), \ord_{D_i}(t)\},\qquad
	p_{i,y} = \min\{\ord_{D_i}(g_s), \ord_{D_i}(g_t)\}.
	\]
	Using \Cref{lem:c_01_recursive}, one also verifies that
	\begin{equation} \label{eq:c1y}
		c_{1,i,y} = c_{1,i} - c_{0,i}.
	\end{equation}
	We can repeat the construction of the graph $\Upsilon$ as in
	\Cref{def:upsilon} for the function $g$.
	We denote the result of this subgraph of $\Gamma_y$ by $\Upsilon_y$.
	Denote by $D_y = \{y=0\} \subset \C^2$. By abuse of notation, we also 
	denote
	by $D_y$ the strict transform of $D_y$ in any modification of $\C^2$.
	Then $D_y$ intersects a unique exceptional component in $\Ymin$.
	By renaming of vertices, we assume that this component is $D_1$, and 
	that
	the geodesic from $1$ to $0$ in $\Gamin$ is $1,2,\ldots,k,0$.
	Set $\V' = \{1,\ldots,k\} \subset \Vmin$. We refer the reader to 
	\cref{fig:diag} for an example showing the use of this notation.
\end{block}

\begin{example}
	Let $f(x,y) = y^7 + x^3y^4 + x^6y^2 + x^{11}$. This is a Newton 
	nondegenerate
	germ, which can be resolved using a toric modification with $5$ 
	exceptional
	components, as seen in \cref{fig:diag}.
	Using the notation in \cref{block:Z_induction}, we find
	\[
	g(s,t)
	= f(s,st)
	= (st)^7 + s^3(st)^4 + s^6(st)^2 + s^{11}
	= s^7  \left( t^7 + t^4 + st^2 + s^4 \right).
	\]
	This curve has three branches at the origin $(s,t) = (0,0)$, one of 
	which
	is not reduced. The graph $\Gamma_y$ contains the four vertices
	to the right in \cref{fig:diag}.
\end{example}

\begin{lemma}
	With the above notation, there exists a morphism $\pi':Y' \to Y_0$ 
	which is
	a sequence of blow-ups, so that the minimal resolution factorizes as
	\[
	\begin{tikzcd}[row sep=tiny]
		& Y' \arrow[dd,"\pi'"] \\
		\Ymin \arrow[ur] \arrow[dr] &
		\\
		& Y_0
	\end{tikzcd}
	\]
	and the exceptional divisor of $\pi'$ is $D_1 \cup \ldots \cup D_k$.
\end{lemma}

\begin{proof}
	The minimal resolution is obtained by repeatedly blowing up any point of
	the total transform of $C$ which is not normal crossing. If we add the
	rule that we only blow up intersection points on the total transform of
	$D_y$, we get a shorter sequence of blow-ups $\pi':Y'\to Y_0$.
	The dual graph to the exceptional divisor of $Y'$ is a string, and
	the strict transform of $C$ in $Y'$ only intersects smooth points
	on the total transform of $D_y$.
	Finishing the resolution process gives a map
	$\Ymin \to Y'$, and an identification
	of the exceptional divisor of $Y'$ with
	$D_1 \cup \ldots \cup D_k \subset \Ymin$.
\end{proof}

\begin{block} \label{blc:yprime}
	The composed morphism $Y' \to Y_0 \to \C^2$ is a toric morphism with 
	respect to 
	the
	standard toric structure on $\C^2$, as in \Cref{ex:toric}.
	Let $\triangle_y$ be the corresponding fan.
	Along with the natural basis vectors $(1,0)$ and $(0,1)$,
	the vectors
	\[
	w_l = (\ord_{D_l}(x), \ord_{D_l}(y)) = (c_{0,l}, c_{1,l}),\qquad
	l = 0,1,\ldots,k
	\]
	generate the rays in $\triangle_y$.
	In particular, $w_0 = (1,1)$.
	Denote throughout \Cref{ss:Upsilon} by $\wt_l(f)$ and $\init_l f$ the 
	weight
	and initial part of a function $f\in\O_{\C^2,0}$
	with respect to the weight vector $w_l$. Thus, if
	\[
	f(x,y) = \sum_{i,j} a_{i,j} x^i y^j,\qquad
	\supp(f) = \set{(i,j) \in \Z_{\geq 0}^2}{a_{i,j} \neq 0},
	\]
	then
	\[
	\begin{split}
		\wt_l(f)
		&= \min\set{\langle w_l,(i,j)\rangle}{(i,j)\in\supp(f)},\\
		\init_l(f)
		&= \sum\set{a_{i,j}x^iy^j}{\langle w_l,(i,j)\rangle = \wt_l(f)}.
	\end{split}
	\]
	Furthermore,
	\[
	\ord_{D_l}(f) = \wt_l(f).
	\]
\end{block}
In the rest of the results of the present subsection we use the notation 
introduced in 
\Cref{block:Z_induction} and \Cref{blc:yprime}.
\begin{lemma} \label{lem:ord_y}
	Assume $i \in \V_y$. The following holds
	\begin{enumerate}
		\item if $i \in \V'$, then $\ord_{D_i}(y) = c_{1,i}$ and
		\item if $i \notin \V'$, then $\ord_{D_i}(y) < c_{1,i}$ .
	\end{enumerate}
	
\end{lemma}

\begin{proof}
	The first statement follows from the computations made in
	\Cref{ex:toric}.
	By construction, the strict transform of any branch of $C$
	with tangent $\{y=0\}$ in $Y'$ intersects the exceptional divisor
	in a smooth point. If $D_l$ has such a point, and we blow it up to
	get a new vertex $j$, then by \Cref{lem:c_01_recursive},
	we have $c_{1,j} = c_{1,l}+1$, whereas $\ord_{D_j}(y) = \ord_{D_l}(y)$.
	By continuing the resolution process, the difference $c_{1,j} - 
	\ord_{D_j}(y)$
	can only increase. This proves the second statement.
\end{proof}

\begin{cor} \label{cor:yx_const}
	If $i \in \V_y \setminus \V'$, then the (pullback of the) function
	\begin{equation} \label{eq:yx_const}
		\frac{y^{\ord_{D_i}(x)}}{x^{\ord_{D_i}(y)}}
	\end{equation}
	is a nonzero constant along $D_i^\circ$.
\end{cor}
\begin{proof}
	Set $k = \ord_{D_i}(y)$, and note that $\ord_{D_i}(x) = c_{0,i}$.
	Let $U$ be a coordinate chart around some point in $D^\circ_i$ as in
	\cref{block:c_1}, with coordinates $u,v$ so that $x = u^{c_{0,i}}$.
	Then, the fraction
	\cref{eq:yx_const}, restricted to $D_i^\circ$ equals $g_k(v)^{c_{0,i}}$.
	Since $k < c_{1,i}$ by the previous lemma,
	the function $g_k(v)$ is constant, by construction in \cref{block:c_1}.
	Since the fraction \cref{eq:yx_const} has vanishing order zero,
	this constant is not zero.
\end{proof}

\begin{lemma} \label{lem:varpi_char}
	Let $i \in \V_y$. The following are equivalent
	\begin{enumerate}
		\item \label{it:varpi_char_inv}
		$i$ is invariant, i.e. $\varpi_i = 0$,
		\item \label{it:varpi_char_x}
		the fraction
		\[
		\frac{f^{c_{0,i}}}{x^{m_i}}
		\]
		restricts to a nonzero constant function along $D_i^\circ$,
		\item \label{it:varpi_char_h}
		If $h \in \O_{\C^2,0}$ is any function satisfying
		$\ord_{D_i}(g) = c_{0,i}$, then the fraction
		\[
		\frac{f^{c_{0,i}}}{h^{m_i}}
		\]
		restricts to a nonzero constant function along $D_i^\circ$,
	\end{enumerate}
\end{lemma}
\begin{proof}
	The equivalence of \cref{it:varpi_char_inv} and \cref{it:varpi_char_x}
	follows from \Cref{lem:varpi}.
	If $\ord_{D_i}(h) = c_{0,i}$, then
	$h$ vanishes at $0$, and has a linear term $ax+by$, with $a\neq 0$, 
	since
	all monomials other than $1$ and $x$ vanish with strictly higher order
	than $c_{0,i}$. Then
	\[
	\frac{f^{c_{0,i}}}{h^{m_i}}
	= \frac{f^{c_{0,i}}}{(ax + by + \hot)^{m_i}}
	= a^{-m_i}\frac{f^{c_{0,i}}}{x^{m_i}}.
	\]
	along $D_i^\circ$, which proves that
	\cref{it:varpi_char_x} and \cref{it:varpi_char_h} are equivalent.
\end{proof}

\begin{lemma} \label{lem:char_l}
	If $l \in \V'$, then the following are equivalent
	\begin{enumerate}
		\item \label{it:char_l_U}
		$l$ is in $\Upsilon$,
		\item \label{it:char_l_mon}
		$\init_l f$ does depend on $y$, i.e. $\init_l f$ is not
		a monomial of the form $a x^n$,
		\item \label{it:char_l_inv}
		$l$ is invariant.
	\end{enumerate}
\end{lemma}

\begin{proof}
	We first show  \cref{it:char_l_mon} $\Leftrightarrow$ 
	\cref{it:char_l_inv}.
	If \cref{it:char_l_mon} does not hold,
	then, since $\partial_y x^n = 0$, we have
	\[
	c_{1,l} + p_i = \wt_l(yf_y) > \wt_l(f) = m_i,
	\]
	where we use \Cref{lem:ord_y}. This gives $\varpi_i > 0$, i.e.
	\cref{it:char_l_inv} does not hold.
	Conversely, if \cref{it:char_l_mon} holds,
	then the partial operator $\partial_y$ does not kill the initial part
	of $f$ with respect to the weight vector $w_l$, and we have an equality 
	above,
	i.e. $\varpi_i = 0$.
	
	To prove \cref{it:char_l_U} $\Leftrightarrow$ \cref{it:char_l_mon},
	consider first the case when the $x$-axis $D_y$ is a component
	of $C$. In this case, all vertices in $\V'$ are in $\V_\Upsilon$ by
	construction, since $\V'$ is the set of vertices along the geodesic
	connecting the strict transform of $D_y$ with the first blow-up.
	That is, all elements $l\in \V'$ satisfy \cref{it:char_l_U}.
	In this case, $f$ is divisible by $y$, and so all elements of $\V'$
	also satisfy \cref{it:char_l_mon}.
	
	In the case when the strict transform of $C$ passes through the 
	intersection
	point of $D_y$ and $D_1$, we also find that all elements of $\V'$
	are in $\V_\Upsilon$, and also satisfy \cref{it:char_l_mon}.
	
	Otherwise, the strict transform of $C$ intersects only the smooth part
	of the total transform of $D_y$ in $Y'$. For $l\in \V'$, the strict
	transform intersects $D_l$ if and only if $\init_l(f)$ is not a 
	monomial,
	in which case the convex hull of $\supp(\init_l(f))$ is a face
	of the Newton polyhedron of $f$. If $l$ and $j$ are two such vectors
	corresponding to adjacent faces, then $\init_s(f)$ is the monomial
	at the intersection of the two faces, if $s$ is a vertex in between
	$l$ and $j$. These are precisely the vertices in $\Upsilon$, and for the
	rest of the vertices, the initial part of $f$ is a monomial whose 
	exponent
	is a vertex on the Newton diagram which also lies on the $x$-axis.
\end{proof}

\begin{figure}[ht]
	\begin{center}
		\includegraphics*{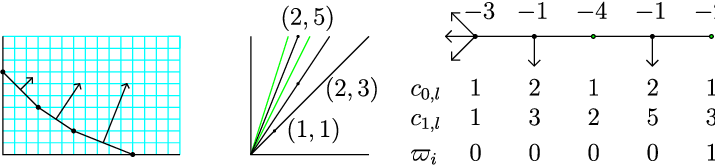}
		\caption{
			A Newton diagram for the function
			$f(x,y) = y^7 + x^3y^4 + x^6y^2 + x^{11}$, along with the 
			corresponding
			regular subdivision of $\R_{\geq 0}^2$. The vertex to the left 
			on the
			resolution graph corresponds to the first blow-up, the other 
			vertices are
			in $\V'$. The vertex to the right is not invariant. 
			Correspondingly, the
			initial part of $f$ with respect to the weight vector $(1,3)$ is
			$x^{11}$.}
		\label{fig:diag}
	\end{center}
\end{figure}

\begin{lemma} \label{lem:char_l_np}
	If $i \in \V_y \setminus \V'$, then $i \in \Upsilon_y$
	if and only if $i \in \Upsilon$.
\end{lemma}
\begin{proof}
	The first blow-up in the
	resolution of $g$ creates a vertex in $\V'$.
	Therefore, the geodesic in $\Gamma_y$ from an arrowhead to the
	first blow-up in $\Gamma_y$, and the geodesic from
	the same arrowhead in $\Gamin$ to $0 \in \Vmin$ both start as the 
	geodesic
	starting from this arrowhead, going towards $\V'$.
	It follows that
	\[
	\V_\Upsilon \cap (\V_y \setminus \V')
	=
	\V_{\Upsilon_y} \setminus \V'. \qedhere
	\]
\end{proof}

Finally, we are ready to prove the lemma stated at the beginning of this 
subsection.

\begin{proof}[Proof of \Cref{lem:inv}]
	We prove the statement by induction on the cardinality of  $\Vmin$. The 
	base 
	case is when $\Vmin$ consists of $1$ single vertex, which means that 
	the 
	singularity consists of smooth branches (not necessarily reduced) with 
	different tangents. In this case 
	the result follows by \Cref{ex:base_case}.
	Assume now that $\Vmin$ has more than one element. If $i \in \V'$, then 
	the 
	statement follows from \Cref{lem:char_l}.
	Assume, therefore, that $i \notin \V'$. By induction, we have
	$i \in \Upsilon_y$ if and only if $\varpi_{i,y} = 0$.
	Thus, by \Cref{lem:char_l_np}, it suffices to show that
	\begin{equation} \label{eq:varpi_iff}
		\varpi_i = 0,\quad
		\text{if and only if}\quad
		\varpi_{i,y} = 0.
	\end{equation}
	Recall the coordinates $s,t$ in \cref{block:Z_induction}, with
	$(x,y) = (s,st)$.
	Since $c_{0,i,y} = \min\{\ord_{D_i}(s),\ord_{D_i}(t)\}$
	(recall \Cref{rem:vanishing_linear}), we consider two cases.
	
	First, assume that $\ord_{D_i}(s) \leq \ord_{D_i}(t)$.
	Then $c_{0,i,y} = \ord_{D_i}(s) = \ord_{D_i}(x) = c_{0,i}$,
	and
	\[
	\frac{g^{c_{0,i,y}}}{s^{m_{i,y}}}
	=
	\frac{f^{c_{0,i}}}{x^{m_i}}.
	\]
	Using the characterization of \Cref{lem:varpi_char}, the equivalence
	\cref{eq:varpi_iff} follows.
	
	Assume next that $\ord_{D_i}(t) < \ord_{D_i}(s)$.
	Then $c_{0,i,y} = \ord_{D_i}(t) = \ord_{D_i}(y/x)$,
	and $c_{0,i} = \ord_{D_i}(x)$. We find
	\[
	\left(
	\frac{g^{c_{0,i,y}}}{t^{m_{i,y}}}
	\right)^{c_{i,0}}
	=
	\left(
	\frac{f^{\ord_{D_i}(y/x)}}{x^{-m_i} y^{m_i}}
	\right)^{\ord_{D_i}(x)}
	=
	\left(
	\frac{f^{c_{0,i}}}{x^{m_i}}
	\right)^{c_{0,i,y}}
	\cdot
	\left(
	\frac{y^{\ord_{D_i}(x)}}{x^{\ord_{D_i}(y)}}
	\right)^{-m_i}
	\]
	By \Cref{cor:yx_const}, the factor to the right is constant
	along $D_i^\circ$. Since $c_{0,i}$ and $c_{0,i,y}$ are positive
	integers, \cref{eq:varpi_iff} again follows from \Cref{lem:varpi_char}.
\end{proof}

\section{The invariant subgraph of $\Gamma_{\mathrm{pol}}$}

In this subsection we show that the invariant subgraph $\Upsilon$ of 
$\Gamin$ does not
get modified by the extra blow-ups  $\Ypol \to \Ymin$ needed to get to the 
polar resolution (other than the Euler numbers may be lowered).
That is, no invariant edges are split up.
We conclude with \Cref{cor:ups_pol}, which is used in the proofs of
\Cref{lem:transversality_number,lem:potential_xiro}. Our
\Cref{lem:Pmin_Cmin} has been proved in the case of an isolated plane
curve in \cite{LMWcomp}.

For this subsection we choose some $[w] \in \Omega$ that defines a generic 
relative polar curve $P_w$. We recall that we denote its strict transform 
in $\Ymin$ by $\Pmin$.

\begin{lemma} \label{lem:Pmin_edge}
	If $ij$ is an invariant edge in $\Gamin$, then $\Pmin$ does not pass
	through the intersection point $D_i \cap D_j$.
\end{lemma}

\begin{proof}
	We consider coordinates $x,y$  in $\C^2$ so that $\{y=0\}$ is the 
	tangent associated with $D_i$ and such that $\{f_w=0\}= \{f_y =0\}$.
	Take coordinates $u,v$ in a small
	chart $U \subset \Ymin$ near $D_i \cap D_j$ so that
	$x = u^{c_{0,i}} v^{c_{0,j}}$. We can write
	$\pi^*f(u,v) = u^{m_i} v^{m_j}(a + \hot)$ for some $a \in \C^*$
	in these coordinates.
	A chain rule computation gives
	\begin{equation} \label{eq:fyc0}
		f_y
		= \frac{-x_v f_u + x_u f_v}{\det \Jac \pi}
		= \frac{u^{c_{0,i}-1} v^{c_{0,j}-1}}{\det\Jac\pi}
		\cdot
		(-c_{0,j} u f_u + c_{0,i} v f_v).
	\end{equation}
	The first factor on the right hand side does not vanish on
	$U \setminus (D_i \cup D_j)$. But since
	\[
	uf_u = u^{m_i} v^{m_j}(m_i a + \hot),\qquad
	vf_v = u^{m_i} v^{m_j}(m_j a + \hot),
	\]
	we get
	\[
	-c_{0,j} u f_u + c_{0,i} v f_v
	=
	u^{m_i} v^{m_j}
	\left(
	a
	\left|
	\begin{matrix}
		m_i & c_{0,i} \\
		m_j & c_{0,j}
	\end{matrix}
	\right|
	+\hot
	\right)
	\]
	The determinant in this formula does not vanish, by \Cref{lem:c0ord}.
	Indeed, \Cref{lem:c0ord}
	guarantees that this determinant vanishes if and only if
	the branch of $\Gamma$ at $i$ does not meet an arrowhead.
	Now because the edge $ij$ is
	invariant, by \Cref{lem:inv} it is an edge of $\Upsilon$,
	so it lies on a geodesic between $i$ and an arrowhead.
	Therefore, the second factor on the right hand side of \cref{eq:fyc0}
	does not vanish outside $U \cap (D_i \cup D_j)$ either, and so
	the strict transform of the polar curve does not intersect $U$.
\end{proof}

\begin{cor} \label{lem:Pmin_Cmin}
	If $P'$ is a component of $\Pmin$, 
	then either $P' \cap \Cmin = \emptyset$, or $P' \subset \Cmin$.
	The inclusion only happens when $C$ is not reduced.
\end{cor}
\begin{proof}
	As in the previous lemma, we consider coordinates $x,y$  in $\C^2$ so 
	that $\{y=0\}$ is the tangent associated with $D_i$ and such that 
	$\{f_w=0\}= \{f_y =0\}$.
	Let $i\in \V$ and $p\in D_i$, and assume that there is a component $C'$
	of $\Cmin$ passing through $p$.
	Choose a small chart $U$
	containing $p=(0,0)$ with coordinates $u,v$ so that $\pimin^*x = 
	u^{c_{0,i}}$.
	Since $\pimin$ is an embedded resolution,
	we can write $\pi^*f(u,v) = u^{m_i} (au + bv + \hot)^{m'}$, where
	$a\in\C$, $b\in\C^*$, and $m'$ is a positive integer.
	A computation like in the previous proof gives
	\[
	f_y
	= \frac{x_u f_v}{\det\Jac\pi}
	= \frac{c_{0,i} u^{c_{0,i}-1+m_i} m' (b + \hot)}{\det\Jac\pi}
	(au + bv + \hot)^{m'-1}.
	\]
	The fraction on the right hand side does not vanish in $U \setminus 
	D_i$.
	The second factor is $1$ if $m' = 1$, in which case $\Pmin$ does
	not pass through $U$, and vanishes precisely along $C'$ if $m'>1$,
	in which case $\Ppol \cap U = \Cpol \cap U$ as reduced sets.
\end{proof}

\begin{lemma}\label{lem:only_non_invariant}
	The subgraph $\Upsilon \subset \Gamin$ naturally embeds in $\Gapol$
	as the smallest connected subgraph containing the vertex $0$ and
	all $i\in \Vpol$ satisfying $\Cpol \cap D_i \neq \emptyset$.
\end{lemma}

\begin{proof}
	By \Cref{lem:Pmin_edge} and \Cref{lem:Pmin_Cmin}, invariant edges or
	intersection points with $\Cmin$ are not modified by the resolution 
	process
	used to obtain $\Ypol$.
\end{proof}

\begin{lemma} \label{lem:inv_pol}
	If $i \in \Vpol$, then $i \in \V_\Upsilon$ if and only if $\varpi_i = 
	0$.
\end{lemma}

\begin{proof}
	If $i \in \Vmin \subset \Vpol$, then the statement is clear from
	\Cref{lem:inv}. To prove the lemma, it suffices to show that if
	$i \in \Vpol \setminus \Vmin$, then $\varpi_i \neq 0$.
	If $i$ is such a vertex, then $D_i$ is the exceptional divisor of
	a blow-up $(Y_i,D_i) \to (Y_{i-1},q_i)$, where $q_i$ is a point on
	the exceptional divisor in $Y_{i-1}$.
	Let us assume that $q_i$ is an intersection point of $D_j$ and $D_k$.
	The case when $q_i$ lies on a smooth point of the exceptional divisor is
	similar.
	
	By \Cref{rem:resolv_jac} \cref{it:repeat_blow}, $q_i$ is an 
	intersection point of $\Pmin$ and $D$. Therefore, we have
	\[
	p_i > p_j + p_k.
	\]
	By \Cref{lem:Pmin_Cmin}, the strict transform of $C$
	does not pass through $q_i$, and so we have
	\[
	m_i = p_j + p_k.
	\]
	Using \Cref{lem:c_01_recursive}, we find
	\[
	\varpi_i
	=
	c_{1,i} - m_i + p_i
	>
	c_{1,j} - m_j + p_j
	+c_{1,k} - m_k + p_k
	=
	\varpi_j + \varpi_k
	\geq 0. \qedhere
	\]
\end{proof}

\begin{cor} \label{cor:Gapol_ah}
	Let $i,j$ be neighbors in $\Gapol$, with an edge going from $j$ to $i$.
	The following are equivalent:
	\begin{enumerate}
		
		\item \label{it:Gapol_ah_i}
		The vertex $i$ is invariant.
		
		\item \label{it:Gapol_ah_ji}
		The edge $ji$ is invariant.
		
		\item \label{it:Gapol_ah_ah}
		The branch at $i$ contains arrowheads
		(see \Cref{def:Ga_branch} and \Cref{rem:branch_ah}).
		
		\item \label{it:Gapol_ah_m}
		We have
		\[
		\left|
		\begin{matrix}
			c_{0,i} & m_i \\
			c_{0,j} & m_j
		\end{matrix}
		\right|
		\neq 0.
		\]
		
		\item \label{it:Gapol_ah_pos}
		The vector $(m_i, m_j)$ is not a positive rational multiple of
		$(c_{0,i}, c_{0,j})$.
	\end{enumerate}
\end{cor}

\begin{proof}
	By construction of $\Upsilon$, and \Cref{lem:only_non_invariant},
	the branch of $\Gapol$ at $i$ contains arrowheads
	if and only if $i \in \Upsilon$, which, by the previous lemma, is
	equivalent to $\varpi_i = 0$, which is equivalent to
	$\varpi_i = \varpi_j = 0$, proving
	\ref{it:Gapol_ah_i} $\Leftrightarrow$
	\ref{it:Gapol_ah_ji} $\Leftrightarrow$
	\ref{it:Gapol_ah_ah}.
	\Cref{lem:c0ord} gives
	\ref{it:Gapol_ah_ah} $\Leftrightarrow$
	\ref{it:Gapol_ah_m}.
	Finally, \ref{it:Gapol_ah_m} is equivalent to \ref{it:Gapol_ah_pos},
	since both $c_{0,i}$ and $m_i$ are positive.
\end{proof}

\begin{cor} \label{cor:ups_pol}
	Let $i,j$ be vertices in $\Gapol$, with a directed edge $j \to i$. Then
	\[
	\left|
	\begin{matrix}
		c_{0,i} & \varpi_i \\
		c_{0,j} & \varpi_j
	\end{matrix}
	\right|
	\leq 0,
	\]
	with equality if and only if $\varpi_i = \varpi_j = 0$.
\end{cor}

\begin{proof}
	If $\varpi_i = \varpi_j = 0$, then the statement is clear.
	It follows from the construction of $\Upsilon$ and \Cref{lem:inv_pol}, 
	that if $\varpi_j \neq 0$, then $\varpi_i \neq 0$ as well,
	so let us assume that $\varpi_i \neq 0$.
	We have
	\[
	\left|
	\begin{matrix}
		c_{0,i} & \varpi_i \\
		c_{0,j} & \varpi_j
	\end{matrix}
	\right|
	= 
	\left|
	\begin{matrix}
		c_{0,i} & c_{1,i} \\
		c_{0,j} & c_{1,j}
	\end{matrix}
	\right|
	-
	\left|
	\begin{matrix}
		c_{0,i} & m_i \\
		c_{0,j} & m_j
	\end{matrix}
	\right|
	+
	\left|
	\begin{matrix}
		c_{0,i} & p_i \\
		c_{0,j} & p_j
	\end{matrix}
	\right|.
	\]
	The first term on the right hand side is negative by 
	\Cref{def:directed}, and
	the last term is nonpositive by \Cref{lem:c0ord}.
	Since $\varpi_i \neq 0$, the middle term vanishes by 
	\Cref{cor:Gapol_ah} \cref{it:Gapol_ah_m}.
\end{proof}

\begin{rem}
	One can change  $\Ypol$ for any other resolution that arises from 
	blowing 
	up 
	$\Ymin$ either at non-invariant vertices or base points of the polar 
	locus. In 
	this case, all the results in this section are still true.
\end{rem}

\section{The Jacobian ideal in $Y_{\mathrm{min}}$ and polar curves in 
$Y_{\mathrm{pol}}$}

The main result of this subsection is \Cref{lem:polar}, which gives a 
description
of the strict transform of the polar curve in $\Ymin$ and $\Ypol$.
The lemma follows from some of the properties of these modifications
we have seen so far, and a strong result of Fran\c{c}ois Michel
\cite{Mich_Jac}.

Recall (\Cref{not:geodesic_degree}) that if,  $i \in \Vmin$, we denote by 
$\degmin(i)$ the
\emph{degree}
of $i$ in $\Gamin$, i.e. the number of adjacent edges
(this includes edges adjacent to arrowheads). Since $\Upsilon$ is naturally 
embedded both in $\Gamin$ and $\Gapol$, a vertex $i \in \V_\Upsilon$ has a 
well-defined $\degmin(i)$.

\begin{definition}
	Let $\Ndmin \subset \Vmin$ be the set of
\index{node}
\emph{nodes}, i.e. vertices
	with degree $\geq 3$.
	Let $\E\subset \Vmin$ be the set of
\index{end}
\emph{ends}, i.e. vertices of 
	degree $1$.
\end{definition}

This notion of node is used in some classical references as in 
\cite{neuwahl}. 
Sometimes the corresponding divisors are referred to as {\em branching} or 
{\em 
	rupture}.
\begin{lemma}\label{lem:neigh_invariant}
	An invariant vertex has at most one non-invariant neighbor in $\Gamin$.
	Furthermore, a connected component of the complement of $\Upsilon$
	in $\Gamin$ is a bamboo, that is, a string of vertices, as shown in red 
	on the 
	left part of
	\cref{fig:redinv}.
\end{lemma}

\begin{proof}
	This statement can be seen to follow from the explicit process of 
	resolving
	a plane curve. Alternatively, we can use induction,
	see \cref{block:Z_induction} for notation.
	If $\Xi$ is a connected component of $\Gamin \setminus \Upsilon$, then 
	$\Xi$
	is either the induced subgraph on the vertices in
	$\V' \setminus \V_\Upsilon$, or we have $\V_\Xi \subset \Vmin \setminus 
	\V'$.
	In the first case, the connected component $\Xi$ has precisely one 
	neighbor, 
	which
	does not have other non-invariant neighbors.
	In the second case, use induction.
\end{proof}

\begin{block} \label{block:nodes}
	Let $n\in\Ndmin$ be a node, and assume that $n$ has a non-invariant 
	neighbor $i_1$.
	Denote by $i_1,i_2,\ldots,i_h$ the vertices on the connected component 
	of
	$\Gamin \setminus\Upsilon$ containing the $i_1$,
	such that $i_j$ and $i_{j+1}$
	are adjacent for $j=1,\ldots,h-1$, and $i_h$ is an end.
	Let $-b_j = D_{i_j}^2$ be the Euler numbers
	of these vertices. We use the notation
	\[
	[b_1,\ldots,b_h] = b_1 - \frac{1}{b_2 - \frac{1}{\cdots-\frac{1}{b_h}}}
	\]
	for the \emph{negative continued fraction} associated with the sequence
	$b_h,\ldots,b_1$. Recursively define
	\[
	\alpha_j
	=
	\begin{cases}
		1   & j = h,\\
		b_h & j = h-1,\\
		b_{j+1} \alpha_{j+1} - \alpha_{j+2}, & j=0,\ldots,h-2.
	\end{cases}
	\]
\end{block}

\begin{definition}
	Let $n \in \Ndmin \setminus \{0\}$. If $n$ has a non-invariant neighbor 
	in $\Gamin$,
	set $\alpha_n = \alpha_0$, as defined above. Otherwise, set
	$\alpha_n = \infty$.
\end{definition}

\begin{rem}
	If $\alpha_n = \infty$, then it is understood that $1/\alpha_n = 0$.
\end{rem}

\begin{definition}
	Let $n \in \Ndmin$ be a node in $\Gamin$.
	With the above notation, set $\Vmin(n) = \{i_0, i_1, \ldots, i_h\}$ 
	with $i_0 =n$.
	Let $\Vpol(n)$ be the set of vertices in $\Vmin(n)$,
	as well as any vertex appearing when blowing up a point
	in the preimage of the union of the $D_i$ for $i \in \Vmin(n)$. See 
	\cref{fig:redinv}.
\end{definition}

\begin{rem}
	The set $\Vmin(n) \subset \Vmin$
	is the set of vertices in the connected component
	containing $n$ in the graph obtained from
	$\Gamin$ by removing all invariant edges.
	A similar description holds for $\Vpol(n) \subset \Vpol$.
\end{rem}

\begin{figure}[ht]
	\begin{center}
		\includegraphics*{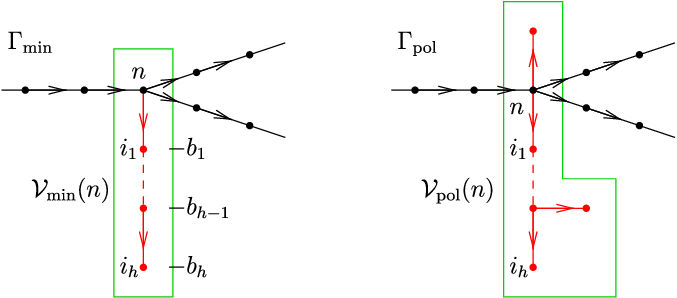}
		\caption{On the right hand side, we see two new vertices appear in
			$\Vpol(n)$.}
		\label{fig:redinv}
	\end{center}
\end{figure}

By recursion, we find that $\alpha_{j-1} > \alpha_j$, and
$\gcd(\alpha_{j+1}, \alpha_j) = 1$, as well as
the formula
\begin{equation} \label{eq:bj_exp}
	[b_j, b_{j+1},\ldots,b_h]
	=
	\frac{\alpha_{j-1}}{\alpha_j}.
\end{equation}
Indeed, by induction
\[
[b_j, b_{j+1},\ldots,b_h]
=
b_j - \frac{1}{[b_{j+1},\ldots,b_h]}
=
b_j - \frac{\alpha_{j+1}}{\alpha_j}
=
\frac{b_j \alpha_j - \alpha_{j-1}}{\alpha_j}
=
\frac{\alpha_{j-1}}{\alpha_j}.
\]

With this in mind, we have the following lemma.

\begin{lemma} \label{lem:alpha_ratio}
	For $j=0,\ldots,h$, we have
	\[
	\alpha_j
	= \frac{m_{i_j}}{m_{i_h}}
	= \frac{c_{0,i_j}}{c_{0,i_h}}.
	\]
\end{lemma}
\begin{proof}
	Since the edges $i_ji_{j+1}$ are non-invariant, it suffices, by
	\Cref{cor:Gapol_ah}, to prove the first equality.
	
	Since the strict transform $\Cmin$ of the curve defined by $f$ does not 
	intersect
	$D_{i_j}$ for $j>0$, and the divisor $\sum_i m_i D_i + \Cmin = 
	(\pimin^*f)$ is principal,
	we find
	\[
	0 = (D_{i_j},\pimin^*f) =
	\begin{cases}
		m_{i_{h-1}} - b_{h}m_{i_h}, & j = h, \\
		m_{i_{j+1}} - b_{j}m_{i_j} + m_{i_{j-1}}, & 1 \leq j < h.
	\end{cases}
	\]
	Clearly, $\alpha_h = 1 = m_{i_h} / m_{i_h}$. The case $j=h$ above
	gives $\alpha_{h-1} = b_h = m_{i_{h-1}} / m_{i_h}$. For $j < h-1$, 
	induction gives
	\[
	\frac{m_{i_j}}{m_{i_h}}
	=
	b_{j+1} \frac{m_{i_{j+1}}}{m_{i_h}} - \frac{m_{i_{j+2}}}{m_{i_h}}
	=
	\alpha_{j+1}\left(b_{j+1} - \frac{1}{[b_{j+2},\ldots,b_h]}\right)
	=
	\alpha_{j}. \qedhere
	\] 
\end{proof}

We are ready to state the main result of this section whose proof is 
postponed to the end.
\begin{lemma} \label{lem:polar}
	With the notation introduced above, the following hold.
	\begin{enumerate}
		\item \label{it:polar_0}
		The Jacobian ideal has no base points on $D_0$.
		If $t$ is the number of tangents of $C$, then
		\[
		(\Ppol, D_0)
		= (\Pmin, D_0)
		= \degmin(0) - 1 = t - 1.
		\]
		
		\item \label{it:polar_bamboo}
		If $i \in \V_\Upsilon \setminus \Ndmin$, then
		\[
		D_i \cap \Ppol = D_i \cap \Pmin = \emptyset.
		\]
		
		\item \label{it:polar_node}
		If $n \in \Ndmin$, then, with the notation as in \cref{block:nodes},
		\begin{equation} \label{eq:polar_node}
			\sum_{i\in\Vpol(n)} (\Ppol,m_i D_i)
			= \sum_{j=0}^h (\Pmin, m_{i_j}D_{i_j})
			= m_n\left(\degmin(n) - 2 - \frac{1}{\alpha_n}\right).
		\end{equation}
		Furthermore, every intersection point in $\Pmin \cap D^\circ_n$ is
		a base point of the Jacobian ideal.
	\end{enumerate}
\end{lemma}

\begin{example}
	The curve singularities defined by the polynomials
	\[
	f(x,y) = y^4 + x^6 + x^5y,\qquad
	g(x,y) = y^4 + x^3y^2 + x^6 + x^5y.
	\]
	at the origin have the same topological type. In fact, they are both
	Newton nondegenerate, and have the same Newton polyhedron.
	Thus, $f$ and $g$ are resolved by a toric morphism $\pi_\triangle$,
	as described in \Cref{ex:toric}. The Newton polyhedron has one
	compact face, with normal vector $(2,3)$.
	
	The Newton polyhedron of a generic partial $af_x + bf_y$
	has two compact faces, one with normal vector $(2,3)$, and
	one with normal vector $(1,2)$.
	
	The Newton polyhedron of a generic partial $ag_x + bg_y$
	has one compact face with normal vector $(3,5)$.
	
	To the left of \cref{fig:4_6_ex}, we see the minimal resolution of $f$,
	with two red arrows representing the strict transform of the polar 
	curve,
	as well as, we see the same resolution, with one additional
	blow-up, giving an embedded resolution of a generic polar curve of $g$.

	\begin{figure}[ht]
		\begin{center}
			\includegraphics*{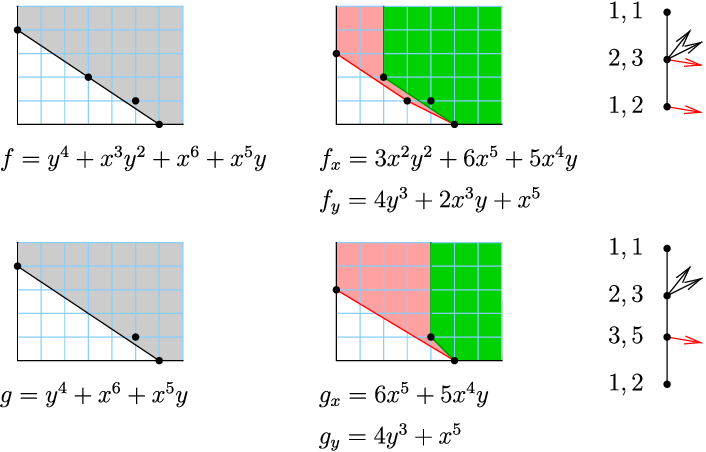}
			\caption{Left and center: Newton polyhedrons of $f$ and $g$ and 
				their partials.
				To the right: resolution graphs showing the generic
				polar curve in red.}
			\label{fig:4_6_ex}
		\end{center}
	\end{figure}
\end{example}

\begin{lemma} \label{lem:xfx}
	Use the notation introduced in \cref{block:Z_induction}.
	If $i \in \V_y \cap \V_\Upsilon$, then
	\[
	\ord_{D_i}(x f_x) \geq \ord_{D_i}(y f_y).
	\]
	Equality holds,
	unless $i \in \V'$ and $\init_i(f)$ does not depend on $x$,
	i.e. $\init_i(f) = by^m$ for some $b\in\C^*$ and $m\in \Z_{>0}$.
\end{lemma}
\begin{proof}
	
	\textbf{Case 1.} Assume first that $i\in \V'$. Since we assumed that 
	$i\in 
	\V_\Upsilon$, that means by \Cref{lem:char_l} that $\init_i(f)$ does 
	depend on 
	$y$, so
	\[
	\ord_{D_i}(y f_y) = \ord_{D_i}(f) \leq \ord_{D_i}(xf_x).
	\]
	If $\init_i(f)$ does depend on $x$, then the above inequality is 
	actually an equality.
	
	\textbf{Case 2.} Assume next that $i \in \V_y \setminus \V'$. Since $i 
	\in \V_\Upsilon$,
	$i$ is an invariant vertex, i.e. we have $\varpi_i = 0$, by
	\Cref{lem:inv}. If we take coordinates $u,v$ in a neighborhood
	of some point $p \in D_i^\circ$, such that $\pimin^*x = u^{c_{0,i}}$,
	then the chain rule gives, since $x_v = 0$,
	\begin{equation} \label{eq:partials_chain}
		\begin{aligned}
			f_u &= x_u f_x + y_u f_y,\qquad
			&f_x &= \frac{f_u}{x_u} - \frac{y_u f_v}{x_u y_v}, \\
			f_v &= y_v f_y,
			&f_y &= \frac{f_v}{y_v}.
		\end{aligned}
	\end{equation}
	Since $f$ and $x$ vanish along $D_i$, we find
	\[
	\ord_{D_i}(f_u) = m_i - 1,\quad
	\ord_{D_i}(x_u) = \ord_{D_i}(x) - 1,\quad
	\ord_{D_i}\left(x \frac{f_u}{x_u}\right) = m_i.
	\]
	By \Cref{lem:ord_y}, we have $\ord_{D_i}(y) < c_{1,i}$, and so
	$\ord_{D_i}(y_u) < c_{1,i}-1$.
	By \cref{block:c_1} we have $\ord_{D_i}(y_v) = c_{1,i}$. 
	By \Cref{lem:inv,lem:varpi}, we have $\ord_{D_i}(f_v) = m_i$.
	Put together,
	\[
	\ord_{D_i}\left( x \frac{y_u f_v}{x_u y_v} \right)
	<
	c_{0,i} + (c_{1,i}-1) + m_i - (c_{0,i}-1) - c_{1,i}
	=
	m_i.
	\]
	These two lines, along with \cref{eq:partials_chain}, give
	\[
	\ord_{D_i}(x f_x)
	=
	\ord_{D_i}\left( x \frac{y_u f_v}{x_u y_v} \right)
	=
	m_i - c_{1,i} + \ord_{D_i}(y).
	\]
	Similarly,
	\[
	\ord_{D_i}(y f_y)
	=
	\ord_{D_i}\left(\frac{y f_v}{y_v}\right)
	= m_i - c_{1,i} + \ord_{D_i}(y).\qedhere
	\]
\end{proof}

\begin{cor}\label{cor:strict_ine}
	Use the notation introduced in \cref{block:Z_induction}.
	If $i\in\V_y$ is an invariant vertex, and $i\neq 0$,
	then $\ord_{D_i}(f_y) < \ord_{D_i}(f_x)$.
\end{cor}
\begin{proof}
	If $i \in \V'$ and $\init_i(f) = b y^m$, with $b\in \C^*$, then
	\[
	\ord_{D_i}(f_y)
	= \wt_i(y^{m-1})
	< \wt_i(x^{-1}y^m)
	< \ord_{D_i}(f_x).
	\]
	We are using that by  \Cref{blc:yprime}, {\em orders} can be computed 
	as {\em weights}. Also, the last strict inequality follows since the
	operator  $\partial_x$  kills the initial part $\init_i(f) = b y^m$.
	
	Otherwise, by \Cref{lem:xfx}, we find
	\[
	\ord_{D_i}(f_y)
	= \ord_{D_i}(f_x) + \ord_{D_i}(x) - \ord_{D_i}{y}
	< \ord_{D_i}(f_x).\qedhere
	\]
\end{proof}

\begin{cor}
	If $i \in \V_\Upsilon \setminus \{0\}$
	and $p \in D_i^\circ \cap \Pmin$, then $p$ is
	a base point of the Jacobian ideal $(f_x, f_y)$.
\end{cor}

\begin{proof}
	We can assume that $i \in \V_y \cap \V_\Upsilon$.
	Let $p$ be an intersection point of $\Pmin$ and $D_i$.
	Since $\ord_{D_i}(f_y) < \ord_{D_i}(f_x)$ by \Cref{cor:strict_ine}, the 
	strict transform of
	the curve $af_x + f_y = 0$ passes through $p$ for any $a \in \C$.
\end{proof}

\begin{proof}[Proof of \Cref{lem:polar}]
	\cref{it:polar_0} was proved in \Cref{lem:polar_comp_0} and the fact 
	that, by 
	\Cref{rem:tangents}, the Jacobian ideal has no base points on $D_0$.
	
	The statements \cref{it:polar_bamboo} and \cref{it:polar_node} follow 
	from
	the work of Michel \cite{Mich_Jac}, as follows.
	Assume that $x$ is any generic
	linear function $\C^2 \to \C$, and consider the finite morphism
	$(f,x):(\C^2,0) \to (\C^2,0)$.
	The polar curve $P$ is called the \emph{Jacobian curve} in 
	\cite{Mich_Jac},
	and the
\index{Hironaka number}
\emph{Hironaka number}
of $i\in\Vmin$ is defined as
	\[
	q_i
	= \frac{\ord_{D_i}(f)}{\ord_{D_i}(x)}
	= \frac{m_i}{c_{0,i}}.
	\]
	By \Cref{lem:c0ord} and \Cref{cor:Gapol_ah},
	the Hironaka number is strictly increasing along invariant edges, 
	and constant along non-invariant edges, in either $\Gamin$ or $\Gapol$.
	Thus, if $i \in \Upsilon$ and
	$\degmin(i) = 2$, then the vertex $i$ has no neighbor with the same 
	Hironaka number, and thus by \cite[Theorem 4.9]{Mich_Jac}, the polar
	curve $\Pmin$ does not pass through $D_i$, proving
	the second equality in \cref{it:polar_bamboo}.
	The first equality follows immediately, since the modification
	$\Ypol\to\Ymin$ therefore does not modify $D_i$.
	
	If $n \in \Ndmin$, then, using the notation in \cref{block:nodes}, the
	set $\{n,i_1,\ldots,i_h\}$ is a \emph{rupture zone}, as defined in
	\cite[Definitions 1.4]{Mich_Jac}, i.e. a connected component
	in $\Gamin$ of the induced subgraph of $\Gamin$
	on the vertices with Hironaka number $q_n$.
	Using $i_0 = n$ as before, then
	\cite[Theorem 4.9]{Mich_Jac} states that
	\begin{equation} \label{eq:Michel_statement}
		\sum_\gamma V_f(\gamma)
		=
		-\sum_{j=0}^h m_{i_j} \chi(D_{i_j}^\circ).
	\end{equation}
	Here, $\gamma$ runs through the connected components of $\Pmin$ which 
	pass 
	through
	a point $p \in D_i$ for some $i \in \Vmin(n)$, and if
	$\nu:(\C,0)\to(\gamma,p)$ is the normalization of such a branch, then
	$V_f(\gamma) = \ord_t(\nu^* \pimin^*f(t))$ where $t$ is a coordinate in 
	$\C$.
	As a result, we have
	\[
	V_f(\gamma)
	= \left(\gamma, (\pimin^*f)\right)
	= \sum_{i\in\Vmin} (\gamma, D_i),
	\]
	since $\gamma$ does not intersect the strict transform $\Cmin$, by
	\Cref{lem:Pmin_Cmin}.
	A similar statement holds for $\pipol$.
	The terms on the right hand side of \cref{eq:Michel_statement}
	all vanish, except for $j=0$ and $j=h$, and
	we have
	\[
	-\chi(D_{i_0}^\circ) = -\chi(D_n^\circ) = \degmin(n) - 2,\qquad
	-\chi(D_{i_h}^\circ) = \degmin(i_h) - 2 = -1.
	\]
	Furthermore, $m_{i_h} = m_n / \alpha_n$ by \Cref{lem:alpha_ratio}. Now
	\cref{eq:polar_node} follows, since
	\[
	\left(\Pmin, \sum_{j=0}^h m_{i_j} D_{i_j} \right)
	= \sum_\gamma V_f(\gamma)
	= \sum_{i\in\Vpol(n)} (\Ppol,m_i D_i).\qedhere
	\]
\end{proof}

%
%
%
	\chapter{The $1$st Blow-up}
\label{s:1st_blowup}
In this section we extend the pullback of the vector field $\xi$ to the 
exceptional divisor of the $1$st blow up
\index{$\pi_0$}
$\pi_0:Y_0 \to \Tu$.
As we will 
notice, this construction cannot be further continued to divisors of 
subsequent 
blow-ups since, in general, the vector field is not invariant in the sense 
of 
\Cref{def:equivariant}.

After that, we also extend the vector field to the boundary of the real 
oriented blow up of $Y_0$ along $D_0$. We introduce some genericity 
conditions 
on the metric (but not on the defining function $f$) in order for these 
extensions to admit 
potentials that are Morse functions.

\section{Special coordinates}

For this subsection fix the standard metric on $\C^2$. Let $x,y$ be 
\index{isometric coordinates}
isometric linear coordinates
(\Cref{def:isometric_coords}) on $\C^2$ with 
respect to some $A \in U(2)$ such that $\{x=0\}$ is not a tangent of 
$(C,0)$. In the $x,y$ coordinates the vector field $\xi$, by 
\Cref{lem:cplex_gradient} is represented by the column vector 
\begin{equation}\label{eq:xi_wrt_ga}
	-\nabla^{g_{\mathrm{std}}} \log |f| = -\left(\frac{\bar{f}_x}{\bar{f}}, 
	\frac{\bar{f}_y}{\bar{f}}\right)^\top.
\end{equation}

Let $\pi_0:Y_0 \to \C^2$ be the first blow up. The vector field $\pi_0^* 
\xi$ is a vector field well defined on $Y_0 \setminus \pi_0^{-1}(C)$. In 
this section we re-scale the vector field by a positive real function so 
that it can be extended over the exceptional divisor $D^\circ_0$. Let $U 
\subset Y_0$ be the standard chart with coordinates $u,v$ such that 
$\pi_0(u,v)= (u, uv)$ on $U$. Then the Jacobian of $\pi_0$ in those 
coordinates is given by:
\[
\Jac \pi_0 = 
\left(
\begin{matrix}
	x_u & x_v \\
	y_u & y_v
\end{matrix}
\right) = \left(
\begin{matrix}
	1 & 0 \\
	v & u
\end{matrix}
\right).
\] 
Here we are using the notation $x_u = \partial_u \pi^*_0 x(u,v)$ and 
similar for the other variables. 

In particular, $\det\Jac\pi_0 = u $ and the expression for $\pi^*_0\xi$ 
leaves us with 
\begin{equation} \label{eq:xi_coordinates_0}
	\pi^*_0\xi = (\Jac\pi_0)^{-1} \cdot \xi
	= \frac{-1}{u \bar f}
	\left(
	\begin{matrix}
		u \bar f_x  \\
		-v \bar f_x + \bar f_y
	\end{matrix}
	\right) = 	\left(
	\begin{matrix}
		\xi^u  \\
		\xi^v
	\end{matrix}
	\right).
\end{equation}
The last equality defines the real analytic functions $\xi^u$ and $\xi^v$ 
on $U \setminus D_0$ as the coordinates of the vector field $\pi^*_0\xi$.

In this chart the divisor $D_0$ is given by $\{u=0\}$ and $\pi^*_0 \xi$ has 
a pole along it.  Next we compute the order of this pole. 

\begin{lemma}\label{lem:pole_order} 
	$\ord_{u} \xi^u \geq - 1 $ and $\ord_{u} \xi^v = - 2$. 
\end{lemma}

\begin{proof}
	The first inequality is true regardless of the chosen coordinates. 
	\[
	\ord_{u} \xi^u = \ord_u \frac{u \bar{f}_x}{u \bar{f}} = \ord_u 
	\frac{\bar{f}_x}{\bar{f}} \geq -1.
	\]
	
	For the second equality, we use that we have chosen special coordinates 
	as in the beginning of this subsection such that $\{x=0\}$ is not a 
	tangent of $(C,0)$. This implies that $\init (\bar{f})$ contains the 
	monomial $\bar{y}^{e}$ with a non-zero coefficient, where $e$ is the 
	multiplicity of $(C,0)$. So $\bar{f}_y$ contains the monomial 
	$\bar{y}^{e-1}$. This gives us that $\ord_u \bar{f}_y = e-1$. We 
	observe that $\ord_u (-v\bar{f}_x) = \ord_u \bar{f}_x \geq e-1$ as the 
	first inequality of the lemma shows. Also, $\init (v\bar{f}_x)$ does 
	not contain antiholomorphic monomials, so the initial parts of 
	$-v\bar{f}_x$ and $\bar{f}_y$ do not cancel. We conclude that $\ord_u 
	(-v \bar{f}_x + \bar{f}_y ) = \ord_u \bar{f}_y = e-1$. Finally, since 
	$\ord_u (u \bar{f})=e+1$ and $e-1 - (e+1) =-2$ we get:
	\[
	\ord_{u} \xi^v = \ord_u \frac{-v \bar{f_x} + \bar{f}_y}{u \bar{f}}  = 
	-2. \qedhere
	\]
\end{proof}

\section{Scaling and extending the vector field}

Recall that we denote by $\tilde{C} \subset Y_0$ the strict transform of $C 
\subset \Tu$.
\begin{lemma}\label{lem:exist_xi_0}
	There exists a unique vector field on $Y_0 \setminus \tilde C$ which is 
	equal to 
	\[\pi^*_0 \left(\delta \cdot \xi \right) 
	\]
	outside $D_0^\circ$, where $\delta(x,y) = |x|^2+|y|^2$. Furthermore, it 
	is tangent to $D_0$ and does not vanish identically along $D_0^\circ$.
\end{lemma}

\begin{proof}
	With the previous notation, $\pi^*_0 \delta = |u|^2(1+ |v|^2)$ has 
	vanishing order $2$ along $D_0$. Since $(1+ |v|^2)$ is an unit,  it 
	suffices to show that the functions $|u|^2\xi^u$ and $|u|^2\xi^v$ 
	extend analytically across $D^\circ_0$.
	
	We can express each of these two functions explicitly. First we have
	\[
	|u|^2\xi^u = |u|^2 \frac{\bar{f}_x}{\bar{f}}= u \frac{\bar{u} 
		\bar{f}_x}{\bar{f}}.
	\]
	We observe that $\frac{\bar{u} \bar{f}_x}{\bar{f}}$ is an 
	anti-holomorphic function because, by \Cref{lem:pole_order}, $\ord_u 
	\xi^u \geq -1$ and so $\ord_u \frac{\bar{u} \bar{f}_x}{\bar{f}} \geq 
	0$. We conclude that  $u \frac{\bar{u} \bar{f}_x}{\bar{f}}$ is real 
	analytic and that $\ord_u u \frac{\bar{u} \bar{f}_x}{\bar{f}} \geq 1$ 
	which means that the extension of $\xi^u$ vanishes along $D_0$.
	
	Secondly, we have 
	\[
	|u|^2\xi^v =  |u|^2\frac{-v \bar{f_x} + \bar{f}_y}{u \bar{f}} = \bar{u} 
	\frac{-v \bar{f_x} + \bar{f}_y}{\bar{f}} = -v \frac{\bar{u} 
		\bar{f}_x}{\bar{f}} + \frac{\bar{u} \bar{f}_y}{\bar{f}}.
	\]
	
	Similarly as in the previous case, the function $|u|^2\xi^v$ extends as 
	a real analytic function over $D_0^\circ$. Furthermore, by 
	\Cref{lem:pole_order} $\ord_u |u|^2\xi^v = 2-2= 0$. This means that the 
	extension of $|u|^2\xi^v$ over $D_0^\circ$ is not identically $0$.
\end{proof}

\begin{definition}\label{def:xi_0}
	Let $\hat\xi_0$ be the vector field on $Y_0 \setminus \tilde{C}$ 
	described by \Cref{lem:exist_xi_0} and let 
\index{$\xi_0$}
	$\xi_0=\hat\xi_0|_{D_0^\circ}$ be its restriction.
\end{definition}

One of the goals of this section is to prove 
(\Cref{thm:restrict_xi_0_morse}) that, except for a measure $0$ set in the 
set of Hermitian metrics,  $\hat\xi_0$ is an elementary vector field 
(recall \Cref{def:elementary_vf}).

\begin{rem}\label{rem:extension_pullback}
	Let $\piro_0$ be the composition
	\[
	\piro_0=\pi_0 \circ \sigma_0: \Yro_0 \to \Tu
	\] 
	of the blow-up at the origin and the real oriented blow up of $Y_0$ 
	along $D_0$. 
	
	Then, we can also extend $(\piro_0)^* (\delta \cdot \xi)$ to a vector 
	field on 
	all $\Yro_0 \setminus \sigma_0^{-1}(\tilde{C})$ (recall 
	\Cref{lem:arg_equivariant}). Equivalently, we can simply take the 
	pullback of 
	$\hat\xi_0$ by $\sigma_0$.
\end{rem}

\begin{definition}\label{def:xiro0}
	Let $\hxiro_0 = \sigma_0^* \hat\xi_0$ be the pullback of the previously 
	constructed vector field to the real oriented blow up outside the 
	strict transform $\Yro_0 \setminus \sigma_0^{-1}(\tilde{C})$. Let 
	$\xiro_{0,\theta} = \hxiro_0|_{\Droc_{0,\theta}}$ for $\theta \in \R / 
	2 \pi \Z $.
\end{definition}

\section{$\hat\xi_0$ in the normal direction}

Let $p\in D^\circ_0$ be a zero of $\xi_0$. In this subsection, we change 
conventions with respect to the previous subsection and consider a chart 
$V\subset Y_0$ with coordinates $u,v$ in such a way that $p=(0,0)$ and 
$\{u=0\}$ defines $D_0 \cap V$. In other words, $p$ corresponds to the line 
$\{x=0\} \subset \C^2$. In these coordinates the blow up is expressed as  
$\pi(u,v)= (uv,u)$.

In these coordinates we have 
\[\Jac \pi_0 = 	\left(
\begin{matrix}
	v &  u  \\
	1 & 0
\end{matrix}
\right)  \]
and so 
\[\Jac \pi_0^{-1} = 	\frac{1}{-u}\left(
\begin{matrix}
	0 & -u  \\
	-1 &  v
\end{matrix}
\right).  \]
In particular, we find that $\xi^u = - \bar{f}_y / \bar{f}$.

Since $\hat\xi_0$ is tangent to $D_0^\circ$, the Hessian of $\hat\xi_0$ at 
$p$ has $T_pD_0$ as invariant subspace. Therefore, the spectrum of 
$\mathrm{Hess}_p (\hat\xi_0)$ is the union of the spectrum of its 
restriction of  to $T_pD_0$ and the spectrum of the induced operator on the 
quotient $T_p Y_0 / T_pD_0$.

Observe that $\hat\xi_0$ only has zeroes on $D_0$. Therefore, by the 
starting 
paragraph of this subsection, it is enough to show that the parts of the 
Hessian corresponding to $T_p Y_0 / T_pD_0$ and to $T_p D_0$ are elementary 
for 
each zero $p \in D_0^\circ$.

\begin{lemma}\label{lem:sink_in_normal}
	Let $p\in D^\circ_0$ be a singularity of $\hat\xi_0$. Then 
	$\mathrm{Hess}_p(\xi_0)$ has two negative eigenvalues in the part 
	corresponding to  $T_p Y_0 / T_pD_0$. 
\end{lemma}
\begin{proof}

	Using the coordinates described at the beginning of this section, we 
	have 
	that $p= (0,0)$.  We also have
	\[\pi_0^* f = au^e + \mathrm{h.o.t.}\]
	with $a\neq0$ some complex number. Therefore,
	\[\pi_0^* f_y = a e u^{e-1} + \mathrm{h.o.t.},\]
	and, 
	\[\frac{\bar{u} \bar{f}_y}{\bar{f}}= e+ \mathrm{h.o.t.}\]
	Finally,
	\[|u|^2 \xi^u = - u \frac{\bar{u} \bar{f}_y}{\bar{f}}= -e u + 
	\mathrm{h.o.t.}\]
	Since $\hat\xi_0$ is tangent to $D_0^\circ$, we find that its 
	linearization matrix, is of the form
	\[
	\left(
	\begin{matrix}
		-e & 0 & \multicolumn{2}{c}{\multirow{2}{*}{\scalebox{2}{$0$}}} \\
		0 & -e & \multicolumn{2}{c}{} \\
		\multicolumn{2}{c}{\multirow{2}{*}{\scalebox{2}{$\bigast$}}}
		& 	\multicolumn{2}{c}{\multirow{2}{*}{\scalebox{2}{$\bigast$}}} \\
		\multicolumn{2}{c}{}
		& \multicolumn{2}{c}{}
	\end{matrix}
	\right). \qedhere
	\]
\end{proof}

\section{$\hat\xi_0$ in the tangent direction}

Consider the space of constant Hermitian metrics $GL(2,\C) / U(2)$ of 
$\C^2$. This space is a real manifold of dimension $4$. We identify a small 
neighborhood of the identity matrix $\id \in GL(2,\C) / U(2)$ with a small 
neighborhood
\index{$\mathcal{H}$}
$\mathcal{H} \subset \R^4$  of $(1,0,0,1) \in \R^4$ such that 
for all points $(a,b,c,d) \in \mathcal{H}$, the matrix
\begin{equation*}
	H=
	\begin{pmatrix}
		a &	b+ci \\
		b-ci &	d 
	\end{pmatrix}
\end{equation*}
is a Hermitian inner product. Finally, let $\delta_H$ be the squared norm 
with respect to the metric $H$, that is $\delta_H(x,y)= \irow{\bar{x} & 
	\bar{y}} H \icol{x\\y}$. In this setting, we define the family of real 
valued functions:

\begin{align*}
	\Phi:\mathcal{H} \times \C^2 \setminus \{C\} &\to \R\\
	(a,b,c,d,x,y) &\mapsto - \log 
	\frac{|f(x,y)|}{\left(\delta_H(x,y)\right)^{e/2}}.
\end{align*}

Assume that one of the tangents of $(C,0)$ is $\{y=0\}$. We can always do 
so up to a linear change of coordinates that preserves the Hermitian metric 
(simply a complex rotation). We take a standard chart of $Y_0$ with 
coordinates $u,v$ such that  $D_0^\circ \subset Y \setminus \tilde{C}$ is 
defined by $u=0$ and $v$ is a coordinate for it. As usual, in these 
coordinates $\pi_0$ is given by $\pi_0(u,v)= (u,uv)$. Now we pullback the 
family $\Phi$ to  $\mathcal{H} \times Y_0 \setminus \pi_0^{-1}(C)$ and 
prove that it actually extends over $Y_0 \setminus \tilde{C}$. First we 
observe that 
\[\pi_0^*\delta_H(u,v) = |u|^2\left({\left(b + i  c\right)} v + {\left(d v 
	+ b - i  c\right)} \overline{v} + a\right).
\]
Then, writing $\pi^*_0f(u,v)=u^e(\tilde{f}(u,v))$, a direct computation 
yields:

\begin{equation}\label{eq:phi_0}
	\begin{split}
		- \log \frac{|\pi^*_0f(u,v)|}{\left(\pi^*_0 
			\delta_H(u,v)\right)^{\frac{e}{2}}}   
		= & - \log |\pi^*_0f(u,v)| +  \log{\left(\pi^*_0 
			\delta(u,v)\right)^{\frac{e}{2}}} \\
		= &-  \log |\tilde{f}(u,v)| -  \log |u|^e + \log{\pi^*_0 
			\delta(u,v)}^{\frac{e}{2}}  \\
		= &  -  \log |\tilde{f}(u,v)| -   e \log |u| +  e \log |u| \\
		& +  \frac{e}{2} \log\left({\left(b + i  c\right)} v + {\left(d v + 
			b - i  c\right)} \overline{v} + a\right) \\
		=& - \log |\tilde{f}(u,v)| + \frac{e}{2} \log\left({\left(b + i  
			c\right)} v + {\left(d v + b - i  c\right)} \overline{v} + a\right).
	\end{split}
\end{equation}
Let $g(v)= \tilde{f}(0,v)$ be the restriction of $\tilde{f}$ to 
$D_0^\circ$. Let $s,t$ be real coordinates for $D_0^\circ$, that is, 
$v=s+ti$. Thus, we have constructed a well defined family of functions
\begin{align*}
	\Phi_0:\mathcal{H} \times D_0^\circ &\to \R\\
	(a,b,c,d,s,t) &\mapsto -  \log |g(v)| + \frac{e}{2}  \log\left(d s^{2} 
	+ d t^{2} + 2  b s - 2  c t + a\right).
\end{align*}
For each $(a,b,c,d) = H \in \mathcal{H}$, we define
\begin{align*}
	\phi_{H,0}:D_0^\circ &\to \R\\
	(s,t) &\mapsto \Phi_0(a,b,c,d,s,t).
\end{align*}
\begin{lemma}\label{lem:construction_morse_function}
	There exists a Lebesgue measure $0$ set $\mathcal{H}_0$ in the space 
	$\mathcal{H}$, such that for all  $H \in \mathcal{H} \setminus 
	\mathcal{H}_0$ the function 
	\[
	\phi_{H,0}:D_0^\circ \to \R
	\]
	is Morse.
\end{lemma}

\begin{proof}
	The theorem follows after checking that the hypotheses of  
	\cite[Theorem 1.2.1]{Nic} are satisfied by our family  $\Phi_0$. More 
	concretely, we only need to check that
	\begin{equation*}
		\rk 
		\begin{pmatrix}
			(\Phi_0)_{as} &	(\Phi_0)_{bs} &	(\Phi_0)_{cs} &	(\Phi_0)_{ds} \\
			(\Phi_0)_{at} &	(\Phi_0)_{bt} &	(\Phi_0)_{ct} &	(\Phi_0)_{dt}
		\end{pmatrix}
		=2
	\end{equation*}
	for all $(a,b,c,d,s,t) \in \mathcal{H} \times D_0^\circ$. For 
	convenience of notation, let 
	\[
	h(a,b,c,d,s,t) =  d s^{2} + d t^{2} + 2  b s - 2  c t + a.
	\] We observe that $g(v)$ does not depend on $a,b,c$ or $d$, so we find 
	that
	\begin{align*}
		& \phantom{\frac{e}{2}}\begin{pmatrix}
			(\Phi_0)_{as} &	(\Phi_0)_{bs} &	(\Phi_0)_{cs} &	(\Phi_0)_{ds} \\
			(\Phi_0)_{at} &	(\Phi_0)_{bt} &	(\Phi_0)_{ct} &	(\Phi_0)_{dt}
		\end{pmatrix}
		\\
		=& \frac{e}{2} \begin{pmatrix}
			\frac{h_{as}h - h_ah_s}{h^2} &	\frac{h_{bs}h - h_bh_s}{h^2} &
			\frac{h_{cs}h - h_ch_s}{h^2} &	\frac{h_{ds}h - h_dh_s}{h^2} \\
			\frac{h_{at}h - h_ah_t}{h^2} &	\frac{h_{bt}h - h_bh_t}{h^2} &	
			\frac{h_{ct}h - h_ch_t}{h^2} &	\frac{h_{ct}h - h_ch_t}{h^2}
		\end{pmatrix}
		\\
		=& \frac{e}{2} \begin{pmatrix}
			\frac{- h_s}{h^2} &	\frac{h_{bs}h - h_bh_s}{h^2} &
			\frac{h_{cs} - h_ch_s}{h^2} &	\frac{h_{ds}h - h_dh_s}{h^2} \\
			\frac{- h_t}{h^2} &	\frac{h_{bt}h - h_bh_t}{h^2}&	
			\frac{h_{ct}h - h_ch_t}{h^2} &	\frac{h_{ct}h - h_ch_t}{h^2}
		\end{pmatrix}
	\end{align*}
	Where the last equality follows from the fact that $h_a=1$. We are also 
	using that $h$ is never $0$ in order to divide by $h^2$.
	
	We perform elementary column 
	operations on the last matrix of the above equation and get:
	\begin{align*}
		&\rk 
		\begin{pmatrix}
			\frac{- h_s}{h^2} &	\frac{h_{bs}h - h_bh_s}{h^2} &
			\frac{h_{cs}h - h_ch_s}{h^2} &	\frac{h_{ds}h - h_dh_s}{h^2} \\
			\frac{- h_t}{h^2} &	\frac{h_{bt}h - h_bh_t}{h^2}&	
			\frac{h_{ct}h - h_ch_t}{h^2} &	\frac{h_{ct}h - h_ch_t}{h^2}
		\end{pmatrix}
		\\
		= &\rk
		\begin{pmatrix}
			- h_s&	h_{bs}h - h_bh_s &
			h_{cs}h - h_ch_s &	h_{ds}h - h_dh_s \\
			- h_t &	h_{bt}h - h_bh_t&	
			h_{ct}h - h_ch_t &	h_{ct}h - h_ch_t
		\end{pmatrix}
		\\
		= &\rk
		\begin{pmatrix}
			- h_s &	h_{bs}h  &
			h_{cs}h &	h_{ds}h \\
			- h_t &	h_{bt}h &	
			h_{ct}h  &	h_{ct} h
		\end{pmatrix}
		\\
		= &\rk
		\begin{pmatrix}
			- h_s &	h_{bs}  &
			h_{cs} &	h_{ds}\\
			- h_t &	h_{bt} &	
			h_{ct} &	h_{ct}
		\end{pmatrix}
		\\
		= &\rk
		\begin{pmatrix}
			- h_s &	2  &
			0 &	h_{ds} \\
			- h_t &	0 &	
			-2  &	h_{ct} 
		\end{pmatrix} 
		=  \, 2 
	\end{align*}
	The cited criterion, \cite[Theorem 1.2.1]{Nic} gives exactly the 
	statement of this lemma.
\end{proof}

\begin{definition} \label{def:grad_0}
	For $(a,b,c,d)=(1,0,0,1)$, that is, for the standard Hermitian metric,
	we denote 
\index{$\phi_0$}
	\[
	\phi_0=\phi_{\id, 0}.
	\]
\end{definition}
\begin{lemma}\label{lem:xi_0_gradient}
	The vector field $\xi_0$ is the gradient of  $\phi_0(v)$ with respect 
	to the
\index{metric!Fubini-Study}
Fubini-Study metric
on $D_0^\circ \subset \CP^1.$
\end{lemma}

\begin{proof}
	We make a few observations on the second coordinate $\xi^v$ of 
	$\pi^*_0\xi$. First we compute it from the formula $\pi^*_0 \xi = - 
	(\Jac 
	\pi_0)^{-1} \icol{\overline{f_x(u,uv)} / \overline{f(u,uv)} \\ 
		\overline{f_y(u,uv)} / \overline{f(u,uv)}}$ and we find that 
	\begin{equation}\label{eq:pullback_second_coordinate}
		\xi^v(u,v)= \frac{-1}{u\overline{f(u,uv)}} \left( 
		-v\overline{f_x(u,uv)} + \overline{f_y(u,uv)}\right)
	\end{equation}
	Let $\init (f)$ be the initial part (recall \Cref{def:initial_part}) of 
	$f$, that is, \[f= \init (f) + \hot\] and $\init (f)$ is a homogeneous 
	polynomial of degree $e$. Thus, $\pi^*_0 f$ is divisible by $u^e$. Then
	\[
	\left. \frac{\pi^*_0 f}{u^e}\right|_{D_0^\circ} =
	\left. \frac{\pi^*\init (f)}{u^e}\right|_{D_0^\circ}
	= \init (f)(1,v)= g(v)
	\]
	with $g(t) \in \C[t]$.  Similarly, $\pi^*_0 f_x$ and $\pi^*_0 f_y$ are 
	divisible by $u^{e-1}$, and
	\[
	\left. \frac{\pi^*_0 f_x}{u^{e-1}}\right|_{D_0^\circ} = \left. 
	\frac{\pi_0^*(\init (f))_x}{u^{e-1}}\right|_{D_0^\circ}= (\init 
	(f))_x(1,v)= e g(v) - vg'(v)
	\]
	and
	\[
	\left. \frac{\pi^*_0 f_y}{u^{e-1}}\right|_{D_0^\circ} = \left. 
	\frac{\pi^*(\init (f))_y}{u^{e-1}}\right|_{D_0^\circ}= (\init 
	(f))_y(1,v)= g'(v).
	\]
	Using the above two equations in \cref{eq:pullback_second_coordinate} 
	and, after scaling the vector field with $\delta$, we find
	\begin{equation}\label{eq:pullback_xi0v}
		\begin{split}
			\left.
			\pi^*_0(\delta)\cdot \xi^v\right|_{D_0^\circ}
			& = \left. \frac{|u|^2(1+|v|^2)}{u\overline{f(u,uv)}}
			\left(
			ve \bar{u}^{e-1} g(\bar{v})
			- |v|^2\bar{u}^{e-1}g'(\bar{v})
			- e\bar{u}^{e-1} g'(\bar{v})
			\right)
			\right|_{D_0^\circ} \\
			& = \frac{(1+|v|^2)}{g(\bar{v})}
			\left( veg(\bar{v}) - |v|^2g'(\bar{v}) -  g'(\bar{v})\right) \\
			& = (1+|v|^2)\left( ve - (1 + 
			|v|^2)\frac{g'(\bar{v})}{g(\bar{v})} \right)
		\end{split}
	\end{equation}
	
	On another hand we compute the differential of $\phi_0$:
	
	\begin{equation}\label{eq:diff_phi0}
		\begin{split}
			D \phi_0 
			& = - \frac{g'(\bar{v})}{g(\bar{v})} + \frac{ve}{1+|v|^2} \\
			& = \frac{1}{1+|v|^2}\left(ve - (1+|v|^2) 
			\frac{g'(\bar{v})}{g(\bar{v})} \right)
		\end{split}
	\end{equation}
	Finally, if we denote by $\nabla^\mathrm{FS} \phi_0$ the gradient of 
	$\phi_0$ with respect to the
\index{metric!Fubini-Study}
 Fubini-Study metric, then 
	\[\nabla^\mathrm{FS} \phi_0 = (1+|v|^2)^2 \cdot D \phi_0.
	\]
	So the lemma follows from \cref{eq:pullback_xi0v} and 
	\cref{eq:diff_phi0}.
\end{proof}

Let $\hat\xi_{H,0}$ be the vector field analogous to $\hat\xi_{0}$ but 
computed with the metric $H$.

\begin{cor}\label{cor:xi_0_gradient}
	The vector field $\hat\xi_{H,0}|_{D_0^\circ}$ is the gradient of  
	$\phi_{H,0}(v)$.
\end{cor}
\begin{proof}
	Find $A\in \GL(\C, 2)$ such that $A^*A= H$, then take the isometric 
	coordinates induced by $A$ and apply the previous lemma.
\end{proof}
\begin{thm}\label{thm:restrict_xi_0_morse}
	There exists a measure $0$ set $\mathcal{H}_0$ in $\mathcal{H}$, such 
	that for all Hermitian metrics $H \in \mathcal{H} \setminus 
	\mathcal{H}_0$, the restriction of the vector field 
	$\hat\xi_{H,0}|_{D_0^\circ}$ is elementary.
\end{thm}
\begin{proof}
	The proof is a direct application of 
	\Cref{lem:construction_morse_function} and \Cref{cor:xi_0_gradient} 
	plus the observation that gradients of Morse functions are elementary 
	vector fields.
\end{proof}

\begin{definition}\label{def:genericity_morse}
	Let $G_0 \subset \GL (2, \C)$ be the set of matrices $A$ such that 
	\[
	A^*A \in \mathcal{H} \setminus \mathcal{H}_0
	\]
\end{definition}

\begin{definition}\label{def:sigma_0}
	We denote by
\index{$\Sigma_0$}
$\Sigma_0 \subset D_0^\circ$ the set of singularities of 
	$\xi_0$ on $D_0^\circ$.
\end{definition}

The following lemma shows that $\xi_0$ may only have
\index{fountain!of $\xi_0$}
fountains or
\index{saddle!of $\xi_0$}
saddle 
points (\cref{fig:elementary_sing}).

\begin{lemma}\label{lem:attractors}
	If $H \in \mathcal{H} \setminus \mathcal{H}_0$, then $\xi_{H,0}$ does 
	not have
\index{sink!of $\xi_0$}
sink
singularities.
\end{lemma}

\begin{proof}
	If $p \in D_0^\circ$ is a sink singularity of $\xi_{H,0}$, that is, if 
	the 
	vector field $\xi_{H,0}$ has a sink in the direction tangent to $D_0$, 
	then 
	the vector field $\hat\xi_{H,0}$ has an sink at $p$ by 
	\Cref{lem:sink_in_normal}. In particular, for each Milnor fiber, there 
	is a 
	whole disk in it that forms part of the spine, which contradicts the 
	fact 
	that the spine contains no disks with positive volume which is proved 
	in  
	\cite[Theorem 2.2]{AC_lag}.
\end{proof}

\begin{lemma}\label{lem:sinks_near_strict}
	Let $p$ be a point in $D_0 \cap \tilde{C}$. Then, the vector field 
	$\xi_0$ points towards $p$ in a small neighborhood around $p$. As a 
	consequence, the
\index{Poincar\'e-Hopf index}
Poincar\'e-Hopf index
of $\xi_0$ at the puncture $p$ 
	is $1$.
\end{lemma}

\begin{proof}
	By the expression of $\phi_0$ given in \cref{eq:phi_0} we get that 
	$\phi_0\to + \infty$ for points approaching $p$. By 
	\Cref{lem:xi_0_gradient}, the vector field $\xi_0$ is the gradient of 
	$\phi_0$ and so it points in the direction where $\phi_0$ grows. 
	
	The statement about the index is a direct consequence of the dynamics 
	of $\xi_0$ near $p$ which behaves effectively as a sink (recall 
	\Cref{not:index_vf}).
\end{proof}

\section{The spine at radius zero over the $1$st blow-up}
\label{ss:vfld_prop_spine_0}

We are now ready to construct the spine at radius zero for homogeneous 
singularities. Consider the real oriented blow up 
\[
\sigma_0:\Yro_0 \to Y_0
\]
of $Y_0$ along $D_0 \cup C$. 

First we construct a $1$-dimensional $CW$-complex which is a spine for 
$D_0^\circ$. Let $I_0 \subset \Sigma_0$ be the set of
\index{saddle!of $\xi_0$}
saddle points
of 
$\xi_0$. For each saddle point $p \in I_0$ we consider its stable manifold 
$\tilde{B}_p$ which consists of the union of two trajectories $B^+_p$ and 
$B^-_p$ together with $p$, that is,
\[
\tilde{B}_p = B^+_p \sqcup B^-_p \sqcup \{p\}.
\]

Let $\Iro_{0,\theta} \subset \Siro_{0,\theta}$ be the set of saddle points 
of $\xiro_0|_{\Droc_{0,\theta}}$. For each saddle point $q \in 
\Iro_{0,\theta}$ we consider the corresponding stable manifold,
\[
\tBro_q = B^+_q \sqcup B^-_q \sqcup \{q\}.
\]
We recall that by \Cref{def:xiro0}, the vector field 
$\xiro_0|_{\Droc_{0,\theta}}$ is just the pullback of $\xi_0$ by the map 
$\sigma_0|_{\Droc_{0,\theta}} : \Droc_{0,\theta} \to D_0^\circ$, which is a 
$m_0$-fold regular covering. In particular,
\[
\left(\sigma_0|_{\Droc_{0,\theta}}\right)^{-1}(\tilde{B}_p)= \bigcup_{q \in 
	\left(\sigma_0|_{\Droc_{0,\theta}}\right)^{-1}(p)} \tBro_q.
\]

By \Cref{lem:attractors} all the singularities in $R_0=\Sigma_0 \setminus 
I_0$ are
\index{repeller!of $\xi_0$}
repellers
and thus, have empty stable manifolds. Analogously, all 
singularities in $\Rro_{0,\theta}=\Siro_{0,\theta} \setminus 
\Iro_{0,\theta}$ are also repellers. 
We define the set
\index{$S_0$}
\[
S_0 = \bigsqcup_{p \in I_0} \tilde{B}_p \bigsqcup_{ p \in R_0} \{p\}.
\]

\begin{definition}\label{def:spine_D0}
	We define the set
	\[
	\Sro_{0,\theta} = \bigsqcup_{q \in \Iro_{0,\theta}} \tBro_q 
	\bigsqcup_{q \in \Rro_{0,\theta}} \{ q \}
	\]
	and call it the {\em spine at radius $0$ and angle $\theta$ over $D_0$}.
\end{definition}

\begin{lemma}\label{lem:Sro0_is_spine}
	The set $S_0$ is a
\index{spine!for $D_0^\circ$}
spine
for $D_0^\circ$ and the set $\Sro_0$ is a 
	spine for $\Droc_{0,\theta}$ for each $\theta \in \R / 2 \pi \Z$.
\end{lemma}

\begin{proof}
	We know that near the punctures of $D_0$ the vector field points 
	towards the punctures (\Cref{lem:sinks_near_strict}). By definition, 
	the complement of $S_0$ consists of trajectories that necessarily 
	converge to one of these punctures because there are no sinks in 
	$D_0^\circ$ (\Cref{lem:attractors}). Pick a small simple closed curve 
	around each of these punctures and let it flow by $\xi_0$ forward and 
	backwards. The union of all these curves and their translates by  
	$\xi_0$ equals an union of disjoint cylinders, as many as punctures in 
	$D_0^\circ$. This shows that $S_0$ is a spine of $D_0^\circ$.
	
	The second statement follows directly after the observation that  
	$\sigma|_{\Droc_{0,\theta}} : \Droc_{0,\theta} \to D_0^\circ$ is a  
	regular covering for each $\theta$.
\end{proof}

\begin{figure}[h!]
	\centering
	\includegraphics*[scale=0.8]{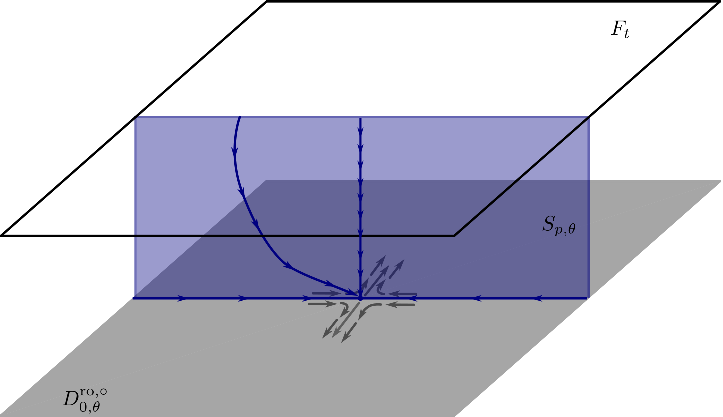}
	\caption{In blue, the stable-manifold associated with a saddle point of 
		$\xiro_{0,\theta}$. In darker blue we can see two trajectories 
		contained in $S_{p,\theta}$. On top, the Milnor fiber over $t\in 
		D_\eta$ with $\arg(t)= \theta$. Below, in gray, a part of 
		$\Droc_{0,\theta}$.}
\end{figure}

\section{Some homogeneous examples}

In this subsection we apply the theory developed on this section to some 
base case examples. 
\begin{example}
	We start by considering the
\index{Morse singularity}
Morse singularity $f(x,y)=y^2+x^2$.  Let 
	$\pi_0:Y_0 \to \Tu$ be the blow up at the origin $0 \in \Tu$. Take the 
	local chart $U \subset Y_0$ with coordinates $u,v$ such that
	\[
	\pi_0(u,v) = (u, uv).
	\]
	In these local coordinates we have:
	\[
	\pi_0^*f={\left(v^{2} + 1\right)} u^{2},
	\quad 
	\Jac \pi_0 = 
	\begin{pmatrix}
		1 &	0 \\
		v &	u 
	\end{pmatrix},
	\quad
	\det \Jac \pi_0 = u.
	\]
	
	We use \cref{eq:xi_coordinates_0} and we compute
	\[
	\left(
	\begin{matrix}
		\xi^u  \\
		\xi^v
	\end{matrix}
	\right)
	=
	\frac{-1}{u {\left(\bar{v}^{2} + 1\right)} \bar{u}^{2}}
	\left(
	\begin{matrix}
		2  u \bar{u}  \\
		2 \bar{u} \bar{v} -2v\bar{u}
	\end{matrix}
	\right)
	=
	\left(
	\begin{matrix}
		-\frac{2}{\bar{u} \bar{v}^{2} + \bar{u}} \\
		\frac{2  {\left(v - \bar{v}\right)}}{u \bar{u} \bar{v}^{2} + u 
			\bar{u}}
	\end{matrix}
	\right).
	\]
	
	\begin{figure}
		\centering
		\includegraphics*[scale=0.7]{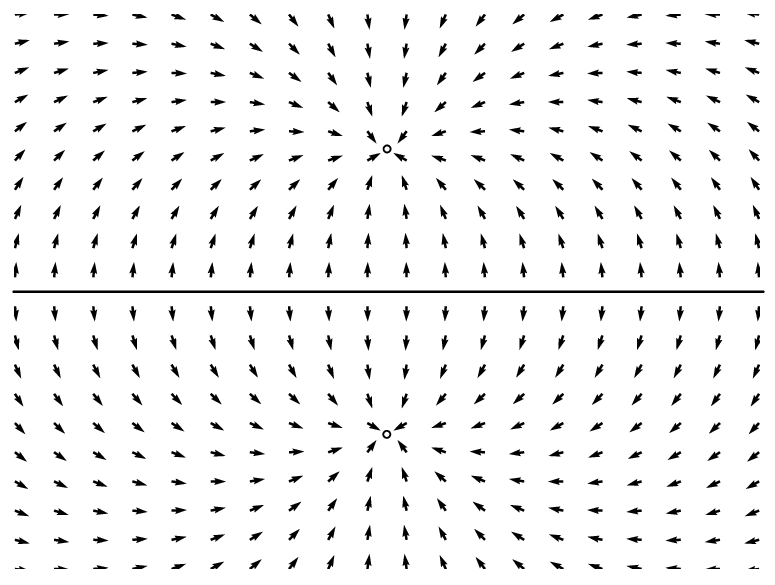}
		\caption{In black the vector field on the chart $U$. The horizontal 
			line on the center is part of a whole $S^1$ of zeroes of $\xi_0$}
		\label{fig:22_deg_case}
	\end{figure}
	
	We compute
	\[
	\lim_{u\to 0} |u|^2 
	\left(
	\begin{matrix}
		-\frac{2}{\bar{u} \bar{v}^{2} + \bar{u}} \\
		\frac{2  {\left(v - \bar{v}\right)}}{u \bar{u} \bar{v}^{2} + u 
			\bar{u}}
	\end{matrix}
	\right)
	=
	\left(
	\begin{matrix}
		0 \\
		\frac{2  {\left(v - \bar{v}\right)}}{\bar{v}^{2} + 1}
	\end{matrix}
	\right).
	\]
	That is, for the chart $U\cap D_0^\circ$ with coordinate $v$,
	\[
	\xi_0
	=
	\frac{2 {\left(v -\bar{v}\right)}}{\bar{v}^{2} + 1}.
	\]
	We can see from the expression above that the vector field has 
	non-isolated 
	zeroes (see \cref{fig:22_deg_case}). More concretely, the whole line 
	$\Re(v)=0$ is a line of zeros in $U\cap D_0^\circ$. Actually, this line 
	forms part of a $S^1$ of zeroes in $D_0^\circ$. In particular, this 
	shows 
	that the standard metric here does not come from a matrix $A$ in $G_0$ 
	(\Cref{def:genericity_morse}) since the vector field cannot be the 
	gradient 
	of any Morse function.
\end{example}

In the next example we consider the Morse singularity obtained from 
$y^2+x^2$ after the linear change of coordinates $y \mapsto y$ and $x 
\mapsto x + y$.
\begin{example}
	Consider the polynomial $f(x,y)=y^2+(y+x)^2$. Let $\pi_0:Y_0 \to \Tu$ 
	and  $U \subset Y_0$ as before. We find:
	\[
	\pi_0^*f=\left(2  v^{2} + 2  v + 1\right) u^{2},
	\quad 
	\Jac \pi_0 = 
	\begin{pmatrix}
		1 &	0 \\
		v &	u 
	\end{pmatrix},
	\quad
	\det \Jac \pi_0 = u.
	\]
	And so, following \cref{eq:xi_coordinates_0}, we find
	\[
	\left(
	\begin{matrix}
		\xi^u  \\
		\xi^v
	\end{matrix}
	\right)
	=
	\frac{-1}{u \left(2  \bar{v}^{2} + 2  \bar{v} + 1\right) \bar{u}^{2}}
	\left(
	\begin{matrix}
		2  \bar{u}(\bar{u} + \bar{v})  \\
		-2v\bar{u}(1+\bar{v}) + 2 \bar{u}(1+2\bar{v})
	\end{matrix}
	\right)
	=
	\left(
	\begin{matrix}
		-\frac{2 {\left(\bar{v} + 1\right)}}{2  \bar{u} \bar{v}^{2} + 2  
			\bar{u} \bar{v} + \bar{u}} \\
		\frac{2 {\left({\left(v - 2\right)} \bar{v} + v - 1\right)}}{2  u 
			\bar{u} \bar{v}^{2} + 2 u \bar{u} \bar{v} + u \bar{u}}
	\end{matrix}
	\right).
	\]
	
	As before, we now compute the scaling and extension of
	\[
	\lim_{u\to 0} |u|^2 
	\left(
	\begin{matrix}
		-\frac{2 {\left(\bar{v} + 1\right)}}{2  \bar{u} \bar{v}^{2} + 2  
			\bar{u} \bar{v} + \bar{u}} \\
		\frac{2 {\left({\left(v - 2\right)} \bar{v} + v - 1\right)}}{2  u 
			\bar{u} \bar{v}^{2} + 2 u \bar{u} \bar{v} + u \bar{u}}
	\end{matrix}
	\right)
	=
	\left(
	\begin{matrix}
		0 \\
		\frac{2 {\left({\left(v - 2\right)} \bar{v} + v - 1\right)}}{2 
			\bar{v}^{2} + 2 \bar{v} + 1}
	\end{matrix}
	\right).
	\]
	So on $U\cap D_0^\circ$ we have the expression,
	\[
	\xi_0
	=
	\frac{2 {\left({\left(v - 2\right)} \bar{v} + v - 1\right)}}{2 
		\bar{v}^{2} + 2 \bar{v} + 1}.
	\]
	This vector field has isolated singularities on $D_0^\circ$ at the 
	points
	\[
	v_0 =(1 - \sqrt{5})/2, \quad v_1 = (1 + \sqrt{5})/2.
	\]
	We compute the Hessian of $\xi_0$ at those singularities. For that we 
	write 
	$v=s+it$ and compute the partials of the vector field evaluated at 
	those 
	points. For $v_0$ we get:
	\[
	\begin{split}
		\left. \partial_s\xi_0 \right|_{((1 - \sqrt{5})/2,0)}
		&=
		-\frac{4  {\left({\left(s + i  t - 2\right)} {\left(s - i  
					t\right)} + s + i  t - 1\right)} {\left(2  s - 2 i  t + 
				1\right)}}{{\left(2  {\left(s - i  t\right)}^{2} + 2  s - 2 i  
				t + 
				1\right)}^{2}} \\
		&\left. + \frac{2  {\left(2  s - 1\right)}}{2  {\left(s - i  
				t\right)}^{2} + 2  s - 2 i  t + 1} \right|_{((1 - 
				\sqrt{5})/2,0)} \\
		&= \frac{-2 \sqrt{5}}{5 - 2  \sqrt{5}}
	\end{split}
	\]
	and
	\[
	\begin{split}
		\left. \partial_t\xi_0 \right|_{(1 - \sqrt{5})/2,0)}
		&=
		-\frac{4  {\left({\left(s + i  t - 2\right)} {\left(s - i  
					t\right)} + s + i  t - 1\right)} {\left(-2 i  s - 2  t - 
				i\right)}}{{\left(2  {\left(s - i  t\right)}^{2} + 2  s - 2 i  
				t + 
				1\right)}^{2}} \\
		& \left. + \frac{2  {\left(2  t + 3 i\right)}}{2  {\left(s - i  
				t\right)}^{2} + 2  s - 2 i  t + 1} \right|_{(1 - 
				\sqrt{5})/2,0)} \\
		&= \frac{6 i}{5- 2  \sqrt{5}} 
	\end{split}
	\]
	So the Hessian is, up to a positive real constant, 
	\[
	\Hess_{((1 - \sqrt{5})/2,0)} \xi_0 
	=
	\begin{pmatrix}
		\frac{-2 \sqrt{5}}{5 - 2  \sqrt{5}} &	0 \\
		0 &	\frac{6 }{5- 2  \sqrt{5}} 
	\end{pmatrix}, 
	\]
	which indicates that $v_0$ is a saddle point. A similar computation for 
	$v_1$ 
	leads to the Hessian
	\[
	\Hess_{((1 + \sqrt{5})/2,0)} \xi_0 
	=
	\begin{pmatrix}
		\frac{2  \sqrt{5}}{2  \sqrt{5} + 5}&	0 \\
		0 &	\frac{6}{2 \sqrt{5} + 5} 
	\end{pmatrix}.
	\]
	That is, $v_1$ is a fountain (see \Cref{fig:22_nondeg_case}).
	\begin{figure}
		\centering
		\includegraphics*[scale=0.7]{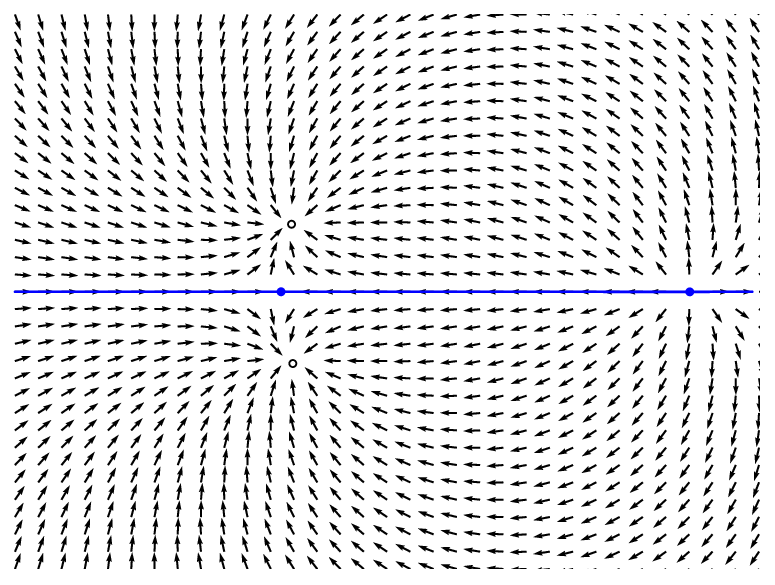}
		\caption{In black the vector field on the chart $U$. In this case 
			the vector field has a potential which is a Morse function so all 
			the zeroes are isolated. We see a saddle point and a fountain. In 
			blue the spine $S_0$}
		\label{fig:22_nondeg_case}
	\end{figure}

	\begin{figure}[h!]
		\centering
		\includegraphics*[scale=0.65]{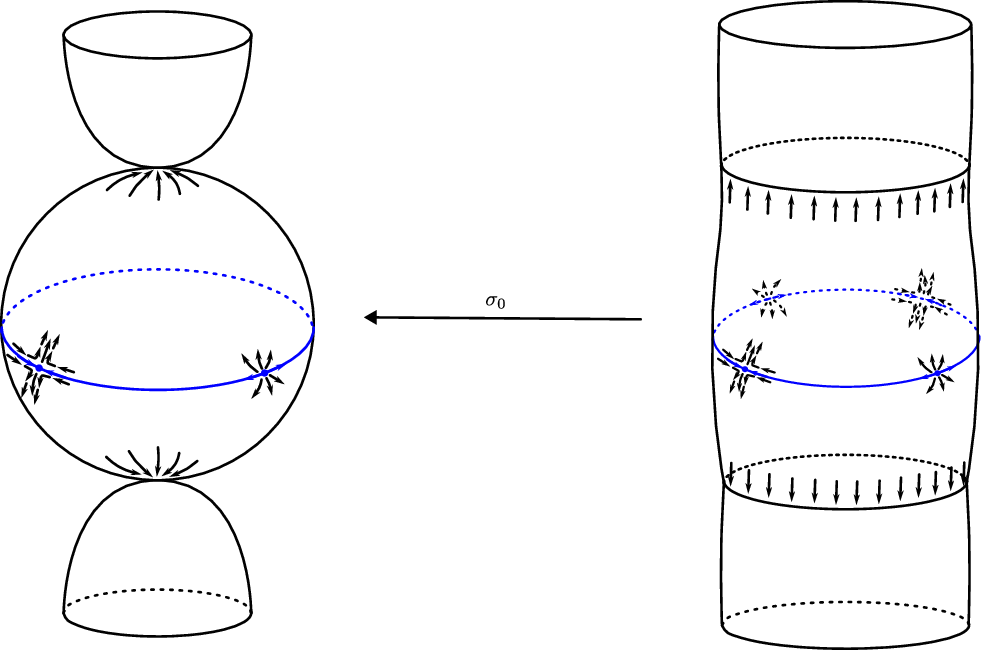}
		\caption{On the left we see the exceptional set of $Y_0$, the 
			vector field in black and the spine $S_0$ in blue. On the right we 
			see the corresponding picture for a fixed angled on the Milnor 
			fiber at radius zero.}
		\label{fig:22_nondeg_case_bl}
	\end{figure}

	\begin{figure}[h!]
		\centering
		\includegraphics*[scale=0.4]{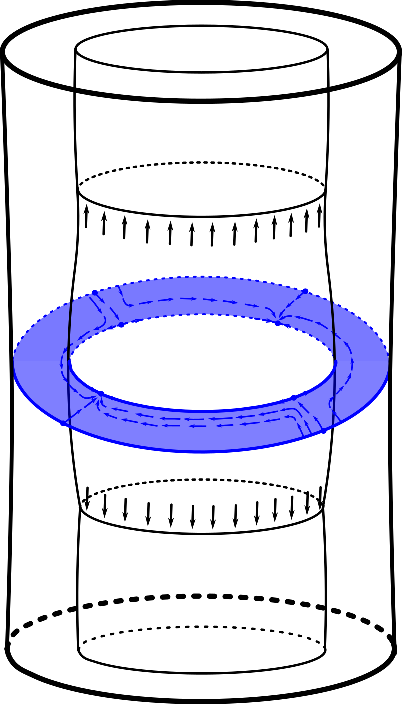}
		\caption{The inner cylinder represents the Milnor fiber at radius 
			$0$ and 
			angle $\theta$. The outer cylinder represents a Milnor fiber 
			over a point 
			$t$ with $\arg(t)= \theta$. In light blue we see the closure of 
			the union 
			of all trajectories starting at the Milnor fiber $F_t$ that 
			converge to 
			some point in the Milnor fiber at radius $0$. In darker blue we 
			see some of 
			the trajectories together with the spine $S_t$ on the Milnor 
			fiber $F_t$ 
			and the spine $\Sro_{0,\theta}$ at radius $0$ and angle 
			$\theta$. }
		\label{fig:22_nondeg_case_bl_stable}
	\end{figure}

\end{example}

\begin{example}\label{ex:brieskorn_33}
	Let $f(x,y)= y^3+x^3$  and let $\pi_0:Y_0 \to \Tu$ be the blow up at 
	the origin $0 \in \Tu$. Take the local chart $U \subset Y_0$ with 
	coordinates $u,v$ such that
	\[
	\pi_0(u,v) = (u, uv).
	\]
	In these local coordinates we have:
	\[
	\pi_0^*f=u^{3} {\left(v^{2} - v + 1\right)}{\left(v + 1\right)},
	\quad 
	\Jac \pi_0 = 
	\begin{pmatrix}
		1 &	0 \\
		v &	u 
	\end{pmatrix},
	\quad
	\det \Jac \pi_0 = u.
	\]
	And so, following \cref{eq:xi_coordinates_0}, we find
	\[
	\left(
	\begin{matrix}
		\xi^u  \\
		\xi^v
	\end{matrix}
	\right)
	=
	\frac{1}{u {\left(\bar{v}^{2} - \bar{v} + 1\right)} {\left(\bar{v} + 
			1\right)} \bar{u}^{3}}
	\left(
	\begin{matrix}
		3 u \bar{u}^{2}  \\
		-3 v \bar{u}^{2} + 3 \bar{u}^{2} \bar{v}^{2}
	\end{matrix}
	\right)
	=
	\left(
	\begin{matrix}
		-\frac{3}{\bar{u} \bar{v}^{3} + \bar{u}} \\
		-\frac{3  {\left(\bar{v}^{2} - v\right)}}{u \bar{u} \bar{v}^{3} + u 
			\bar{u}}
	\end{matrix}
	\right).
	\]
	We can compute
	\[
	\lim_{u\to 0} |u|^2 
	\left(
	\begin{matrix}
		-\frac{3}{\bar{u} \bar{v}^{3} + \bar{u}} \\
		-\frac{3  {\left(\bar{v}^{2} - v\right)}}{u \bar{u} \bar{v}^{3} + u 
			\bar{u}}
	\end{matrix}
	\right)
	=
	\left(
	\begin{matrix}
		0 \\
		-\frac{3 {\left(\bar{v}^{2} - v\right)}}{\bar{v}^{3} + 1}
	\end{matrix}
	\right).
	\]
	That is, for the chart $U\cap D_0^\circ$ with coordinate $v$,
	\[
	\xi_0
	=
	-\frac{3 {\left(\bar{v}^{2} - v\right)}}{\bar{v}^{3} + 1}.
	\]
	A direct computation, shows that this vector field vanishes at the 
	origin and at the $3$rd roots of unity, that is, at the points
	\[
	v=0, \quad v = 1, \quad v = -1/2 (1+i\sqrt{3} ), \quad v = 
	1/2(-1+i\sqrt{3}).
	\]
	Next we compute the Hessian of $\xi_0$ at those points. Take real 
	coordinates 
	$s,t$ such that $v=s+it$. We compute the 
	partials and 
	evaluate at 
	$(s,t)=(0,0)$
	\[
	\left. \partial_s \xi_0\right|_{(0,0)}
	=
	\left. 
	\frac{9  {\left({\left(s - i  t\right)}^{2} - s - i  
			t\right)} {\left(s - i  
			t\right)}^{2}}{{\left({\left(s - i  t\right)}^{3} + 
			1\right)}^{2}} - \frac{3  {\left(2  s - 2 i  t - 
			1\right)}}{{\left(s - i  t\right)}^{3} + 1}
	\right|_{(0,0)} = 3
	\]
	and
	\[
	\left.
	\partial_t \xi_0
	\right|_{(0,0)}
	=
	\left.
	-\frac{9 i  {\left({\left(s - i  t\right)}^{2} - s - 
			i  t\right)} {\left(s - i  
			t\right)}^{2}}{{\left({\left(s - i  t\right)}^{3} + 
			1\right)}^{2}} - \frac{3  {\left(-2 i  s - 2  t - 
			i\right)}}{{\left(s - i  t\right)}^{3} + 1}
	\right|_{(0,0)} = 3i
	\]
	which yields the Hessian
	\[
	\Hess_{(0,0)} \xi_0
	=
	\begin{pmatrix}
		3 &	0 \\
		0 &	3 
	\end{pmatrix},
	\]
	that shows that $\xi_0$ has a fountain at the origin.
	
	When $v=1$, that is $(s,t)=(1,0)$, we have
	\[
	\Hess_{(1,0)} \xi_0
	=
	\begin{pmatrix}
		-3/2 &	0 \\
		0 &	9/2 
	\end{pmatrix},
	\]
	which shows that $\xi_0$ has a saddle point at $(1,0)$. A similar 
	computation shows that  $\xi_0$ has also saddle points at $(-1/2, 
	-\sqrt{3}/2)$ and $(-1/2, \sqrt{3}/2)$. So, in the chart $U$, the 
	vector field looks as in \Cref{fig:33_case_U}.
	
	\begin{figure}
		\centering
		\includegraphics*[scale=0.8]{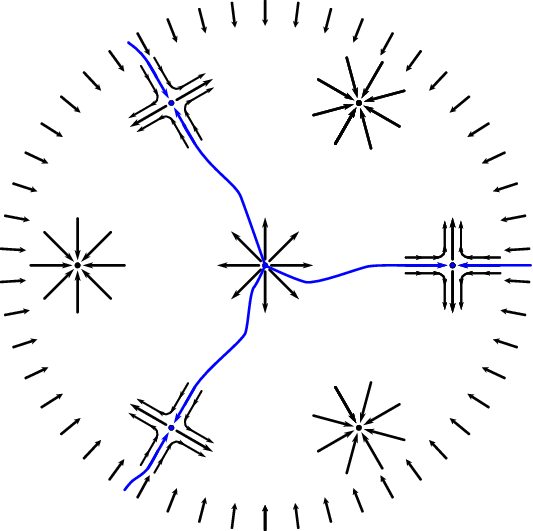}
		\caption{In black the vector field on the chart $U$. The inwards 
			pointing outer arrows indicate that there is a fountain at infinity 
			(or rather at the only point of $D_0$ that $U$ does not see). In 
			blue the stable manifolds of the three saddle points.}
		\label{fig:33_case_U}
	\end{figure}
	
	By taking another chart, an analogous computation yields that  $\xi_0$ 
	has a fountain at the only point in $D_0^\circ$ not seen by the chart 
	$U$.
	
	The Milnor fiber at radius $0$ minus a tubular neighborhood of its 
	boundary is a $3$-fold regular cover of $D_0^\circ$. Therefore, each of 
	the two fountains give rise to $3$ vertices of the spine at radius $0$.
\end{example}

%
%
%
	\chapter{The Vector Field Near the Corners} \label{s:corners}

Near an intersection point $D_i \cap D_j$ of two exceptional components
in $Y$, the pull-back of $\xi$ has poles along $D_i \cap D_j$. Assuming 
that the edge $ij$ is invariant, \Cref{def:polar}
after rescaling $\pi^*(\xi)$ by a positive factor, the vector field
extends to an analytic vector field near $D_i \cap D_j$,
which does not vanish identically along $D_i$ or $D_j$.
If the edge is not invariant, a similar scaling of the further pullback
$(\piro)^*(\xi)$ to the real oriented blow-up extends over the boundary.
The vanishing order of this positive factor is given by
the \emph{radial weights} $\tau_i$ defined below.

In this section we work with the resolution $\pipol: \Ypol \to \C^2$ which, 
in particular, resolves generic polar curves (see \Cref{def:pipol}). We use 
the notation $\pi= \pipol$.

We sometimes write simply $f$ and $f_x$ instead
of $\pipol^*f$ and $\pipol^*f_x$ to avoid cumbersome formulas.
It is clear from the 
context and the coordinates use therein. Furthermore, the partials of
$x = \pipol^*x$ and $\pipol^*f$ with respect to $u,v$ are denoted $x_u, 
x_v, 
f_u,f_v$, and so on.

\begin{definition} \label{def:radial}
	The
\index{radial weight}
\emph{radial weight}
of the vertex $i \in \V$ is the integer
\index{$\tau_i$}
	\[
	\tau_i = c_{1,i} + m_i - p_i.
	\]
\end{definition}

\begin{rem}\label{rem:formula_ci}
	It follows from definition that $\tau_i + \varpi_i = 2c_{1,i}$.
	By \Cref{lem:varpi}, we therefore have $\tau_i \leq 2 c_{1,i}$ for all
	$i\in \V$. 
	Equality holds precisely for the invariant vertices, i.e.
	when $i \in \V_\Upsilon$, by \Cref{lem:inv_pol}.
	In particular, $\tau_i$ is even in this case.
\end{rem}

\begin{block} \label{block:corners}
	Let $i,j \in \V$, and assume that there is a directed
	edge $j \to i$ in $\Gamma$ (see \Cref{def:directed}).
	Note that this implies that $i \neq 0$, since all edges
	adjacent to $0$ point away from $0$.
	Corresponding to the edge $ij$,
	we have an intersection point $p \in D_i \cap D_j$.
	Choose
\index{isometric coordinates}
isometric coordinates
$x,y$ in $\C^2$
	so that $y=0$ defines the tangent associated with $p$
	(see \Cref{def:tangent}).
	Let $U \subset Y$ be a small chart containing $p$ with coordinates
	$u,v$. We choose $u,v$ in such a way that
	\begin{equation}\label{eq:coords_intesection}
		\begin{split}
			\pi^*x &= u^{c_{0,i}} v^{c_{0,j}} \\
			\pi^*y &= u^{k} v^{\ell} \eta(u,v) 
		\end{split}
	\end{equation}
	where $\eta(0,0) \neq0$. Since $i\neq0$, we have $k>c_{0,i}$ 
	(\Cref{lem:tang_van}). Similarly, $\ell \geq c_{0,j}$ with equality if 
	and only if $j=0$.
	
	On $U \setminus D$, we have an explicit description
	of the pullback of $\xi$
	\begin{equation} \label{eq:xi_corner}
		\pi|_{U\setminus D}^*\xi
		=
		\begin{pmatrix}
			\xi^u \\
			\xi^v
		\end{pmatrix}
		=
		\frac{-1}{\det\Jac\pi \cdot \bar f}
		\begin{pmatrix}
			y_v \bar f_x - x_v \bar f_y \\
			-y_u \bar f_x + x_u \bar f_y
		\end{pmatrix}
	\end{equation}
	
	Since $f_y$ defines a generic polar curve (by the definition of 
	$\pipol$ and the assumption on the coordinate $y$), it vanishes along 
	$D_i, D_j$ with multiplicities
	$p_i,p_j$, we can write $\pi^*f_y(u,v) = u^{p_i}v^{p_j}(b + \hot)$
	where $b\in\C^*$. We can therefore explicitly write  
	\[
	x_v \bar{f}_y
	= u^{c_{0,i}} \bar{u}^{p_i} v^{c_{0,j}-1} \bar{v}^{p_j} (c_{0,j} \bar 
	b  + \hot)
	\]
	with $b\in\C^*$. Here, once again, we are identifying $f_y$ with its 
	pullback. We can now write $\pi^* f_x(u,v)= u^{p_i'} v^{p_j'}(b' + 
	\hot)$  with the inequalities $p_i' \geq p_i$ and $p_j' \geq p_j$ 
	because the $x$-axis is a tangent by hypothesis. Again using the 
	observation after \cref{eq:coords_intesection}, we can write
	\[
	y_v \bar{f}_x 
	=  u^{k} \bar{u}^{p_i'} v^{\ell-1} \bar{v}^{p_j'} (\ell b' + \hot)
	=  u^{c_{0,i}} \bar{u}^{p_i} v^{c_{0,j}-1} \bar{v}^{p_j} (0 + \hot)
	\] 
	with $b' \in \C$. Together this yields
	\[
	y_v \bar{f}_x - x_v \bar{f}_y  
	= u^{c_{0,i}} \bar{u}^{p_i} v^{c_{0,j}-1} \bar{v}^{p_j} (-c_{0,j} \bar 
	b + \hot).
	\]
	By an analogous reasoning we find,
	\[
	-y_u \bar{f}_x + x_u \bar{f}_y 
	= u^{c_{0,i}-1} \bar{u}^{p_i} v^{c_{0,j}} \bar{v}^{p_j} (c_{0,i} \bar b 
	+ \hot).
	\]
	Similarly, we can also write 
	\begin{equation}\label{eq:coefficients_a_c}
		\begin{split}
			\pi^*f(u,v) &= u^{m_i} v^{m_j}(a  +\hot) \quad a \in \C^* \\
			\det\Jac\pi(u,v) &=  u^{\nu_i-1} v^{\nu_j-1}(c + \hot) \quad c 
			\in \C^*.
		\end{split}
	\end{equation}
	Setting
\index{$d$}
$d = \bar b / (c\bar a) \in \C^*$, we get
	(recall \Cref{def:c_1})
	\begin{equation} \label{eq:xi_corner_more}
		|u|^{\tau_i} |v|^{\tau_j} \pi|_{U\setminus D}^*\xi
		=
		\left(
		\frac{\bar{u}}{|u|}
		\right)^{\varpi_i}
		\left(
		\frac{\bar{v}}{|v|}
		\right)^{\varpi_j}
		d
		\left(
		\begin{matrix}
			u (c_{0,j} + \hot)\\
			v(-c_{0,i}  + \hot)
		\end{matrix}
		\right).
	\end{equation}
\end{block}

We recall the notation  $r = |u|$ from \Cref{s:real_oriented} and introduce 
the notation $s=|v|$. Do not confuse with the notation $s, t$ 
used in \Cref{ex:brieskorn_33}.

The previous discussion, in particular the above equation, yields the 
following 
result. 
\begin{lemma}\label{lem:extension_corner}
	\begin{blist}
		\item If $p \in D_i \cap D_j$, with $ij$ an invariant edge, then
		the vector field 
		\begin{equation}\label{eq:xiu_ext}
			|u|^{\tau_i}|v|^{\tau_j} \pi|_{U\setminus D}^* \xi
		\end{equation} 
		extends
		to an analytic vector field on $U$ that is tangent to $D_i\cap U$ 
		and $D_j\cap U$. And it does not vanish identically there.
		\item The scaled pull-back
		\begin{equation}\label{eq:xirou_ext}
			r^{\tau_i} s^{\tau_j} \piro|_{\Uro\setminus \partial \Uro}^*\xi
		\end{equation}
		extends to an analytic vector field on $\Uro$, which is tangent to
		both $\Dro_i \cap \Uro$ and $\Dro_j\cap \Uro$. Furthermore, if $ij$ 
		is a non-invariant edge, it does not vanish identically along 
		$\Dro_{i,j}$. 
	\end{blist}
\end{lemma}

\begin{definition}\label{def:xiroU_intersection}
	Denote by $\xiro_U$ the
	analytic vector field on $\Uro$ whose restriction to $\Uro \setminus 
	\partial
	\Uro$ coincides with \cref{eq:xirou_ext}.  We denote by 
	$\xiro_{i,j,\theta}$ the restriction of $\xiro_U$ to a neighborhood of 
	$\Dro_{i,j,\theta}$ in the Milnor fiber $\Fro_\theta$ at radius $0$ and 
	angle $\theta$.
	
	Similarly, for $p$ an invariant intersection point, denote by $\xi_U$ 
	the
	analytic vector field on $U$ whose restriction to $U \setminus D$ 
	coincides
	with \cref{eq:xiu_ext}.
	
\end{definition}
\section{Invariant intersection points}
\label{ss:invariant_int}
In this subsection we consider an invariant edge $ij$ in the graph $\Gamma$.
We use the notation from \cref{block:corners}, and assume
$\varpi_i = \varpi_j = 0$ and that $j \to i$.

\begin{lemma} \label{lem:d_pos}
	The number $d$ from \cref{eq:xi_corner_more} is a positive real number
	if the edge $ij$ is invariant.
\end{lemma}
\begin{proof}
	Denote by $l$ the coefficient of $u^{c_{1,i}} v^{c_{1,j}}$
	in the expansion of $\pi^*y$ in $u,v$. We have
	\[
	\det\Jac\pi(u,v) = x_u y_v - x_v y_v
	= c_{0,i} u^{c_{0,i}-1} v^{c_{0,j}} y_v
	- c_{0,j} u^{c_{0,i}} v^{c_{0,j}-1} y_u,
	\]
	and so we can compute the number $c$ from \cref{eq:coefficients_a_c}:
	\[
	c = (c_{0,i} c_{1,j} - c_{0,j} c_{1,i}) l.
	\]
	A chain rule computation gives
	\[
	\begin{split}
		f_y &= \frac{-x_v f_u + x_u f_v}{x_u y_v - y_u x_v} \\
		&= \frac{
			u^{c_{0,i}+m_i-1}
			v^{c_{0,j}+m_j-1}
			(a   (c_{0,j} m_i-c_{0,i} m_j) + \hot)}
		{u^{\nu_i-1} v^{\nu_j -1}(c + \hot)} \\
		&= -\frac{a}{c} \cdot
		\left|
		\begin{matrix}
			c_{0,i} & m_i \\
			c_{0,j} & m_j
		\end{matrix}
		\right|
		(u^{p_i} v^{p_j} + \hot)
	\end{split}
	\]
	and so $b = a(m_i c_{0,j} - m_jc_{0,i})/c$, i.e.
	$d = \bar b / (c \bar a) = (m_i c_{0,j} - m_jc_{0,i}) / |c|^2 > 0$ where
	the last inequality follows from the fact that $(c_{0,i}m_j - 
	m_ic_{0,j})<0$ by
	\Cref{lem:c0ord}.
\end{proof}

\begin{lemma} \label{lem:elementary_invariant}
	If $p \in D_i \cap D_j$ is an invariant intersection point, then
	the vector field $\xi_U$ has an elementary singularity
	at $p$, whose
\index{manifold!unstable}
unstable manifold is $D_i$ and whose
\index{manifold!stable}
stable
	manifold is $D_j$.
\end{lemma}

\begin{proof}
	Since $d$ is positive by \Cref{lem:d_pos},  the result follows from the
	Grobman-Hartman theorem and \cref{eq:xi_corner_more}.
\end{proof}

The previous lemma talks about the Hessian of the vector field $\xi_U$ in 
$\Ypol$, it is also useful later on the text to be able to speak about the 
Hessian  of the vector field $\xiro_U$ in $\Yropol$.

\begin{block} \label{block:pullback_ro}
	Here we briefly explain how to obtain useful formulas for the pullback 
	of our vector field to $\Yro$. We use polar coordinates $r$, $s$, 
	$\alpha$ and $\beta$
	(in this order), where
	$u = r e^{i\alpha}$, $v = s e^{i\beta}$
	for the chart $\Uro = \sigma^{-1}(U) \subset \Yro$.
	We can write
	\[
	\xiro_U
	=
	(
	\xiror,
	\xiros,
	\xiroal,
	\xirobe
	)^\top.
	\]
	The formulas
	\begin{equation} \label{eq:xiro_corner}
		\begin{aligned}
			\xiror &=  \Re(e^{-i\alpha} \sigma^*\xi^u),\qquad       &
			\xiros &= \Re(e^{-i\beta} \sigma^*\xi^v),         \\
			\xiroal &=  \Im(r^{-1}e^{-i\alpha} \sigma^*\xi^u),\qquad &
			\xirobe &= \Im(s^{-1}e^{-i\beta} \sigma^*\xi^v),
		\end{aligned}
	\end{equation}
	follow from a direct computation in polar coordinates. We carry on this 
	direct 
	computation of the formulas for $\xiror$. We observe that in $\Uro$, 
	the real 
	oriented blow up $\sigma$ is expressed as 
	\[
	(r,s, \alpha, \beta) \mapsto (r e^{i\alpha}, s e^{i\beta})
	\]
	and so its Jacobian is
	\[
	\Jac \sigma=
	\left(
	\begin{matrix}
		\cos \alpha & 0 & -r \sin \alpha & 0\\
		\sin \alpha  & 0 & r \cos \alpha & 0\\
		0 & \cos \beta & 0 & -s \sin \beta \\
		0 & \sin \beta  & 0 & s \cos \beta
	\end{matrix}
	\right).
	\]
	This matrix is, up to permutation of the second and third columns, a 
	block diagonal matrix. So we can compute its inverse by blocks and find 
	that
	
	\begin{equation}\label{eq:inverse_jacsig}
		(\Jac \sigma)^{-1}=
		\left(
		\begin{matrix}
			\cos \alpha & \sin \alpha & 0 & 0\\
			0  & 0 & \cos \beta & \sin \beta \\
			- \frac{\sin \alpha}{r} & \frac{\cos \alpha}{r} & 0 & 0\\
			0 & 0 & - \frac{\sin \beta}{s}  & \frac{\cos \beta}{s}
		\end{matrix}
		\right).
	\end{equation}
	By definition of the pullback of a vector field we know that
	\[
	\xiro_U = (\Jac \sigma)^{-1} (\Re(\sigma^* \xi^u), \Im(\sigma^* 
	\xi^u),\Re(\sigma^* \xi^v), \Im(\sigma^* \xi^v) )^\top.
	\]
	So using \cref{eq:inverse_jacsig}, we compute, on one hand:
	\begin{equation}\label{eq:xiror}
		\xiror = (\cos \alpha) \cdot \Re(\sigma^* \xi^u) +  (\sin \alpha) 
		\cdot \Im(\sigma^* \xi^u).
	\end{equation}
	And on the other hand:
	\begin{equation}\label{eq:xiror_other}
		\begin{split}
			\Re (e^{-i\alpha} \sigma^*(\xi^u)) 
			= &\Re\left( (\cos\alpha -i \sin \alpha) \cdot (\Re(\sigma^* 
			\xi^u) + i \Im(\sigma^* \xi^u) ) \right)\\
			= & \Re \left((\cos \alpha) \cdot \Re(\sigma^* \xi^u) +  (\sin 
			\alpha) \cdot \Im(\sigma^* \xi^u) \right.\\
			& \left. + i \left( (\cos \alpha) \cdot  \Im(\sigma^* \xi^u) - 
			(\sin \alpha) \cdot  \Re(\sigma^* \xi^u)  \right) \right) \\
			= & (\cos \alpha) \cdot \Re(\sigma^* \xi^u) +  (\sin \alpha) 
			\cdot \Im(\sigma^* \xi^u)
		\end{split}
	\end{equation}
	From \cref{eq:xiror} and \cref{eq:xiror_other}, we verify the formula 
	for $\xiror$ given in \cref{eq:xiro_corner}. The other three formulas 
	of \cref{eq:xiro_corner} follow from very similar computations.
\end{block}

\begin{cor} \label{cor:xiro_inv}
	At a point $q \in \Dro_{i,j}$ with $\varpi_i=\varpi_j=0$, the vector 
	field $\xiro_U$  has a Hessian
	\[
	\Hess_q \xiro_U
	=
	\left(
	\begin{matrix}
		d c_{0,j} & 0 & 0 & 0\\
		0  & -d c_{0,i}  & 0  & 0\\
		0 & 0 & 0 & 0 \\
		0 & 0  & 0 & 0
	\end{matrix}
	\right).
	\]
	Moreover, there exists a local change of coordinates such that
	\[
	\xiro_{i,j,\theta}=(dc_{0,j}r', -dc_{0,i}s',0,0)^\top.
	\]
	in coordinates $r',s',\alpha',\beta'.$
\end{cor}
\begin{proof}
	Using \cref{eq:xiro_corner} and \cref{eq:xi_corner_more} we get
	\[
	\xiror = \Re(e^{-i \alpha} d r e^{i \alpha} (c_{0,j} + \hot)) = 
	rdc_{0,j} + \Re(\hot)
	\]
	where $\Re(\hot)$ denotes the real part of the higher (than linear) 
	order terms. This yields the first row of the Hessian in the statement.
	Similarly, we find 
	\[
	\xiroal = \Im(r^{-1} e^{-i \alpha} d r e^{i \alpha} (c_{0,j} + \hot)) ) 
	= \Im(dc_{0,j} + \hot)= 0 + \Im (\hot),
	\]
	which gives the third row of the Hessian. The second and fourth rows 
	are completely analogous.
	
	For the last statement we use that the vector field $\xi_U$ satisfies 
	the
	hypothesis of the Grobman-Hartman theorem by 
	\Cref{lem:elementary_invariant}.
	This implies that there exists coordinates $u', v'$ for $U$ such that 
	$\{u'=
	0\} = D_i \cap U$ and $\{v'= 0\} = D_j \cap U$ and such that $\xi_U$ in 
	the
	coordinates $u',v'$ is
	\begin{equation} \label{eq:GH_vf}
		\left(
		\begin{matrix}
			\xi^{u'}\\
			\xi^{v'}
		\end{matrix}
		\right) 
		=
		\left(
		\begin{matrix}
			dc_{0,j} u'\\
			-d c_{0,i}v'
		\end{matrix}
		\right).
	\end{equation}
	So, after taking its pullback by the real oriented blow up 
	$\sigma(r',\alpha',s',\beta')= (r' e^{i \alpha'},s' e^{i \beta'})$, we 
	get the vector field
	
	\begin{equation}\label{eq:xiro_grob}
		\xiro_U
		=
		(
		d c_{0,j} r',
		-d c_{0,i} s',
		0,
		0
		)^\top
	\end{equation}
	which proves the last statement.
\end{proof}

In \cref{fig:droij_inv} we can see a drawing of the vector field $\xiro_U$ 
restricted to $(\Dro_{i,\theta} \cup \Dro_{j,\theta})\cap U$
\begin{figure}[!ht]
	\centering
	\includegraphics*[scale=1]{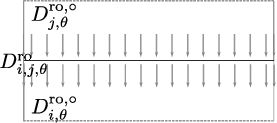}
	\caption{The vector field $\xiro_U$ restricted to $(\Dro_{i,\theta} 
		\cup \Dro_{j,\theta})\cap U$ near a component of $\Dro_{i,j,\theta}$ 
		for an edge $j \to i$ and with $\varpi_i=\varpi_j=0$.}
	\label{fig:droij_inv}
\end{figure}

\section{Non-invariant intersection points}
Next we assume that $\varpi_i, \varpi_j$ are not both 0. In this subsection 
we also use the formulas \cref{eq:xiro_corner}.

\begin{lemma}
	In polar coordinates given by the real oriented blow up, the following 
	expression holds
	\[
	\arg (f )(0,0,\alpha,\beta) = m_i\alpha + m_j\beta + \Delta  
	\]
	with $\Delta \in \R / 2 \pi \Z$ a constant.
\end{lemma}

\begin{proof}
	This follows from \Cref{lem:arg_equivariant} and the observation that 
	$\arg (f )(0,0,\alpha,\beta)$ is equivariant with respect to $\cdot_i$ 
	of weight $m_i$ and equivariant with respect to $\cdot_j$ of weight 
	$m_j$.
\end{proof}
\begin{definition}
	With $d \in \C^*$ as in \cref{block:corners},
	set $\delta = \arg(d) \in \R/2\pi\Z$, and
\index{$\Sigma_{i,j}^\pm$}
	\begin{equation} \label{eq:Sigma_ij_pm}
		\begin{split}
			\Sigma_{i,j}^+
			&=
			\set{(0,0,\alpha,\beta)\in \Uro}
			{\delta + \varpi_i\alpha + \varpi_j\beta \equiv 0 \; (\mod  
				2\pi\Z)} \subset \Dro_{i,j}, \\
			\Sigma_{i,j}^-
			&=
			\set{(0,0,\alpha,\beta)\in \Uro}
			{\delta + \varpi_i\alpha + \varpi_j\beta \equiv \pi \; (\mod 
				2\pi\Z)}\subset \Dro_{i,j},
		\end{split}
	\end{equation}
	and $\Sigma_{i,j} = \Sigma_{i,j}^+ \cup \Sigma_{i,j}^-$.
\end{definition}

\begin{lemma}\label{lem:transversality_number}
	Fix $\theta \in \R / 2 \pi \Z$.
	\begin{blist}
		\item \label{it:trans_i}The set $\Sigma_{i,j}$ is a real 
		$1$-dimensional affine subspace homeomorphic to a disjoint union of 
		$2\gcd(\varpi_i,\varpi_j)$ circles.
		
		\item \label{it:trans_ii} The set $\Dro_{i,j} \cap \Yro_\theta$ is 
		a real $1$-dimensional affine subspace homeomorphic to a disjoint 
		union of $\gcd(m_i,m_j)$ circles.
		
		\item \label{it:trans_iii}The intersection $\Sigma_{i,j} \cap 
		\Yro_\theta$ is transverse and
		\[
		|\Sigma^+_{i,j} \cap \Yro_\theta |=
		|\Sigma^-_{i,j} \cap \Yro_\theta |=
		- 
		\left|
		\begin{matrix}
			m_i & \varpi_i \\
			m_j & \varpi_j
		\end{matrix}
		\right|.
		\]
		Where $| \cdot |$ denotes cardinality in the first two terms and 
		determinant in the last one.
	\end{blist}
\end{lemma}

\begin{proof}
	The map $\chi:\Dro_{i,j} \to \R/2\pi\Z$,
	$(\alpha,\beta) \mapsto \delta + \varpi_i \alpha + \varpi_j\beta$ with 
	$\delta= \arg(d)$,
	is a surjective affine morphism, and therefore its preimages 
	$\chi^{-1}(0)$ and $\chi^{-1}(\pi)$ are submanifolds of dimension 
	$\dim_\R \Dro_{i,j} - \dim_\R (\R / 2 \pi \Z) = 1$. From 
	\cref{eq:Sigma_ij_pm} we see that $\Sigma^\pm_{i,j}$ consists of 
	$\gcd(\varpi_i,\varpi_j)$ connected components. In fact, each component 
	can be parametrized as follows. Set $(\varpi_i', \varpi_j') = 
	(\varpi_i,\varpi_j)/\gcd(\varpi_i,\varpi_j)$. We see that 
	$\Sigma_{i,j}$ consists of components parametrized by $\gamma \in 
	\R/2\pi\Z$
	by
	\[
	r = s = 0, \quad
	\alpha = -\varpi'_j \gamma +\delta',\quad
	\beta = \varpi'_i \gamma +\delta''
	\]
	where $\delta'$ and $\delta''$ are some constants in $\R / 2 \pi \Z$.
	
	A similar argument works for the set $\Dro_{i,j} \cap \Yro_\theta$ 
	since 
	\begin{equation}\label{eq:parame_dijro}
		\Dro_{i,j} \cap \Yro_\theta
		= \set{(0,0, \alpha, \beta) \in \Uro}{m_i \alpha + m_j \beta + 
			\Delta = \theta }.
	\end{equation}
	
	The composition
	\[
	\arg(f)(0,0,-\varpi'_j\gamma,\varpi'_i\gamma)
	= \frac{m_j\varpi_i - m_i\varpi_j}{\gcd(\varpi_i,\varpi_j)} \cdot 
	\gamma 
	+ \Delta
	\]
	is a submersion, since $m_i\varpi_j - m_j\varpi_i < 0$ by 
	\Cref{cor:ups_pol} and \Cref{cor:Gapol_ah} \cref{it:Gapol_ah_pos}.
	which proves transversality. 
\end{proof}

\begin{lemma} \label{lem:noninv_corner}
	Let $\theta \in \R/2\pi\Z$.
	\begin{blist}
		\item \label{it:noninv_corner_ext}
		The vanishing set of $\xiro_U$ is $\Sigma_{i,j}$.
		
		\item \label{it:noninv_corner_elem}
		The Hessian of $\xiro_U$ along the normal bundle of $\Sigma_{i,j}$
		is elementary.
		
		\item \label{it:noninv_corner_plu}
		If $q \in \Sigma_{i,j}^+$ and $\arg(f)(q)=\theta$,
		then
		the 
\index{manifold!unstable}
unstable manifold of $\xiro_U$ at $q$ is the germ 
		$(\Dro_{i,\theta},q)$,
		and
		its
\index{manifold!stable}
stable manifold is a single trajectory in $(\Dro_{j,\theta},q)$.
		
		\item \label{it:noninv_corner_min}
		Similarly,
		if $q \in \Sigma_{i,j}^-$ and $\arg(f)(q)=\theta$,
		then
		the stable manifold of $\xiro_U$ at $q$ is $(\Dro_{i,\theta},q)$,
		and
		its unstable manifold is a single trajectory in 
		$(\Dro_{j,\theta},q)$.
	\end{blist}
\end{lemma}

\begin{lemma} \label{lem:explicit_pullback_xi}
	Let $q\in \Sigma_{i,j}$. Set $\chi(\alpha,\beta) = \delta + 
	\varpi_i\alpha + \varpi_j\beta$ with $\delta = \arg(d)$. Then
	\begin{equation} \label{eq:diff_equi}
		\begin{split}
			r^{\tau_i} s^{\tau_j}  \xiror
			& \in  r |d| c_{0,j} \cos(\chi(\alpha,\beta)) + 
			\mathfrak{m}^2_{\Yro,q}, \\
			r^{\tau_i} s^{\tau_j} \xiros
			& \in -s|d| c_{0,i} \cos(\chi(\alpha,\beta)) + 
			\mathfrak{m}^2_{\Yro,q}, \\
			r^{\tau_i} s^{\tau_j}  \xiroal
			& \in |d| c_{0,j} \sin(\chi(\alpha,\beta)) + (r,s) 
			\mathcal{C}^\omega_{\Yro,q},\\
			r^{\tau_i} s^{\tau_j} \xirobe
			& \in -|d|c_{0,i} \sin(\chi(\alpha,\beta)) + (r,s) 
			\mathcal{C}^\omega_{\Yro,q},
		\end{split}
	\end{equation}
	where $\mathfrak{m}_{\Yro,q}$ is the maximal ideal of the local ring of 
	real analytic functions $\mathcal{C}^{\omega}_{\Yro,q}$ and $(r,s) 
	\mathcal{C}^\omega_{\Yro,q}$ is the ideal generated by $r$ and $s$ in 
	the same local ring.
\end{lemma}
\begin{proof}
	We recall from \cref{eq:xiro_corner} 
	\[
	\xiror 
	= \Re(e^{-i\alpha} \sigma^*\xi^u),\qquad       
	\xiroal 
	= \Im(r^{-1}e^{-i\alpha} \sigma^*\xi^u).   
	\]
	Now, we look at \cref{eq:xi_corner_more} and observe
	\[
	\left(
	\frac{u}{|u|}
	\right)^{\varpi_i}
	\left(
	\frac{v}{|v|}
	\right)^{\varpi_j}
	d 
	= |d|(\cos(\chi) + i \sin(\chi))
	\]
	We compute
	\[
	\begin{split}
		r^{\tau_i} s^{\tau_j}\Re(e^{-i\alpha} \sigma^*\xi^u) 
		&= \Re(e^{-i\alpha} e^{i\chi} |d|r e^{i\alpha} (c_{0,j} + \hot)) \\
		&= \Re(|d|r e^{i\chi}(c_{0,j} + \hot)) \\
		&= |d|r c_{0,j} \cos(\chi) + \Re(r e^{i\chi} \hot)
	\end{split}
	\]
	The function $r$ and all the higher order terms vanish at $q$. This 
	settles the first line of \cref{eq:diff_equi}. On another hand, we have
	\[
	\begin{split}
		r^{\tau_i} s^{\tau_j}\Im(r^{-1}e^{-i\alpha} \sigma^*\xi^u) 
		&= \Im(r^{-1}e^{-i\alpha} |d| e^{i\chi} r e^{i\alpha} (c_{0,j} + 
		\hot)) \\
		&= \Im (|d| e^{i\chi} (c_{0,j} + \hot))) \\
		&=|d|c_{0,j}\sin(\chi) + \Im(e^{i\chi} \hot)
	\end{split}
	\]
	Each higher order term is divisible by $r$ or $s$, this is because the 
	higher order terms in \cref{eq:xi_corner_more} are monomials in 
	$u,\bar{u},v$ and $\bar{v}$. This settles the third line of 
	\cref{eq:diff_equi}. The other two lines are completely analogous.
\end{proof}

\begin{proof}[Proof of \Cref{lem:noninv_corner}]
	
	Given that $U$ has been chosen small, these functions all vanish 
	precisely
	when $r=s=0$ and $\chi \in \Z\pi$, proving \cref{it:noninv_corner_ext}.
	
	If $q \in \Sigma_{i,j}$, then \Cref{lem:explicit_pullback_xi}
	shows that the linear terms of $\xiror$ and $\xiros$ are
	$|d|r c_{0,j}$ and $-|d|s c_{0,i}$, respectively. Furthermore, we have
	\[
	\partial_\alpha \sin(\chi) = \varpi_i \cos(\chi), \qquad
	\partial_\beta \sin(\chi) = \varpi_j \cos(\chi).
	\]
	As a result, the Hessian of $\xiro_U$ at $q \in \Sigma_{i,j}^\pm$ is
	\begin{equation} \label{eq:hessian}
		\pm |d| \cdot
		\left(
		\begin{matrix}
			c_{0,j} & 0 & 
			\multicolumn{2}{c}{\multirow{2}{*}{\scalebox{2}{$0$}}} \\
			0 & -c_{0,i} & \multicolumn{2}{c}{} \\
			\multicolumn{2}{c}{\multirow{2}{*}{\scalebox{2}{$\bigast$}}}
			& c_{0,j}\varpi_i & c_{0,j}\varpi_j \\
			\multicolumn{2}{c}{}
			& -c_{0,i}\varpi_i & -c_{0,i}\varpi_j
		\end{matrix}
		\right).
	\end{equation}
	The lower right block of the above matrix has rank $1$. Therefore its 
	only non-zero eigenvalue equals its trace $ c_{0,j}\varpi_i 
	-c_{0,i}\varpi_j$. This number is positive by \Cref{cor:ups_pol}.
	In particular, the Hessian has three nonzero real eigenvalues,
	which proves \cref{it:noninv_corner_elem}, since $\dim_\R \Sigma_{i,j} 
	= 1$.
	
	For \cref{it:noninv_corner_plu}, let $q \in \Sigma_{i,j}^+$.
	Then from the previous discussion we have that \cref{eq:hessian} has 
	two positive eigenvalues and one negative eigenvalue. As a result, 
	following Kelley \cite[Theorem 1]{kelley_stable}, $\xiro_U$ has a 
	unique $2$-dimensional unstable manifold and a unique $1$-dimensional 
	stable manifold.
	The submanifold $\Dro_{i,\theta}$, given by $r = 0$ and $\chi = 0$, is 
	an invariant  submanifold and, again,  it is unstable since the 
	restriction of the Hessian \cref{eq:hessian} to $T_q \Dro_{i,\theta}$ 
	is positive definite. A similar observation gives that the restriction 
	of the Hessian to $T_q \Dro_{j,\theta}$ has one positive and one 
	negative eigenvalue. The statement \cref{it:noninv_corner_min} is 
	analogous to the previous one. 
\end{proof}

\begin{figure}
	\centering
	\includegraphics[scale=1.5]{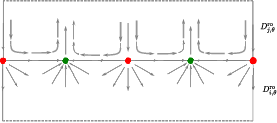}
	\caption{
		Trajectories of the scaled vector field along
		the boundary, near the corner locus.
		The winding number on $\Dro_{j,\theta}$ around the boundary
		component seen here, is computed by a path going from left
		to right in the upper half of the figure, and so coincides
		with the number of red (or green) points.
		On $\Dro_{i,\theta}$ we find the same number with the opposite sign.
		\label{fig:corner_sings}}
\end{figure}

\begin{notation}\label{not_red_green}
	Take a connected component of some $\Dro_{i,j,\theta}$ with $\varpi_i$ 
	or $\varpi_j$ not zero and such that $j \to i$. A point of 
	$\Sigma_{i,j}^+$ will be called a
\index{red point}
{\em red point},
that is, a red point 
	looks like half a
\index{fountain!half}
fountain
on $\Dro_{i,\theta}$ and half a
\index{saddle!half}
 saddle point 
	on $\Dro_{j,\theta}$. Analogously, a point of $\Sigma_{i,j}^-$ will be 
	called a
\index{green point}
{\em green point},
that is, a green point looks like half a 
	sink on $\Dro_{i,\theta}$ and half a saddle point on $\Dro_{j,\theta}$. 
	See \cref{fig:corner_sings} for a drawing of the vector field 
	$\xiro_{i,j,\theta}$.
\end{notation}
\section{The strict transform} \label{ss:strict_tranform}

In this subsection we rescale the pull-back $\pipol^*\xi$ in a neighborhood
of a component $D_a$ of the strict transform $\Cpol$, for some $a \in \A$.
With the right rescaling, this vector field looks like a Morse-Bott
attractor along $D_a$. As a result, the flow of the vector field
retracts a neighborhood onto $D_a$, and, as we shall see,
the strict transform $\overline{\pipol^{-1}(S \setminus \{0\})} \subset 
\Ypol$
of the total spine does not intersect a neighborhood of $\Cpol$.

\begin{block} \label{block:str}
	Let $a \in \A$ be an arrowhead in the graph $\Gapol$, and let $i \in \V$
	be the unique neighbor of $a$ in $\Gapol$. Then $D_i$ is an exceptional
	component, and $D_a$ is a component of the strict transform
	$\Cpol \subset \Ypol$. These divisors intersect in a unique point,
	choose a small coordinate chart $U$ around this point.
	We can assume that the coordinates $u,v$ in $U$ satisfy
	\[
	\pi^* x|_U = u^{c_{0,i}}, \qquad
	\pi^* f|_U = u^{m_i} v^{m_a}.
	\]
	Since $\pipol$ restricts to an isomorphism outside the exceptional 
	divisor,
	and $\xi$ is a well defined vector field outside the curve $C$,
	we have a well defined vector field $\pipol^*\xi$ on the set
	$U \setminus D = U \setminus (D_i \cup D_a)$.
	
	After choosing a small neighborhood $U$,
	we can choose $\epsilon$ and $\eta$ small enough that we can assume
	that $D_a \cap \pipol^{-1}(\Tu(\epsilon,\eta)) \subset U$.
\end{block}

\begin{lemma} \label{lem:str_tr}
	The vector field
	\begin{equation} \label{eq:str}
		|u|^{\tau_i} |v|^2 \cdot  \pipol^*\xi|_{U \setminus D}
	\end{equation}
	extends to an analytic vector field on $U$.
	Furthermore,
	\begin{enumerate}
		\item
		the singular set of this extension is the submanifold $U \cap D_a$,
		\item
		the Hessian of this extension has strictly negative real 
		eigenvalues along
		the normal bundle of the singular set.
	\end{enumerate}
\end{lemma}

\begin{proof}
	By the definition of a pull-back of a vector field, we find
	\[
	\pipol^*\xi|_{U\setminus D}
	=
	\left(
	\begin{matrix}
		\xi^u \\
		\xi^v
	\end{matrix}
	\right)
	=
	-(\Jac \pipol)^{-1}
	\cdot
	\left(
	\begin{matrix}
		\bar f_x / \bar f \\
		\bar f_y / \bar f
	\end{matrix}
	\right)
	=
	- \Jac \pipol^{-1} \cdot
	(\Jac \pipol^{-1})^*
	\cdot
	\left(
	\begin{matrix}
		m_i \bar u^{-1} \\
		m_a \bar v^{-1}
	\end{matrix}
	\right)
	\]
	We recall that in this context $(\Jac \pipol^{-1})^*$ stands for the 
	hermitian 
	adjoint of $\Jac \pipol^{-1}$. Here we use the assumption that $f = 
	u^{m_i} 
	v^{m_a}$
	is a monomial, and so its log derivatives
	are easily computed:
	\[
	f_u / f = m_i u^{-1},\qquad
	f_v / f = m_a v^{-1}.
	\]
	A computation gives (note that $x_v = 0$ and so $\det\Jac\pipol = x_u 
	y_v$)
	\[
	\Jac \pipol^{-1}
	(\Jac \pipol^{-1})^*
	=
	\dfrac{1}{|x_u y_v|^2}
	\left(
	\begin{matrix}
		|y_v|^2 &
		-y_v \bar y_u \\
		-y_u \bar y_v &
		|x_u|^2 + |y_u|^2
	\end{matrix}
	\right)
	=
	\dfrac{1}{|y_v|^2}
	\left(
	\begin{matrix}
		\left|\dfrac{y_v}{x_u}\right|^2 &
		-\dfrac{y_v}{x_u} \dfrac{\bar y_u}{\bar x_u} \\[1.1em]
		-\dfrac{y_u}{x_u} \dfrac{\bar y_v}{\bar x_u} &
		1 + |\dfrac{y_u}{x_u}|^2
	\end{matrix}
	\right).
	\]
	As a result, since $\tau_i = 2c_{1,i}$ because $i$ is invariant (recall 
	\Cref{rem:formula_ci} and \Cref{def:polar}), we find
	\begin{equation} \label{eq:str_vector}
		|u|^{\tau_i} |v|^2 \pipol^*\xi
		=
		v
		\cdot
		\left|\frac{u^{c_{1,i}}}{y_v}\right|^2
		\cdot
		\left(
		\begin{matrix}
			-m_i\dfrac{\bar v}{\bar u}\left|\dfrac{y_v}{x_u}\right|^2
			+
			m_a\dfrac{y_v}{x_u} \dfrac{\bar y_u}{\bar x_u} \\[1.1em]
			m_i\dfrac{\bar v}{\bar u}\dfrac{y_u}{x_u} \dfrac{\bar y_v}{\bar 
				x_u}
			-
			m_a - m_a\left|\dfrac{y_u}{x_u}\right|^2
		\end{matrix}
		\right).
	\end{equation}
	Here, the factor $\left|\frac{u^{c_{1,i}}}{y_v}\right|^2$ is a positive
	unit, since
	$y_v$ vanishes with order $c_{1,i}$ along $D_i$ (see \cref{block:c_1}).
	In fact, since $y_v$ and $y_u$ both vanish with order strictly higher
	than $x_u$, all the summands in each coordinates of the vector to the 
	right
	in \cref{eq:str_vector} are real analytic functions.
	As a result, \cref{eq:str} has a real analytic extension.
	The factor $v$ in \cref{eq:str_vector} guarantees that the extension
	vanishes along $D_i$.
	In fact, the extension of both functions
	\[
	|u|^{\tau_i} |v|^2 \xi^u,\qquad
	|u|^{\tau_i} |v|^2 \xi^v
	\]
	vanish along $D_a$, and the latter has the linear term $v$ with
	coefficient given by
	\[
	-m_a - m_a\left|\frac{y_u}{x_u}\right|^2 < 0
	\]
	at each point in $D_a$.
	It follows that the extension of \cref{eq:str} vanishes precisely along
	$D_a$, and that the Hessian has negative eigenvalues along the normal
	bundle.
\end{proof}

\begin{lemma}
	Let $\Uro = \sigma^{-1}(U) \subset \Yropol$ with polar coordinates
	$(r,s,\alpha,\beta)$ so that $u = re^{i\alpha}$ and $v = se^{i\beta}$.
	The vector field
	\begin{equation} \label{eq:strro}
		r^{\tau_i} s (\piropol)^*\xi,
	\end{equation}
	which is well defined on $\Uro \setminus \partial\Uro$,
	extends to a real analytic vector field on $\Uro$ which does not vanish
	on $\Uro$ and is tangent to $\Dro_i$ and is transverse to $\Dro_a$, 
	pointing
	outwards there.
\end{lemma}

\begin{proof}
	The vector field
	\[
	(\piropol)^*\xi|_{\Uro\setminus \partial \Uro}
	=
	\left(
	\begin{matrix}
		\xiror \\
		\xiroal \\
		\xiros \\
		\xirobe
	\end{matrix}
	\right)
	\]
	is the pull-back of $\pipol^*\xi|_{U\setminus D}$ to
	$\Uro\setminus \partial\Uro$, where the coordinates are given by
	\cref{eq:xiro_corner}.
	Since $\xi_U$ is tangent to $D_i$, the analytic function
	$|u|^{\tau_i}|v|^2\xi^u$ is
	divisible by $u$. Furthermore, by \cref{eq:str_vector},
	$|u|^{\tau_i}|v|^2\xi^u$ is
	also divisible by $v$. As a result, $\sigma^*(|u|^{\tau_i}|v|^2\xi^u)$
	is divisible by both $r$ and $s$, and so the functions
	\[
	\begin{split}
		r^{\tau_i} s \xiror
		&= s^{-1} \Re(e^{-i\alpha} \sigma^*(|u|^{\tau_i}|v|^2\xi^u)) \\
		r^{\tau_i} s \xiroal
		&= s^{-1} r^{-1} \Im(e^{-i\alpha} \sigma^*(|u|^{\tau_i}|v|^2\xi^u))
	\end{split}
	\]
	extend to real analytic functions on $\Uro$.
	Similarly, we have
	\[
	\begin{split}
		r^{\tau_i} s \xiros
		&= s^{-1} \Re(e^{-i\alpha} \sigma^*(|u|^{\tau_i}|v|^2\xi^v)) \\
		r^{\tau_i} s \xirobe
		&= s^{-2} \Im(e^{-i\alpha} \sigma^*(|u|^{\tau_i}|v|^2\xi^v))
	\end{split}
	\]
	The analytic function $|u|^{\tau_i} |v|^2 \xi^v$ is divisible by
	$v$, and so $\Re(e^{-i\alpha} \sigma^*(|u|^{\tau_i} |v|^2) \xi^v)$ is
	divisible by $s$. As a result, $r^{\tau_i} s \xiros$ extends
	as an analytic function.
	
	In a similar way, the function
	$\Im(e^{-i\alpha} \sigma^*(|u|^{\tau_i} |v|^2 \xi^v))$
	is divisible by $s$. But the only term in the expansion
	of $|u|^{\tau_i} |v|^2 \xi^v$ which has order $1$ in $v$ is a
	real multiple of the variable $v$, it follows that
	$\Im(e^{-i\alpha} \sigma^*(|u|^{\tau_i} |v|^2 \xi^v))$
	is divisible by $s^2$. As a result, $r^{\tau_i} s \xirobe$ has
	an analytic extension to $\Uro$ which is never $0$.
	
	We have seen that the function $r^{\tau_i} s \xiror$ is divisible
	by $r$, and so vanishes along the boundary piece $\Dro_i \cap \Uro$.
	As a result, the vector field $r^{\tau_i} s (\piropol)^*\xi$ is
	tangent to $\Dro_i \cap \Uro$.
	
	Similarly, since $\xi^v$ has linear part which is a negative
	multiple of $v$, the function
	\[
	r^{\tau_i} s \xiros
	= s^{-1} \Re(e^{-i\alpha} \sigma^*(|u|^{\tau_i}|v|^2\xi^v))
	\]
	takes negative values along the set defined by $s=0$. This means that
	the extension of $r^{\tau_i} s (\piropol)^*\xi$ is transverse
	to the boundary piece $\Dro_a$ and points outwards.
\end{proof}

\begin{definition}
	With $U$ as in \cref{block:str}, and $\Uro = \sigma^{-1}(U)$, denote by
	\[
	\xi_U,\qquad \xiro_U
	\]
	the unique extensions of the vector fields
	\cref{eq:str} and \cref{eq:strro} to $U$ and $\Uro$, respectively.
\end{definition}

\begin{block}
	For $p \in U$, denote by $\gamma_p$ the trajectory of $\xi_U$ starting
	at $p$. If $p \in U \setminus D$, then $\gamma_p$ is, up to 
	reparametrization,
	the trajectory of $\xi$ pulled back to $U$ via $\pipol$.
\end{block}

\begin{cor}
	Let $U$ be as in \cref{block:str}. Then, there exists a smaller
	neighborhood $V \subset U$ satisfying
	\begin{enumerate}
		\item
		if $p \in V$, then $\gamma_p$ does not escape $U$,
		\item
		if $p \in V \cap D_i$, then $\gamma_p$ converges to the unique 
		intersection
		point in $D_i \cap D_a$,
		\item
		if $p \in V \setminus D_i$, then $\gamma_p$ converges to a point in
		$D_a \setminus D_i$.
		
	\end{enumerate}
\end{cor}

\begin{cor}\label{cor:strict_trans}
	There exists an open set $V \subset \Ypol$ containing $\Cpol$ such that
	\[
	\pushQED{}
	\overline{\pipol^{-1}(S) \setminus \{0\}} \cap V = \emptyset.
	here
	\popQED
	\]
\end{cor}

\section{Poincar\'e-Hopf indices}

If $ij$ is an invariant vertex, then the vector field $\xi_U$
is tangent to the disks $D_i \cap U$ and $D_j \cap U$, and does not
vanish outside the intersection point.
As a result, restricting the vector field to either disk produces
an
index, see \Cref{not:index_vf}.

\begin{lemma} \label{lem:int_ind}
	\begin{enumerate}
		\item \label{it:int_ind_inv}
		Assume that $ij$ is an invariant edge.
		Then the restriction $\xi_U|_{D_i\cap U}$ has an isolated 
		singularity
		at $D_i \cap D_j$ of index one.
		\item \label{it:int_ind_i} The set
		$\Droc_{i,\theta} \cap \Uro$ consists of $\gcd(m_i,m_j)$ punctured 
		disks.
		At each puncture, the restriction of $\xiro_U$ has index
		\[
		\frac{1}{\gcd(m_i,m_j)}
		\left|
		\begin{matrix}
			\varpi_i & m_{i} \\
			\varpi_j & m_{j} \\
		\end{matrix}
		\right|
		+ 1.
		\]
		\item \label{it:int_ind_j}
		The set
		$\Droc_{j,\theta} \cap \Uro$ consists of $\gcd(m_i,m_j)$ punctured 
		disks.
		On each disk, the restriction of $\xiro_U$ has index
		\[
		\frac{-1}{\gcd(m_i,m_j)}
		\left|
		\begin{matrix}
			\varpi_i & m_{i} \\
			\varpi_j & m_{j} \\
		\end{matrix}
		\right|
		+ 1.
		\]
	\end{enumerate}
\end{lemma}

\begin{proof}
	First we prove \cref{it:int_ind_inv}. From
	\Cref{lem:d_pos,lem:elementary_invariant}, it 
	follows that the Hessian of the restriction has
	two real nonzero eigenvalues with the same sign. As a result, it  has 
	index $1$.
	
	For \cref{it:int_ind_i} and \cref{it:int_ind_j}, it suffices to have a 
	look
	at \cref{fig:corner_sings}, and use \cref{eq:wind_ph}. Indeed, observe 
	that, by \Cref{lem:transversality_number}, \cref{it:trans_ii} and 
	\cref{it:trans_iii}, the number of red points is  
	\begin{equation}\label{eq:pronged}
		\frac{1}{\gcd(m_i,m_j)}
		\left|
		\begin{matrix}
			\varpi_i & m_{i} \\
			\varpi_j & m_{j} \\
		\end{matrix}
		\right|.
	\end{equation}
	Which is the same as the number of green points. After contracting each 
	connected component of $\Dro_{i,j,\theta}$ to a point, we get a 
	\cref{eq:pronged}-pronged singularity on $\Dro_{j,\theta}$ and a 
	\cref{eq:pronged}-petal singularity on $\Dro_{i,\theta}$. So the result 
	follows after  \Cref{ex:ind_k}.
\end{proof}

\begin{lemma} \label{lem:varpi_index}
	Assume that $ik$ is an edge from $i$ to $k$, and that $i$ is invariant
	and $k$ not invariant. Then the
\index{Poincar\'e-Hopf index}
Poincar\'e-Hopf index of the vector 
	field
	$\xi_i$ at the intersection point $D_i \cap D_k$ equals $1-\varpi_k$.
\end{lemma}

\begin{proof}
	Consider the the vector field $\xiro_{i,k,\theta}|_{\Dro_{i,\theta}}$. 
	Let $L$ 
	be one of the boundary components which is contained in 
	$\Dro_{i,k,\theta}$. 
	Then,  by \Cref{lem:noninv_corner} and \Cref{lem:transversality_number} 
	\cref{it:trans_iii}, the boundary component $L$ has $\varpi_k m_i/ 
	\gcd(m_i,m_k)$ red vertices. So, after contracting $L$ to a point, the 
	vector 
	field $\xiro_{i,k,\theta}|_{\Dro_{i,\theta}}$ has a $\varpi_k m_i/ 
	\gcd(m_i,m_k)$-pronged singularity. Now recall that for an invariant 
	vertex 
	$i$, the vector field $\xiro_{i,\theta}$ is just the pullback by 
	$\sigma$ of 
	$\xi_i$ restricted to $\Dro_{i,\theta}$. Furthermore, the map 
	$\sigma|_{\Dro_{i,\theta}}:\Dro_{i,\theta} \to D_i$ is an $m_i$ regular 
	cover, 
	and, restricted to each connected component of $\Dro_{i,\theta}$ is an 
	$m_i/\gcd(m_i,m_k)$ regular cover.  We find that $\xi_i$ has at $D_i 
	\cap D_k$ 
	a $\varpi_k$-pronged singularity. So, after \Cref{ex:ind_k}, 
	\cref{eq:winding_index}, it has Poincar\'e-Hopf index $1-\varpi_k$.
\end{proof}

%
%
%
	\chapter{Other Exceptional Divisors} \label{s:others}
In this section, we consider an embedded resolution $\pipol:\Ypol\to \C^2$ 
(\Cref{def:pipol})
of the curve $(C,0)$ that resolves the generic polar curves, and fix a 
vertex $i \in \V$, which does not
correspond to the first blow-up, $i \neq 0$.
It may be impossible to scale and extend the
vector field $\xi$ over the exceptional divisor $D_i$, to get a vector
field on $D_i^\circ$, similar to $\xi_0$.
In terms of numerical invariants on the graph, we characterize
those divisors having such an extension. More concretely we show that it is 
precisely the vanishing of the polar weight $\varpi_i$ that controls when 
it is 
possible to do such an extension prior without doing the real oriented 
blow-up.
We describe a similar construction on the real oriented blow-up
$\Yro$, which can be done for any exceptional divisor, invariant or 
non-invariant.
This is 
further exemplified in \Cref{ex:angle_dependence}.

\begin{notation} \label{not:other_exc_coord}
	We fix a metric induced by a coordinate change in $\Gpol$ 
	(\Cref{def:gen_pol}). Fix also
\index{isometric coordinates}
isometric coordinates (recall 
	\Cref{def:isometric_coords}) $x,y$ in $\C^2$ with respect to this 
	metric chosen also so that the line $\{y=0\}$ is the tangent to $i$. 
	This can always be achieved by a further unitary change of coordinates.
	
	As we have been doing so far, we usually simplify the notation by
	identifying functions defined in a neighborhood
	of the origin in $\C^2$ with their pullback to $U$ via $\pipol$
	or $\Uro$ via $\piropol$. That is, we sometimes write simply $f$ and 
	$f_x$ instead
	of $\pipol^*f$ and $f_x$, etc. Furthermore, the partials of
	$x = \pipol^*x$ with respect to $u,v$ are denoted $x_u, x_v$, and so on.
	
	As in \Cref{s:real_oriented}, we have the real oriented
	blow-up $\sigma:\Yropol \to \Ypol$, and the chart $U$
	induces the chart $\Uro = \sigma^{-1}(U)$ with coordinates
	$r=|u|$, $\alpha = \arg(u)$ and $v$.
	The boundary of $\Uro$ is $\partial\Uro = (\piropol)^{-1}(D_i)$,
	given by $r = 0$.
	
	We choose a point $p\in D_i^\circ$, and a coordinate neighborhood
	$U$ containing $p$ so that $u(p) = v(p) = 0$.
	Since $x$ vanishes with order $c_{0,i}$ along $D_i$, and does
	not vanish on $U\setminus D_i$, we can choose the coordinates
	$u,v$ so that
	\[
	x = u^{c_{0,i}}.
	\]
	Since the restriction $\pipol|_{U\setminus D}:U\setminus D \to \C^2 
	\setminus \{0\}$
	is a local diffeomorphism, we can pull back $\xi$ to get
	$\pipol|_{U\setminus D}^*\xi$.
	
	In coordinates, we have complex functions $\xi^u, \xi^v$ on $U\setminus 
	D$
	so that this vector field is given by the vector
	\begin{equation} \label{eq:xi_U}
		\pipol|_{U\setminus D}^*\xi
		=
		\left(
		\begin{matrix}
			\xi^u\\
			\xi^v
		\end{matrix}
		\right)
		=
		\frac{-1}{\det\Jac\pi \bar f}
		\left(
		\begin{matrix}
			y_v \bar f_x - x_v \bar f_y \\
			-y_u \bar f_x + x_u \bar f_y
		\end{matrix}
		\right).
	\end{equation}
	Similarly, on $\Uro \setminus \partial \Uro$, we have a vector field
	\[
	\left.\piropol \middle|_{\Uro \setminus \partial \Uro}^* \xi \right.
	=
	\left(
	\begin{matrix}
		\xiror\\
		\xiroal\\
		\xirov
	\end{matrix}
	\right)
	\]
	Here, $\xiror$ and $\xiroal$ are real analytic functions,
	while $\xirov$ is, as before, a real analytic complex-valued function. 
	We write $v = s +it$ an correspondingly $\xirov = \xiros + i\xirot$.
	Recall, \cref{block:pullback_ro}, that a  computation in polar 
	coordinates gives 
	\begin{equation} \label{eq:xiro_coords}
		\xirov = \sigma^*\xi^v, \qquad
		\xiror = \Re\left(e^{-i\alpha} \sigma^*\xi^u \right),\qquad
		\xiroal = \Im\left( (re^{i\alpha})^{-1} \sigma^*\xi^u \right).
	\end{equation}
	
\end{notation}

\begin{definition}\label{def:relative_polar}
	Let $P_i$ be the relative polar curve defined by the partial $f_y$,
	and denote by $\tilde{P}_{i}$ its strict transform by $\pipol$.
\end{definition}

\begin{definition}
	Let $i \in \V$ with $i\neq 0$. Set
\index{$\Sigma_i$}
	\[
	\Sigma_i = \tilde{P}_{i} \cap D_i^\circ,\qquad
	\Siro_i = \sigma^{-1}(\Sigma_i), \qquad
	\Siro_{i,\theta} = \Siro \cap D_{i,\theta}^\circ.
	\]
	and
	\[
	\Sigma_U = \Sigma_i\cap U,\qquad
	\Siro_U = \Siro \cap \Uro.
	\]
\end{definition}

Note that we have a precise choice for the polar curve $P_i$,
which depends on the metric and the vertex $i$.
More precisely, $P_i$ is the vanishing set of the partial derivative of $f$
with respect to the direction orthogonal to the tangent associated with $i$.
Therefore, even if we fix the underlying metric, the definition of
$\Sigma_i$ for different vertices $i$ involves a different choice
of the polar curve.

In contrast, $\Sigma_0$ cannot be seen as the intersection of
$D_0^\circ$ with the strict transform of a particular polar curve. A signed 
count of singularities of the vector field
$\xi_0$, however, allows for estimating the number of points in $\Sigma_0$ 
as the following lemma shows:

\begin{lemma}\label{lem:estimate_sig_0}
	Let $F_0$ and $S_0$ be the number of fountains and saddle points, 
	respectively, of $\xi_0$. If $t$ is the number of tangents of $C$, then
	\[
	F_0 - S_0 = 2-t
	\]
\end{lemma}
\begin{proof}
	By \Cref{lem:attractors}, $\xi_0$ has no sinks. By 
	\Cref{lem:sinks_near_strict}, the index of $\xi_0$ near $D_0 \cap 
	\tilde{C}$ is $1$. The Euler characteristic of $D_0$ is $2$. A direct 
	application of Poincar\'e-Hopf index theorem yields the result.
\end{proof}

\begin{rem}
	The intersection $\Sigma_i$ consists
	of finitely many points.
	The preimage $\Siro_i$
	is therefore the union of finitely many fibers
	of the $S^1$-bundle $\Droc_i \to D^\circ_i$.
	These fibers are transverse to the foliation on $\Droc_i$ given by
	$\arg(f)$, whose leaves are $\Droc_{i,\theta}$ for $\theta \in 
	\R/2\pi\Z$.
	In fact,
	$|\Siro_{i,\theta}| = m_i \cdot |\Sigma_i| < \infty$.
	
\end{rem}

\section{Extension over the boundary} \label{ss:others_ext}

\begin{lemma} \label{lem:ext_others}
	We use \Cref{not:other_exc_coord}.
	\begin{blist}
		\item \label{it:ext_others_noninv}
		The vector field
		\begin{equation} \label{eq:ext_others_noninv}
			r^{\tau_i}
			\piropol|_{\Uro \setminus \partial \Uro}^* \xi
		\end{equation}
		on $\Uro \setminus \partial \Uro$ extends to an analytic vector 
		field
		on $\Uro$ which is tangent to $\partial \Uro$ and whose vanishing 
		set is $\Siro_U$.
		
		\item \label{it:ext_others_inv}
		If $i$ is an invariant vertex, that is, if $m_i - p_i = c_{1,i}$, 
		then the vector field
		\begin{equation} \label{eq:ext_others_inv}
			|u|^{2c_{1,i}} \pipol|_{U \setminus D}^* \xi = |u|^{\tau_i} 
			\pipol|_{U \setminus D}^* \xi
		\end{equation}
		on $U\setminus D$
		extends to an analytic vector field on $U$, which is tangent to
		$D \cap U$ and whose vanishing set is $\Sigma_U$.
	\end{blist}
\end{lemma}

\begin{proof}
	Since $x = u^{c_{0,i}}$, we have $x_v = 0$. As a result
	\[
	r^{\tau_i} \xi^u = r^{c_{1,i} + m_i -p_i} \frac{-y_v \bar{f}_x}{\det 
		\Jac 
		\pipol \bar{f}} = r^{c_{1,i} - (c_{0,i} -1)} r^{c_{0,i}-1} 
	\frac{-y_v}{\det 
		\Jac \pipol} r^{m_i - p_i} \frac{\bar{f}_x}{\bar{f}}
	\]
	The second factor  $r^{c_{0,i}-1} \frac{-y_v}{\det \Jac \pipol}$ 
	extends over the boundary by \Cref{lem:arg_equivariant}. This is 
	because by \cref{block:c_1} we have  $\ord_u y_v = c_{1,i}$ and $\ord_u 
	\det \Jac \pipol = \nu_i -1 = c_{0,i} + c_{1,i}-1$. The third factor 
	$r^{m_i - p_i} \frac{\bar{f}_x}{\bar{f}}$ also extends over the 
	boundary since $\ord_u \bar{f}_x \geq p_i$ (\cref{eq:def_p_eq}) and 
	$\ord_u \bar{f} =m_i$. Finally, since $c_{1,i} - c_{0,i} + 1 \geq 2$ 
	(because $i\neq 0$), the first factor $r^{c_{1,i} - c_{0,i}+1}$ 
	vanishes identically along the boundary $\partial \Uro$. We conclude, 
	by the last two formulas of \cref{eq:xiro_coords} that both $r^{\tau_i} 
	\xiror$ and $r^{\tau_i}\xiroal$ extend over and vanish along $ \partial 
	\Uro$.

	Similarly as above, again using \Cref{lem:arg_equivariant}, the function
	\begin{equation} \label{eq:xiv_x}
		r^{c_{1,i}+m_i-p_i} \frac{x_u \bar f_y}{\det\Jac\pi\bar f}
		=
		r^{c_{1,i}}
		\frac{c_{0,i} u^{c_{0,i}-1}}{\det\Jac\pi}
		\cdot
		r^{m_i - p_i}
		\frac{\bar f_y}{\bar f}
	\end{equation}
	extends to a real analytic $\cdot_i$-equivariant function of weight
	$c_{1,i} - m_i + p_i$ on $\Uro$. We are using that $\ord_u \bar{f}_y = 
	p_i$. Note that this function vanishes precisely along 
	$\{r^{-p_i}f_y=0\}$ because 
	\[
	r^{c_{1,i}}
	\frac{c_{0,i} u^{c_{0,i}-1}}{\det\Jac\pi}
	\cdot
	r^{m_i}
	\frac{1}{\bar f}
	\]
	is an unit. Also note that $\{r^{-p_i}f_y=0\} \cap \partial \Uro = 
	\tilde{P}_i \cap \partial \Uro = \Siro_U$.

	By \Cref{lem:tang_van}, and our choice of coordinates $x,y$, we have
	\[
	\ord_{u}(y_u)
	= \ord_{u}(y) - 1
	> \ord_{u}(x) - 1
	= \ord_{u}(x_u).
	\]
	Furthermore,
	$\ord_u(f_x) \geq \ord_u(f_y)$ by \cref{eq:def_p_eq}.
	Thus, as in the previous argument, the function
	\begin{equation} \label{eq:xiv_y}
		r^{c_{1,i}+m_i-p_i}\frac{-y_u \bar f_x}{\det\Jac\pi\bar f}
		=
		r^{c_{1,i}}\frac{-y_u}{\det\Jac\pi}
		\cdot
		r^{m_i-p_i}\frac{\bar f_x}{\bar f}
	\end{equation}
	on $\Uro\setminus\partial\Uro$ extends over the boundary and vanishes
	there.
	
	By \cref{eq:xi_U} and \cref{eq:xiro_coords}, the function
	$r^{\tau_i} \xirov$
	is the sum of \cref{eq:xiv_x} and \cref{eq:xiv_y},
	and so extends to a function on $\Uro$. Outside the boundary, the 
	vector field   $r^{\tau_i}
	\piropol|_{\Uro \setminus \partial \Uro}^* \xi$ does not vanish since 
	$\piropol$ is a local diffeomorphism on $\Uro \setminus \partial \Uro$ 
	and $\xi$ does not vanish outside the curve $C \subset \C^2$. Since the 
	function in \cref{eq:xiv_y} vanishes along $\partial \Uro$, the 
	vanishing set of $\xirov|_{\partial \Uro}$ is that of the function in 
	\cref{eq:xiv_x} which we have already seen to be $\Siro_U$. This 
	finishes the proof of \cref{it:ext_others_noninv}.
	
	\cref{it:ext_others_inv} follows from a completely analogous argument 
	by replacing
	$r^{c_{1,i}+m_i-p_i}$ by $|u|^{2c_{1,i}}$.
\end{proof}

\begin{lemma} \label{lem:ext_others_tot}
	We use \Cref{not:other_exc_coord}.
	\begin{blist}
		\item \label{it:ext_others_noninv_total}
		The vector field
		\begin{equation} \label{eq:ext_others_noninv_tot}
			\piropol|_{\Yro \setminus  \Dro}^* (|x|^{\tau_i / c_{0,i}}\xi)
		\end{equation}
		on $\Yro \setminus \Dro$ extends to an analytic vector field
		on $(\Yro \setminus \Dro) \cup \Droc_i$ which is tangent to 
		$\Droc_i$ and whose vanishing set is $\Siro_i$. Moreover, for each 
		$\theta \in \R / 2 \pi \Z$, the extension is tangent to 
		$\Droc_{i,\theta}$.
		
		\item \label{it:ext_others_inv_tot}
		If $i$ is an invariant vertex, that is, if $m_i - p_i = c_{1,i}$, 
		then the vector field
		\begin{equation} \label{eq:ext_others_inv_tot}
			\pipol|_{Y \setminus D}^* (|x|^{2c_{1,i}/c_{0,i}} \xi) = 
			\pipol|_{Y \setminus D}^*(|x|^{\tau_i/c_{0,i}}  \xi)
		\end{equation}
		on $Y\setminus D$
		extends to an analytic vector field on $(Y \setminus D) \cup 
		D^{\circ}_i$, which is tangent to
		$D^{\circ}_i$ and whose vanishing set is $\Sigma_i$.
	\end{blist}
\end{lemma}
\begin{proof}
	It is always possible to cover $\Droc_i$ by charts $\Uro$ with 
	coordinates $r,\alpha,v$ as in \Cref{not:other_exc_coord}. In those 
	coordinates, we have that 
	\[
	{\piropol}^*|x|^{\tau_i/c_{0,i}} = r^{\tau_i}.
	\]
	Similarly, it is possible to cover $D^{\circ}_i$ by charts $U$ with 
	coordinates $u,v$ such that 
	\[
	{\pipol}^*|x|^{\tau_i/c_{0,i}} = |u|^{\tau_i}.
	\]
	Then, the proof follows from \Cref{lem:ext_others}.
\end{proof}

\begin{definition}\label{def:xiroU}
	Denote by
\index{$\xiro_i$}
$\xiro_U$ the analytic vector field on $\Uro$
	whose restriction to $\Uro\setminus\partial\Uro$ coincides
	with \cref{eq:ext_others_noninv}. If $i$ is an invariant vertex then
	denote by $\xi_U$ the analytic vector field on $U$
	whose restriction to $U\setminus D$ coincides
	with \cref{eq:ext_others_inv}.
	
	Similarly, denote by $\hxiro_i$ the analytic vector field on $(Y 
	\setminus \Dro)\cup \Droc_i$
	whose restriction to $\Yro \setminus \Dro$ coincides
	with \cref{eq:ext_others_noninv_tot}. If $i$ is an invariant vertex then
	denote by $\hat \xi_i$ the analytic vector field on $(Y \setminus D) 
	\cup D^\circ_i$
	whose restriction to $Y\setminus D$ coincides
	with \cref{eq:ext_others_inv_tot}. Denote the restriction by 
	$\xi_i=\hat{\xi}_i|_{D_i^\circ}$.
	
	Finally, denote the other corresponding restrictions by
	$\xiro_i = \hxiro_i|_{\Droc_i}$
	and $\xiro_{i,\theta} = \xiro_i |_{\Droc_{i,\theta}}$.
	
\end{definition}

\begin{rem} \label{rem:an_in}
	Let $i \in \V_\Upsilon \setminus \{0\}$.
	In local coordinates on $D_i$, the vector field $\xi_i$ is determined by
	the initial part of the pullback of the function $f_y$.
	If we start with a different metric, then the same construction
	gives a vector field which is a linear combination of the pullbacks
	of $f_x$ and $f_y$. Since $f_x$ vanishes with a higher order than
	$f_y$, by \Cref{cor:strict_ine}, this construction gives the same
	result up to a real scalar.
	Therefore, the vector fields $\xi_i$ and $\xiro_i$ do not depend
	on the chosen generic metric, up to a positive scalar.
\end{rem}

\section{A potential for $\xi_i$ on invariant vertices}
\label{ss:potential_inv}
In this subsection, we consider an invariant vertex $i\in\Vpol\setminus 
\{0\}$.
We show that the vector field $\xi_i$,
defined in the previous subsection, has a scalar potential with respect
to the metric $g_i$, defined in \Cref{def:exc_metric}.
An analogous statement for $\xiro_{i,\theta}$, and its proof,
can be found in the next subsection.
We choose a coordinate chart around a point $p \in D_i^\circ$ as
in \cref{block:c_1}, and set $\Uro = \sigma^{-1}(U)$.

Let $i\in \V\setminus\{0\}$ be an invariant vertex. The function 
\begin{equation} \label{eq:fxmc}
	\pi^* \left(- \log \frac{|f|}{|x|^{m_i/c_{0,i}}}\right)
\end{equation}
extends as an analytic function over $D_i^\circ$. 
In fact, in coordinates $u,v$ in $U$, satisfying $x = u^{c_{0,i}}$,
the function $f$ is divisible by $u^{m_i}$, and
$|x|^{m_i / c_{0,i}} = |u^{m_i}|$.

\begin{definition}\label{def:grad_i}
	Denote by
\index{$\phi_i$}
$\phi_i:D_i^\circ \to \R$ the restriction of the extension of
	\cref{eq:fxmc} to $D_i^\circ$.
\end{definition}

In coordinates $u,v$, we can write $\pi^*f = u^{m_i}g(v) + \hot$,
where the higher order terms are divisible by $u^{m_i+1}$.
Therefore, \cref{eq:fxmc} equals
\[	
\begin{split}
	-\log |f| + \log |u|^{m_i/c_{0,i}} 
	& = -(\log |g(v) + \hot| + m_i\log|u|) + \log |u|^{m_i} \\ 
	& = - \log  |g(v)|
\end{split}
\]
along $D_i$, where $u$ vanishes. As a result,
\begin{equation}\label{eq:phi_i}
	\phi_i(v) = -\log |g(v)|.
\end{equation}

\begin{lemma}\label{lem:morse_invariant}
	Let $i\neq 0$ be an invariant vertex.
	The vector field $\xi_i$ is the gradient of $\phi_i$ with respect
	to the Riemannian
\index{metric!on $D_i^\circ$}
metric $g_i$ from  \Cref{def:exc_metric}.
\end{lemma}

\begin{proof}
	From \cref{eq:xi_U} and observing that $x_v=0$, we have the formula
	\begin{equation}\label{eq:xivex}
		|u|^{\tau_i}\xi^v = |u|^{\tau_i} \frac{y_u \bar{f}_x - x_u 
			\bar{f}_y}{x_u y_v \bar{f}}.
	\end{equation}
	We recall from \Cref{def:radial} that $\tau_i = c_{1,i} + m_i - p_i$. 
	Since $\varpi_i=0$ (recall \Cref{def:polar}), we have that $\tau_i = 2 
	c_{1,i}$ and so
	\[
	|u|^{\tau_i} =|u|^{2 c_{1,i}}= u^{c_{1,i}} \bar{u}^{c_{1,i}}.
	\]
	Similarly as in \cref{eq:xiv_y},
	the term $|u|^{\tau_i} \frac{y_u \bar{f}_x}{x_u y_v \bar{f}}$
	vanishes along $D_i^\circ$. The other term is
	\[
	|u|^{2 c_{1,i}} \frac{-x_u \bar{f}_y}{x_u y_v \bar{f}} 
	=  - \frac{|u|^{2 c_{1,i}}\bar{f}_y}{y_v\bar{f}}
	= - \frac{|u|^{2 c_{1,i}}\bar{f}_v}{|y_v|^2\bar{f}}.
	\]
	In the last inequality we have used that $\bar f_v = \bar{y}_v 
	\bar{f}_y$.
	Since $\varpi_i=0$, \Cref{lem:varpi} implies that
	$\ord_u \bar{f}_v = \ord_u \bar{f} = m_i$.
	Therefore, we have
	$\init_u f = u^{m_i}g(v)$ and $\init_u f_v = u^{m_i} g'(v)$.
	Along $D_i \cap U$, we find
	\[
	\left.
	|u|^{\tau_i} \xi^v
	\right|_{D_i\cap U}
	=
	h^{-1}_U(v) \frac{g'(\bar v)}{g(\bar{v})}.
	\]
	where $h_U$ is the real function defining the metric $g_i$, defined in
	\Cref{def:exc_metric}.
	The result now follows from \Cref{lem:cplex_gradient}, and
	\Cref{rem:exc_metric}.
\end{proof}

Before dealing with non-invariant vertices, we motivate the announced 
necessity 
to consider the real oriented blow-up with an example. In the next example 
we 
carefully analyze the cusp singularity and {\em attempt} an extension of 
the 
vector field in the style of the present subsection and observe the failure 
and 
its reason.

\begin{example}[The cusp $f(x,y) = y^2+x^3$] \label{ex:angle_dependence}
	We derive the minimal embedded resolution and associated invariants 
	step-by-step.
	
	\subsubsection*{Step 1: First Blow-up and Divisor $D_0$}
	We begin by blowing up the origin in $\C^2$. In the chart $(x,y) = 
	(u_0, 
	u_0v_0)$, the total transform of $f$ is:
	$$ \pi_0^*f = (u_0v_0)^2 + u_0^3 = u_0^2(v_0^2+u_0) $$
	This creates the first exceptional divisor $D_0$, given locally by 
	$u_0=0$. 
	The strict transform of the curve, $C_1$, has the equation 
	$g_1(u_0,v_0) = 
	v_0^2+u_0=0$. At the new origin $(u_0,v_0)=(0,0)$, the curve $C_1$ is 
	tangent to the divisor $D_0$. This tangency must be resolved.

	\textbf{Invariants for $D_0$:}
	\begin{itemize}
		\item By definition, $c_{0,0}=1, c_{1,0}=1$.
		\item $m_0 = \ord_{D_0}(f) = \mathrm{mult}_0(f) = 2$.
		\item $p_0 = \ord_{D_0}(f_y) = \ord_{D_0}(y) = \ord_{u_0}(u_0v_0) = 
		1$.
		\item $\varpi_0 = c_{1,0}-m_0+p_0 = 1-2+1=0$. $D_0$ is invariant.
	\end{itemize}
	
	\subsubsection*{Step 2: Second Blow-up and Divisor $D_1$}
	We blow up the point of tangency $P_1=(u_0,v_0)=(0,0)$. This creates 
	the 
	divisor $D_1$. Let's use the  chart, $(u_0,v_0)=(u_1v_1, v_1)$. The 
	total 
	transform 
	of $g_1$ is $v_1^2+u_1v_1 = v_1(v_1+u_1)$. Here, $D_1$ is given by 
	$v_1=0$, 
	the strict transform $\tilde{D}_0$ is $u_1=0$, and the strict transform 
	of 
	the curve is $\tilde{C}_1: v_1+u_1=0$. These three smooth curves 
	intersect 
	at the new origin $(u_1,v_1)=(0,0)$, which is a triple point and not a 
	normal crossing.
	
	\textbf{Invariants for $D_1$:}
	\begin{itemize}
		\item The center $P_1$ is a smooth point on $D_0$, so by 
		\Cref{lem:c_01_recursive}, $c_{0,1}=c_{0,0}=1$ and 
		$c_{1,1}=c_{1,0}+1=2$.
		\item We need the map from a chart seeing $D_1$ to $\C^2$. Using 
		the 
		chart $(u_0,v_0)=(u_1, u_1v_1)$ centered on $P_1$, the map is 
		$\pi(u_1,v_1)=(x,y) = (u_1, u_1^2v_1)$, and $D_1$ is $\{u_1=0\}$.
		$$ m_1 = \ord_{D_1}(f) = \ord_{u_1}((u_1^2v_1)^2+u_1^3) = 
		\ord_{u_1}(u_1^3(u_1v_1^2+1)) = 3 $$
		$$ p_1 = \ord_{D_1}(y) = \ord_{u_1}(u_1^2v_1) = 2 $$
		\item $\varpi_1 = c_{1,1}-m_1+p_1 = 2-3+2=1$. So $D_1$ is the 
		\textbf{non-invariant divisor}.
	\end{itemize}
	
	\subsubsection*{Step 3: Third Blow-up and Divisor $D_2$}
	We blow up the triple point $P_2=(u_1,v_1)=(0,0)$. This creates the 
	divisor 
	$D_2$ and separates the three intersecting curves. The final 
	configuration 
	is a normal crossings divisor. The resulting dual graph is:
	\begin{center}
		\begin{tikzpicture}[scale=1.5, every node/.style={scale=1.2}]
			\node[draw, circle] (D2) at (0,0) {$2$};
			\node[draw, circle] (D0) at (-1,0.8) {$0$};
			\node[draw, circle] (D1) at (-1,-0.8) {$1$};
			\node (C) at (1.2,0) {$\tilde{C}$};
			\draw (D2) -- (D0);
			\draw (D2) -- (D1);
			\draw[->, thick] (D2) -- (C);
		\end{tikzpicture}
	\end{center}
	
	\textbf{Invariants for $D_2$:}
	\begin{itemize}
		\item The center $P_2$ is the intersection $\tilde{D}_0 \cap D_1$, 
		so 
		by \Cref{lem:c_01_recursive}, $c_{0,2} = c_{0,0}+c_{0,1}=2$ and 
		$c_{1,2} = c_{1,0}+c_{1,1}=3$.
		\item A chart seeing $D_2=\{u_2=0\}$ is given by $\pi(u_2,v_2) = 
		(u_2^2v_2, u_2^3v_2^2)$.
		$$ m_2 = \ord_{D_2}(f) = \ord_{u_2}((u_2^3v_2^2)^2+(u_2^2v_2)^3) = 
		6 $$
		$$ p_2 = \ord_{D_2}(y) = \ord_{u_2}(u_2^3v_2^2)=3 $$
		\item $\varpi_2 = c_{1,2}-m_2+p_2 = 3-6+3=0$. $D_2$ is invariant.
	\end{itemize}
	The complete table of invariants confirms that $D_1$ is the unique 
	non-invariant divisor:
	\begin{center}
		\begin{tabular}{|l||c|c|c|} \hline
			$i$ & 0 & 1 & 2 \\ \hline \hline
			$c_{0,i}$ & 1 & 1 & 2 \\ \hline $c_{1,i}$ & 1 & 2 & 3 \\ \hline
			$m_i$ & 2 & 3 & 6 \\ \hline $p_i$ & 1 & 2 & 3 \\ \hline
			$\varpi_i $ & \textbf{0} & \textbf{1} & \textbf{0} \\ \hline
			$\tau_i $ & 2 & 3 & 6 \\ \hline
		\end{tabular}
	\end{center}
	
	\section*{Angle dependence}
	Now we show the angular dependence when approaching 
	the non-invariant divisor, $D_1$. To do this, we first find a local 
	coordinate 
	system $(u,v)$ where $D_1$ is given by $\{u=0\}$.
	
	As derived above, the second blow-up $\pi_1$ (centered at $p_1=(0,0)$ 
	in the $(u_0,v_0)$ chart) can be described by the map 
	$(u_0,v_0)=(u,uv)$. 
	The new exceptional divisor $D_1$ is $\{u=0\}$. Composing this with the 
	first blow-up map $\pi_0(u_0,v_0)=(u_0,u_0v_0)$ gives the map from our 
	local chart to $\C^2$:
	$$ (x,y) = \pi_0(\pi_1(u,v)) = \pi_0(u, uv) = (u, u(uv)) = (u, u^2v) $$
	This map $\pi(u,v)=(u, u^2v)$ is therefore the correct local model for 
	analyzing the vector field near $D_1$. It correctly yields the 
	invariants 
	calculated for $D_1$ in the table above: $\varpi_1=1$ and $\tau_1=3$.  
	We 
	recall the formula for the pullback of $\xi = -\nabla\log|f|$ 
	from \cref{eq:xi_U}:
	$$
	\pi^*\xi = \begin{pmatrix} \xi^u \\ \xi^v \end{pmatrix} = 
	\frac{-1}{\det\Jac\pi 
		\cdot \bar f} \begin{pmatrix} y_v \bar f_x - x_v \bar f_y \\ -y_u 
		\bar 
		f_x + 
		x_u \bar f_y \end{pmatrix}
	$$
	The appropriate scaling factor, as per \Cref{lem:ext_others_tot}, 
	involves 
	the 
	radial weight $\tau_1=3$. We must analyze the vector field 
	$|u|^{\tau_1}\pi^*\xi = |u|^3\pi^*\xi$. Let's compute the components.
	
	\begin{itemize}
		\item \textbf{Jacobian:}
		$$ \Jac\pi = \begin{pmatrix} x_u & x_v \\ y_u & y_v \end{pmatrix} = 
		\begin{pmatrix} 1 & 0 \\ 2uv & u^2 \end{pmatrix} \implies 
		\det\Jac\pi = 
		u^2. $$
		\item \textbf{Partials and pullbacks:}
		$$ f_x=3x^2, \quad f_y=2y $$
		$$ \pi^*f = u^3(uv^2+1), \quad \pi^*f_x = 3u^2, \quad \pi^*f_y = 
		2u^2v 
		$$
		\item \textbf{Components of $\pi^*\xi$:}
		$$ \xi^u = \frac{-1}{u^2 \bar{f}}(y_v \bar{f}_x - x_v \bar{f}_y) = 
		\frac{-1}{u^2 \bar{f}}(u^2 \cdot 3\bar{u}^2 - 0) = 
		\frac{-3\bar{u}^2}{\bar{f}} = 
		\frac{-3\bar{u}^2}{\bar{u}^3(1+\bar{u}\bar{v}^2)} = 
		\frac{-3}{\bar{u}(1+\bar{u}\bar{v}^2)} $$
		$$ \xi^v = \frac{-1}{u^2 \bar{f}}(-y_u \bar{f}_x + x_u \bar{f}_y) = 
		\frac{-1}{u^2 \bar{f}}(-2uv \cdot 3\bar{u}^2 + 1 \cdot 
		2\bar{u}^2\bar{v}) = 
		\frac{-2\bar{u}^2(-3uv+\bar{v})}{u^2 \bar{f}} = 
		\frac{-2(-3uv+\bar{v})}{u^2 
			\bar{u}(1+\bar{u}\bar{v}^2)} $$
	\end{itemize}
	
	Now we apply the scaling factor $|u|^3$:
	\begin{itemize}
		\item \textbf{u-component:}
		$$ |u|^3\xi^u = |u|^3 \frac{-3}{\bar{u}(1+\bar{u}\bar{v}^2)} = 
		\frac{-3 
			u\bar{u}|u|}{\bar{u}(1+\bar{u}\bar{v}^2)} = 
		\frac{-3u|u|}{1+\bar{u}\bar{v}^2} $$
		As $u \to 0$, this component tends to $0$. This is expected, as the 
		extended vector field must be tangent to the divisor $D_1=\{u=0\}$.
		
		\item \textbf{v-component:} This is the crucial part.
		$$ |u|^3\xi^v = |u|^3 \frac{-2(-3uv+\bar{v})}{u^2 
			\bar{u}(1+\bar{u}\bar{v}^2)} = \frac{u\bar{u}|u| \cdot 
			(-2(-3uv+\bar{v}))}{u^2 \bar{u}(1+\bar{u}\bar{v}^2)} = 
		\frac{|u|}{u} 
		\frac{-2(-3uv+\bar{v})}{1+\bar{u}\bar{v}^2} $$

		Let's analyze the limit of the scaled $v$-component as we approach 
		the 
		divisor 
		$D_1$, i.e., as $u \to 0$. We introduce polar coordinates for $u$, 
		setting 
		$u=re^{i\alpha}$ where $r=|u|$. The term $\frac{|u|}{u}$ becomes:
		$$ \frac{|u|}{u} = \frac{r}{re^{i\alpha}} = e^{-i\alpha} $$
		Now we can compute the limit:
		$$
		\lim_{u \to 0} \left( |u|^3\xi^v \right) = \lim_{r \to 0} \left( 
		e^{-i\alpha} 
		\frac{-2(-3re^{i\alpha}v+\bar{v})}{1+re^{-i\alpha}\bar{v}^2} 
		\right) = 
		e^{-i\alpha} \frac{-2(0+\bar{v})}{1+0} = -2\bar{v}e^{-i\alpha}
		$$
	\end{itemize}
	
	The limit of the scaled vector field on the divisor $D_1$ is not a 
	single 
	well-defined vector field. Instead, it depends on the angle $\alpha = 
	\arg(u)$ 
	of approach to the divisor. For each direction of approach, we get a 
	different 
	resulting vector field on $D_1$.
	
	This explicit angular dependence, captured by the term $e^{-i\alpha}$, 
	is 
	precisely the reason a standard extension fails.  As shown in the proof 
	of 
	\Cref{lem:ext_others}, the limit of the rescaled vector field is an 
	equivariant 
	function of weight $\varpi_i$. Here, for $i=1$, we have $\varpi_1=1$, 
	and 
	indeed our resulting limit $-2\bar{v}e^{-i\alpha}$ is equivariant of 
	weight 
	$1$ 
	with respect to the $\cdot_1$-action (which corresponds to rotating 
	$u$).
	
	The \textbf{real oriented blow-up} $\sigma: \Yro \to Y$ resolves this 
	indeterminacy. It replaces the divisor $D_1$ (a complex manifold) with 
	its 
	pre-image $\Dro_1 = \sigma^{-1}(D_1)$, which is the space of oriented 
	real 
	normal 
	directions 
	to $D_1$. In our local chart, $\Dro_1$ is parametrized by $(\alpha, 
	v)$. 
	The 
	limit vector field now becomes a well-defined vector field on this new, 
	larger 
	space $\Dro_1$, with its value at a point $(\alpha, v)$ being precisely 
	the 
	limit we calculated for that specific angle of approach.
\end{example}

\section{A potential for $\xi^{\mathrm{ro}}_{i,\theta}$} \label{ss:potential_ro}

In this subsection, we fix a non-invariant vertex $i \in \Vpol$,
$\varpi_i \neq 0$, and a value $\theta \in \R/2\pi\Z$ of $\arg(f)$. We 
prove an analogous result to \Cref{lem:xi_0_gradient,lem:morse_invariant} 
but for non-invariant vertices, that is, we show the existence of a 
potential function for the vector field $\xiro_{i,\theta}$.

\begin{block}
	Consider the Riemann surface $\Droc_{i,\theta}$, defined in
	\Cref{def:angled_ray}, with the metric
	$\gro_{i,\theta}$ defined in \Cref{def:exc_metric}.
	If we contract each boundary component of $\Dro_{i,\theta}$ to a point,
	we get a branched covering over $D_i$. Thus, $\Droc_{i,\theta}$ is a 
	Riemann
	surface with punctures. If $j\in\Vpol$ is a neighbor of $i$, let
	$U \subset \Ypol$ be a small coordinate chart containing the 
	intersection
	$D_i \cap D_j$, and let $\Uro = \sigma^{-1}(U)$. We can assume that
	$D_i \cap U$ is a disc. Thus, $\Droc_{i,\theta} \cap \Uro$ covers
	a punctured disc, and so is a disjoint union of some number of punctured
	discs.
\end{block}

\begin{lemma}
	Let $D \subset \Droc_{i,\theta}$ be a connected component, and
	take $j\in\Vpol$ so that $ji$ is the unique edge pointing to $i$.
	Let $U$ and $\Uro$ be as in the previous paragraph.
	Then $D$ has genus zero, and $D \cap \Uro$ is connected.
\end{lemma}

\begin{proof}
	It follows from \Cref{lem:varpi_char}\cref{it:varpi_char_x} that
	the function $\arg(x)$ takes a constant value, say $\theta'$, along $D$.
	Thus, $D$ is a connected component of the set
	$\arg(x)^{-1}(\theta') \cap \Droc_i$.
	Thus, $D$ embeds into the Milnor fiber at radius zero of the function
	$x$. This Milnor fiber is a disk, and so $D$ has genus zero.
	The modification $\Ypol\to\C^2$, and its real oriented blow-up,
	induce a decomposition of the Milnor fiber of the function $x$ at 
	radius zero 
	(recall \Cref{def:milnor_fib_radius_0}) 
	consisting of nested discs. Thus, each connected component
	has a unique outer boundary component. For the vertex $0$, this outer
	component is the boundary of the Milnor fiber, whereas for other 
	vertices $i$,
	the outer component corresponding to the unique edge pointing to $i$.
	Since the closure of $D$ has a unique outer boundary component,
	the set $D \cap \Uro$ is a single punctured disc, in particular,
	it is connected.
\end{proof}

\begin{lemma} \label{lem:potential_xiro}
	Let $i$ be any vertex in $\Vpol \setminus \{0\}$. Then, there exists
	a function
\index{$\phiro_{i,\theta}$}
 $\phiro_{i,\theta}:\Droc_{i,\theta}\to\R$ so that
	\[
	\xiro_{i,\theta} = \nabla \phiro_{i,\theta}
	\]
	where the gradient is taken with respect to the metric $\gro_{i,\theta}$
\index{metric!on $\Droc_{i,\theta}$}
	defined in \Cref{def:exc_metric}.
	
	Furthermore, if $k$ is a neighbor of $i$ so that the edge $ik$ is 
	oriented
	from $i$ to $k$, and $\varpi_k \neq 0$,
	then $\phiro_{i,\theta}$ extends over the corresponding
	boundary components with constant value there.
\end{lemma}

\begin{proof}
	If $i$ is invariant, then we can take $\phiro_{i,\theta} = 
	\sigma^*\phi_i$,
	where $\phi_i:D_i \to \R$ is defined in the previous subsection.
	We leave the second statement for invariant vertices to the end of
	this proof.
	
	So assume that $\varpi_i \neq 0$. By the previous lemma, 
	$\Droc_{i,\theta}$
	is a punctured sphere. To each neighbor of $i$, there corresponds
	a set of punctures. If $j$ is the unique neighbor of $i$ so that
	there is an edge from $j$ to $i$, then there is only one puncture 
	corresponding
	to $j$ on each connected component.
	Let $\Drosc_{i,\theta}$ be the Riemann surface obtained by filling
	in all the other punctures, thus, $\Drosc_{i,\theta}$
	is a disjoint union of copies of $\C$.
	Denote by $\karo_{i,\theta}$ the differential form dual to
	$\xiro_{i,\theta}$ with respect to the metric $\gro_{i,\theta}$.
	Let $u,v$ be a holomorphic coordinate system on a chart in
	$\Ypol$, near a point $p\in D_i$ as before, and expand
	
	\[
	f(u,v) = u^{m_i}f_{m_i}(v) + \ldots + u^{m_i+\varpi_i}f_{m_i+\varpi_i} 
	+ \hot
	\]
	Then, by \Cref{lem:varpi}, we have
	\[
	f'_{m_i} = \ldots = 
	f'_{m_i+\varpi_i - 1} = 0,
	\]
	and
	\[
	f'_{m_i+\varpi_i} \not\equiv 0.
	\]
	As a result, we can write
	\[
	\init_u f = u^{m_i} f_{m_i}(v),\qquad
	\init_u f_v = u^{m_i+\varpi_i} f'_{m_i+\varpi_i}(v).
	\]
	Similarly as in the proof of \Cref{lem:ext_others}, the restriction of
	$|u|^{\tau_i}\xi^v$ to $\partial \Uro$ is given by
	\[
	|u|^{\tau_i} \frac{x_u \init_u \bar f_y}{x_u y_v \init_u \bar f}
	=
	|u|^{\tau_i} \frac{\init_u \bar f_v}{|y_v|^2 \init_u \bar f}
	=
	\frac{|u|^{\tau_i} \bar u^{\varpi_i}}{|y_v|^2}
	\frac{\bar f'_{m_i+\varpi_i}}{\bar f_{m_i}}
	=
	\left(
	\frac{u}{|u|}
	\right)^{-\varpi_i}
	\frac{|u|^{2c_{1,i}} }{|y_v|^2}
	\frac{\bar f'_{m_i+\varpi_i}}{\bar f_{m_i}}.
	\]
	Recall that the metric on $\Droc_{i,\theta}$ is given by the function
	\[
	h_U(v) = \lim_{u\to 0}\frac{|y_v|}{|u|^{c_{1,i}}}.
	\]
	Furthermore, the function $u/|u|$ is locally constant on
	$\Droc_{i,\theta}$ because of \Cref{lem:varpi}.
	
	By \Cref{rem:exc_metric}, the differential form $\karo_{i,\theta}$ is
	given, in the coordinate $v$, by the complex function
	\[
	\left(
	\frac{u}{|u|}
	\right)^{-\varpi_i}
	\frac{\bar f'_{m_i+\varpi_i}}{\bar f_{m_i}},
	\]
	which is antiholomorphic on $\Droc_{i,\theta}$ for $\theta$ fixed. 
	By the Cauchy-Riemann equations, the form $\karo_{i,\theta}$ is closed.
	
	Now, assume that $v$ is a holomorphic coordinate in a chart around
	a puncture, corresponding to a vertex $k$, so that the edge $ik$ is 
	oriented
	from $i$ to $k$. Assume that the puncture is given by $v=0$.
	As a result, we can represent $\karo_{i,\theta}$ by an antiholomorphic
	complex function $\alpha + i \beta$ on a punctured disk. Since
	\[
	\lim_{v\to 0} |u|^{\tau_i} |v|^{\tau_k} \xi^v = 0,
	\]
	the function $|v|^{\tau_k}(\alpha + i\beta)$ is bounded on the 
	punctured disk.
	It follows from the Riemann extension theorem that the holomorphic 
	function
	$\alpha - i \beta$ has a meromorphic singularity at $v = 0$.
	
	As a result, the vanishing order of $\alpha - i \beta$ at
	$v=0$ is minus the Poincar\'e-Hopf index of $\xiro_{i,\theta}$ at the 
	puncture.
	By \Cref{lem:int_ind}, this index is
	\[
	\frac{-1}{\gcd(m_i,m_k)}
	\left|
	\begin{matrix}
		\varpi_k &
		m_k \\
		\varpi_i &
		m_i
	\end{matrix}
	\right|
	+ 1.
	\]
	By  \Cref{cor:ups_pol} and \Cref{cor:Gapol_ah} \cref{it:Gapol_ah_pos}, 
	the index
	is nonpositive, and so $\karo_{i,\theta}$ does not have a pole at $v=0$.
	Therefore, $\karo_{i,\theta}$ extends as
	a closed form on $\Drosc_{i,\theta}$, which is a disjoint union of
	copies of $\C$.
	It follows that $\karo_{i,\theta}$ is exact, i.e. there exists
	$\phiro_{i,\theta}:\Droc_{i,\theta}\to\R$
	with $\karo_{i,\theta} = d\phiro_{i,\theta}$,
	or equivalently, $\xiro_{i,\theta} = \nabla \phiro_{i,\theta}$.
	Furthermore, $\phiro_{i,\theta}$ extends continuously over the punctures
	corresponding to the edge $ik$, and so, has a constant limit
	at each corresponding boundary component.
	
	If $i$ is invariant, then
	the same argument as above applies to show that $\phi_i$ has a limit at
	a puncture corresponding to $ik$, if $i$ is such that $i \to k$,
	using, again, \Cref{lem:int_ind}.
	This finishes the proof.
\end{proof}

%
%
%
	\chapter{Singularities on the Boundary of $Y^{\mathrm{ro}}_{\mathrm{pol}}$}
\label{s:sings_on_boundary}

In this section we study the singularities of $\xiro_U$ (\Cref{def:xiroU}) 
on
$\partial \Uro$. Fix a point $p \in \Sigma_{i}$ with $i \neq 0$ and assume 
that
the strict transform of $\{f_y=0\}$ defines the polar curve $P_i$.
Since $\pipol$ resolves the polar curve $P_i$
(\Cref{def:relative_polar}), the function $\pipol^*f_y$ vanishes along a 
smooth
curve $\tilde{P}_i$ passing through $p$ meeting $D_i^\circ$ transversely. 
As a
consequence, on top of choosing $u$ so that $x=u^{c_{0,i}}$ as in
\Cref{not:other_exc_coord}, we can choose $v$ so that
\[
\pipol^*f_y = u^{p_i}v.
\]

\section{Spinal cells}

In this subsection we give name to the sets of trajectories of the vector 
field $\pipol^*\xi$ that converge to points in the sets $\Sigma_i$.

\begin{definition}\label{def:spinal_cell}
	Let $i\in \Vpol \setminus \{0\}$, let $p \in \Sigma_i$.
	We define the set
\index{$S(p)$}
	\[
	S(p) = \set{b \in \Tu^*}{\text{the trajectory } (\pipol)^* \gamma_b 
		\text{ converges to } p \text{ in } \Ypol}
	\]
	Now fix $q \in \sigma^{-1}(p)$ and $\theta = \arg(q)$. We define the 
	sets
	\[
	S_\theta(p) = \set{b \in \Tu_\theta^*}{\text{the trajectory } 
		(\pipol)^* \gamma_b \text{ converges to } p \text{ in } \Ypolte}
	\]
	and
	\[
	S(q) = \set{b \in \Tu^*}{\text{the trajectory } (\piropol)^* \gamma_b 
		\text{ converges to } q \text{ in } \Yropolte}
	\]
\end{definition}

\begin{rem}
	We describe a few easy properties of the sets defined above. Recall the 
	definition of the total spine $S$ in \Cref{def:total_spine}. 
	
	Let $p \in \Sigma_i$. We have the containments,
	\[
	S(q) \subset S_\theta(p) \subset S(p) \subset S.
	\]
	with $\arg(q) = \theta$. Also, by definition 
	\[
	S(p) = \bigcup_{q \in \sigma^{-1}(p)} S(q)
	\]
	and
	\[
	S_\theta(p) = \bigcup_{\substack{q \in \sigma^{-1}(p) \\
			\arg(q)= \theta}} S(q).
	\]
	Finally observe that $\sigma^{-1}(p)$ is a circle and that there are 
	exactly $m_i$ points $q \in \sigma^{-1}(p)$ with $\arg(q) = \theta$ for 
	each $\theta \in \R / 2 \pi \Z$.
\end{rem}

All the sets in \Cref{def:spinal_cell} are contained, by definition, in 
$\Tu^*$. In the next definition we describe the counterpart of $S(q)$ in 
$\Yropol$.
\begin{definition}
	For $q \in \Siro_{i,\theta}$ we define the set
	\[
	\Sro(q) = \set{a \in \Yropolte \setminus \partial \Yropolte}{\text{the 
			trajectory } \gamma_a \text{ of } (\piropol)^*\xi \text{ converges 
			to } 
		q \text{ in } \Yropolte}.
	\]
	Equivalently, $\Sro(q) = (\piropol)^{-1}(S(q))$. Analogously, we define 
	\[
	\Sro(p) = \set{a \in \Yropol \setminus \partial \Yropol}{\text{the 
			trajectory } (\pipol)^* \gamma_b \text{ converges to } p \text{ in 
			} 
		\Ypol},
	\]
	or, equivalently, $\Sro(p) = (\piropol)^{-1}(S(p))$.
\end{definition}

\begin{definition}\label{def:local_spinal}
	Let $U$ be a given coordinate chart around $p$ as in the previous 
	sections.
	Recall \Cref{def:xiroU} of $\xiro_U$.  We define
	\[
	\Sro_U = \set{a \in \Uro }{\text{the trajectory } \gamma_a \text{ of } 
		\xiro_U \text{ converges to a point in } \sigma^{-1}(p)}.
	\]
\end{definition}
Note that the set $\Sro_U$ depends on $p$. We drop $p$ from the notation to 
avoid cumbersome formulas but it should be remembered since $U$ is a chart 
around $p$. Also note that it follows from the definition that the 
trajectory 
$\gamma_a$ should be completely contained in $\Uro$ since it is a 
trajectory of 
$\xiro_U$.

The main goal of this section is to prove the following proposition. Its 
proof is relegated to \Cref{ss:center_stable} after enough theory is 
developed.
\begin{prop}\label{prop:properties_SU}
	The set $\Sro_U$ satisfies
	\begin{enumerate}
		\item It is a submanifold of $\Uro$ of real dimension $3$ 
		transverse to $\partial \Uro$,
		\item it contains $\sigma^{-1}(p)$,  and
		\item \label{it:hessian_explicit} if $q \in \sigma^{-1}(p)$ then 
		the tangent space $T_q\Sro_U$ is the sum of the kernel and the 
		negative eigenspace of $\Hess_q \xiro_U$ (\cref{eq:hess_xiro}), 
		which is generated by the three vectors
		\[
		\partial_r + P \partial_s + Q \partial_t, \qquad \partial_\alpha, 
		\qquad P' \partial_s + Q' \partial_t, 
		\]
		where $P+iQ = -\bar{\lambda}_p^{-1} e^{-i \varpi_i \alpha} 
		\overline{\partial_r r^{\tau_i}\xirov(q)}$ and $P' + i Q' = 
		\sqrt{-e^{-i \varpi_i \alpha} \lambda_p}$.
	\end{enumerate}
\end{prop}
\section{The Hessian in $U^{\mathrm{ro}}$}
In this subsection we compute the Hessian of $\xiro_U$ along its
singular set $\sigma^{-1}(p)$, for some $p \in \Sigma_i$.
This Hessian is block-triangular.

\begin{definition}
	For $\lambda \in \C$, define $2\times 2$ real matrices
	\[
	A(\lambda)
	=
	\left(
	\begin{matrix}
		\Re(\lambda) & -\Im(\lambda) \\
		\Im(\lambda) & \Re(\lambda) \\
	\end{matrix}
	\right),\qquad
	\bar A(\lambda)
	=
	\left(
	\begin{matrix}
		\Re(\lambda) & \Im(\lambda) \\
		\Im(\lambda) & -\Re(\lambda) \\
	\end{matrix}
	\right).
	\]
\end{definition}

\begin{rem}
	A direct computation yields
	\[
	\bar{A}(\lambda)^{-1} = \frac{1}{|\lambda|^2} \bar{A}(\lambda).
	\]
\end{rem}
\begin{rem}\label{rem:anti_hessian}
	The matrices $A(\lambda)$, $\bar A(\lambda)$ describe real linear
	maps $\C \to \C$ which are either complex linear on antilinear.
	If $V\subset \C$ is an open set, and $h:V\to\C$ is differentiable,
	then the Jacobian of $h$ is given by the formula
	\begin{equation} \label{eq:jacob_hesse}
		\Jac h(v) = A(\partial_v h(v)) + \bar A(\partial_{\bar v} h(v)).
	\end{equation}
	If $\lambda \neq 0$. Denote by $\sqrt{\lambda}$ a root of $\lambda$ and 
	by 
	$\sqrt{-\lambda}$ a root of $-\lambda$. Then the matrix $\bar 
	A(\lambda)$ 
	has
	eigenvectors $\sqrt{\lambda}$, with eigenvalue $|\lambda| > 0$,
	and $\sqrt{-\lambda}$, with eigenvalue $-|\lambda| < 0$. Furthermore, 
	these 
	two eigenspaces are orthogonal with respect to any inner product that 
	is 
	compatible with the complex structure of $\C$.
	
	The matrix in \cref{eq:jacob_hesse} is the Jacobian matrix of the 
	function
	$h:V \to \C$, where we identify $\C$ with $\R^2$ as usual
	This way,
	we can consider $h$ as a vector field (recall \Cref{rem:exc_metric}).
	If $h$ vanishes at some point
	$v \in V$, then this Jacobian matrix is the Hessian matrix of $h$
	at $v$.
\end{rem}

\begin{lemma}\label{lem:antiholo}
	Let $V \subset \C$ be an open set, and $h,k:V \to \C$ differentiable 
	functions.
	Assume that $h$ is antiholomorphic and has a simple zero at $p \in V$,
	and that $k(p) \neq 0$.
	Then the function $k\cdot h$, seen as a vector field on $V$,
	has an elementary singularity at $p$
	whose Hessian is
	$\bar{A} (k(p) \cdot \partial_{\bar v} h(p))$.
\end{lemma}

\begin{proof}
	The lemma follows from \Cref{rem:anti_hessian}, since
	\[\begin{split}
		\partial_v(h\cdot k)(p)
		&= \partial_v h(p) \cdot k(p) + h(p) \cdot \partial_v k(p)
		= 0,
		\\
		\partial_{\bar v}(h\cdot k)(p) 
		&=  \partial_{\bar v} h(p) \cdot k(p) + h(p) \cdot \partial_{\bar 
			v} k(p) 
		= k(p) \cdot \partial_{\bar v} h(p). here
	\end{split}
	\]
\end{proof}
In this subsection we use real coordinates $r,\alpha, s,t$ for $\Uro$ with 
$r,\alpha$ (the pullback of) polar coordinates for $u$ and $s+it=v$, and 
$u,v$ coordinates in $U$.
\begin{lemma} \label{lem:zeros_not_0}
	There exists a $\lambda_p \in \C^*$ so that
	if $q\in \sigma^{-1}(p)$, and we set $\alpha = \arg(u)(q)$. Then, the 
	matrix $\Hess_q \xiro_U $ expressed in $2 \times 2$ blocks with respect 
	to the coordinates $r,\alpha, s, t$, is 
	\begin{equation}\label{eq:hess_xiro}
		\Hess_q \xiro_U 
		=
		\left(
		\begin{matrix}
			\multicolumn{2}{c}{\multirow{2}{*}{\scalebox{2}{$0$}}} &  
			\multicolumn{2}{c}{\multirow{2}{*}{\scalebox{2}{$0$}}} \\
			\multicolumn{2}{c}{} & \multicolumn{2}{c}{} \\
			(\partial_r \Re (r^{\tau_i} \xirov))|_q & 0 &  
			\multicolumn{2}{c}{\multirow{2}{*}{$\bar 
					A(e^{-i\varpi_i\alpha}\lambda_p)$}} \\
			(\partial_r \Im (r^{\tau_i} \xirov)|_q  & 0 & 
			\multicolumn{2}{c}{} 
		\end{matrix}
		\right).
	\end{equation}
\end{lemma}

\begin{proof}
	We split the proof into three parts.
	
	Part 1. In this first part we take care of the first $2$ rows of the
	matrix. In order to show that this $2 \times 4$ matrix is null, it is
	enough to show that both $ r^{\tau_i} \xiror$ and $r^{\tau_i} \xiroal$
	are divisible by $r^2$, that is, in particular they are in the square of
	the maximal ideal of the local ring at $q$.
	
	First we take care of $r^{\tau_i} \xiror$ which, by
	\cref{eq:xiro_coords}, equals $r^{\tau_i} \Re (e^{-i \alpha} \sigma^*
	\xi^u)$. Now, using \cref{eq:xi_U} and our choice of coordinates we find
	that 
	\[
	\ord_u(y_v\bar{f}_x) >  c_{0,i}  + p_i.
	\]
	Because $\ord_u(\bar{f}_x) \geq \ord_u(\bar{f}_y)$
	by \ref{eq:def_p_eq} and
	$\ord_u (y_v) = c_{1,i} > c_{0,i}$ by
	\Cref{lem:tang_van}. On another hand, the order of vanishing of the 
	denominator is 
	\[
	\ord_u(\det\Jac\pipol \bar{f})= c_{0,i} + c_{1,i} -1 + m_i.
	\]
	Finally, taking into account the factor $r^{\tau_i}$, we estimate
	\[
	\begin{split}
		\ord_r(r^{\tau_i} \xiror)
		&=\tau_i + \ord_r(y_v\bar{f}_x) - c_{0,i} - c_{1,i} +1 - m_i \\
		&> c_{1,i} + m_i - p_i + c_{0,i} + p_i - c_{0,i} - c_{1,i} +1 -m_i 
		\\
		&= 1.
	\end{split}
	\]
	This implies that $r^{\tau_i}\xiror$ vanishes with order at least $2$ 
	along the boundary.
	
	Now we take care of $r^{\tau_i} \xiroal$. The argument for this part is
	more delicate. Notice that in the expression for $\xiroal$ in
	\cref{eq:xiro_coords} there is a factor $r^{-1}$ that was not present in
	the step before. In this case, a similar computation yields only the
	weaker inequality $\ord_r(r^{\tau_i}\xiroal) \geq 1$. Next, we are going
	to argue that this function actually vanishes with order $2$ at $q$. If
	the above inequality is strict, there is nothing to prove. Assume
	otherwise that
	\begin{equation} \label{eq:ord_xiroal}
		\ord_r(r^{\tau_i}\xiroal) = 1.
	\end{equation}
	That is equivalent to
	saying that $\ord_r (r^{\tau_i} \xi^u) = 2$, which in turn, is 
	equivalent to saying that 
	\begin{equation}\label{eq:equality_ord}
		\ord_r(y_v\bar{f}_x) =  c_{0,i} + 1  + p_i.
	\end{equation}
	In general we know that $\ord_r(y_v) \geq c_{0,i} + 1$ and 
	$\ord_r(\bar{f}_x) \geq p_i$ so, in order to have 
	\cref{eq:equality_ord}, we need these two inequalities to be 
	equalities, that is:
	
	\[
	c_{1,i} = \ord_r(y_v) = c_{0,i} +1 
	\]
	and
	\[
	\ord_r(\bar{f}_x) = p_i
	\]
	A computation gives (since $u_v = 0$):
	\[
	\left(
	\begin{matrix}
		f_u\\
		f_v
	\end{matrix}
	\right)
	=
	\left(
	\begin{matrix}
		x_u & y_u\\
		0   & y_v
	\end{matrix}
	\right)
	\left(
	\begin{matrix}
		f_x\\
		f_y
	\end{matrix}
	\right),\quad
	\left(
	\begin{matrix}
		f_x\\
		f_y
	\end{matrix}
	\right)
	=
	\frac{1}{x_u y_v}
	\left(
	\begin{matrix}
		y_v & -y_u\\
		0   & x_u
	\end{matrix}
	\right)
	\left(
	\begin{matrix}
		f_u\\
		f_v
	\end{matrix}
	\right)
	=
	\left(
	\begin{matrix}
		\frac{1}{x_u} f_u - \frac{y_u}{x_uy_v} f_v\\
		\frac{1}{y_v} f_v
	\end{matrix}
	\right).
	\]
	In particular, we use that $\det\Jac\pipol = x_u y_v$, and so
	\begin{equation} \label{eq:xi_u_ord}
		\xi^u
		= \frac{-y_v \bar f_x}{\det\Jac\pipol \cdot \bar f}
		= - \frac{1}{|x_u|^2} \cdot \frac{\bar f_u}{\bar f}
		+ \frac{\bar y_u}{|x_u|^2 \bar y_v} \cdot \frac{\bar f_v}{\bar f}.
	\end{equation}
	To compute the order of $\xi^u$, we find first
	\begin{equation} \label{eq:xi_u_ord_1}
		\ord_r \left( - \frac{1}{|x_u|^2} \cdot \frac{\bar f_u}{\bar f} 
		\right)
		= -2c_{0,i} + 1,
	\end{equation}
	since $\ord_r x_u = c_{0,i} - 1$, and the logarithmic derivative of $f$ 
	has
	a simple pole along the zero set of $f$.
	Next, we estimate the order of the second term on the right hand side of
	\cref{eq:xi_u_ord}. Since we assume that $c_{1,i} = c_{0,i} + 1$,
	we have $i\neq 0$. Indeed, $c_{0,0} = c_{1,0} = 1$, by
	\Cref{lem:c_01_recursive}.
	Furthermore, since we assumed
	$\ord_r f_x = p_i$,  we find that $i \notin \V_\Upsilon$ by
	\Cref{cor:strict_ine}. By \Cref{lem:inv}, the vertex $i$ is not
	invariant, i.e. $\varpi_i > 0$. By \Cref{lem:varpi}, we have
	\[
	\ord_r f_v > m_i.
	\]
	Since $c_{0,i} < \ord_r y \leq c_{1,i} = c_{0,i} + 1$, we have
	$\ord_r y = c_{0,i} + 1$, and so $\ord_r y_u = c_{0,i}$.
	Also, $\ord_r x_u = c_{0,i} - 1$ and $\ord_r y_v = c_{1,i} = c_{0,i}+1$,
	and so
	\begin{equation} \label{eq:xi_u_ord_2}
		\ord_r
		\left(
		\frac{\bar y_u}{|x_u|^2 \bar y_v} \cdot \frac{\bar f_v}{\bar f}
		\right)
		> -2 c_{0,i} + 1.
	\end{equation}
	By \cref{eq:xi_u_ord_1,eq:xi_u_ord_2}, we find
	$\ord_r \xi^u = -2c_{0,i} + 1$.
	This also gives
	\[
	1
	=
	\ord_r
	\left(
	r^{\tau_i}r^{-1} e^{-i\alpha} 
	\xi^u
	\right)
	=
	c_{0,i}+1 + m_i - p_i - 1 - 2c_{0,i} + 1
	=
	-\varpi_i + 2,
	\]
	by \cref{eq:ord_xiroal}, and so $\varpi_i = 1$. Similarly,
	\[
	\ord_r
	\left(
	r^{\tau_i} r^{-1} e^{-i\alpha}
	\frac{\bar y_u}{|x_u|^2 \bar y_v} \cdot \frac{\bar f_v}{\bar f}
	\right)
	> -\varpi_i + 2 = 1,
	\]
	and so the order on the left hand side of the equation right above is 
	$\geq 2$.
	Similarly,
	\[
	\ord_r
	\left(
	- \frac{1}{|x_u|^2} \cdot \frac{\bar f_u}{\bar f}
	r^{\tau_i} r^{-1} e^{-i\alpha}
	\right)
	=
	-\varpi_i + 2 = 1,
	\]
	However, we can write
	\[
	\frac{f_u}{f} = \frac{m_i}{u} + g,
	\]
	where $g$ is holomorphic. As a result,
	\[
	r^{\tau_i} r^{-1} e^{-i\alpha}
	\frac{-1}{|x_u|^2} \cdot \frac{\bar f_u}{\bar f}
	=
	- \frac{r^{\tau_i-2} m_i}{|x_u|^2}
	- \frac{r^{\tau_i - 1} e^{-i\alpha} \bar g}{|x_u|^2}.
	\]
	The first term on the right hand side above is real. Therefore,
	\[
	r^{\tau_i}\Im\left(r^{-1} e^{-i\alpha}
	\frac{-1}{|x_u|^2} \cdot \frac{\bar f_u}{\bar f} \right)
	=
	\Im\left(
	\frac{r^{\tau_i-1} e^{-i\alpha} \bar g}
	{|x_u|^2}\right).
	\]
	We find
	\[
	\begin{split}
		\ord_r\left(\frac{r^{\tau_i-1} e^{-i\alpha} \bar g}
		{|x_u|^2}\right)
		&\geq
		\tau_i - 1 - 2c_{0,i} + 2 \\
		&=
		c_{1,i} + m_i - p_i - 2c_{1,i} + 3 \\
		&=
		-\varpi_i + 3 \\
		&= 2.
	\end{split}
	\]
	
	Part 2. Next, we take care of the lower right block of the Hessian in 
	the statement. 
	
	As in
	\cref{eq:xi_U,eq:xiro_coords}, the $v$-component, i.e. the 
	$(s,t)$-component
	of the vector field $\xiro_U$ is given by the complex function
	\[
	r^{\tau_i} \xirov =r^{\tau_i}( \xiros + i \xirot )=r^{\tau_i} \sigma^* 
	\xi^v
	= r^{\tau_i}\frac{-y_u \bar f_x + x_u \bar f_y}{\det\Jac\pipol \bar f}.
	\]
	We consider the first term of the above expression
	\begin{equation}\label{eq:first_term_xiv}
		r^{\tau_i}\frac{-y_u \bar f_x}{\det\Jac\pipol \bar f},
	\end{equation}
	and observe that because $\ord(y_u \bar{f}_x) \geq c_{0,i} + p_i$ and  
	$\ord(\det\Jac\pipol \bar f) = c_{0,i}+c_{1,i} -1 + m_i$,
	\[
	\tau_i +\ord(y_u \bar{f}_x) - \ord(\det\Jac\piropol \bar f) \geq 
	c_{1,i} + m_i - p_i + c_{0,i} + p_i - c_{0,i} - c_{1,i} + 1 - m_i = 1.
	\]
	This means that \cref{eq:first_term_xiv} is divisible by $r$ and in 
	particular, its partial derivatives with respect to $s, t$ (and even 
	$\alpha$) vanish at $q$.
	Now we consider the second term
	\begin{equation}\label{eq:second_term}
		r^{\tau_i} \frac{x_u \bar f_y}{\det\Jac\pi \bar f} 
		= r^{\tau_i} \frac{c_{0,i} u^{c_{0,i}-1} \bar{u}^{p_i} 
			\bar{v}}{\det\Jac\pi \bar f}.
	\end{equation}
	By taking,
	\[
	\lambda_p = 
	\left.
	\frac{c_{0,i}}{u^{-\nu_i+1}\det\Jac\pi}
	\cdot
	\partial_{\bar v}
	\left(
	\frac{\bar u^{-p_i+m_i} \bar{u}^{p_i} \bar{v}}{ \bar f}
	\right)
	\right|_{\substack{u=0\\v=0}}. 
	\]
	and applying \Cref{lem:antiholo} we get the expression for the lower 
	right block.
	
	Part 3. In this part we compute the $2 \times 2$ block under the 
	diagonal. 
	
	Observe that the first column is just by the definition of Hessian of a 
	vector field.
	
	Now we take care of the partial derivatives with respect to $\alpha$. 
	Indeed, to prove that the $2\times2$ block matrix under the diagonal is 
	as in the statement, we need to show that
	\[
	\partial_\alpha (r^{\tau_i} \xirov)|_q = 0.
	\]
	We already showed in Part 2, that $\partial_\alpha$ of 
	\cref{eq:first_term_xiv} vanishes at $q$. So we just need justify that
	\[
	\partial_\alpha \left. \left(r^{\tau_i} \frac{x_u \bar f_y}{\det\Jac\pi 
		\bar f}\right) \right|_q = 0.
	\]
	but again this is clear since \cref{eq:second_term} shows that there is 
	a factor of $\bar{v}$ that vanishes at $q$.
\end{proof}

\section{Center-stable manifolds}
\label{ss:center_stable}

In this subsection we prove the central technical result
\Cref{prop:properties_SU},
using the theory developed in previous subsections.
As before, $p \in \Sigma_i$, where $i\neq 0$.

\begin{block}
	In \Cref{lem:zeros_not_0}, we find that
	the Hessian of $\xiro_U$ is a block triangular matrix.
	Using the following coordinates in $\Uro$
	\[
	r' = r, \quad
	\alpha' = \alpha,\quad
	v' = v - r \partial_r (r^{\tau_i} \xirov),
	\]
	the Hessian of $\xiro_U$ is the block diagonal matrix
	\[
	\left(
	\begin{matrix}
		\multicolumn{2}{c}{\multirow{2}{*}{\scalebox{2}{$0$}}}
		&  \multicolumn{2}{c}{\multirow{2}{*}{\scalebox{2}{$0$}}} \\
		\multicolumn{2}{c}{} 
		&  \multicolumn{2}{c}{} \\
		\multicolumn{2}{c}{\multirow{2}{*}{\scalebox{2}{$0$}}}
		&  \multicolumn{2}{c}{\multirow{2}{*}
			{$\bar A(e^{-i\varpi_i\alpha'}\lambda_p)$}} \\
		\multicolumn{2}{c}{} 
		&  \multicolumn{2}{c}{}
	\end{matrix}
	\right).
	\]
	In order to get a diagonal Hessian, we must take a double cover of 
	$\Uro$.
	Define $\tUro$ by a Cartesian diagram
	\[
	\begin{tikzcd}
		\tUro \arrow{r}{\tilde\alpha} \arrow[swap]{d}{c}
		& \R/4\pi\Z \arrow{d}{2:1} \\
		\Uro \arrow{r}{\alpha'}
		& \R/2\pi\Z
	\end{tikzcd}
	\]
	where the vertical arrow on the right is the natural double cover of
	$\R/2\pi\Z$, inducing a double cover $c:\tUro \to \Uro$. Define
	\[
	\tSiro_p = c^{-1}(\sigma^{-1}(p)) \simeq \R / 4 \pi \Z.
	\]
	Along with
	$\tilde\alpha$, we define coordinates
	\[
	\tilde r = c^*(r'),\qquad
	\tilde v = e^{-i\varpi_i\tilde\alpha/2}\nu_p c^*(v'),
	\]
	where $\nu_p$ is a fixed
	square root of $\lambda_p$, i.e. $\nu_p^2 = \lambda_p$.
	Let $\txiro_U$ be the pullback of $\xiro_U$ to $\tUro$.
	At a point $\tilde{q} \in \tSiro_p$, in the coordinates $\tilde r$, 
	$\tilde \alpha$
	and $\tilde v = \tilde s + i\tilde t$, this vector field has Hessian
	\begin{equation} \label{eq:hess_txiro}
		\Hess_{\tilde{q}} \txiro_U = \left(
		\begin{matrix}
			\multicolumn{2}{c}{\multirow{2}{*}{\scalebox{2}{$0$}}}
			&  \multicolumn{2}{c}{\multirow{2}{*}{\scalebox{2}{$0$}}} \\
			\multicolumn{2}{c}{} 
			&  \multicolumn{2}{c}{} \\
			\multicolumn{2}{c}{\multirow{2}{*}{\scalebox{2}{$0$}}}
			&  |\lambda_p| & 0 \\
			\multicolumn{2}{c}{} 
			&  0 & -|\lambda_p|
		\end{matrix}
		\right).
	\end{equation}
	Thus, the system of ordinary differential equations corresponding to the
	vector field $\txiro_U$ is
	\begin{equation} \label{eq:kelleys_form}
		\begin{split}
			\dot{\tilde{r}}
			&= R(\tilde r, \tilde \alpha, \tilde s, \tilde t), \\
			\dot{\tilde{\alpha}}
			&= A(\tilde r, \tilde \alpha, \tilde s, \tilde t), \\
			\dot{\tilde{s}}
			&= |\lambda_p| \tilde s + S(\tilde r, \tilde \alpha, \tilde s, 
			\tilde t), \\
			\dot{\tilde{t}}
			&= -|\lambda_p| \tilde t + T(\tilde r, \tilde \alpha, \tilde s, 
			\tilde t),
		\end{split}
	\end{equation}
	where the functions $R,A,S,T$ are real analytic, vanish along
	$\tSiro_p$, and their partials with respect to $\tilde r,\tilde 
	s,\tilde t$
	also vanish along $\tSiro_p$.
	
	We can assume that $\tUro$ is of the form
	\[
	\tUro =
	[0,\eta_r] \times \R/4\pi\Z \times [-\eta_v, \eta_v]^2.
	\]
	where $\eta_r$ and $\eta_v$ are small.
	Set
	\begin{equation} \label{eq:hUro_form}
		\hUro =
		[-\eta_r,\eta_r] \times \R/4\pi\Z \times [-\eta_v, \eta_v]^2.
	\end{equation}

	\begin{figure}[h]
		\centering
		\includegraphics*[scale=0.7]{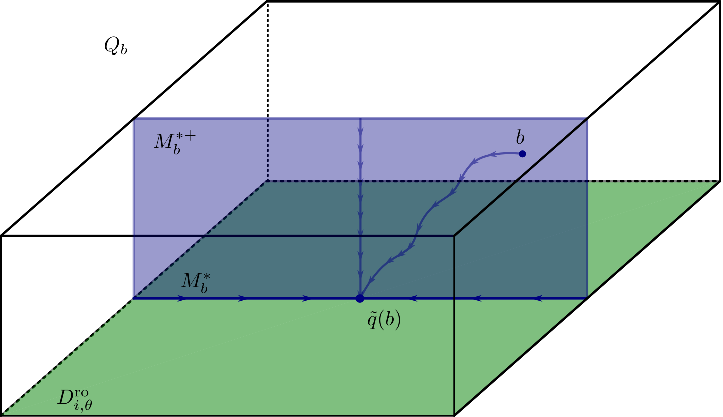}
		\caption{The space $Q_b$. In light blue we see  $M^{*+}_b$ 
			containing the point $b$ and intersecting transversely 
			$\Dro_{i,\theta}$. The manifold  $M^{*}_b$ consists of two 
			trajectories on  $\Dro_{i,\theta}$ converging to the point 
			$\tilde{q}(b)$.}
		\label{fig:qube_ext}
	\end{figure}
	
	We consider the system \cref{eq:kelleys_form} in $\hUro$.
	Therefore, this is a system in one {\em periodic variable} 
	$\tilde\alpha$ (by {\em periodic} here we mean that $\tilde\alpha \in 
	\R / 4 \pi \Z$.), and
	three real variables, where $\tilde r$ takes negative values as well.
	By \cite[Theorem 1]{kelley_stable}, 
	this system has a
\index{manifold!center-stable}
\emph{center-stable manifold}
	$M^{*+} \subset \hUro$.
	This means that, assuming that $\eta_r$ and $\eta_v$ are chosen small 
	enough, $M^{*+}$ is a manifold which is invariant by the vector field 
	$\txiro_U$, and
	there exists a $\Cinf$ function
	$\hat w^{*+}$ in the variables $\tilde\alpha, \tilde t, \tilde r$,
	so that
	\[
	M^{*+}
	=
	\set{(\tilde r, \tilde\alpha, \tilde s, \tilde t) \in \hUro}
	{\tilde s = \hat w^{*+}(\tilde \alpha, \tilde t, \tilde r)}.
	\]
	and $\hat w^{*+}$ and its derivatives with respect to
	$\tilde t,\tilde r$ vanish
	when $\tilde r,\tilde t$ are zero. Similarly,
	the system \cref{eq:kelleys_form} has a
	\index{manifold!center}
	\emph{center manifold}
	\[
	M^*
	=
	\set{(\tilde r, \tilde\alpha, \tilde s, \tilde t) \in \hUro}
	{\tilde s = \hat w^*(\tilde \alpha, \tilde r),\;
		\tilde t = \hat u^*(\tilde \alpha, \tilde r)}.
	\]
	where $\hat w^*,\hat u^*$ and their first partials with respect to
	$\tilde{r}$ vanish along $\tSiro_p$.
\end{block}

\begin{block}
	Let $b \in \tUro$, and let $\theta = \arg(f)(c(b))$. Then, the set
	$c^{-1}(\arg(f)^{-1}(\theta))$ has $2m_i$ connected components. Denote 
	by
	$Q_b$ the connected component which contains $b$. Then $Q_b$ has 
	coordinates
	$\tilde r, \tilde s, \tilde t$, and
	\[
	Q_b \simeq [0,\eta_r] \times [-\eta_v,\eta_v]^2.
	\]
	We define 
	\begin{equation}\label{eq:qb}
		\tilde{q}(b) \in Q_b
	\end{equation}
	as the origin in this coordinate system, i.e., the unique point of 
	$Q_b$ that maps to $p$ by $\sigma \circ c$. The restriction of $\tilde 
	\alpha$ to $Q_b$ is a
	function in the variables $\tilde r, \tilde s, \tilde t$.
	Since $\xiro_U$ is tangent to $\Yro_\theta$, the system
	\cref{eq:kelleys_form} reduces to a system in three variables
	\begin{equation} \label{eq:kelleys_form_Q}
		\begin{split}
			\dot{\tilde{r}}
			&= R(\tilde r, \tilde \alpha(\tilde r, \tilde s, \tilde t),
			\tilde s, \tilde t), \\
			\dot{\tilde{s}}
			&= |\lambda_p| \tilde s + S(\tilde r,
			\tilde \alpha(\tilde r, \tilde s, \tilde t), \tilde s, \tilde 
			t), \\
			\dot{\tilde{t}}
			&= -|\lambda_p| \tilde t + T(\tilde r,
			\tilde \alpha(\tilde r, \tilde s, \tilde t), \tilde s, \tilde 
			t),
		\end{split}
	\end{equation}
	This system has center-stable and center manifolds
	\[
	M^{*+}_b = M^{*+} \cap Q_b,\qquad
	M^*_b = M^* \cap Q_b.
	\]
	Note that $\dim M^{*+}_b = 2$ and $\dim M^*_b = 1$.
	In particular, $M^*_b$ consists of two trajectories, and one point.
\end{block}

Define $\tSro_U = M^{*+} \cap \tUro$. Then
\begin{equation} \label{eq:graph_SU}
	\tSro_U
	=
	\set{(\tilde r, \tilde\alpha, \tilde s, \tilde t) \in \tUro}
	{\tilde s = \tilde w^{*+}(\tilde \alpha, \tilde r, \tilde t)},
\end{equation}
where $\tilde w^{*+}$ is the restriction of $\hat w^{*+}$. Moreover, the 
properties satisfied by $M^{*+}$ imply the following lemma.

\begin{lemma}\label{lem:properties_tSro}
	The following properties for
	$\tSro_U$ hold:
	\begin{enumerate}
		\item \label{it:tSro_inv}
		$\txiro_U$ is tangent to $\tSro_U$,
		\item \label{it:tSro_axis}
		$\tSiro_p \subset \tSro_U$,
		\item \label{it:tangent_SU}
		for $\tilde q \in \tSiro_p$, we have
		$T_{\tilde q} \tSro_U =
		\R\langle \partial_{\tilde r}, \partial_{\tilde \alpha},
		\partial_{\tilde t} \rangle$. In other words, the $T_{\tilde q} 
		\tSro_U$ is the sum of the kernel and the negative eigenspace of 
		$\Hess_{\tilde{q}} \txiro_U$ (see \cref{eq:hess_txiro}).
	\end{enumerate}
	In particular, $\tSro_U$ is transverse to $\partial\tUro$.
	For $b \in \tUro$, define also
	\[
	\tSro_{U,b} = \tSro_U \cap Q_b.
	\]
	The functions $\tilde r$ and $\tilde t$ are coordinates on $\tSro_{U,b}$
	so that $\tSro_{U,b} \simeq [0,\eta_r] \times [-\eta_v,\eta_v]$.
\end{lemma}
\begin{proof}
	Property \cref{it:tSro_inv} follows from the fact that $M^{*+}$ is an
	invariant manifold. Property \cref{it:tSro_axis} follows since 
	$\tilde{w}^{*+}$ 
	vanishes along
	the $\alpha$ axis, i.e. when $\tilde r$ and $\tilde t$ vanish. Finally,
	\cref{it:tangent_SU} follows in the same way from the definition of
	a center-stable manifold, i.e. the vanishing of partials of
	$\tilde{w}^{*+}$. These properties of $\tilde{w}^{*+}$ and the 
	vanishing of its 
	partials are also part of the statement of \cite[Theorem 
	1]{kelley_stable}.
\end{proof}

\begin{lemma} \label{lem:local_dyn}
	Assuming $U$ is chosen small enough, then we have the following 
	properties
	\begin{enumerate}
		\item \label{it:local_dyn_S}
		If $b \in \tSro_U$, then $\gamma_b$ converges to
		$\tilde q(b)$.
		\item \label{it:local_dyn_nS}
		If $b \in \tUro \setminus \tSro_U$, then $\gamma_b$ exits $\tUro$.
	\end{enumerate}
\end{lemma}
\begin{proof}
	In this proof, we make a coordinate change, replacing
	$\tilde s$ by $\tilde s - \hat w^{*+}(\tilde r,\tilde \alpha,\tilde t)$
	and
	$\tilde t$ by $\tilde t - \hat u^*(\tilde r,\tilde \alpha)$. Thus,
	the center-stable and center manifolds $M^{*+}$ and $M^*$ are
	coordinate subspaces.
	Similarly, if we fix a point $b \in \tUro$, then $M^{*+}_b$ and $M^*_b$
	are coordinate subspaces in $Q_b$.
	We assume that $\hUro$ has the form
	\cref{eq:hUro_form} in these new coordinates.

	Assume first that $b \in \tSro_U$.
	Name the boundary pieces of $\tSro_{U,b}$ as follows:
	\[
	\begin{split}
		\partial_0 \tSro_{U,b}
		&=
		\{0\} \times [-\eta_v,\eta_v],\\
		\partial_r \tSro_{U,b}
		&=
		\{\eta_r\} \times [-\eta_v,\eta_v],\\
		\partial_t \tSro_{U,b}
		&=
		[0,\eta_r] \times \{-\eta_v,\eta_v\}.
	\end{split}
	\]
	The set $M_q^* \cap \tUro$ consists of a single trajectory
	and the point $\tilde q(b)$.
	The trajectory must be oriented towards $\tilde q(b)$,
	since no trajectory of $\xi$ in $\C^2$ emanates from the origin.
	Therefore, there exists a neighborhood of the point $(\eta_r, 0)$
	in $\partial_r \tSro_{U,b}$, where $\txiro_U$ points inwards.
	Fixing $\eta_r$ and taking $\eta_v$ small enough, we can assume
	that $\txiro_U$ points inwards along all of
	$\partial_r \tSro_{U,b}$.
	
	We have
	\[
	\frac{\partial}{\partial \tilde t} \dot{\tilde{t}} (\tilde q(b))
	= -|\lambda_p| < 0.
	\]
	Therefore, we can assume that
	\begin{equation} \label{eq:decr}
		\frac{\partial}{\partial \tilde t} \dot{\tilde{t}}
		< 0
	\end{equation}
	on $\tSro_{U,b}$. Furthermore, $\dot{\tilde t}|_{\tSro_{U,b}}$ vanishes 
	precisely along
	$M^*_b \cap \tUro$. This implies that $\tilde t$ and $\dot{\tilde t}$ 
	have
	opposite signs on $\tSro_{U,b}$. As a result,
	$|\tilde t|$ decreases along the trajectory $\gamma_b$.
	Therefore, $|\tilde t|$ has a limit along $\gamma_b$, call it $\tilde 
	t_0$.
	
	It also follows from \cref{eq:decr} that the vector field $\txiro_U$
	restricted to $\tSro_{U,b}$ points inwards along
	$\partial_t \tSro_{U,b}$. As a result, a trajectory of $\txiro_U$
	which starts at $b \in \tSro_U$ does not escape from
	$\tSro_{U,b}$. Since $\tSro_{U,b}$ is compact, $\gamma_b$ is well 
	defined on
	the whole nonnegative real axis.
	
	The function $c^* (\piro)^* |f|$ has the same
	vanishing set as $\tilde r$ on $\tSro_{U,b}$, that is,
	$\partial_0 \tSro_{U,b}$.
	Furthermore, the limit of $c^* (\piro)^* |f|$ along $\gamma_b$ is zero.
	Therefore, $\tilde r$ has limit zero along $\gamma_b$, and so,
	in coordinates $\tilde r, \tilde t$ in $\tSro_{U,b}$, $\gamma_b$
	converges to $(0,\tilde t_0)$. Since $\txiro_U$ does not vanish at
	$(0,\tilde t_0)$ unless $\tilde t_0 = 0$, we have $\tilde t_0 = 0$,
	and so $\gamma_b$ converges to $\tilde q(b) = (0,0)$, which
	concludes the proof of \cref{it:local_dyn_S}.
	
	Next, we prove \cref{it:local_dyn_nS}.
	Assume that $b \in \tUro \setminus \tSro_U$.
	Similarly as above, we can assume that
	\begin{equation} \label{eq:decr_1}
		\frac{\partial}{\partial \tilde s} \dot{\tilde{s}} > 0
	\end{equation}
	in $Q_b$. Therefore, the absolute value of the
	$\tilde s$ coordinate of the trajectory
	$\gamma_b$ increases. Therefore, $\gamma_b$ escapes $\tUro$.
\end{proof}

\begin{figure}[h]
	\centering
	\includegraphics*[scale=0.7]{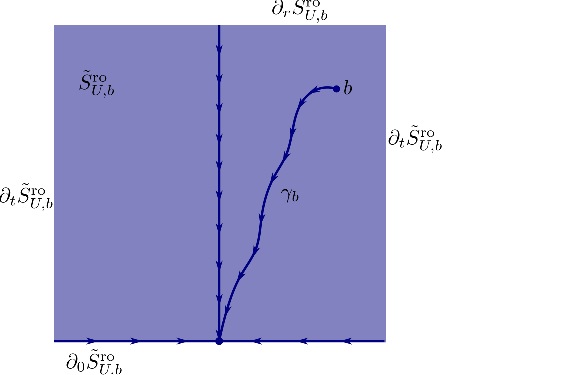}
	\caption{We see $\Sro_{U,b}$ containing a trajectory $\gamma_b$.}
	\label{fig:SUb}
\end{figure}

\begin{cor} \label{cor:unicity_tilde}
	The function $\tilde w^{*+}$ is the only function in variables
	$\tilde \alpha, \tilde t$ and $\tilde r \geq 0$, which vanishes together
	with its partials along $\tilde t = \tilde r = 0$,
	whose graph as in \cref{eq:graph_SU} is an invariant submanifold.
\end{cor}
\begin{proof}
	If $w$ is another such function, and
	$b = (\tilde r, \tilde \alpha, \tilde t, w(\tilde r, \tilde 
	\alpha,\tilde t))$,
	then, by the proof of
	\cref{it:local_dyn_S}, the trajectory $\gamma_b$ converges to $\tilde 
	q(b)$. 
	But by \cref{it:local_dyn_nS}, it diverges.
\end{proof}

\begin{cor} \label{cor:SU_invariant}
	The set $\tSro_U \subset \tUro$
	is invariant under the Galois transformation of the covering map $c$.
\end{cor}
\begin{proof}
	Let $\iota$ be the Galois transformation. Then, $\tilde w^{*+} \circ 
	\iota$
	satisfies the conditions in the previous corollary, and
	so $\tilde w^{*+} = \tilde w^{*+} \circ \iota$.
\end{proof}
\begin{rem}
	The unicity statement \Cref{cor:unicity_tilde} may only apply to
	$\tSro_U$, the graph of $\tilde w^{*+}$,
	and not the whole center-stable submanifold $M^{*+}$, the graph of
	$\hat w^{*+}$.
\end{rem}
\begin{cor}\label{cor:local_spine_cell}
	We have the following equality of sets
	\[
	\Sro_U = c(\tSro_U).
	\]
	Furthermore, the following properties are satisfied
	\begin{enumerate}
		\item \label{it:ls_i} $\xiro_U$ is tangent to $\Sro_U$,
		\item \label{it:ls_ii}$\sigma^{-1}(p) \subset \Sro_U$,
		\item \label{it:ls_iii}for $q \in \sigma^{-1}(p)$, and $\tilde{q} 
		\in c^{-1}(q)$, we have $T_q \Sro_U = D_{\tilde{q}} c 
		(T_{\tilde{q}} \tilde{S}_U)$.
	\end{enumerate}
\end{cor}

\begin{proof}
	The first statement follows from \Cref{lem:local_dyn} and  
	\Cref{cor:SU_invariant}. The three properties follow from 
	\Cref{lem:properties_tSro}. 
\end{proof}

\begin{proof}[Proof of \Cref{prop:properties_SU}]
	The three properties follow from the results \Cref{cor:local_spine_cell} 
	and 
	\Cref{lem:properties_tSro}, \cref{it:tangent_SU}. That the three 
	vectors 
	are like \Cref{prop:properties_SU}, \cref{it:hessian_explicit} is a 
	direct 
	computation using the previously cited results. Below you can find this 
	direct (but tedious) computation.
\end{proof}

\begin{proof}[Direct computation for \Cref{prop:properties_SU}, 
	\cref{it:hessian_explicit} ]
	As established by the application of center-stable manifold theory in 
	\Cref{lem:properties_tSro}, (\cref{it:tangent_SU}) and 
	\Cref{cor:local_spine_cell}, the tangent 
	space $T_q\Sro_U$ is the center-stable eigenspace of the linearized 
	vector field $\xiro_U$ at the singular point $q$. This space is spanned 
	by the generalized eigenvectors of the Hessian matrix, $L = \Hess_q 
	\xiro_U$, corresponding to eigenvalues with non-positive real part. We 
	compute this basis explicitly.
	
	From \cref{eq:hess_xiro}, the Hessian at $q$ in coordinates $(r, 
	\alpha, s, 
	t)$ is the block lower-triangular matrix:
	$$
	L = \Hess_q \xiro_U =
	\begin{pmatrix}
		\mathbf{0}_{2 \times 2} & \mathbf{0}_{2 \times 2} \\
		M_{bl} & M_{br}
	\end{pmatrix}
	$$
	where the blocks are defined as follows:
	\begin{itemize}
		\item $M_{br} = \bar A(k)$, where $k = 
		e^{-i\varpi_i\alpha}\lambda_p$. 
		This is the Hessian of the $\xirov$ component with respect to 
		$v=s+it$.
		\item $M_{bl}$ is a $2 \times 2$ matrix describing the partials of 
		$\xirov$ with respect to $(r, \alpha)$. Its second column is zero, 
		and 
		its first column is the vector representation of $D = \partial_r 
		(r^{\tau_i} \xirov)|_q$.
	\end{itemize}
	The eigenvalues of $L$ are the eigenvalues of its diagonal blocks, 
	which 
	are $0$ (with multiplicity 2) from the top-left block, and $\pm|k| = 
	\pm|\lambda_p|$ from the bottom-right block. The center-stable 
	eigenspace 
	corresponds to the eigenvalues $0, 0,$ and $-|\lambda_p|$. We now 
	compute 
	the basis for this space.
	
	\paragraph{Stable Eigenvector ($\lambda = -|\lambda_p|$)}
	We seek an eigenvector $V = (V_r, V_\alpha, V_s, V_t)^\top$ such that 
	$LV = 
	-|\lambda_p|V$. The equation for the first two components is 
	$\mathbf{0}_{2\times 2} \icol{V_r \\ V_\alpha} = -|\lambda_p| \icol{V_r 
		\\ 
		V_\alpha}$, which implies $V_r = V_\alpha = 0$. The equation for 
	the last 
	two components is $M_{br} \icol{V_s \\ V_t} = -|\lambda_p| \icol{V_s \\ 
		V_t}$.
	This identifies $(V_s, V_t)$ as the eigenvector of $M_{br} = \bar A(k)$ 
	for 
	the eigenvalue $-|\lambda_p| = -|k|$. In complex notation, this 
	eigenvector 
	is a multiple of $\sqrt{-k}$.
	$$ \sqrt{-k} = \sqrt{-e^{-i\varpi_i\alpha}\lambda_p} = i 
	\sqrt{e^{-i\varpi_i\alpha}\lambda_p} = i e^{-i\varpi_i\alpha/2}\nu_p, 
	\quad 
	\text{where } \nu_p^2 = \lambda_p. $$
	This matches the definition of $P'+iQ'$ in the proposition. The 
	corresponding basis vector is $(0,0,P',Q')^\top$, which is:
	\[ \mathbf{V_3} = P' \partial_s + Q' \partial_t. \]
	
	\paragraph{Center Eigenvectors ($\lambda=0$)}
	We seek generalized eigenvectors for the eigenvalue $0$, i.e., vectors 
	$V$ 
	in the kernel of $L$. The equation is $LV=\mathbf{0}$. The 
	component-wise 
	equations are:
	\begin{enumerate}
		\item $\mathbf{0}_{2\times 2} \icol{V_r \\ V_\alpha} = \mathbf{0}$, 
		which is trivially satisfied.
		\item $M_{bl} \icol{V_r \\ V_\alpha} + M_{br} \icol{V_s \\ V_t} = 
		\mathbf{0}$.
	\end{enumerate}
	We find a basis for the two-dimensional solution space.
	\begin{itemize}
		\item \textbf{First center vector:} Choose $(V_r, V_\alpha) = (0, 
		1)$. 
		The second equation becomes $M_{br} \icol{V_s \\ V_t} = 
		\mathbf{0}$. 
		Since $M_{br}$ is invertible ($|\lambda_p|\neq 0$), this implies 
		$V_s=V_t=0$. The resulting eigenvector is $(0,1,0,0)^\top$:
		\[ \mathbf{V_2} = \partial_\alpha. \]
		\item \textbf{Second center vector:} Choose $(V_r, V_\alpha) = (1, 
		0)$. 
		The second equation becomes $M_{br} \icol{V_s \\ V_t} = -M_{bl} 
		\icol{1 
			\\ 0}$. In complex notation, this is $k \cdot 
		\overline{(V_s+iV_t)} = 
		-D$. We solve for $V_s+iV_t$:
		$$ \overline{V_s+iV_t} = -k^{-1}D \implies V_s+iV_t = 
		-\overline{k^{-1}} \bar{D}. $$
		We compute the coefficient $\overline{k^{-1}}$:
		$$ \overline{k^{-1}} = \frac{1}{\bar k} = \frac{k}{|k|^2} = 
		\frac{e^{-i\varpi_i\alpha}\lambda_p}{|\lambda_p|^2} = 
		e^{-i\varpi_i\alpha} \frac{\lambda_p}{|\lambda_p|^2} = 
		e^{-i\varpi_i\alpha} (\bar{\lambda}_p)^{-1}. $$
		Thus, the complex number representing the $(s,t)$ components of the 
		eigenvector is
		$$ P+iQ = V_s+iV_t = -e^{-i\varpi_i\alpha} (\bar{\lambda}_p)^{-1} 
		\bar{D}, $$
		which is precisely the formula given in the proposition. The 
		eigenvector is $(1,0,P,Q)^\top$:
		\[ \mathbf{V_1} = \partial_r + P\partial_s + Q\partial_t. \]
	\end{itemize}
	The three vectors $\mathbf{V_1}, \mathbf{V_2}, \mathbf{V_3}$ form a 
	basis 
	for the center-stable eigenspace, which is tangent to $\Sro_U$ at $q$.
\end{proof}

\begin{rem} \label{rem:solid_tK}
	It follows from construction that the restriction of
	the Galois transformation $\iota$ to $\tSro_U$ either respects or 
	reverses
	orientation, depending on whether $\varpi_i$ is even or odd.
	Since $\tSro_U$ is a solid torus, the quotient by $\iota$
	is a
\index{solid torus}
solid torus if $\varpi_i$ is even, and a
\index{solid Klein bottle}
solid Klein bottle if
	$\varpi_i$ is odd.
	
\end{rem}

%
%
%
	\chapter{The Total Spine of the Milnor Fibration}
\label{s:total_spine}

\section{Broken trajectories}
\label{s:broken_trajectories}

\begin{block}
	Construct a $C^\infty$ real function $\psi:\Ypol \to \R_{\geq 0}$ as
	follows. If $p \in D = \pipol^{-1}(0)$, take a chart $U$ containing $p$ 
	with
	coordinates $u,v$. Let $i,j \in \W$, in particular $i$ or $j$ might be 
	an
	arrowhead. If $p \in D_i^\circ$ assume that $D\cap U = D_i \cap U  = 
	\{u=0\}$.
	Otherwise, if $p \in D_i \cap D_j$ is an intersection point, assume that
	$U\cap D = U\cap  (D_i \cup D_j)$ and that $D_i \cap U = \{u=0\}$ and
	$D_j \cap U = \{v=0\}$.
	If $i\in \A $, that is, if $i$ is an arrowhead, set $\tau_i = 2$
	(as in \Cref{lem:str_tr}).
	In the former case set $\psi_U = |u|^{\tau_i}$ and in the latter
	$\psi_U = |u|^{\tau_i} |v|^{\tau_j}$. Set
	\[
	\psi = \sum_{U} \chi_U \psi_U
	\]
	where $\{\chi_U\}$ is a partition of unity subordinate to some finite 
	open
	cover by charts $U$.
\end{block}

\begin{definition}
	The vector field
\index{$\xiropol$}
$\xiropol$ is the unique continuous extension over all 
	$\Yropol$ of the vector field 
	\[
	\sigma^*\psi \cdot (\piropol)^* \xi.
	\]
\end{definition}

\begin{rem}
	By construction,
	\begin{enumerate}
		\item the vector field $\xiropol|_{\Uro}$ coincides with the vector 
		field $\xiro_U$ (see \Cref{def:xiroU,def:xiroU_intersection}) on 
		$\Uro$ up to multiplication by a positive real function $\Uro \to 
		\R_{>0}$, and
		
		\item the vector field $\xiropol|_{\Droc_{i,\theta}}$ coincides, up 
		to a positive real factor, with the vector field $\xiro_{i,\theta}$.
	\end{enumerate}
\end{rem}

\begin{definition}
	For $\theta \in \R/2\pi\Z$, set
\index{$\Siro_\theta$}
	\[
	\Sigma = \bigcup_{i\in \V}\Sigma_i,\quad
	\Siro_\theta = \bigcup_{i\in \V}\Siro_{i,\theta}.
	\]
\end{definition}

\begin{definition}\label{def:broken_trajectory}
	A
\index{broken trajectory}
\emph{broken trajectory}
$\gamma$ in $\Dpol=(\piropol)^{-1}(0)$
	a family of trajectories $\{\gamma_k\}_{k \in I}$
	of the vector field $\xiropol$ in $\Dpol$, satisfying
	\[
	\lim_{t\to+\infty} \gamma_k(t) = 
	\lim_{t\to-\infty} \gamma_{k+1}(t),\quad
	k \in I.
	\]
	If $|I|=n < \infty$ we say that the broken trajectory has length $n$.
\end{definition}
\begin{rem}\label{rem:broken_trajectories}
	\begin{enumerate}
		\item The fact that $\xiropol$ is tangent to $\Yropolte$ 
		(\Cref{rem:R_and_xi} \cref{it:iii_and_xi_tang}) implies that there 
		exists a $\theta \in \R / 2 \pi\Z$ such that all parts $\gamma_i$ 
		of a broken trajectory are contained in $\Yropolte$.
		
		\item \label{it:non_invariant_broken} The definition allows for 
		parts
		$\gamma_k$  of a broken trajectory $\gamma$ to be included in the
		connecting parts $\Dro_{i,j,\theta}$. Observe that this might only
		happen when the edge $ij$ is non-invariant, that is, when at least 
		one
		of $\varpi_i$ or $\varpi_j$ is non-zero.
		Indeed, if $ij$ is invariant, then the restriction of
		$\xiropol$ to $\Droc_{ij}$ is identically zero,
		see \cref{block:pullback_ro}.
		
	\end{enumerate}
\end{rem}

A broken trajectory $\gamma$ defines a path in the graph $\Gapol$. Here we 
understand a path in the following sense: a sequence of vertices (it is not 
necessary to specify the edges since $\Gapol$ is a tree). The path is 
defined by a sequence of vertices $i_k$ and edges $e_k$, with $k = 1, 
\ldots n$ where $i_k \in \V$ appears in the sequence if $\gamma_k \subset 
\Droc_{i_k,\theta}$ and $e_k = ij$ appears in the sequence if $\gamma_k$ is 
contained in $\Droc_{i,j,\theta}$.

\begin{lemma} \label{lem:broken_edges}
	If $\gamma_k \subset \Dro_{i,j,\theta}$, that is, if $e_k=ij$ and $j 
	\to i$, then $i_{k+1} = j$.
\end{lemma}

\begin{proof}
	By \Cref{rem:broken_trajectories},\cref{it:non_invariant_broken} we 
	know that $ij$ is a non-invariant edge. By \cref{fig:corner_sings} the 
	path $\gamma_k$ must go from a red vertex to a green vertex on  
	$\Dro_{i,j,\theta}$. But there is a unique trajectory emanating from a 
	green vertex and it is contained in $\Droc_{j,\theta}$.
\end{proof}
We introduce some notation that is useful for this section and later on.

\begin{notation}\label{not:parts_Milnor_fiber}
	Let $\Gamma$ be the dual graph of some resolution of a plane curve 
	singularity and let $\Gamma'$ be a subgraph. We denote by 
\index{Milnor fibration!at radius zero}
	$\Fro_\theta[\Gamma']$ the part of the Milnor fiber at radius $0$ and 
	angle $\theta$ that corresponds to $\Gamma'$, that is, if $\V_{\Gamma'} 
	\subset \V$ is the set of vertices of $\Gamma'$, then (recall 
	\Cref{def:angled_ray}):
	\[
	\Fro_\theta[\Gamma'] = \bigcup_{i\in \V_{\Gamma'}} \Dro_{i,\theta}.
	\]
\end{notation}

\begin{lemma}\label{lem:finite_broken_traj}
	A broken trajectory $\gamma$ can pass through a non-invariant edge $ij$ 
	at most $2 \gcd(m_i,m_j)$ times. 
\end{lemma}

\begin{proof}
	Assume that $j \to i$. We know that $\Droc_{i,j,\theta}$ has $\gcd(m_i, 
	m_j)$ connected components. If $\gamma$ passes through 
	$\Droc_{i,j,\theta}$ more than 
	$2 \gcd(m_i,m_j)$ times, then by the pigeonhole principle, it must pass 
	at least $3$ times through the same connected component, let's say 
	$\tDroc_{i,j,\theta}$ of $\Droc_{i,j,\theta}$. Let $\Xi_j$ be the 
	branch of $\Gapol$ at the vertex $j$. Then, since $\gamma'$ passes at 
	least $3$ times through $\tDroc_{i,j,\theta}$, there exists a 
	sub-broken trajectory $\gamma' \subset \gamma$ of length $n'$ such that 
	\begin{enumerate}
		\item $\gamma'_1$ starts at $\tDroc_{i,j,\theta}$,
		\item $\gamma'_{n'}$ ends at $\tDroc_{i,j,\theta}$
		\item $\gamma' \subset \Fro_{\theta}[\Gapol \setminus \Xi_j]$
	\end{enumerate}
	The last condition means that $\gamma'$ is contained in parts of the 
	Milnor fiber at radius $0$ corresponding to vertices that are closer 
	$0$ than $j$. The \cref{fig:petri_dish_loop} makes it easier to follow 
	this proof.
	
	Let $\Gamma' \subset \Gapol$ be the smallest subgraph such that 
	$\Fro_\theta[\Gamma']$ contains $\gamma'$. Observe that $\Gamma'$ is 
	connected since $\gamma'$ is connected. The graph $\Gamma$ is a 
	directed tree naturally rooted at $0$, so $\Gamma'$ is also a directed 
	tree rooted at some vertex $k$ which minimizes the distance to $0$ 
	among all vertices in $\V_{\Gamma'}$. By definition of $\Gamma'$, we 
	have that $\Dro_{k,\theta}$ contains a part of $\gamma'$ and moreover, 
	there is at least a boundary component of $\Dro_{k,\theta}$ that 
	$\gamma'$ intersects twice since $\gamma'$ must return to 
	$\tDroc_{i,j,\theta}$.
	
	Now, for each $\ell\in \V_{\Gamma'}$ such that $k\ell$ is an edge of
	$\Gamma'$ with $k \to \ell$, we contract each of the boundary components
	of $\Dro_{k,\theta}$  contained in  $\Dro_{k,\ell,\theta}$ to a point.
	We obtain a compact surface $\widehat{\Dro_{k,\theta}}$. The image of
	$\gamma'$ in this quotient surface must consist of some loops. But, by
	\Cref{lem:potential_xiro} (if $k \neq 0$) or by \Cref{lem:xi_0_gradient}
	(if $v_i=0$ ), the vector field $\xiro_{k,\theta}$ admits a potential
	that takes constant values on all the contracted boundary components (so
	the potential is equivariant under the quotient and yields a potential
	in $\widehat{\Dro_{k,\theta}}$) and it is impossible to have loops
	formed by integral lines of potentials. \end{proof}

\begin{figure}[!ht]
	\centering
	\includegraphics*[scale=0.95]{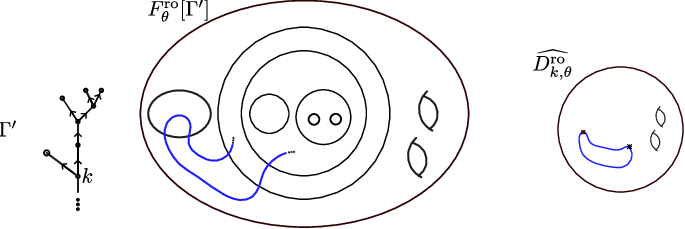}
	\caption{On the left hand side we see the surface 
		$\Fro_\theta[\Gamma']$. On the right hand side, we see the surface 
		$\widehat{\Dro_{k,\theta}}$ after contracting two boundary components 
		of $\Dro_{k,\theta}$.}
	\label{fig:petri_dish_loop}
\end{figure}

\begin{lemma} \label{lem:it_ends}
	There exists a number $M = M(\Gapol, (\pi^*f)) \in \Z_{>0}$ which only 
	depends on
	the graph $\Gapol$ and the multiplicities of the total transform of 
	$f$, so that all continuous broken trajectories
	in $\Ypol$ have length $\leq M$.
\end{lemma}

\begin{proof}
	The dynamics of $\xiropol$ near $\Dro_{i,j,\theta}$ for $ij$ an 
	invariant 
	edge (\cref{fig:droij_inv}) and the fact that $\Gamma$ is a tree, imply 
	that a broken trajectory can pass at most once through a invariant 
	edge. By 
	\Cref{lem:finite_broken_traj}, a broken trajectory can pass at most $2 
	\gcd(m_i,m_j)$ times through a non-invariant edge. Also, by the 
	dynamics of 
	$\xiropol$ near the corner parts 
	(\cref{fig:droij_inv,fig:corner_sings}) 
	and by \Cref{lem:broken_edges}, two consecutive vertices  $v_k, 
	v_{k+1}$ of 
	the path defined by a broken trajectory cannot be on the same connected 
	component $\Droc_{i,\theta}$. The result follows by the finiteness of 
	$\Gamma$. 
\end{proof}

\section{The total spine}
So far, we have described certain families of trajectories
of $\xi$ which converge to the origin in $\C^2$, i.e. that are contained
in the total spine $S$, \Cref{def:total_spine}.
These are the stable (for $i = 0$) and center stable (for $i \neq 0$)
manifolds $S(p)$, for $p \in \Sigma_i$,
constructed in \Cref{s:1st_blowup} and \Cref{ss:center_stable}.
In this section, we show that this is the whole total spine.

\begin{thm} \label{thm:spine}
	The
\index{total spine}
total spine
admits a finite partition
	\begin{equation} \label{eq:spine_partition}
		S = \{0\} \sqcup \bigsqcup_{p\in \Sigma} S(p),
	\end{equation}
	\begin{itemize}
		\item
		If $p\in \Sigma_0$ is a
\index{repeller}
repeller
		then $S(p) \setminus f^{-1}(S^1_\eta)$ is an open punctured disk.
		\item
		If $p\in \Sigma_0$ is a
\index{saddle}
saddle point,
		then $S(p) \setminus f^{-1}(S^1_\eta)$ is an open
\index{solid torus}
		solid torus.
		\item
		If $p \in \Sigma_i$ with $i \neq 0$,
		then $S(p) \setminus f^{-1}(S^1_\eta)$ is an open
		solid torus if $\varpi_i$ is even, and an open
\index{solid Klein bottle}
		solid Klein bottle if $\varpi_i$ is odd.
	\end{itemize}
\end{thm}

\begin{proof}
	Let $\gamma \subset \Tu^*$ be a trajectory of the vector field $\xi$
	which converges
	to the origin, and assume that $\gamma$ is not contained in a (center) 
	stable
	manifold $S(p)$, $p \in \Sigma$.
	Then $\gamma$ lifts to a trajectory $\gapol$ of $\xipol$ in $\Ypol$,
	and to a trajectory $\garopol$ of $\xiropol$ in $\Yropol$.
	Denote by $\Apol$ and $\Aropol$ the accumulation sets of $\gapol$
	and $\garopol$. Since $\sigma$ is a proper map, we have
	$\Apol = \sigma(\Aropol)$. Furthermore, the set $\Aropol$ is invariant
	under the flow of $\xiropol$.
	
	\emph{Claim:} If $p \in \Aropol$ and $\xiropol(p) = 0$, then
	there exists a trajectory $\gamma_p \subset \Aropol$ satisfying
	\[
	\lim_{t \to -\infty} \gamma_p(t) = p.
	\]
	
	To prove the claim, we consider two cases; either $p \in \Droc_i$ for
	some vertex $i \in \Vpol$, or $p \in \Dro_{ij}$ for some edge $ij$.
	
	In the first case, we have $p \in \Siro_i$, and
	$\gamma \cap S(p) = \emptyset$, by assumption.
	Fix some small half-ball $B_p$ around $p$ in $\Yropol$.
	Let $p_i = \garopol(t_i) \in \garopol \cap B_p$ be
	a sequence that converges to $p$. Since the trajectory of $\xiropol$
	starting at $p_i$ diverges from $p$, there exists some $t'_i > t_i$
	so that $p'_i = \garopol(t_i') \in \partial B_p$, and $\gamma$ is
	exiting $B_p$ at $p_i'$.
	By compactness, we can
	restrict this sequence to a subsequence, and assume that
	$p'_i \to p' \in \partial B_p$. By the assumption that $\gamma$
	converges to the origin, we must have $p' \in (\piropol)^{-1}(0)$, and 
	so
	$p' \in \Droc_i$.
	The trajectory of $\xiropol$ in $\Dro_i$ which passes through $p'$
	satisfies the conditions of the claim.
	
	Consider next the case when $p \in \Dro_{ij}$ for some edge $ij$.
	By \Cref{cor:strict_trans}, both vertices $i$ and $j$ are nonarrowheads,
	i.e. $i,j \in \Vpol$.
	As a result, this case follows by
	\Cref{lem:elementary_invariant} if the edge $ij$ is invariant, and
	\Cref{lem:noninv_corner} if it is not, which
	finishes the proof of the claim.
	
	We will now construct a broken trajectory in $\Aropol$,
	as follows.
	The set $\Aropol$ is nonempty, since the map $\piropol$ is proper.
	Take some $p_1 \in \Aropol$. If $\xiropol(p_1) \neq 0$, let
	$\gamma_1$ be the trajectory passing through $p_1$. Otherwise, choose
	$\gamma_1$ as some trajectory in $\Aropol$ as in the claim.
	Assuming that we have constructed a
	broken trajectory $\gamma_1,\ldots,\gamma_{n-1}$ in $\Aropol$, let
	$p_n$ be the limit of $\gamma_{n-1}$ at plus infinity, and choose 
	$\gamma_n$
	using the point $p = p_n$ in the claim.
	This process never ends, in contradiction to \Cref{lem:it_ends}.
	This proves \cref{eq:spine_partition}.
	
	If $p$ is a repeller on $D_0$, then $S(p)$ is the punctured stable
	manifold at $p$, in this case a disk. Similarly, if $p \in D_0$
	is a saddle point, then the punctured (open) stable manifold at $p$
	is an open solid torus. The cases when $p \in D_i$ with $i \neq 0$
	follow from \Cref{rem:solid_tK}.
\end{proof}

%
%
%
	\chapter{The Invariant Spine}
\label{s:invariant_spine}

In this section we construct a $\Cinf$ fibration equivalent to the Milnor
fibration by contracting the parts of the Milnor fiber at radius $0$
corresponding to non-invariant vertices. Furthermore, we construct a 
piecewise
smooth spine for all the fibers of this fibration at the same time.

\begin{lemma}\label{lem:four_pos}
	Let $i \in \Vpol$ and $\theta \in \R/2\pi\Z$.
	Any trajectory of the vector field $\xiro_{i,\theta}$ in
	$\Droc_{i,\theta}$ has a limit at each end in $\Dro_{i,\theta}$,
	each of which is either: 
	\begin{enumerate}
		\item a singularity of $\xiro_{i,j,\theta}$ in the corner circle 
		$\Dro_{i,j,\theta}$,
		\item a
\index{saddle}
saddle point $q$ of $\xiro_{i,\theta}$, or
		\item a
\index{repeller}
repeller on $\Dro_{0,\theta}$.
	\end{enumerate}
	The last case can only happen if $i = 0$.
\end{lemma}

\begin{proof}
	Let $\gamma \subset \Droc_{i,\theta}$ be a maximal trajectory, defined 
	on
	an interval $(a,b)$. It suffices to show that $\gamma(t)$ has a limit, 
	when
	$t \to b^-$, the case $t\to a^+$ follows similarly.
	
	Assume first that $\gamma$ has an accumulation point $q$ in 
	$\Droc_{i,\theta}$
	when $t \to b^-$. Since the potential $\phiro_{i,\theta}$ increases
	along $\gamma$, it is bounded on $\gamma$ by $\phiro_{i,\theta}(q)$,
	and it follows that necessarily $\xiro_{i,\theta}(q) = 0$.
	It now follows from \cite{loj_coll} that
	$b=+\infty$, and $\gamma(t) \to q$ when $t\to+\infty$.
	
	Otherwise, $\gamma$ has an accumulation point on the boundary
	of the compact space $\Dro_{i,\theta}$.
	In this case, it is clear from the description of
	$\xiro_{i,\theta}$ near the boundary in \Cref{s:corners} that $\gamma$
	has a limit at some point, see
	\cref{fig:droij_inv,fig:corner_sings}.
\end{proof}

\begin{figure}
	\centering
	\includegraphics{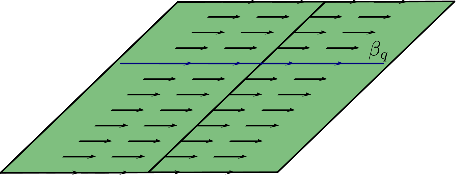}
	\caption{The vector field near a corner where both adjacent vertices 
		are invariant.}
	\label{fig:both_invariant}
\end{figure}

\begin{figure}
	\centering
	\includegraphics{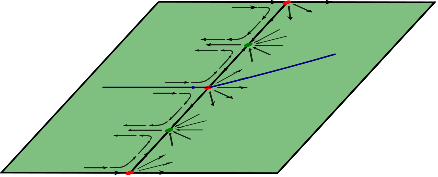}	
	\caption{The vector field near a non-invariant corner where the broken 
		trajectory (followed backwards) arrives at a
\index{red point}
red point.  In this case, 
		the continuation of the broken trajectory is determined by the dynamics 
		of the vector field.}
	\label{fig:droij_non_inv_2}
\end{figure}

\begin{figure}
	\centering
	\includegraphics{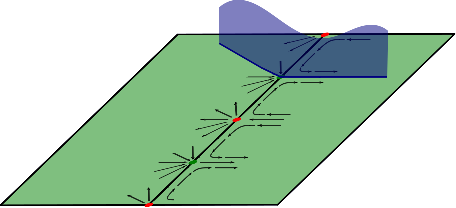}
	\caption{The vector field near a non-invariant corner where the broken 
		trajectory (followed backwards) arrives at a
\index{green point}
green point. In this case, 
		the continuation of the broken trajectory is determined by the 
		center-stable manifold of the saddle point where the broken trajectory 
		started.}
	\label{fig:droij_non_inv}
\end{figure}

\section{Broken trajectories} \label{ss:brotraj}
In this subsection we construct two broken trajectories (recall 
\Cref{def:broken_trajectory}) associated to each
saddle point of $\xiropol$ 
and define a topological path, the
{\em total broken trajectory}
as the 
union of their closures.

Let $\beta^+_q[0]$ be one of the two trajectories contained in the stable 
manifold of $\xiro_{i,\theta}$ at a saddle point $q[0]=q \in 
\Droc_{i,\theta}$. We observe that since $\beta^+_q[0]$ is a trajectory of 
$\xiro_{i,\theta}$,  \Cref{lem:four_pos} says that its closure consists of 
$q[0]$ and another point $q[1] \neq q[0]$. Furthermore, since the 
\index{manifold!center-stable}
center-stable manifold $\Sro(q)$ meets transversely $\Droc_{i,\theta}$ at 
$q$ and it is an invariant (by $\hxiro_i$) manifold we also have
\[
\beta^+_q[0] \subset \overline{\Sro(q)} \cap \Yro_\theta.
\]
Assume that we have constructed a topological path $\beta^+_q[\ell]$ 
endowed with a filtration 
\[
\beta^+_q[0] \subseteq \beta^+_q[1] \subset \cdots \beta^+_q[\ell]
\]
such that for each $k=0, \ldots, \ell$
\begin{enumerate}[label=(\alph*)]
	\item the set \label{it:a_integral} $\mathrm{int} \left( \beta^+_q[k] 
	\, \setminus \, \beta^+_q[k-1] \right)$ is an integral line of 
	$\xiro_{j,\theta}$ for some $j \in \Vpol$.
	
	\item \label{it:b_two_points} the boundary of the closure 
	$\overline{\beta^+_q[k] \setminus \beta^+_q[k-1]}$ consists of two 
	points $q[k]$ and $q[k+1]$ which, by \Cref{lem:four_pos} satisfy that 
	$q[k] \neq q[k+1]$.
	
	\item \label{it:c_centerstable} $\beta^+_q[k] \subset 
	\overline{\Sro(q)} \cap \partial \Yro_\theta.$
\end{enumerate}

The next is an exhaustive list, depending on where the point $q[\ell + 1]$ 
lies, that indicates  if the construction of the broken trajectory finishes 
or one has to construct $\beta^+_q[\ell+1]$ and iterate. At each step in 
which the algorithm is not finished, we verify that 
\cref{it:a_integral,it:b_two_points,it:c_centerstable} above are satisfied.
\begin{enumerate}
	\item \label{it:1stcase} The point  $q[\ell+1] \in \Droc_{j,\theta}$ is 
	a saddle point of $\xiro_{j,\theta}$ for some $j \in \V$. Define 
	\[
	\beta^+_q = \beta^+_q[\ell] \cup \{q[\ell+1]\}
	\]
	and finish the construction here.
	
	\item \label{it:2ndcase}  The point  $q[\ell+1] \in \Droc_{0,\theta}$ 
	and it is a fountain of $\xiro_{0,\theta}$. Define
	\[
	\beta^+_q = \beta^+_q[\ell] \cup \{q[\ell+1]\}
	\]
	and finish the construction here.
	
	\item \label{it:3rdcase} The point $q[\ell+1] \in \Dro_{j,k}$ and both 
	$j$ 
	and $k$ are invariant vertices with $k \to j$ (as in figure 
	\Cref{fig:both_invariant}). In this case there is a unique trajectory 
	$\gamma[\ell+1]$ of $\xiro_{k,\theta}$ that has $q[\ell+1]$ in its 
	closure. 
	We define 
	\[
	\beta^+_q[\ell+1] = \beta^+_q[\ell] \cup q[\ell+1] \cup \gamma.
	\]
	Now we check that in this case  the induction hypothesis 
	\cref{it:a_integral,it:b_two_points,it:c_centerstable} above are still 
	satisfied. Indeed:
	\begin{enumerate}[label=(\alph*)]
		\item $\mathrm{int}\left( \beta^+_q[\ell+1] \, \setminus \, 
		\beta^+_q[\ell] \right)= \gamma[\ell+1]$ which is, by definition, 
		an integral line.
		\item By \Cref{lem:four_pos}, we have that the boundary of 
		$\overline{\gamma[\ell+1]}$ consists exactly of two different 
		points.
		
		\item By the last statement of \Cref{cor:xiro_inv}, the closure of 
		the
		center stable manifold $\overline{\Sro(q)}$ meets $\partial \Uro$
		precisely at $\left(\{q[\ell+1]\}\cup\gamma[\ell] \cup
		\gamma[\ell+1]\right) \cap \partial \Uro$. Here we are using that
		$\Sro(q)$ is an invariant manifold contained in $\Yro_\theta$ that 
		has
		$\gamma[\ell]$ in its closure. A complete proof of this point is 
		below.
		
		\begin{proof}[Proof of  (c)]
			Our goal is to demonstrate that the property $\beta^+_q[k] 
			\subset 
			\overline{\Sro(q)} \cap \partial \Yro_\theta$ is preserved at 
			each step of 
			the construction. We assume, by induction, that the constructed 
			path so 
			far, $\beta^+_q[\ell]$, lies within the closure of the 
			center-stable 
			manifold $\overline{\Sro(q)}$. The terminal segment, 
			$\gamma[\ell]$, is a 
			trajectory of a vector field $\xiro_{j,\theta}$ which converges 
			to the 
			point $q[\ell+1] \in \Dro_{j,k}$.
			
			The current step of the construction addresses Case (iii), 
			where the edge 
			$jk$ is invariant. This implies that the corner point 
			$q[\ell+1]$ is an 
			elementary singularity of the extended vector field, as 
			established in 
			\Cref{lem:elementary_invariant}.
			
			By \Cref{cor:xiro_inv}, there exists a local coordinate system 
			$(r', s', 
			\alpha', \beta')$ in a neighborhood $\Uro$ of $q[\ell+1]$ such 
			that the 
			vector field $\xiro_U$ is linearized to the form:
			\[
			\xiro_U = (C_k r', -C_j s', 0, 0)
			\]
			where $C_k = d c_{0,k}$ and $C_j = d c_{0,j}$ are positive 
			constants. The 
			flow is described by the equations $\dot{r'} = C_k r'$ and 
			$\dot{s'} = -C_j 
			s'$. In this chart, the point $q[\ell+1]$ corresponds to the 
			origin. The 
			local manifold $\Droc_{j,\theta}$ is represented by the 
			half-space $\{s' = 
			0, r' > 0\}$, and $\Droc_{k,\theta}$ is represented by $\{r' = 
			0, s' > 
			0\}$. The trajectory $\gamma[\ell]$ approaches the origin along 
			the 
			positive $r'$-axis. The next segment of the broken trajectory, 
			$\gamma[\ell+1]$, must be the unique trajectory emanating from 
			the origin 
			within $\Droc_{k,\theta}$, which is the positive $s'$-axis.
			
			The center-stable manifold $\Sro(q)$ is, by definition, 
			invariant under the 
			flow of $\xiropol$. The induction hypothesis, $\gamma[\ell] 
			\subset 
			\overline{\Sro(q)}$, implies that $\overline{\Sro(q)}$ contains 
			the 
			positive $r'$-axis in the chart $\Uro$. Consequently, there 
			exists a 
			sequence of points $\{p_n\}_{n \in \N}$ in $\Sro(q) \cap \Uro$ 
			that 
			converges to the positive $r'$-axis. Let the coordinates of 
			these points be 
			$p_n = (r'_n, s'_n, \alpha'_n, \beta'_n)$, where $r'_n > 0$ and 
			$s'_n \to 
			0$ as $n \to \infty$.
			
			By the invariance of $\Sro(q)$, the full trajectory 
			$\gamma_{p_n}$ through 
			each point $p_n$ is contained in $\Sro(q)$. The coordinates of 
			this 
			trajectory at time $t$ are given by $\gamma_{p_n}(t) = (r'_n 
			e^{C_k t}, 
			s'_n e^{-C_j t}, \alpha'_n, \beta'_n)$.
			
			Consider the backward flow ($t < 0$) of these trajectories. As 
			$t \to 
			-\infty$, the $r'$-coordinate, $r'_n e^{C_k t}$, converges to 
			$0$, while 
			the $s'$-coordinate, $s'_n e^{-C_j t}$, grows without bound. 
			This shows 
			that trajectories originating in $\Sro(q)$ near the $\{s'=0\}$ 
			plane are 
			swept towards the $\{r'=0\}$ plane under the backward flow.
			
			Since $\overline{\Sro(q)}$ is a closed set, it must contain the 
			accumulation points of all trajectories contained within it. 
			Let $p_s = (0, 
			s'_0, \alpha', \beta')$ be an arbitrary point on the positive 
			$s'$-axis, 
			which represents the trajectory $\gamma[\ell+1]$. For any such 
			$p_s$, we 
			can construct a sequence of points in $\Sro(q)$ that converges 
			to it. 
			Specifically, for each $n$, we can choose a time $t_n < 0$ such 
			that $s'_n 
			e^{-C_j t_n} = s'_0$. Then the sequence of points $q_n = 
			\gamma_{p_n}(t_n)$ 
			lies in $\Sro(q)$, and as $n \to \infty$, we have $s'_n \to 0$, 
			which 
			forces $t_n \to -\infty$. Consequently, the $r'$-coordinate 
			$r'_n e^{C_k 
				t_n}$ converges to $0$. Thus, the sequence $\{q_n\}$ 
			converges to $p_s$.
			
			This demonstrates that any point on the positive $s'$-axis is a 
			limit point 
			of $\Sro(q)$. Therefore, the entire trajectory $\gamma[\ell+1]$ 
			is 
			contained within $\overline{\Sro(q)}$. This completes the 
			induction.
		\end{proof}
	\end{enumerate} 
	
	\item \label{it:4rdcase} The point $q[\ell+1] \in \Dro_{j,k}$ and, at 
	least 
	one of the two vertices $j$ or $k$ is a non-invariant vertex. Assume $k 
	\to 
	j$. In this case, one of the following happens,
	\begin{enumerate}[label = (\arabic*), ref= (\arabic*)]
		\item \label{it:acase} See \Cref{fig:droij_non_inv_2}. This case is 
		similar to the previous one: 
		there is a unique trajectory $\gamma[\ell+1]$ of $\xiro_{k,\theta}$ 
		that has $q[\ell+1]$ in its closure. We define 
		\[
		\beta^+_q[\ell+1] = \beta^+_q[\ell] \cup q[\ell+1] \cup 
		\gamma[\ell+1].
		\]
		Again, the  \cref{it:a_integral,it:b_two_points} are verified 
		directly. And \cref{it:c_centerstable} follows because by 
		\Cref{lem:noninv_corner}, \cref{it:noninv_corner_elem} the Hessian 
		of $\xiro_U$ is elementary along the normal bundle and, since 
		$\Sro(q)$ is an invariant manifold contained in $\Yro_\theta$ that 
		has $\gamma[\ell] \cup \{q[\ell]\}$ in its closure, then  $\Sro(q)$ 
		must also have the only trajectory of $\xiro_{k,\theta}$ that has 
		$q[\ell]$ in its closure.
		
		\item \label{it:bcase}
		Assume now that $\gamma[\ell]$ is the single unstable trajectory 
		going out of a 
		green point in its half-saddle part as in 
		\Cref{fig:droij_non_inv}. In this case there is not a unique 
		trajectory 
		of $\xiro_{j,\theta}$ that has $q[\ell+1]$ in its closure. Note 
		that in this 
		case, the right-hand part of \Cref{fig:droij_non_inv} corresponds 
		with 
		$\Dro_{k,\theta}$.
		
		Using the same methods as above, it is possible to choose a 
		trajectory
		in $\Droc_{j,\theta}$, which continues the broken trajectory, and 
		is contained
		in $\overline{\Sro(q)}$. We choose any such trajectory for the sake 
		of this 
		algorithm. However,
		uniqueness is not guaranteed by this argument.
		Observe that the chosen trajectory may necessarily be the corner 
		locus, in 
		which case, we follow the corner locus until we hit an adjacent red 
		point, and 
		continue 
		as before.
		In subsequent work, we plan on describing in more details the 
		structure
		of the cell near the boundary. As a result of this work, we
		expect the choice made in this step to be unique. 
		Furthermore, in this subsequent work we show that this $(iv) (2)$ 
		case can be 
		avoided by a small equisingular perturbation {\em a la Kodaira} of 
		the analytic 
		type of the defining function $f$.

	\end{enumerate}

\end{enumerate}

\begin{definition}\label{def:total_broken_trajectory}
	For each $q \in \Siro_{i,\theta}$ that is a saddle point of 
	$\xiro_{i,\theta}$ define the set 
	\[
	\beta_q = \beta^+_q \cup \beta^-_q \cup \{q\}
	\]
	and call it
\index{broken trajectory!total}
{\em the total broken trajectory of $\xiro_{i,\theta}$} at 
	$q$.
\end{definition}

\begin{lemma}\label{lem:equivalence_broken_trajectory}
	We have 
	\[
	\beta_q
	\subset \overline{\Sro(q)}
	\cap \partial \Yro_\theta.
	\]
\end{lemma}
\begin{proof}
	It follows from the construction of the broken trajectory above and 
	that the
	property  \cref{it:c_centerstable}, which is satisfied at each step of 
	the
	construction, says that that
	$\beta_q^+[k] \subset \overline{\Sro(q)} \cap \partial \Yro_\theta$.
\end{proof}

\begin{definition}\label{def:spine_at_rad_0}
	We define the
\index{spine!at radius zero}
spine at radius $0$ and angle $\theta$, and denote it by 
	$\Sro_{\theta}$, as the union of all the total broken trajectories 
	associated with all the saddle points of the vector fields 
	$\xiro_{i,\theta}$ for all $i\in\V$. That is,
	\[
	\Sro_{\theta}= \bigcup_{\substack{ i\in\V \\ q \text{ is a saddle point 
				of } \xiro_{i,\theta}}} \beta_q
	\]
\end{definition}

As we will see later in \Cref{ex:not_spine}, the spine at radius $0$ and
angle $\theta$ is not in general a spine as in \Cref{def:spine}.
In a sequel to this work we
will work out some further genericity conditions under which for a
generic analytical type representing a particular topological type and
for a generic angle, the spine at radius $0$ and angle $\theta$ is a
spine. In \Cref{ss:invariant_milnor} we construct a quotient of the spine
at radius $0$ which is always a spine for a quotient of the Milnor fiber
at radius $0$ which is homeomorphic to the Milnor fiber.

\section{The Petri dishes}
\label{ss:petri_dishes}

In this section we analyze the vector field and the broken trajectories 
inside the connected components of the part of the Milnor fiber that 
corresponds to the non-invariant graph. We use 
\Cref{not:parts_Milnor_fiber}.

\begin{lemma}\label{lem:noninv_petri}
	
	Let $\Xi$ be a branch (recall \Cref{def:Ga_branch}) of $\Gapol$ at a
	vertex $i$ of $\Gapol \setminus \Upsilon$. Then $\Fro_\theta[\Xi]$ is a
	disjoint union of disks. Moreover, if $ij$ is the edge that joins $\Xi$
	with $\Gapol \setminus \Xi$, then  $\Fro_\theta[\Xi]$ is a disjoint
	union of $\gcd(m_i,m_j)$ disks.
	
\end{lemma}

\begin{proof}
	That $\Fro_\theta[\Xi]$  is a disjoint union of disks follows from 
	\Cref{lem:neigh_invariant} that implies, that in $\Gamin$, the graph 
	$\Xi$ is a bamboo. The process of blowing up more to get to $\Gapol$ 
	may introduce new vertices, but does not change the topology of the 
	Milnor fiber. The statement on the number of connected components 
	follows because near the intersection point of $\Xi$ with the rest of 
	$\Gapol$, the Milnor fiber at radius $0$ looks like the Milnor fiber of 
	$u^{m_i}v^{m_j}$ which is a disjoint union of $\gcd(m_i, m_j)$ 
	cylinders.
\end{proof}

\begin{definition}\label{def:petri_dish}
	Let $\Xi$ be a branch of $\Gapol$ at a vertex $i$ of $\Gapol \setminus 
	\Upsilon$. We call each of the disks in $\Fro_\theta[\Xi]$, a
\index{Petri dish}
{\em Petri dish}.
We denote Petri dishes by
\index{$\D_\theta$}
$\D_\theta$.
	
	We say that a Petri dish is {\em maximal} if it is not strictly 
	contained in any other Petri dish.
\end{definition}

Next we introduce some useful notation for this subsection.

\begin{notation}\label{not:petri_dish}
	Let $\D_\theta$ be a Petri dish coming from some branch $\Xi$ as in 
	\Cref{def:petri_dish}. 
	
	We denote by $\beta[\D_\theta]$ the union of all the broken 
	trajectories that start at some $\Dro_{i,\theta}$ with $i \in \V_\Xi$, 
	intersected with $\D_\theta$. That is,
	\[
	\beta[\D_\theta] = \left(\bigcup_{\substack{ q \in \Sigma_{i,\theta} \\ 
			i \in \V_\Xi}} \beta_q \right) \cap \D_\theta.
	\]
	We denote by $S[\D_\theta]$ the part of the spine at radius $0$ that 
	lies in $\D_\theta$. That is
	\[
	S[\D_\theta] = \Sro_{\theta} \cap \D_\theta. 
	\]
	
\end{notation}

A Petri dish $\D_\theta$ inherits a the structure of a stratified 
topological manifold with smooth strata. But it does not have, a priori, a 
smooth structure. In the the rest of this section we make use of 
Poincar\'e-Hopf index theorem to prove results about the set 
$\beta[\D_\theta]$ and, in order to formalize such arguments we need to put 
a smooth structure on $\D_\theta$. 

\begin{lemma}\label{lem:smooth_str_petri}
	Let $\D_\theta$ be a Petri dish associated with a branch $\Xi$ with 
	vertex set
	$\V_\Xi$. Then, $\D_\theta$ admits a smooth structure,
	which is compatible with that of $\Droc_{i,\theta} \subset \D_\theta$,
	for each $i$ on $\Xi$, and the restricted vector
	field $\xiro_{\D_\theta}= \xiropol|_{D_\theta}$ is a continuous vector 
	field on
	$\D_\theta$.
\end{lemma}

\begin{proof}
	Each connected component $\D_{i,\theta}$ of $\Dro_{i,\theta}$ can be 
	thought of as a cobordism between the connected components $\partial_j 
	\D_{i,\theta}$ of $\partial \D_{i,\theta}$ corresponding to edges $j 
	\to i$ and the connected components $\partial_k \D_{i,\theta}$ of 
	$\partial \D_{i,\theta}$ corresponding to edges $i \to k$. 
	
	Furthermore, let $ij$ be an edge $j \to i$. Then, there is a 
	diffeomorphism 
	\[
	\iota_{ij}: \partial_j \D_{i,\theta} \to \partial_i \D_{j,\theta}
	\]
	induced by the structure of fiber product  that identifies both sets of 
	boundary components. Then, by \cite[Theorem 1.4]{hMilnor}, there exists 
	a 
	smooth structure on $ \D_{i,\theta} \sqcup_{\iota_{ij}} \D_{j,\theta}$ 
	which is compatible with each of the smooth structures on 
	$\D_{i,\theta}$ 
	and $\D_{j,\theta}$. 
	
	By definition, the vector field $\xiro_{\D_\theta}$ is well defined on 
	(and tangent to) $\Dro_{i,j,\theta}$.
	
	By iteratively applying this reasoning, we get a smooth structure on 
	all $\D_\theta$ and a continuous vector field on it. Furthermore, this 
	vector field coincides with $\xiro_{i,\theta}$, up to multiplication by 
	a real positive function, on each of the pieces of $\D_\theta$.
	
\end{proof}
\begin{rem}
	Note that the vector field $\xiro_{\D_\theta}$ is a continuous (by 
	construction) vector field but, in general, it is not $C^1$. 
\end{rem}

\begin{rem}\label{rem:petri_dish}
	
	\begin{enumerate}
		
		\item By definition, $\beta[\D_\theta] \subseteq S[\D_\theta]$. 
		And, in general, the above is not an equality as we will see later 
		in an example.
		
		\item \label{it:non_zero_ind} The vector field $\xiro_{\D_\theta}$ 
		does not have any nonzero index singularities in the interior of 
		$\D_\theta \setminus \beta[\D_\theta]$. This is because 
		$\beta[\D_\theta]$ contains all the saddle points by construction 
		and the other singularities are by \Cref{lem:noninv_corner} 
		half-fountain-half-saddle or half-sink-half saddle points, (i.e. 
		they are green and red points like those of 
		\cref{fig:corner_sings}) which all have index $0$.
		
		\item  \label{it:no_ver_val_1} By definition of broken trajectory, 
		the
		set $\beta[\D_\theta]$ seen as graph, if it has vertices of valency 
		$1$,
		then they have to lie on $\partial \D_\theta$.
		
	\end{enumerate}
\end{rem}

\begin{block}\label{blc:rel_poincare}
	Here we deduce an easy relative version of the
\index{Poincar\'e-Hopf index}
Poincar\'e-Hopf index 
	theorem for vector fields on surfaces with boundary. Let $F$ be an 
	oriented compact surface of genus $g$ with $b>0$ boundary components, 
	and let $\xi_F$ be a vector field on $F$ which is also tangent to 
	$\partial F$. Assume furthermore that $\xi$ has isolated singularities 
	(some of them might be on the boundary).  Now take the double-surface 
	$2F$ of $F$: that is, consider $-F$ (which is the surface with opposite 
	orientation) and glue $-F$ to $F$ along their boundaries by the 
	identity map (note that the identity map inverts the orientation from 
	$\partial F$ to $\partial( -F)$). The surface $-F$ comes naturally 
	equipped with a vector field $\xi_{-F}$ which is identified with 
	$\xi_F$ on $\partial (-F)$ by the identity map. Thus, we have a well 
	defined vector field $\xi_{2F}$ defined on $2F$ with a isolated 
	singularities. Observe that the surface $2F$ is a closed oriented 
	surface of genus $2g + b -1$.  This construction allows us to do two 
	things:
	\begin{enumerate}
		\item Assign an index to isolated singularities on the boundary: 
		for an isolated singularity $p\in \partial F$ of $\xi_F$, we define 
		the index of $\xi_f$ at $p$ by the formula
		
		\begin{equation}\label{eq:half_index}
			\Ind_{\xi_F} (p) = \frac{1}{2} \Ind_{\xi_{2F}} (p)
		\end{equation}
		where, on the right-hand side, $p$ is seen as a point in $2F$.
		
		\item We compute a relative version of the Poincar\'e-Hopf index 
		theorem. By applying the classical theorem to the surface $2F$, we 
		get:
		\[
		2\sum_{p \in \mathring{F}} \Ind_{\xi_F} (p) + \sum_{p \in \partial 
			F \subset F} \Ind_{\xi_{2F}} (p) = 2 - 2(2g+b-1).
		\]
		which, after dividing by $2$ and applying the previous definition 
		\cref{eq:half_index} yields,
		\begin{equation}\label{eq:rel_poincare}
			\sum_{p \in F} \Ind_{\xi_F} (p)  = 2 - 2g - b = \chi(F).
		\end{equation}
		Where the term on the left-hand side allows for $\xi_F$ to have 
		isolated singularities on the boundary of $F$.
	\end{enumerate}
	Next we classify all the singularities that might appear on the 
	boundary of connected components of $\D_\theta \setminus 
	\beta[\D_\theta]$.
	\begin{lemma}\label{lem:class_singularities}
		
		Let $\Delta$ be the closure of a connected component of $\D_\theta 
		\setminus
		\beta[\D_\theta]$ and assume that $\partial \Delta$ is an embedded 
		circle. Then
		the vector field $\left. \xiro_{\D_\theta}\right|_{\Delta}$ has no 
		nonzero
		index singularities on the interior of $\Delta$ and, on $\partial 
		\Delta$ it
		may have
		
		\begin{enumerate}
			\item \label{it:zero_ind} Quarter-fountain-quarter-saddle, 
\index{saddle!quarter}
\index{fountain!quarter}
\index{sink!quarter}
			quarter-sink-quarter-saddle, quarter saddle points as the only 
			singularities of index $0$ (see \cref{fig:half_index_0}),
			\item \label{it:pos_ind}
\index{sink!half}
\index{fountain!half}
 half-sink, half-fountain points as the 
			only singularities of index $1/2$ (see \cref{fig:half_index_1}),
			
			\item \label{it:neg_ind} half-saddle points as the only 
			singularities of index $-1/2$ (see the right-hand side of 
			\cref{fig:half_index_m1}).
		\end{enumerate}
	\end{lemma}
	\begin{proof}
		
		The first part of the statement on the singularities in the 
		interior of
		$\Delta$ is exactly \Cref{rem:petri_dish} \cref{it:non_zero_ind}. 
		Now we
		notice that the only singularities that $\partial\Delta$ might 
		have, happen when
		$\partial \Delta$ crosses a singularity of $\left.
		\xiro_{\D_\theta}\right|_{\Delta}$. So let $p \in \partial \Delta$ 
		be a
		singularity of  $\left. \xiro_{\D_\theta}\right|_{\Delta}$, next we 
		go
		through all the possible cases depending on where $p$ lies.
		
		\textbf{Case 1}. If $p \in \partial \D_\theta$, then $p$ is either 
		a half-fountain or a half-sink point when seen as a singularity of 
		$\xiro_{\D_\theta}$ (because of \Cref{lem:noninv_corner}). In this 
		case, $p$ seen as a singularity of  $\left. 
		\xiro_{\D_\theta}\right|_{\Delta}$ can only be a  half-fountain 
		(red vertex) or a half-sink (green vertex) singularity. This covers 
		\cref{it:pos_ind}.
		
		\textbf{Case 2.} If $p$ is in the interior of $\D_\theta$, then 
		there are two possibilities:
		\begin{enumerate}[label= Case (2.\alph*),leftmargin=6\parindent]
			\item $p$ is a saddle point of  $\xiro_{\D_\theta}$. In this 
			case the segment of $\partial \Delta$ that passes through $p$ 
			might be formed either by a stable trajectory and an unstable 
			trajectory (the blue singularity of \cref{fig:half_index_0}) 
			giving a singularity of index $0$, or by two stable 
			trajectories (the right-hand side of \cref{fig:half_index_m1}) 
			giving a singularity of index $-1/2$. Note that it can't happen 
			that this segment is formed by two unstable trajectories 
			(left-hand side of \cref{fig:half_index_m1}) since 
			$\beta[\D_\theta]$ contains by definition the stable 
			trajectories of all saddle points.
			
			\item $p$ is a red or green point of a smaller Petri dish 
			$\D'_\theta \subset \D_\theta$. In this case the segment of 
			$\partial \Delta$ that passes through $p$ must correspond to a 
			broken trajectory that either enters the smaller Petri dish 
			$\D'_\theta$ or exits $\D'_\theta$ (this corresponds to the two 
			left figures of \cref{fig:half_index_0}). Note that by the 
			construction of broken trajectories, no broken trajectory can 
			contain a segment of the boundary of a Petri dish since 
			boundaries of Petri dishes do not have saddle points of 
			$\xiro_{\D_\theta}$.
		\end{enumerate} 
		
		These two subcases cover all the possibilities in 
		\cref{it:zero_ind} and \cref{it:neg_ind} finishing the proof.
	\end{proof}
	
	\begin{figure}[!ht]
		\centering
		\includegraphics*[scale=1.5]{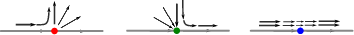}
		\caption{Singularities (red, green and blue points) of index $0$ of 
			a vector field tangent to the boundary (gray thick line) of a 
			surface.}
		\label{fig:half_index_0}
	\end{figure}
	\begin{figure}[!ht]
		\centering
		\includegraphics*[scale=1.5]{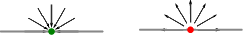}
		\caption{Singularities (green and red points) of index $1/2$ of a 
			vector field tangent to the boundary (gray thick line) of a 
			surface.}
		\label{fig:half_index_1}
	\end{figure}
	\begin{figure}[!ht]
		\centering
		\includegraphics*[scale=1.5]{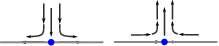}
		\caption{Singularities (blue points) of index $-1/2$ of a vector
			field tangent to the boundary (gray thick line) of a surface.
			The situation on the left hand side does not occur in the above
			algorithm.}
		\label{fig:half_index_m1}
	\end{figure}
\end{block}

\begin{lemma}\label{lem:number_ saddle}
	Let $\D_\theta$ be a Petri dish coming from some branch $\Xi$. Let
	$i$ be the root of $\Xi$ and let $j$ be the vertex in $\Gapol \setminus
	\Xi$ to which it is connected. Then the number of saddle points $I$ of
	$\xiro_{\D_\theta}$ is equal to:
	\begin{equation}\label{eq:number_saddle}
		I = -
		\frac{
			\left|
			\begin{matrix}
				\varpi_i & m_i  \\
				\varpi_j & m_j
			\end{matrix}
			\right|
		}{\gcd(m_i,m_j)} - 1.
	\end{equation}
	In particular, $\beta[\D_\theta]$ is non-empty if and only if $I \geq 
	1$, equivalently, if there are at least $2$ red points on $\partial 
	\D_\theta$.
\end{lemma}

\begin{proof}
	
	The term with the fraction on \cref{eq:number_saddle} equals the number 
	of red
	points (which equals the number of green points, too)
	on $\partial \D_\theta$ (\Cref{lem:int_ind}). These
	singularities look on $\D_\theta$ like half sinks (green points) and 
	half
	fountains (red points). See \cref{fig:half_index_1}. These have relative
	Poincar\'e-Hopf index equal to $1/2$. By \Cref{rem:petri_dish}
	\cref{it:non_zero_ind} the only singularities of nonzero index on the 
	interior
	of $\D_\theta$ are saddle points. A direct application of
	\cref{eq:rel_poincare} yields the result since:
	
	\[
	-
	\frac{
		\left|
		\begin{matrix}
			\varpi_i & m_i  \\
			\varpi_j & m_j
		\end{matrix}
		\right|
	}{\gcd(m_i,m_j)} - 
	I = 1.
	\]
	That $\beta[\D_\theta]$ is non-empty if and only if there is at least 
	$1$ saddle point in $\D_\theta$ follows from the definition of the 
	broken trajectories which start at all saddle points.
\end{proof}

\begin{lemma}\label{lem:no_reentry}
	Let $\D_\theta$ be a maximal Petri dish and let $\gamma$ be a 
	broken trajectory. If a segment of $\gamma$ is contained in $\D_\theta$ 
	and 
	exits through a point $p \in \partial \D_\theta$, then no subsequent 
	segment of $\gamma$ can re-enter $\D_\theta$.
\end{lemma}

\begin{proof}
	The Petri dish correponds to a connected component of
	$\Gamma \setminus \Upsilon$, adjacent to an invariant vertex, say $i 
	\in \V$.
	Let $k$ be the neighbor of $i$ in this component, i.e. the unique 
	non-invariant
	neighbor of $i$.
	After exiting $\D_\theta$, we have a segment of $\gamma$, which we 
	follow
	(backwards) until we end in a limit point, which exists by
	\Cref{lem:four_pos}. If this limit point is a saddle point of $\xiro_i$,
	then this is the end of the broken trajectory, and the proof is
	finished, in this case.
	
	Otherwise, this limit point is on the
	corner locus $\Dro_{i,j,\theta}$ for some neighbor $j$.
	Since the potential $\phiro_{i,\theta}$ decreases along this trajectory,
	and is constant along $\Dro_{i,k,\theta}$ it cannot happen that $j = k$.
	Therefore, $j$ is invariant.
	It cannot happen that we have a directed edge $i \to j$, since
	any trajectory of $\xiro_{i,\theta}$ projects to a trajectory of
	$\xi_i$ on $D_i$, and $\xi_i$ has an attractor at $D_i\cap D_j$, if
	$i \to j$. Therfore, we have $j\to i$.
	Now, continuing the trajectory beyond
	this step, we can never return through $\Dro_{i,j,\theta}$ again, since
	$D_j \cap D_i$ is a repeller of $\xi_i$.
\end{proof}

\begin{cor} \label{cor:no_green_points}
	Let $\D_\theta$ be a maximal Petri dish. The set $\beta[\D_\theta]$ 
	does not 
	contain any green points in $\partial \D_\theta$.
\end{cor}
\begin{proof}
	A green point is, by definition, half a sink. In order for 
	$\beta[\D_\theta]$ 
	to contain a green point it must have  happened that a broken 
	trajectory 
	starting at $q \in \D_\theta$ must have exited $\D_\theta$ and re-enter 
	it 
	through that green point. But \Cref{lem:no_reentry} prevents this.
\end{proof}

\begin{lemma}\label{lem:connected_comp}
	Let $\D_\theta$ be a maximal Petri dish. Each connected component of 
	$\beta[\D_\theta]$:
	\begin{enumerate}
		\item \label{it:conn_red} is connected to some red vertex on 
		$\partial \D_\theta$ and
		\item \label{it:no_loops} is contractible.
	\end{enumerate}
\end{lemma}

\begin{proof}
	Let $q \in \beta[\D_\theta]$ be a saddle point of some 
	$\xiro_{i,\theta}$ with $i \in \V_\Xi$. 
	
	\textbf{Claim 1.} The broken trajectory $\beta^+_q$ can't end at $q$.
	\begin{proof}[Proof of Claim 1.]
		If $\beta^+_q$ ends at $q$, because of \Cref{lem:no_reentry}, it 
		must 
		be contained entirely in $\D_\theta$. Now we apply a similar 
		reasoning 
		as in the proof of \Cref{lem:no_reentry}: contract each of the 
		smaller 
		Petri dishes contained in $\D_\theta$ to a point. We obtain a loop, 
		formed by integral lines of a vector field that, by 
		\Cref{lem:potential_xiro}, admits a potential which is 
		contradictory.
	\end{proof}
	By the previous claim, either the broken trajectory ends at another 
	saddle point $q' \in \D_\theta$ or it has to exit $\partial \D_\theta$ 
	by some red point on $\partial \D_\theta$ (note that it can only exit 
	$\D_\theta$ through a red point since broken trajectories are 
	constructed following the flow backwards). If it exits $\D_\theta$ 
	through a red point, we are done proving \cref{it:conn_red}. 
	
	Assume otherwise it arrives at another saddle point $q'$. Now consider 
	one of the broken trajectories $\beta'^+_{q'}$ associated with $q'$.
	
	\textbf{Claim 2.} The broken trajectory  $\beta'^+_{q'}$ can't end at 
	$q$.
	\begin{proof}[Proof of Claim 2.]
		The proof is very similar to the previous claim:  contract each of 
		the smaller Petri dishes contained in $\D_\theta$ to a point. We 
		obtain a loop, formed by integral lines of a vector field that, by 
		\Cref{lem:potential_xiro}, admits a potential which is 
		contradictory.
	\end{proof}
	
	After claim 2, we get that the  broken trajectory $\beta'^+_{q'}$ 
	either ends at a red point in $\partial \D_\theta$ or it gets to 
	another saddle point $q''$ that, following a similar reasoning to that 
	of Claim 1 and 2, must be different from $q$ and $q'$. Now we can 
	iterate this process and get that, after each step, we either arrive at 
	a red point in $\partial \D_\theta$ or to a saddle point that we have 
	not visited before. But there are only finitely many saddle points 
	contained in $\D_\theta$. This proves \cref{it:conn_red}. 
	
	Now we prove \cref{it:no_loops}. The proof of \cref{it:conn_red} 
	already shows
	that there can be no loops in the interior of
	$\D_\theta$. Assume that we have a loop $L$ that contains at least a 
	point $p$
	in $\partial \D_\theta$. This point is necessarily red by
	\Cref{cor:no_green_points}. Furthermore, assume that this loop does not 
	contain
	any other loops of $\beta[\D_\theta]$ in the disk that it encloses, for 
	if it
	does, we can always choose one of the smaller loops. Then $L$ is a
	concatenation of paths $L_1, \ldots, L_n$ and red points $p_1, \ldots, 
	p_n$.
	The loop $L$ encloses a disk, which by \Cref{rem:petri_dish}
	\cref{it:non_zero_ind} and by our hypothesis on the minimality of $L$, 
	does not
	contain any nonzero index singularities. Since each $L_k$ joins two red
	points, it must contain at least one singularity of index $-1/2$: 
	indeed, if
	for example $L_1$ joins $p_1$ with $p_2$, then there are two 
	trajectories of
	$\xiro_{\D_\theta}$ inside $L_1$ with opposite sign, so necessarily 
	there is a
	saddle point on $L_1$. Thus, using \cref{eq:rel_poincare} we get
	
	\[
	\sum_{i=1}^{n} \left( \Ind_{\xiro_{\D_\theta}} (p_i) + \sum_{q \in L_i} 
	\Ind_{\xiro_{\D_\theta}} (q) \right) \leq \frac{n}{2} -  \frac{n}{2} = 
	0.
	\]
	which contradicts the fact that the Euler characteristic of the disk 
	enclosed by $L$ is equal to $1$.
\end{proof}

\begin{cor}\label{cor:complement_disks}
	The closure of each connected component of $\D_\theta \setminus 
	\beta[\D_\theta]$ is a closed disk.
\end{cor}

\begin{proof}
	By \Cref{lem:connected_comp}, each connected component 
	$\beta[\D_\theta]$ is connected to some red vertex and it is a tree. 
	The interior of the complement of a tree contained in a disk is a 
	disjoint union of open disks provided that the tree is connected to at 
	least a point of the boundary of the disk. Now, by 
	\Cref{rem:petri_dish} \cref{it:no_ver_val_1}, each connected component 
	of $\beta[\D_\theta]$ does not have vertices of valency $1$ in the 
	interior of $\D_\theta$, so the boundary of the closure of each 
	connected component must be an embedded circle.
\end{proof}
\begin{prop}\label{prop:petri_dishes}
	If the set $\beta[\D_\theta]$ is non-empty, then
	\begin{enumerate}
		\item \label{it:con} it is connected,
		\item \label{it:contractible} it is contractible,
		\item \label{it:all_red_points} it contains all the red points in 
		$\partial \D_\theta$.
	\end{enumerate}
\end{prop}
\begin{proof}
	We construct a (possibly disconnected) abstract graph $\tilde{\beta}$:
	\begin{enumerate}
		\item The set of vertices is the union of the saddle points of 
		$\xiro_{\D_\theta}$ and the red vertices on $\partial \D_\theta$.
		
		\item  there is an edge between a saddle point and another saddle 
		point or between a saddle point and a red point if there is a 
		broken trajectory in $\beta[\D_\theta]$ that joins them.
	\end{enumerate}
	The underlying $1$-dimensional $CW$-complex of $\tilde{\beta}$ is
	homeomorphic to $\beta[\D_\theta]$ union all the red vertices. And we
	have a natural injective map $\beta[\D_\theta] \hookrightarrow
	\tilde{\beta}$ that realizes $\beta[\D_\theta]$ in $\tilde{\beta}$.
	Since each red vertex can be thought of as a tree and the connected
	components  of $\beta[\D_\theta]$ are trees by
	\Cref{lem:connected_comp} \cref{it:no_loops}, the Euler characteristic
	of $\tilde{\beta}$ counts the number of connected components. If we let
	$R$ be the number of red points, $I$ the number of saddle points and $e$
	the number edges, then we have $e=2I$, since every edge is a trajectory
	coming out of exactly one saddle point, and every saddle point has two 
	such
	trajectories. Therefore,
	\[
	\chi(\tilde{\beta}) = R + I - e  = R - I.
	\]
	But $R - I = 1$ by \Cref{lem:number_ saddle}. So $\tilde{\beta}$ is 
	connected an contains all the red points, we conclude 
	\cref{it:con,it:all_red_points}.
	
	That $\beta[\D_\theta]$ is contractible \cref{it:contractible} follows 
	directly from \Cref{lem:connected_comp} \cref{it:no_loops}.
\end{proof}
\begin{example}\label{ex:not_spine}
	In this example we show that the union of broken trajectories starting 
	at all
	saddle points does not, in general, form a
\index{spine!at radius zero}
spine for the Milnor fiber 
	at radius
	$0$. 
	
	Consider the  plane curve singularity $C$ defined by $f(x,y)=
	y^6-2y^3x^4-3x^8+yx^7$. The plane curve $C$ consists of two branches 
	which are
	Brieskorn-Pham singularities of type $(3,4)$ with the same tangent. The 
	polar
	curve $P$ is defined by the equation $f_y = 6y^5-6y^2x^4+x^7$ and has 
	two
	branches. The strict transform of one of the branches passes through the
	exceptional divisor $D_2$ and the other one through the divisor $D_3$. 
	The
	divisor $D_3$ is an invariant divisor and so the point $\tilde{P} \cap 
	D_3$ is
	a base point of the Jacobian ideal by \Cref{lem:polar} 
	\cref{it:polar_node}.
	After blowing up this base point once, we get the dual graph on the 
	bottom of
	\cref{fig:68_res_graph}.
	
	\begin{figure}[!ht]
		\centering
		\includegraphics*{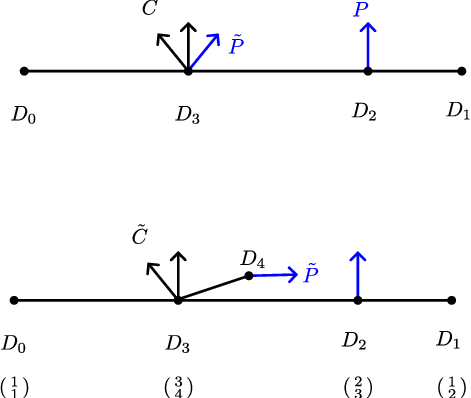}
		\caption{On top the dual graph of the minimal resolution of $C$. On 
			bottom,
			the resolution resulting from blowing up the base point of the 
			divisor $D_3$
			once. In
			blue we have the strict transform of the polar curve $P$ which 
			has two
			branches.} \label{fig:68_res_graph} \end{figure}
	
	Let's look at the divisor $D_3$ in the minimal resolution of the plane 
	curve.
	For that, we consider the monomial transformation
	\[
	\pi(u,v)= (u^3v^2, u^4v^3)
	\]
	which sees a neighborhood around $D_3 \cap D_2$. Using $\tau_3 = 8$ and
	$\tau_2 = 5$ we can compute the restriction of the extension of the 
	vector
	field over $D_3^\circ$ in this chart (see \cref{fig:68_D3} for its
	representation) and we get the expression:
	
	\[
	\xi_3= -\frac{18 {\left(\bar{v} - 1\right)}}{|v|^4\left( \bar{v}^{2} - 
		2 \bar{v} - 3 \right) }.
	\]
	
	The intersection with the two components of the strict transform of $C$ 
	are on
	the real points $-1, 3 \in \C$. We see (\cref{fig:68_D3}) that there is 
	a
	trajectory of $\xi_3$ on $D_3^\circ$ that goes from the saddle point
	corresponding to the intersection point
	$D_3 \cap D_4$. On $D_2$ there is another intersection point
	$\tilde{P} \cap D_2$.

	\begin{figure}[!ht]
		\centering
		\includegraphics*[scale=0.8]{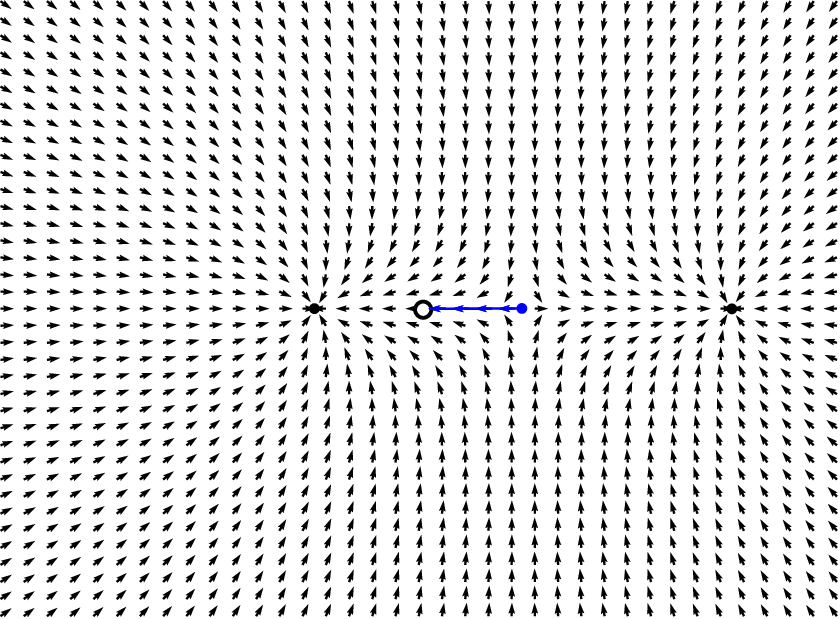}
		\caption{The divisor $D_3$. The two black dots correspond with the
			intersection points of the strict transforms of the curve $C$. 
			The blue dot
			corresponds with the intersection point of $D_3$ and $D_4$.
			The circled dot corresponds with the intersection point of 
			$D_2$ and
			$D_3$. The blue segment is a trajectory.}
		\label{fig:68_D3}
	\end{figure}
	
	Now we analyze what happens at the level of the Milnor fiber at radius 
	$0$. The
	multiplicity $m_2$ of $\pi^*f$ at $D_2$ is $16$. The multiplicity $m_1$ 
	is $8$.
	This says that $\Dro_{1,\theta}$ consists of $8$ disks and 
	$\Dro_{2,\theta}$
	consists of $8$ cylinders for each $\theta \in \R / 2 \pi \Z$. Thus, on 
	each of
	these $8$ cylinders, the vector field $\xiro_{2,\theta}$ has $2$ saddle 
	points.
	Furthermore, a computation yields that $\varpi_2=1$ and we know that
	$\varpi_3=0$. Therefore, using \Cref{lem:transversality_number}
	\cref{it:trans_iii} we get that each connected component of 
	$\Dro_{2,3,\theta}$ 
	has
	$3$ red points (and $3$ green points). We conclude that for
	a generic angle $\theta$,
	one of the three red
	points on each cylinder is chosen by two half broken trajectories 
	starting at
	the two saddle points on each connected component of $\Dro_{2,\theta}$. 
	We
	have, in total $24$ red points and $8$ of these are chosen by two half 
	broken
	trajectories on their corresponding cylinders.
	
	We analyze now what happens in $D_4$. The multiplicity $m_4=m_3 = 24$, 
	so
	$\Dro_{4,\theta}$ consists of $24$ disks. Since $\varpi_4 =2$, each of 
	these
	disks has $2$ red points and $2$ green points ($48$ of each in total).
	Moreover, $D_4$ intersects $\tilde{P}$ and so each of the $24$ disks of
	$\Dro_{4,\theta}$ has a saddle point of $\xiro_{4,\theta}$ on its 
	interior.
	Each of the
	two red points on each component of $\partial \Dro_{4,\theta}$ gets 
	chosen by
	one of the broken trajectories associated with the saddle point. 
	
	The fact that there is a trajectory going from $D_4 \cap D_3$ to 
	$D_3\cap D_2$ yields that there are $24$ trajectories in 
	$\Dro_{3,\theta}$ joining the $24$ red points of $\Dro_{2,3,\theta}$ 
	with $24$ of the $48$ green points lying on $\Dro_{3,4,\theta}$.
	
	Let $p_1,p_2 \in \Dro_{2,\theta}$ be two saddle points of 
	$\xiro_{2,\theta}$
	lying on the same connected component of $\Dro_{2,\theta}$. We know 
	that the
	Hessian of $\xiro_{2,3\theta}$ at the red points is non degenerate, so 
	the two
	broken trajectories that start at $p_1$ and $p_2$ and end at the
	same red point,
	arrive (flowing backwards) at this red point, with different tangents.
	After passing through the red points, both broken trajectories choose 
	the same
	trajectory on $\Dro_{3, \theta}$ and this is one of the $24$ 
	trajectories that
	lie above the trajectory on $D_3^\circ$ connecting $D_2 \cap D_3$ with 
	$D_3
	\cap D_4$. Here, the closures of $\Sro(p_1)$ and $\Sro(p_2)$ intersect
	transversely along the trajectory in $\Droc_{3,\theta}$.
	
	\begin{figure}[!ht]
	    \includegraphics*{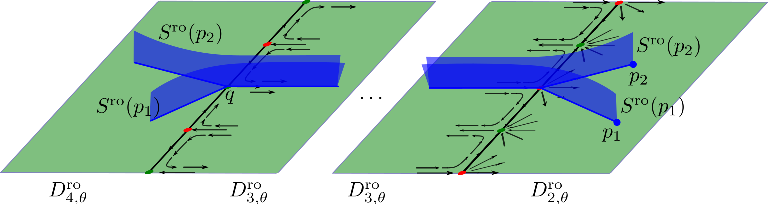}
		\caption{On the right part of the picture we see a neighborhood of 
			the  $\Dro_{2,3,\theta}$. On the right we also see, in blue, the 
			center-stable manifolds $\Sro(p_1)$ and $\Sro(p_2)$. On the left 
			part of the picture we see a neighborhood of $\Dro_{3,4,\theta}$. 
			We see how the two center-stable manifolds define different 
			trajectories in $\Dro_{4,\theta}$ even though they define the same 
			in $\Dro_{3, \theta}$.}
		\label{fig:coinciding_broken}
	\end{figure}
	
	If $q$ is the green point in
	$\Dro_{3,4,\theta}$ chosen by the common segment of the two broken
	trajectories, then the germs of $\overline{\Sro(p_1)}$ and
	$\overline{\Sro(p_2)}$ at $q$ define different manifolds, which by the
	non-degeneracy of the vector field $\xiropol$ at $q$, define different
	trajectories in $\Dro_{4,\theta}$, that is $\overline{\Sro(p_1)} \cap
	\Droc_{4,\theta}$ and $\overline{\Sro(p_2)} \cap \Droc_{4,\theta}$ are 
	two
	distinct trajectories of $\xiro_{4,\theta}$. In \cref{fig:one_of_three} 
	we see
	the three qualitatively different possibilities  and we see how all of 
	them
	form $CW$ complexes that enclose a disk and so we conclude that in 
	general the
	spine at radius $0$ and angle $\theta$ (\Cref{def:spine_at_rad_0}) is 
	not a
	spine.

	\begin{figure}[!ht]
		\centering
		\includegraphics*[scale=0.6]{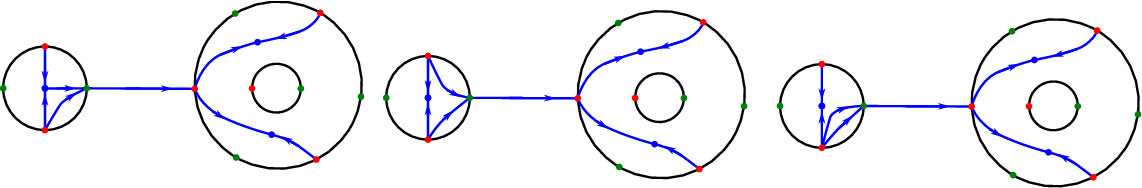}
		\caption{The three possibilities for the broken trajectories coming 
			from
			one Petri dish and entering the Petri dish (flowing backwards)
			corresponding to the component of the polar
			curve that intersects $D_4$.}
		\label{fig:one_of_three}
	\end{figure}
	
\end{example}

\begin{rem}
	In the previous example, a special situation occurs, which can be seen
	in \cref{fig:68_D3}. The blue point has two trajectories emerging from
	it, and one of the trajectories converges to the white point.
	In a generic situation, each trajectory emerging from the blue point
	would converge to one of the sinks. This picture cannot be made
	generic by perturbing the metric by \Cref{rem:an_in}.
	It can, however be made generic by an equisingular deformation
	of the analytic structure of $(C,0)$.
	Such deformations will be further studied in a subsequent manuscript.
\end{rem}

\section{The invariant Milnor fibration}
\label{ss:invariant_milnor}

In this subsection we introduce a $\Cinf$-fibration which is equivalent to 
the
Milnor fibration and whose total space
is a quotient of the exceptional part of the boundary of $\Yropol$.

\begin{block}
	Let
	\[
	\partial_\V \Yropol = \bigcup_{i \in \Vpol} \Dro_i.
	\]
	We consider the equivalence relation $\sim_\Upsilon$ that collapses
	each Petri dish $\D_\theta$ to a point.
	Thus, $\sim_\Upsilon$ is the finest equivalence relation of
	$\partial_\V \Yropol$ for which $p \sim_\Upsilon p'$ for all
	$p,p' \in \Dro_{i,\theta}$, for any $\theta$ and
	$i \in \Vpol \setminus \V_\Upsilon$.
	Denote the quotient space by
	\[
	\Dro_\Upsilon = \partial_\V \Yropol / \sim_\Upsilon
	\]
	Thus, the image of the Milnor fiber at radius $0$ in this quotient 
	space is homeomorphic to the Milnor fiber at radius $0$. Since 
	$\arg^\mathrm{ro}$ takes constant value along Petri dishes, we have a 
	well defined map
	\[
	\arg_\Upsilon^\mathrm{ro}:\Droup \mapsto \R / 2 \pi \Z 
	\]
	on the quotient space.
\end{block}

\begin{block} \label{block:inv_diff}
	We construct a $\Cinf$ structure on the space $\Dro_\Upsilon$ as 
	follows.
	First, we construct a $\Cinf$ structure on $\Dro_i / \sim_\Upsilon$,
	for $i \in \V_\Upsilon$.
	This space contains some number of copies of embedded $S^1$,
	which are the collapsed Petri dishes. More precisely,
	if $ik$ is an edge corresponding to an invariant vertex $i$ and a
	non-invariant vertex $k$, take a coordinate $v$ in a neighborhood $V 
	\subset
	D_i$ of $D_i \cap D_k$. Then $\sigma^*v$ takes constant value along
	$\Dro_{i,k}$. This induces a function $v^\mathrm{ro}: \sigma^{-1}(V) 
	\to \C$,
	which, along with $\arg^{\mathrm{ro}}(f)$ endows $\sigma^{-1}(V)$ with 
	a smooth
	structure.
	This way, $\Dro_i / \sim_\Upsilon$ is a compact $3$ dimensional manifold
	with boundary components which are tori, corresponding to invariant 
	edges.
	
	The space $\Dro_\Upsilon$ is obtained by gluing together the pieces
	$\Dro_i / \sim_\Upsilon$ along the boundary tori (via the identity map).
	Let $ij$ be an invariant
	edge with $j\to i$. We obtain coordinates near the embedded torus
	$\Dro_{ij}$ as follows. Take Grobman-Hartman coordinates $u,v$ in a
	chart $U$ around
	the intersection point $D_i \cap D_j$ as given by
	\Cref{lem:elementary_invariant} such that $\{u=0\} = D_i \cap U$
	and $\{v=0\} = D_j \cap U$.
	In particular, as in \cref{eq:GH_vf}, the vector field $\xi_U$ is
	given as
	\[
	\xi_U
	=
	d
	\left(
	\begin{matrix}
		c_{0,j} u \\
		-c_{0,i} v \\
	\end{matrix}
	\right)
	\]
	where $d > 0$ by \Cref{lem:d_pos}.
	This way, we get polar coordinates
	$u = re^{i\alpha}$ and $v = se^{i\beta}$ in $\Uro = \sigma^{-1}(U)$.
	The functions $\alpha, \beta:\Uro \to S^1$ restrict to functions
	in a neighborhood of $\Dro_{ij} = \Dro_i \cap \Dro_j$, giving two
	coordinates in $\Dro_\Upsilon \cap \Uro$.
	A third coordinate $t$ is constructed by setting
	\[
	t:\Dro_\Upsilon \to \R,\qquad
	t(q) = 
	\begin{cases}
		s(q) & q \in \Dro_i, \\
		-r(q) & q \in \Dro_j. \\
	\end{cases}
	\]
\end{block}

\begin{lemma}\label{lem:invariant_mf}
	
	The map $\arg_\Upsilon^\mathrm{ro}(f):\Droup \mapsto \R / 2 \pi \Z $, 
	induced
	by $\arg^\mathrm{ro}(f)$ is a $\Cinf$ locally trivial fibration with 
	fiber
	$\Froupte$, which is equivalent to  the Milnor fibration.
	
\end{lemma}
\begin{proof}
	The space $\Dro_\Upsilon$ is compact, and the function
	$\arg^{\mathrm{ro}}_\Upsilon(f)$, and its restriction to the boundary
	of $\Dro_\Upsilon$, is a submersion. Therefore, the lemma follows from
	Ehresmann's fibration theorem.
	
	Indeed, take any $q \in \Dro_\Upsilon$. If $q \in \Droc_i$ for some
	$i \in \V_\Upsilon$, then it is clear that 
	$\arg^{\mathrm{ro}}_\Upsilon(f)$
	is a submersion at $q$. If $q$ is the image of some Petri dish, then
	$\arg^{\mathrm{ro}}_\Upsilon(f)$ is a coordinate near $q$, in 
	particular,
	it is a submersion at $q$. If $q \in \Dro_{ij}$, where $ij$ is
	an invariant edge, then $\arg^{\mathrm{ro}}_\Upsilon(f) |_{\Dro_{ij}}$ 
	is
	already a submersion, as seen
	in \Cref{lem:transverse_strata},
	and the same argument applies to the boundary
	components of $\Dro_\Upsilon$.
\end{proof}

\begin{definition}\label{def:invariant_milnor}
	We call 
	\[
	\arg_\Upsilon^\mathrm{ro}(f):\Droup \mapsto \R / 2 \pi \Z
	\] 
	the
\index{Milnor fibration!invariant}
{\em invariant Milnor fibration}.
For each $\theta \in \R / 2 \pi 
	\Z$ we denote by 
	\[
	\Froupte = (\arg_\Upsilon^\mathrm{ro})(f)^{-1}(\theta)
	\] 
	the corresponding fiber and call it {\em the invariant Milnor fiber}.
\end{definition}

\begin{block}
	Next, we construct a vector field defined on $\Droup$.
	In $\Droc_i$,
	outside a small neighborhood of the boundary tori $\Dro_{ij}$, the 
	vector field coincides with $\xiro_i$.
	Let $ij$ be an invariant edge, directed as $j\to i$.
	Then, with a coordinate neighborhood $U$ with coordinates $u,v$
	as in the above construction, define a vector field
	\[
	\xiro_{U,\Upsilon} = (1,0,0)
	\]
	in the coordinates $(t,\alpha,\beta)$.
	Note that this vector field is parallel to $\xiro_i$ and $\xiro_j$ 
	outside
	$\Dro_{ij}$. Using partition of unity, these vector fields add up to a
	global vector field on all of $\Dro_\Upsilon$.
	This vector field is tangent to the Milnor fibers $\Froupte$ for
	$\theta \in \R/2\pi\Z$.
\end{block}

\begin{definition}
	Denote by $\xiroup$ the vector field on $\Droup$
	constructed in the previous paragraph, and by
	$\xiroupte$ its restriction to the invariant Milnor fiber
	$\Froupte$, for $\theta \in \R/2\pi\Z$.
\end{definition}

\begin{rem} \label{rem:trajs}
	By construction, the trajectories of $\xiroup$ pass transversely
	through the tori $\Dro_{ij}$. In fact, as a set, a trajectory of
	$\xiroup$ is a broken trajectory of $\xiro$ in the invariant part
	of $\partial \Yropol$.
\end{rem}

\begin{lemma} \label{lem:inv_potential}
	There exists a
	real function
\index{$\phiroup$}
$\phiroup: \Droup \to \R$ such that $\xiroup$ is
	gradient-like for it. That is,
	\[
	d\phiroup(\xiroup) > 0
	\]
	outside the singular points of $\xiroup$.
	Setting $\phiroupte = \phiroup|_{\Froupte}$, we also have
	\[
	d\phiroupte(\xiroupte) > 0
	\]
	outside the singular points of $\xiroupte$.
\end{lemma}
\begin{proof}
	For any $i \in \V_\Upsilon$, the vector field $\xi_i$ has a potential
	$\phi_i$, given in \Cref{def:grad_0,def:grad_i}.
	We set $\phiro_i = \sigma^* \phi_i$.
	If $i$ has distance $k$ to $0$ in $\Upsilon$, rescale
	$\phiro_i$ to take values in $[k,k+1]$.
	These functions glue
	together to the desired global function $\phiroupte$.
\end{proof}

\begin{rem}\label{rem:pronged_sings}
	Let $ij$ be and edge of $\Gapol$ with $j$ invariant and $i$ non-invariant
	(necessarily $j \to i$). Then, by \Cref{lem:noninv_corner} (see also
	\cref{fig:corner_sings}), the vector field $\xiroupte$ (by
	\Cref{lem:int_ind} \cref{it:int_ind_j}) has $\varpi_i m_j/
	\gcd(m_j,m_i)$-pronged singularities near the intersection points of
	$\Dro_{j,\theta}$ with $\sigma_\Upsilon^{-1}(D_i)$. See
	\cref{fig:petri_collapse}.
\end{rem}

\begin{block}
	We extend the equivalence relation $\sim_\Upsilon$ trivially over
	$\Yropol$, so that equivalence classes outside
	$\partial_\V \Yropol \subset \Yropol$ contain only one point.
	This way, the invariant Milnor fibration sits inside the space
	$\Yropol / \sim_\Upsilon$, which projects by $\piropol$ to $\C^2$
	isomorphically outside $C$.
\end{block}

\begin{definition}
	Denote by
	\[
	p_\Upsilon:\Yropol \to \Yropol/\sim_\Upsilon 
	\]
	the natural projection.
	Let $\Sroupte$ be the intersection of $\Froupte$ with the closure of the
	strict transform of the total spine, that is,
	\[
	\Sroupte
	= \Froupte \cap
	\overline{p_\Upsilon \left((\piropol)^{-1}(S\setminus\{0\})\right)}.
	\]
	We call $\Sroupte$ the
\index{spine!invariant}
\emph{invariant spine}. Define also
	the \emph{total invariant spine} as
	\[
	\Sroup = \bigcup_{\theta \in \R/2\pi\Z} \Sroupte.
	\]
\end{definition}

\begin{thm}\label{thm:invariant_spine}
	The invariant spine $\Sroupte$ satisfies the following:
	\begin{enumerate}
		\item \label{it:union}
		It coincides with the union of all trajectories of $\xiroupte$
		that do not escape to the boundary
		of the invariant Milnor fiber $\partial \Froupte$.
		\item \label{it:spine}
		It is a spine of the invariant Milnor fiber $\Froupte$.
		\item \label{it:smooth}
		It is a piecewise smooth $CW$-complex of dimension
		$1$ and all $1$-cells meet transversely at $0$-cells. 
	\end{enumerate} 
\end{thm}
\begin{proof}
	By \Cref{prop:petri_dishes}\cref{it:all_red_points} and
	\Cref{rem:trajs}, $\Sroupte$ contains all trajectories
	converging to singularities of the vector field $\xiroupte$.
	It follows from \Cref{thm:spine} that no other point in
	$\Froupte$ is contained in the invariant spine.
	This proves \cref{it:union}.
	
	For \cref{it:spine}, the flow of $\xiroupte$ trivializes the
	complement of the spine as a collar neighborhood of the boundary.
	\cref{it:smooth} follows from construction.
\end{proof}

\begin{figure}[h!]
	\centering
	\includegraphics*[scale=0.7]{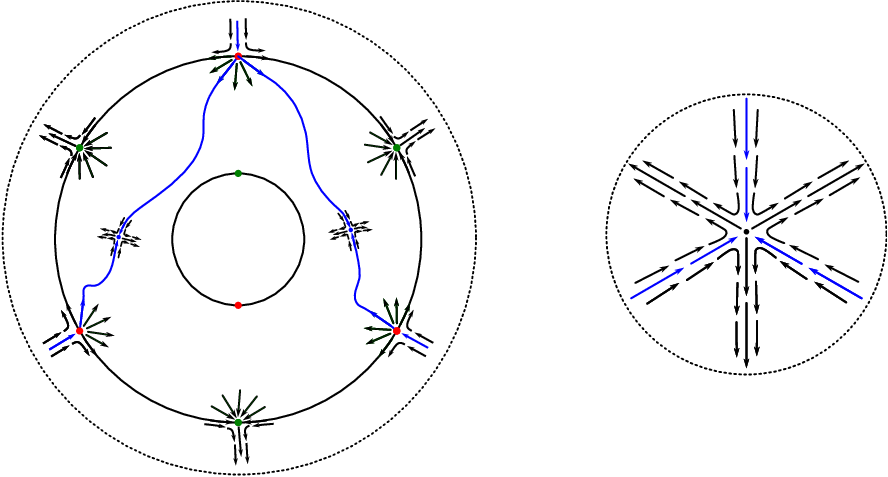}
	\caption{On the left we see a a small neighborhood (dotted disk) of a 
		Petri dish, in blue the broken trajectories that appear in this 
		neighborhood. On the right we see the same neighborhood after 
		collapsing the Petri disk to a point. The contracted Petri dish behaves 
		as a trivalent vertex}
	\label{fig:petri_collapse}
\end{figure}

%
%
%
	\chapter{Example}
\label{s:examples}
\begin{example}
	Consider the plane curve defined by the function 
	\[
	f(x,y)= y^3+x^4+x^3y.
	\]
	Since the curve is irreducible, it has only one tangent $\{y=0\}$. Also 
	note that we are already using the coordinates of 
	\Cref{not:other_exc_coord} and, in particular, the polar curve $P_i$ 
	associated with $i=1,2,3$ is the same, and it is the curve defined by 
	\[
	P= \{f_y=0\}=\{3y^2+x^3=0\}.
	\]
	
	Next we compute the minimal resolution $\pi:Y \to \Tu$ of the curve $f$ 
	and observe that the curve $P$ gets resolved by this resolution. See 
	\Cref{fig:34_res_graph}.
	\begin{figure}[h!]
		\centering
	\includegraphics*{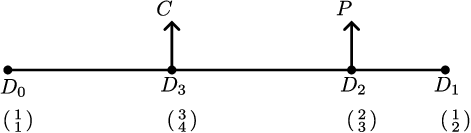}
		\caption{The dual resolution graph of the curve defined by $f$. The 
			blue arrow represents the strict transform $\tilde{P}$ of the polar 
			curve $\{f_y = 0\}$.}
		\label{fig:34_res_graph}
	\end{figure}
	
	Now we compute all the numerical invariants introduced in this work: 
	$c_{0,i}, c_{1,i}, m_i, p_i, \tau_i$ and $\varpi_i$ for $i=0,1,2,3$. 
	See \ref{tb:34}.
	
	\begin{equation}\label{tb:34}
		\begin{tabular}{|l||c|c|c|c|}
			\hline
			$i$ & 0 & 3 & 2 & 1  \\
			\hline
			\hline
			$c_{0,i}$ & 1 & 3 & 2 & 1 \\
			\hline
			$c_{1,i}$ & 1 & 4 & 3 & 2  \\
			\hline
			$m_i$ & 3 & 12 & 8 & 4 \\
			\hline
			$p_i$ & 2 & 8 & 6 & 3 \\
			\hline
			$\tau_i $ &2 & 8 & 5 & 3\\
			\hline
			$\varpi_i $ & 0 & 0 & 1 & 1 \\
			\hline
		\end{tabular}
	\end{equation}
	
	\subsection*{The topology of the Milnor fiber}
	
	Next, we compute the topology of the Milnor fiber. We do this by 
	looking at the
	multiplicities $m_i$ of $\pi^*f$: near $D_i^\circ$ the Milnor fiber is
	an
	$m_i$-fold cover of $D_i^\circ$. Moreover, the part of the Milnor fiber 
	that
	covers $D_i^\circ$ can be identified with $\Droc_{i,\theta}$ for each 
	$\theta
	\in \R / 2 \pi \Z$. Also, we know that $\Droc_{i,\theta}$ has \[ \sum_j
	\gcd(m_i,m_j) \] boundary components, where $j$ runs over vertices 
	adjacent to
	$i$. Using this information we can conclude that $\Droc_{0,\theta}$ is a
	disjoint union of $3$ disks, $\Droc_{3,\theta}$ is the surface of genus 
	$3$ and
	$8$ boundary components, $\Droc_{2,\theta}$ is a disjoint union of $4$ 
	annuli
	and $\Droc_{1,\theta}$ is a disjoint union of $4$ disks, see 
	\cref{fig:34_top}.
	With these invariants, we are ready to start computing the vector fields
	$\xi_i$ for $i \in \V$.
	
	\begin{figure}[h!]
		\centering
		\includegraphics*[scale=0.8]{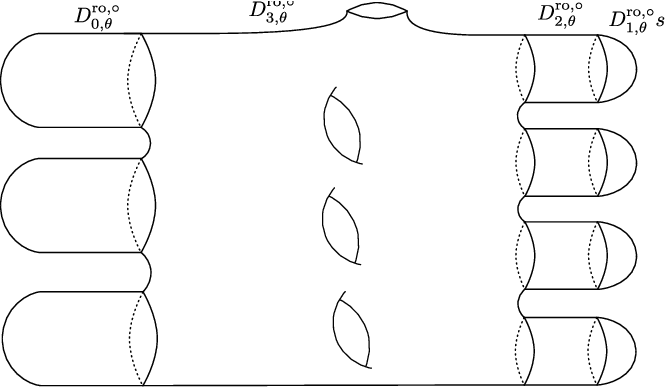}
		\caption{The Milnor fiber at radius $0$ at an angle $\theta$ 
			separated in the parts corresponding to each exceptional divisor. 
			The marked simple closed curves correspond with parts of the Milnor 
			fiber that lie in the corners of $\Yro$.}
		\label{fig:34_top}
	\end{figure}
	
	\subsection*{The 1st blow up}
	We follow \Cref{s:1st_blowup} for this part of the example. We are 
	going to check that the standard metric is already generic in the sense 
	of  (\Cref{def:genericity_morse}). Let $\pi_0:Y_0 \to \Tu$ be the first 
	blow up and take a coordinate chart $U \subset Y_0$ with coordinates 
	$u,v$ so that $\pi_0(u,v)=(u, uv)$ in these coordinates. In order to 
	check that the standard metric is generic, we need to verify that 
	\[
	\left. - \log \frac{|\pi_0^*f|}{(\pi_0^* d_{\id} 
		)^{3/2}}\right|_{D_0^\circ}
	\]
	is a Morse function. Equivalently, by \Cref{lem:xi_0_gradient}, we need 
	to verify that the vector field
	\[
	\left. \xi_0 = \pi_0^*(d \cdot \xi) \right|_{D_0^\circ}
	\]
	is elementary. A direct computation in the coordinates $u,v$ gives,
	(following \cref{eq:xi_coordinates_0}):
	\[
	\pi_0^*\xi=\left(
	\begin{matrix}
		\xi^u  \\
		\xi^v
	\end{matrix}
	\right)
	=
	\left(
	\begin{matrix}
		-\dfrac{3  \bar{v} + 4}{\bar{v}^{3} + \bar{u} \bar{v} + \bar{u}} 
		\\[10pt]
		\dfrac{3  v \bar{u} \bar{v} + 4  v \bar{u} - 3  \bar{v}^{2} - 
			\bar{u}}{u \bar{u} \bar{v}^{3} + u \bar{u}^{2} \bar{v} + u 
			\bar{u}^{2}}
	\end{matrix}
	\right).
	\]
	So, on $U\cap D_0^\circ$,
	\[
	\xi_0 = (|u|^2+|uv|^2)\xi^v|_{\{u=0\}} = -\frac{3 (1+|v|^2)}{\bar{v}}.
	\]
	Since $4 (1+|v|^2)$ is a unit on $U\cap D_0^\circ$ and $-1/\bar{v} = 
	-v/|v|^2$ is an elementary vector field on $U\cap D_0^\circ$ so we are 
	done in this chart. Actually, on $U\cap D_0^\circ$ the vector field 
	$\xi_0$ never vanishes. Now, we verify the other chart where 
	$\pi_0(u,v) = (uv, u)$. An analogous computations gives, 
	\[
	\xi_0 = 3 (1+|v|^2) v,
	\]
	which has the same dynamics as the vector field $v$ which is 
	elementary. This vector field has a single singularity which is a 
	fountain at the origin of this chart. The vector field $\xi_0$ is 
	defined on $D_0^\circ$, to compute the vector field $\xiro_{0,\theta}$ 
	for some $\theta \in \R / 2 \pi \Z$, we need to compute the real 
	oriented blow up. In this case, since $\varpi_0 =0$, the vector field 
	does not depend on $\arg(u)=\alpha$.
	
	\subsection*{The invariant part}
	Now we observe that the invariant graph $\Upsilon$ is the bamboo 
	induced by the vertices $0$ and $3$, equivalently, it is the smallest 
	connected subgraph containing $0$ and all vertices that intersect with 
	the strict transform of $C$ (recall \Cref{lem:only_non_invariant}). So 
	the only other vertex of $\Gamma$ with $\varpi_i=0$ is the vertex 
	$i=3$. In this case, by \Cref{lem:ext_others_tot}, the zero set of 
	$\xi_3$ should be the intersection points of $D_3$ with $\tilde{P}$, 
	but this intersection is empty. So the vector fields $\xi_3$ and 
	$\xiro_3$ are nowhere vanishing on $D_3^\circ$ and $\Droc_3$ 
	respectively.
	
	\subsection*{The non-invariant part}
	
	The non-invariant part of the graph is the graph induced by the vertices 
	$1$ and $2$. Let us take a chart $U \subset Y$ with coordinates $u,v$ 
	that contains most of these two divisors. For that, following the 
	weight vectors of \Cref{fig:34_res_graph} we take the chart $U$ in 
	which the the resolution map $\pi$ looks like:
	\[
	\pi(u,v) = (u^2v, u^3v^2).
	\]In this chart, $D^\circ_2$ is given by $\{u=0\}$, and $D^\circ_1$ by 
	$\{v=0\}$. We have
	\[
	\begin{split}
		f &= u^8 v^4 + u^9 v^6 + u^9 v^5 = u^8 v^4(1 + uv^2 + uv), \\
		f_x &= u^6 v^3 (4 + 3uv), \\
		f_y &= u^6 v^3 (3v + 1). \\
	\end{split}
	\]
	Recall here we are identifying its function with its pullback by $\pi$ 
	to ease the notation. Notice that the two partials have the same weight 
	w.r.t. $u$ and $v$, i.e.
	the weight vectors $(2,3)$ and $(1,2)$.
	
	We compute the Jacobian of the transformation and its inverse:
	\[
	\Jac \pi
	=
	\left(
	\begin{matrix}
		2uv   & u^2 \\
		3u^2v & 2u^3v
	\end{matrix}
	\right),
	\quad
	\left(\Jac \pi\right)^{-1}
	=
	\frac{1}{u^4 v^2}
	\left(
	\begin{matrix}
		2u^3v   & -u^2 \\
		-3u^2v & 2uv
	\end{matrix}
	\right).
	\]
	We find that
	\[
	\pi^*\xi
	=
	\left(
	\begin{matrix}
		-\dfrac{8  u v + 3  {\left(2  u v \bar{u} - 1\right)} \bar{v} - 
			1}{u^{2} v^{2} \bar{u}^{3} \bar{v}^{3} + u^{2} v^{2} \bar{u}^{3} 
			\bar{v}^{2} + u^{2} v^{2} \bar{u}^{2} \bar{v}} \\[10pt]
		\dfrac{12  u v + 3  {\left(3  u v \bar{u} - 2\right)} \bar{v} - 
			2}{u^{3} v \bar{u}^{3} \bar{v}^{3} + u^{3} v \bar{u}^{3} 
			\bar{v}^{2} + u^{3} v \bar{u}^{2} \bar{v}}
	\end{matrix}
	\right) 
	\]
	Note that this formula holds only outside the exceptional set, that is,
	for $(u,v) \in (\C^*)^2$. Following \Cref{lem:extension_corner} and the 
	table \ref{tb:34}, we use the function $|u|^5 |v|^3$ to rescale and 
	extend the vector field over the boundary of the real oriented blow-up. 
	In particular, to see the extension over $\Droc_2$, we take coordinates 
	$(r,\alpha,s ,\beta) \in (\R_{\geq 0} \times \R / 2 \pi \Z)^2$ and 
	compute the vector field over $\Droc_2$ which, in this coordinates is 
	defined by $\{r=0\}$. First we compute the pullback of $\xi^v$ rescaled 
	using the corresponding $\tau$ exponents:
	\begin{equation} \label{eq:xi}
		(r^5 s^3\sigma^*\xi^v)|_{\Droc_2}
		=
		-2 s{\left(3  \bar{v} + 1\right)} e^{-i  \alpha }
	\end{equation}
	Recall that on $\Droc_2$, the coordinate $s$ does not vanish. Now, 
	following \cref{eq:xiro_corner}, we compute the rescaled $\xiros$ and 
	$\xirobe$. We start by $\xiros$,
	\begin{equation}\label{eq:34_xiros}
		\begin{split}
			r^5 s^3 \xiros 
			& =  \Re(e^{-i\beta} \sigma^*\xi^v) \\
			& =   \Re\left(e^{-i\beta}\left(-2 s{\left(3  s e^{-i\beta} + 
				1\right)} e^{-i  \alpha }\right)\right) \\
			& = \Re\left(-6s^2e^{-i(2\beta +\alpha)} 
			-2se^{-i(\alpha+\beta)}\right) \\
			& = -6s^2 \cos(2\beta + \alpha) - 2s\cos (\alpha+\beta).
		\end{split}
	\end{equation}
	Similarly, we compute $\xirobe$:
	
	\begin{equation}\label{eq:34_xirob}
		\begin{split}
			r^5 s^3 \xirobe
			& =  \Im(s^{-1}e^{-i\beta} \sigma^*\xi^v) \\
			& =   \Im\left(s^{-1}e^{-i\beta}\left(-2 s{\left(3  s 
				e^{-i\beta} + 1\right)} e^{-i  \alpha }\right)\right) \\
			& = \Im\left( -6s 
			e^{-i(2\beta+\alpha)}-2e^{-i(\alpha+\beta)}\right) \\
			& = 6s \sin (2\beta+\alpha) + 2 \sin (\alpha+ \beta)
		\end{split}
	\end{equation}
	Now we fix $\theta$ in order to compute $\xiros_\theta$ and 
	$\xirobe_\theta$. That is, we fix $f^\mathrm{ro}=\frac{f}{|f|}=e^{i 
		\theta}$ so we get:
	\[
	\left.\left(\frac{f}{|f|}\right)\right|_{r=0}= e^{i 8 \alpha} e^{i 
		4\beta}  = e^{i\theta},
	\]
	that is
	\[
	8 \alpha + 4\beta = \theta  \mod 2\pi.
	\]
	For each $\theta \in \R/2 \pi \Z$, this equation has four circles of
	solutions in the torus $(\R/2 \pi \Z)^2$. These four
	circles are defined by the equations $\{2 \alpha + \beta = \theta/4 +
	k\pi/2\}$ with $k=0,1,2,3$. These four circles are exactly
	$\Dro_{3,2,\theta}$. Note that each of these circles admits a
	parametrization 
	\[
	\tilde{\beta} \mapsto \left( \frac{\theta}{8} + \frac{k \pi}{4} 
	-\frac{\tilde{\beta}}{2},\tilde{\beta}\right),
	\quad \tilde{\beta} \in [0, 4\pi)
	\]
	
	Moreover, we know that $\Droc_{2,\theta}$ consists of four cylinders. 
	Each of these cylinders has one of the previous four circles in its 
	closure in $\Yro_\theta$ and is defined by 
	\[
	\{r=0\} \cap \{2 \alpha + \beta = \theta/4 + k\pi/2\} \subset 
	\Yro_\theta.
	\]
	Therefore, we can consider coordinates, $s$ and $\beta$, with $s\in 
	\R_{\geq 0}$ and $\beta \in [0, 4 \pi)$. 
	
	Substituting the above expression for $\alpha$ in 
	\cref{eq:34_xiros,eq:34_xirob} we get the formula for 
	$\xiro_{2,\theta}$ in the coordinate patch $\Uro \cap \Droc_{2,\theta}$ 
	in the connected component where $\delta = 0$:
	\[
	\xiro_{2,\theta} 
	=
	\left(
	\begin{matrix}
		-6s^2 \cos\left(\frac{3\beta}{2} + \frac{\theta}{8} 
		+\frac{k\pi}{4}\right) - 
		2 s\cos \left(\frac{\beta}{2} + \frac{\theta}{8} 
		+\frac{k\pi}{4}\right) \\
		6s \sin \left(\frac{3\beta}{2} + \frac{\theta}{8} 
		+\frac{k\pi}{4}\right) 
		+ 2 \sin \left(\frac{\beta}{2} + \frac{\theta}{8} +\frac{k\pi}{4} 
		\right)
	\end{matrix}
	\right)
	\]
	In \cref{fig:34_div_12} we can see the cylinder corresponding to $k=0$ 
	and $\theta=0$ and in \cref{fig:34_div_23_special} we can see the 
	cylinder corresponding to $\theta=0$ and $k=2$. This last cylinder 
	represents a non-generic picture: the broken trajectory corresponding 
	to the stable manifold of one of the saddle points, goes through one of 
	the discs of $\Droc_{1,\theta}$, comes back again to $\Droc_{2,\theta}$ 
	and has the other saddle point in its closure.  Also, following 
	\Cref{lem:transversality_number} \cref{it:trans_iii} we get
	\[
	|\Sigma_{1,2}^+ \cap \Yro_\theta|= |\Sigma_{1,2}^- \cap \Yro_\theta|
	=
	-\left|
	\begin{matrix}
		4 & 1 \\
		8 & 1
	\end{matrix}
	\right|
	=4.
	\]
	Where $| \cdot |$ denotes cardinality in the first two terms and 
	determinant in 
	the third one.
	Since $\Dro_{1,2,\theta}$ has $4$ connected components, we conclude 
	that there is $4/4=1$ point of $\Sigma_{i,j}^+$ and one point of 
	$\Sigma_{i,j}^-$ on each connected component. These are the green and 
	red points in \cref{fig:34_div_23}.

	\begin{figure}[!ht]
		\centering
		\includegraphics*[scale=0.8]{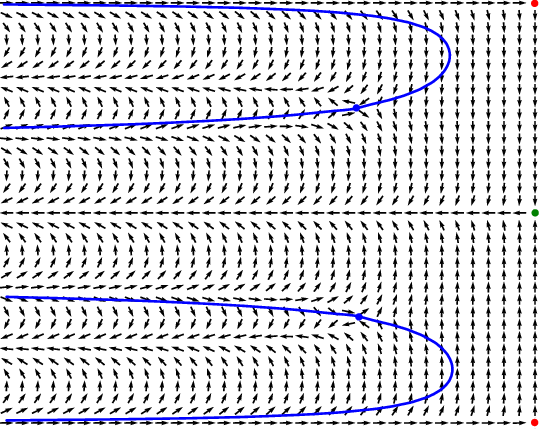}
		\caption{The vector field $\xiro_{2,\theta}$ on one of the four 
			connected components of $\Droc_{2,\theta}$ for $\theta=0$. The 
			right vertical line is one of the connected components of 
			$\Dro_{2,1,\theta}$. In blue the stable manifolds of the two saddle 
			points.}
		\label{fig:34_div_23}
	\end{figure}
	
	\begin{figure}[!ht]
		\centering
		\includegraphics*[scale=0.8]{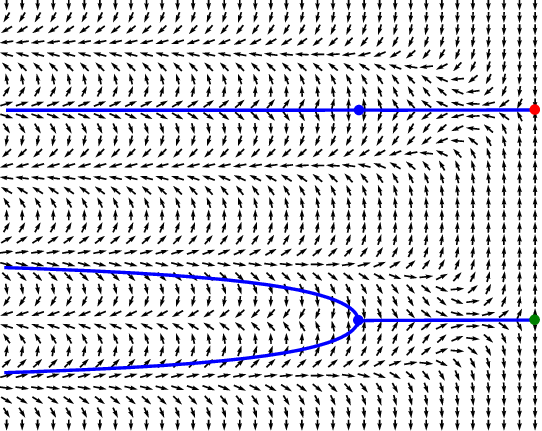}
		
		\caption{The vector field $\xiro_{2,\theta}$ on one of the four
			connected components of $\Droc_{2,\theta}$ for $\theta=4\pi$. 
			The left
			side of the box is one of the connected components of
			$\Dro_{3,2,\theta}$.}
		
		\label{fig:34_div_23_special}
	\end{figure}
	We  take now a chart near $D_{2}\cap D_3$. In this chart the resolution 
	takes the form
	\[
	\pi(u,v)= (u^3v^2,u^4v^3)
	\]
	and $D_2$ is defined by $\{v=0\}$. In this case, we get
	\[
	(r^8s^5 \sigma^*\xi^u)|_{\Droc_2} = r^2 \frac{2 {\left(\bar{u} + 
			3\right)} e^{-i  \beta}}{\overline{u}} =  2 r \left(\bar{u} + 
	3\right)e^{-i  \beta}e^{i\alpha}.
	\]
	As in \cref{eq:34_xiros,eq:34_xirob}, we compute
	\[
	\begin{split}
		r^8 s^5 \xiror &= 2r^2 \cos(\beta+\alpha) + 6 r \cos(\beta) \\
		r^8 s^5 \xiroal &= -2r\sin(\beta+\alpha)-6\sin(\beta).
	\end{split}
	\]
	Now we use the relation $12 \alpha + 8\beta = \theta + k\pi/2$,  and we 
	get
	\[
	\xiro_{2,\theta}
	=
	\left(
	\begin{matrix}
		2r^2 \cos \left(-\frac{\alpha}{2}+\frac{\theta}{8} + \frac{k 
			\pi}{16}\right) 
		+ 6 r \cos \left(\frac{\theta}{8} - \frac{3\alpha}{2} + \frac{k 
			\pi}{16}\right) \\[5pt]
		-2r\sin \left(-\frac{\alpha}{2}+\frac{\theta}{8}  \frac{k \pi}{16} 
		\right)
		-6\sin \left(\frac{\theta}{8} - \frac{3\alpha}{2}  \frac{k \pi}{16} 
		\right)
	\end{matrix}
	\right)
	\]
	in these coordinates. In \cref{fig:34_div_23_other} we can see the 
	cylinder corresponding to $k=0$ and $\theta=0$ and in 
	\cref{fig:34_div_23_other_special} we can see the cylinder 
	corresponding to the non-generic cylinder $\theta=0$ and $k=2$. In this 
	case, the computation from \Cref{lem:transversality_number} 
	\cref{it:trans_iii} yields
	\[
	|\Sigma_{2,3}^+ \cap \Yro_\theta|= |\Sigma_{2,3}^- \cap \Yro_\theta|
	=
	-\left|
	\begin{matrix}
		8 & 1 \\
		12 & 0
	\end{matrix}
	\right|
	=12.
	\]
	Which gives $12/4=3$ green points and three red points.
	\begin{figure}[!ht]
		\centering
		\includegraphics*[scale=0.7]{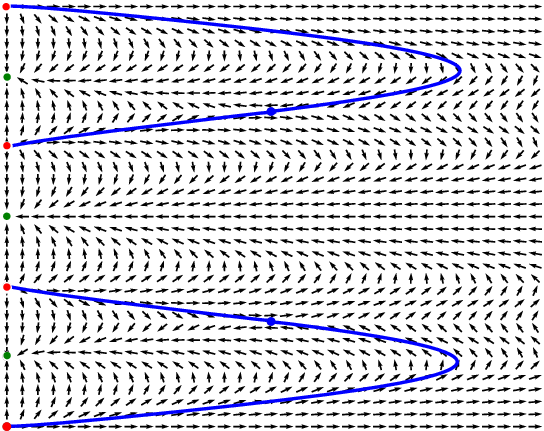}
		\caption{The vector field $\xiro_{2,\theta}$ on one of the four 
			connected components of $\Droc_{2,\theta}$ for $\theta=0$. The left 
			vertical line is one of the connected components of 
			$\Dro_{3,2,\theta}$.}
		\label{fig:34_div_23_other}
	\end{figure}
	
	\begin{figure}[!ht]
		\includegraphics*[scale=0.9]{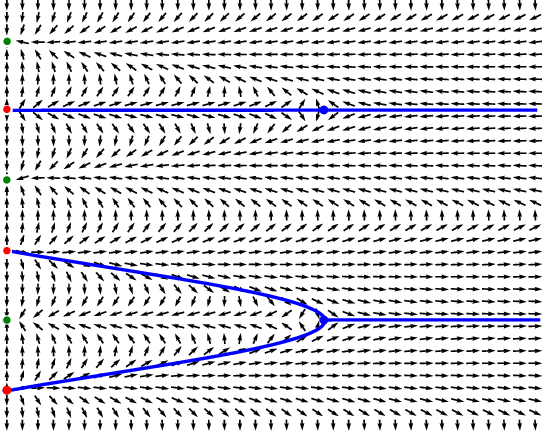}
		\caption{The vector field $\xiro_{2,\theta}$ on one of the four 
			connected components of $\Droc_{2,\theta}$ for $\theta=0$. The left 
			vertical line is one of the connected components of 
			$\Dro_{3,2,\theta}$.}
		\label{fig:34_div_23_other_special}
	\end{figure}
	Similarly, in order to finish the computations for the non-invariant 
	part of the graph, we can compute $\xiro_{1,\theta}$ which is defined 
	on the adjacent divisor $\Droc_{1,\theta}$. We obtain:
	\[
	(r^5 s^3\sigma^*\xi^u)|_{\Droc_1}=re^{-i  \beta}
	\]
	which, after a computation similar to that of 
	\cref{eq:34_xiros,eq:34_xirob}, yields:
	\[
	r^5 s^3 \xiror = r \cos(\beta+\alpha) \, \, r^5 s^3 \xiroal = 
	-\sin(\beta+\alpha).
	\]
	Using the relation $\beta = \theta/4 - 2\alpha$ we get:
	\[
	(\xiro_{1,\theta})^\top = \left( r 
	\cos(\theta/4-\alpha),-\sin(\theta/4-\alpha)\right). 
	\]

	\begin{figure}[h]
		\centering
		\includegraphics*[scale=0.9]{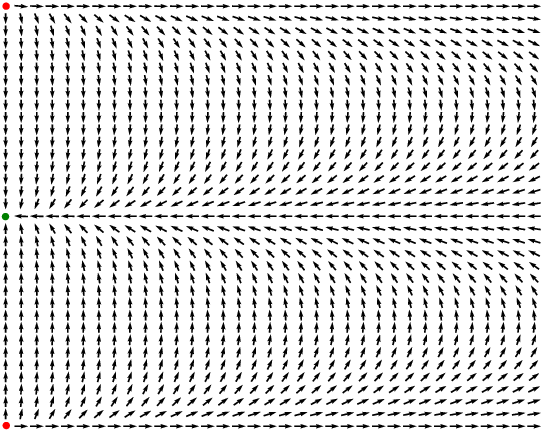}
		\caption{The vector field $\xiro_{1,\theta}$ on one of the four 
			connected components (disks) of $\Droc_{1,\theta}$ for $\theta=0$. 
			The left vertical line is one of the connected components (circles) 
			of $\Dro_{1,2,\theta}$. The two red points and the top and bottom 
			of the picture are identified. The horizontal line converging to 
			the green point actually forms part of the same segment as the 
			horizontal line emerging from the red point.}
		\label{fig:34_div_12}
	\end{figure}

	\begin{figure}[h]
		\centering
		\includegraphics*{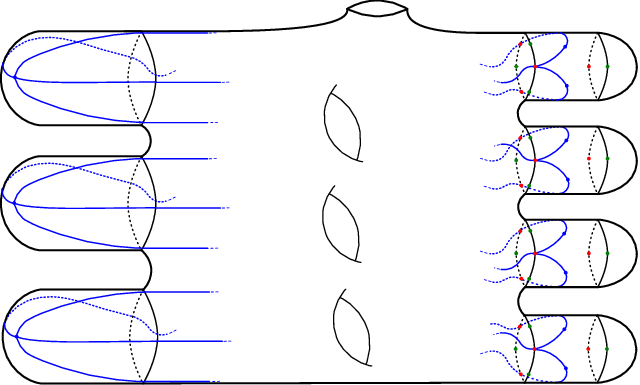}
		\caption{The Milnor fiber with the relevant parts of the spine (in 
			blue) drawn. This is the case for an angle $\theta \neq 0$.}
		\label{fig:mfiber34}
	\end{figure}
	
	\begin{figure}[h]
		\centering
		\includegraphics*{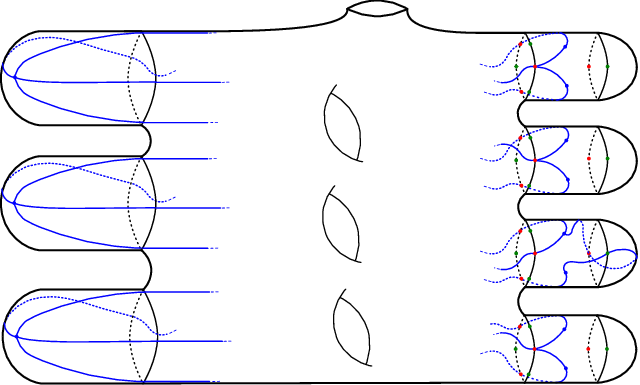}
		\caption{The Milnor fiber with the relevant parts of the spine (in 
			blue) drawn. This is the special angle $\theta = 0$.}
		\label{fig:34milnor_special}
	\end{figure}
	
	\subsection*{The invariant spine and invariant Milnor fiber}
	
	Following \Cref{s:invariant_spine}, we contract each Petri dish to a
	point. In this case there are four Petri dishes, each one corresponding
	to a connected component of the branch $\Xi$ at the vertex $2$. Each of
	these Petri dishes consists of a cylinder and a disk. After contracting
	them, we get four  $3$- pronged singularities (see
	\cref{fig:34milnor_invariant}). The invariant spine in this case is the
	$(3,4)$ bipartite graph.
	
	\begin{figure}[h]
		\centering
		\includegraphics*{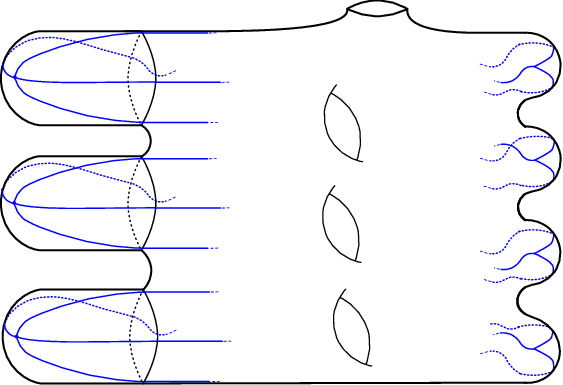}
		\caption{The invariant Milnor fiber with the invariant spine in 
			blue.}
		\label{fig:34milnor_invariant}
	\end{figure}
	
\end{example}

\backmatter

\printindex

\end{document}